%% file: GuideTopMfd.tex
\documentclass[12pt]{amsbook}

\usepackage{etoolbox}
\usepackage{latexsym}

\makeatletter
\def\@secnumfont{\mdseries}
\def\section{\@startsection{section}{1}%
 \z@{.7\linespacing\@plus\linespacing}{.5\linespacing}%
 {\normalfont\scshape\centering}}
\def\subsection{\@startsection{subsection}{2}%
 \z@{.5\linespacing\@plus.7\linespacing}{-.5em}%
 {\normalfont\bfseries}}

\patchcmd{\@thm}{\let\thm@indent\indent}{\let\thm@indent\noindent}{}{}
\patchcmd{\@thm}{\thm@headfont{\scshape}}{\thm@headfont{\bfseries}}{}{}

\makeatother

\DeclareFontFamily{U}{matha}{\hyphenchar\font45}
\DeclareFontShape{U}{matha}{m}{n}{
      <5> <6> <7> <8> <9> <10> gen * matha
      <10.95> matha10 <12> <14.4> <17.28> <20.74> <24.88> matha12
      }{}
\DeclareSymbolFont{matha}{U}{matha}{m}{n}
\DeclareFontSubstitution{U}{matha}{m}{n}
\DeclareMathSymbol{\acap}{2}{matha}{"58}
\DeclareMathSymbol{\acup}{2}{matha}{"59}

\DeclareMathOperator\cone{cone}
\DeclareMathOperator\Hom{Hom}

\usepackage{graphicx}
\usepackage{hyperref,color,mathdots,accents}
\usepackage{extarrows,enumitem,stmaryrd}

\usepackage{amsthm}
\usepackage{xfrac}
\usepackage{color}
\usepackage{enumitem}
\usepackage{rotating}
\usepackage[all]{xypic}
\usepackage[utf8]{inputenc}
\usepackage{amsmath,amsfonts,amssymb}
\usepackage[all,graph]{xy}
\usepackage{tikz-cd}
\usepackage{url}
\usepackage{tikz-cd}
\usepackage{relsize}

\usepackage{pgfplots}

\numberwithin{figure}{chapter}

\theoremstyle{plain}
\newtheorem{theorem}{Theorem}[chapter]
\newtheorem{lemma}[theorem]{Lemma}

\newtheorem{corollary}[theorem]{Corollary}
\newtheorem{proposition}[theorem]{Proposition}
\newtheorem*{proposition*}{Proposition}
\newtheorem*{theorem*}{Theorem}
\newtheorem{question}[theorem]{Question}
\newtheorem{conjecture}[theorem]{Conjecture}
\newtheorem{assumption}[theorem]{Assumption}

\theoremstyle{definition}
\newtheorem{definition}[theorem]{Definition}

\theoremstyle{remark}
\newtheorem{remark}[theorem]{Remark}
\newtheorem{notation}[theorem]{Notation}
\newtheorem{example}[theorem]{Example}
\newtheorem{construction}[theorem]{Construction}

\newtheorem*{claim}{Claim}
\makeatletter
\def\@tocline#1#2#3#4#5#6#7{\relax
  \ifnum #1>\c@tocdepth 
  \else
    \par \addpenalty\@secpenalty\addvspace{#2}%
    \begingroup \hyphenpenalty\@M
    \@ifempty{#4}{%
      \@tempdima\csname r@tocindent\number#1\endcsname\relax
    }{%
      \@tempdima#4\relax
    }%
    \parindent\z@ \leftskip#3\relax \advance\leftskip\@tempdima\relax
    \rightskip\@pnumwidth plus4em \parfillskip-\@pnumwidth
    #5\leavevmode\hskip-\@tempdima
      \ifcase #1
       \or\or \hskip 1em \or \hskip 2em \else \hskip 3em \fi%
      #6\nobreak\relax
    \dotfill\hbox to\@pnumwidth{\@tocpagenum{#7}}\par
    \nobreak
    \endgroup
  \fi}
\makeatother

\def\tf{\op{F}\!}

\newcommand{\map}[3]{\ensuremath{#1\colon#2\rightarrow#3}}
\newcommand{\Ch}[3]{\ensuremath{C_{#1}(#2;#3)}}

\newcommand{\Cech}[3]{\ensuremath{\check{H}^{#1}(#2;#3)}}
\newcommand{\coCh}[3]{\ensuremath{C^{#1}(#2;#3)}}

\newcommand{\SET}[1]{\left\{#1\right\}}
\newcommand{\Ztriv}{\ensuremath{\mathbb{Z}^{\operatorname{triv}} }}

\let\emptyset\varnothing
\newcommand{\inv}{\ensuremath{^{-1}}}

\def\int{\op{Int}}
\def\Q{\mathbb{Q}}

\def\Z{\mathbb{Z}}
\def\R{\mathbb{R}}
\def\C{\mathbb{C}}
\def\N{\mathbb{N}}

\def\O{\mathbb{O}}

\def\shortsms{\!\setminus\!}
\def\shortcup{\!\cup\!}

\def\toiso{\xrightarrow{\simeq}}

\makeatletter
\newcommand{\colim@}[2]{%
  \vtop{\m@th\ialign{##\cr
    \hfil$#1\operator@font colim$\hfil\cr
    \noalign{\nointerlineskip\kern1.5\ex@}#2\cr
    \noalign{\nointerlineskip\kern-\ex@}\cr}}%
}
\newcommand{\colim}{%
  \mathop{\mathpalette\colim@{\rightarrowfill@\scriptscriptstyle}}\nmlimits@
}

\newcommand{\smfrac}[2]{\mbox{\footnotesize$\displaystyle\frac{#1}{#2}$}}
\newcommand{\tmfrac}[2]{\mbox{\large$\frac{#1}{#2}$}}
\newcommand{\tsmfrac}[2]{\mbox{\scriptsize$\displaystyle\frac{#1}{#2}$}}

\newcommand{\tmsum}[2]{\mbox{$\textstyle \sum\limits_{#1}^{#2}$}}

\newcommand{\opnormal}[1]{\operatorname{\textnormal{#1}}}

\DeclareMathAlphabet{\mathbf}{OML}{cmm}{b}{it}

\newcommand{\B}{\operatorname{B}}
\newcommand{\BO}{\operatorname{BO}}
\renewcommand{\O}{\operatorname{O}}
\newcommand{\SO}{\operatorname{SO}}
\newcommand{\BSO}{\operatorname{BSO}}
\newcommand{\Spin}{\operatorname{Spin}}
\newcommand{\BSpin}{\operatorname{BSpin}}
\newcommand{\BPL}{\operatorname{BPL}}
\newcommand{\PL}{\operatorname{PL}}
\newcommand{\BTOP}{\operatorname{BTOP}}
\newcommand{\TOP}{\operatorname{TOP}}
\newcommand{\CAT}{\operatorname{CAT}}
\newcommand{\BCAT}{\operatorname{BCAT}}
\newcommand{\Diff}{\operatorname{Diff}}
\newcommand{\STOP}{\operatorname{STOP}}
\newcommand{\BSTOP}{\operatorname{BSTOP}}
\newcommand{\TOPSpin}{\operatorname{TOPSpin}}
\newcommand{\BTOPSpin}{\operatorname{BTOPSpin}}
\newcommand{\ograph}{\operatorname{graph}}
\newcommand{\PD}{\operatorname{PD}}
\newcommand{\Sq}{\operatorname{Sq}}
\newcommand{\Emb}{\operatorname{Emb}}
\newcommand{\Homeo}{\operatorname{Homeo}}

\def\bnm{\begin{enumerate}[leftmargin=1cm,font=\normalfont]}
\def\bnmt{\begin{enumerate}[leftmargin=0.8cm,font=\normalfont]}
\def\enm{\end{enumerate}}
\def\ba{\begin{array}}
\def\ea{\end{array}}
\def\bpp{\begin{pmatrix}}
\def\epp{\end{pmatrix}}

\def\ext{\operatorname{Ext}}
\def\coker{\operatorname{coker}}
\def\ord{\op{ord}}
\def\ker{\operatorname{ker}}
\def\im{\operatorname{Im}}
\def\hom{\operatorname{Hom}}
\def\sms{\setminus}

\def\op{\operatorname}
\def\up{\textup}
\def\ub{\underbrace}
\def\us{\underset}
\def\wti{\widetilde}

\def\ol{\overline}
\def\id{\operatorname{Id}}
\def\Id{\operatorname{Id}}

\def\rp{\R\textup{P}}
\def\cp{\C\textup{P}}
\def\CP{\C\textup{P}}
\def\sm{\setminus}
\def\wt{\widetilde}

\def\link{\operatorname{Lk}}
\def\pt{\operatorname{pt}}

\def\PL{\operatorname{PL}}

\def\Int{\operatorname{Int}}

\DeclareMathOperator\Ext{Ext}
\DeclareMathOperator\GL{GL}
\DeclareMathOperator{\ex}{ex}
\DeclareMathOperator{\ks}{ks}
\DeclareMathOperator{\pr}{pr}

\numberwithin{equation}{chapter}

\numberwithin{section}{chapter}

\numberwithin{subsection}{section}

\begin{document}
\title[The foundations of $4$-manifold theory]{A survey of the foundations of four-manifold theory in the topological category}

\author[S. Friedl]{Stefan Friedl}
\address{Fakult\"at f\"ur Mathematik\\ Universit\"at Regensburg\\   Germany}
\email{sfriedl@gmail.com}

\author[M. Nagel]{Matthias Nagel}
\email{matthinagel@gmail.com}

\author[P.~Orson]{Patrick Orson}
\address{California Polytechnic State University, San Luis Obispo, CA, USA}
\email{patrickorson@gmail.com}

\author[M. Powell]{Mark Powell}
\address{School of Mathematics and Statistics\\ University of Glasgow\\   UK}
\email{mark.powell@glasgow.ac.uk}

\maketitle

\begin{abstract}
This survey aims to provide a guide to the literature on topological 4-manifolds. 
Foundational theorems on 4-manifolds are stated, especially in the topological category.  Precise references are given, with indications of the strategies employed in the proofs. Where appropriate we give statements for manifolds of all dimensions.

Many intuitively plausible theorems which are standard results in differential topology are either extraordinarily deep results in the topological category, are open, or are known to be false.  Hence one must proceed with caution. This book seeks to help 4-manifold topologists navigate potential pitfalls, and to apply the many powerful results that do exist with confidence.
\end{abstract}

\chapter{Introduction}

Our aim in this survey is to provide readers who have trained in algebraic topology and perhaps differential topology or Riemannian geometry, with a guide to the literature on topological manifolds.
We aim to state some foundational theorems in topological manifolds, with a strong bias towards dimension four. If we were considering manifolds with a smooth atlas, many of these statements would be familiar textbook-level tools. In the topological category, navigating which of these tools can still be used, and where to find proofs of these facts, can be a challenging endeavour.
On the other hand, particularly in dimension four, some of the results we describe are not familiar results in the smooth category and often no result of the sort holds for smooth manifolds. This second type of statement demonstrates one of the attractions of working with 4-manifolds in the topological category, where major classification statements can  be achieved.

Our hope is that with the statements from this book the ``working topologist'' will be equipped to handle most situations.
We make no claims of originality.

\section{High-dimensional topological manifolds}
Though we have a strong bias in this survey towards thinking about 4-dimensional manifolds, in order to do this in the topological category one must have a good understanding of manifold topology in higher dimensions.

Spectacular results on the classification of smooth manifolds in dimensions $\geq 5$ arose from Surgery Theory, due to Smale~\cite{Smale:1962-1}, Kervaire--Milnor~\cite{KM1963}, Browder~\cite{Browder}, Novikov~\cite{Novikov}, Sullivan~\cite{Sullivan}, and Wall~\cite{Wall-Ranicki:1999-1}, among others. These methods were extended to topological manifolds by Newman~\cite{Newman}, Kirby~\cite{Ki69}, and Siebenmann~\cite{KS77} in dimensions at least five, and to dimension four by Freedman and Quinn~\cite{Freedman-82,Qu82,FQ90}. To paraphrase Andrew Ranicki,

\begin{quote}
    \emph{Smooth manifolds in dimensions at least five exhibit a beautiful correspondence between geometry and algebra. Topological manifolds in dimension at least four are in this sense like smooth high-dimensional manifolds,} but even more so.
\end{quote}

To what does ``even more so'' refer?
One instance is the principle that topological manifolds are governed by their homotopy type which, in turn, is often governed by algebraic invariants.
Many specific instances of this principle hold uniformly across all dimensions for topological manifolds and we highlight a few of these now to demonstrate the point.
 Most famous is the (topological) Poincar\'{e} conjecture, which characterises the sphere~$S^n$, up to homeomorphism, in terms of algebraic topological invariants. This result is known to be true in all dimensions, due to Newman, Freedman, and Perelman~\cite{Newman66,Freedman-82,Morgan-Tian}.
Locally flat topological embeddings of spheres in spheres are also well understood in these terms. In codimension zero, every orientation preserving homeomorphism of $S^n$ is isotopic to the identity, due to Fisher, Kirby, and Quinn~\cite{Fisher, Ki69,Qu82}. For codimension 1, we have the Schoenflies conjecture, that every locally flat embedding $S^{n-1} \subseteq S^n$ is trivial. This is true in all dimensions, and due to Brown, Mazur, and Morse~\cite{Brown60,Mazur-schoenflies,Morse-schoenflies}. A codimension two locally flat embedding $S^{n-2} \subseteq S^n$ is topologically unknotted if and only if the complement is homotopy equivalent to $S^1$, due to Papakyriakopoulos~\cite{MR0087944}, Stallings~\cite{Stallings-unknotting}, and Freedman-Quinn~\cite[Theorem 11.7A]{FQ90}. Finally for high codimension, every knot $S^{n-k} \subseteq S^n$ is trivial, for $k \geq 3$, which is again due to Stallings~\cite{Stallings-unknotting}, cf.~\cite{Zeeman-unknotting}.

For manifold classification results, there is also an especially close correspondence between topology and algebra in the topological category in dimensions at least four.
Wall~\cite{Wall62}, Freedman, and Quinn~\cite{FQ90} proved that topological manifolds of dimension $2n$ that are $(n-1)$-connected are classified up to homeomorphism by their intersection forms on the $n$th homology groups, together with a quadratic extension.
The Borel conjecture holds in many cases~\cite{AFW15,Lueck2010}. In particular for every $n \geq 1$, every homotopy equivalence $M^n \to T^n = (S^1)^n$ is homotopic to a homeomorphism, see~\cite[p.~205]{FQ90} and~\cite{Lueck2010}.
Finally, an important example, though somewhat more specialised, is that the Surgery Exact Sequence becomes an exact sequence of abelian groups in the topological category, due to Quinn~\cite{Quinn-geometric} and Nicas~\cite{Nicas-thesis}, with a purely algebraic formulation due to Ranicki~\cite{Ranicki-blue-book}.

These examples demonstrate a tight connection between the study of topological manifolds and the associated homotopy types, and related algebra. In dimension four, many of the significant classification results we state in this survey will also hew to this principle.

\section{A focus on dimension four}
In the Ranicki quotation paraphrased above, the only dimension not common to both ranges is dimension four. Indeed, most of the deep classification results that are known about 4-manifolds are only possible in the topological category.
Studying topological 4-manifolds combines the visual nature of low-dimensional topology, with the ability to apply powerful high-dimensional methods to classification problems.

An additional special feature of dimension four is that here the contrast between the smooth and topological manifold categories is starkest. Indeed, the indications are that smooth 4-manifolds fail in every way imaginable to exhibit a close correspondence to their underlying homotopy type. Although in this survey we focus very heavily on topological 4-manifolds, we will hint at this dramatic divergence in some places, for example in Chapter~\ref{chapter:intersection-form}.

The geometric topologist who wishes to study topological 4-manifolds, with all the rich rewards suggested above, must navigate the challenge that even basic results from differential topology are either false, unknown, or extraordinarily deep results in the absence of a smooth atlas. For example it is not true in general that topological submanifolds admit tubular neighbourhoods, and in cases where this is known, such as for codimension two submanifolds, the proofs use all the available technology developed by Kirby-Siebenmann and Freedman-Quinn. As another example, it was not known until Quinn's work in 1982 that connected sum of topological 4-manifolds is a well-defined operation.

\section{What is in this survey?}

This brings us, at last, to one of the main  practical purposes of this survey.
We seek to clarify exactly which of the familiar  tools of geometric topology are available in the topological category, and to provide a precise guide for where to find proofs. Below, we give a sample of the statements discussed in this book. (Here, and throughout the book, ``manifold'' refers to what is often called a ``topological manifold''; see Chapter~\ref{ch:manifold} for a precise definition.)


\begin{enumerate}[leftmargin=1cm]
\item Existence and uniqueness of collar neighbourhoods (Theorem~\ref{thm:collar}).
\item The Isotopy Extension Theorem~\ref{thm:isotopy-extension-theorem}.
\item Existence of CW structures (Theorem~\ref{thm:topological-manifold-CW complex}).
\item Multiplicativity of the Euler characteristic under finite covers (Corollary~\ref{cor:topological-mfd-homology}).
\item The Annulus Theorem~\ref{thm:annulus} and the Stable Homeomorphism Theorem~\ref{thm:SHT}.
\item Connected sum of two oriented connected 4-manifolds is well-defined (Theorem~\ref{thm:connected-sum-well-defined}).
\item Existence and uniqueness of tubular neighbourhoods of submanifolds (Theorems~\ref{thm:tubular-neighbourhood} and~\ref{thm:uniqueness-of-general-tubular-neighbourhoods}).
\item Stiefel-Whitney classes for topological manifolds (Chapter~\ref{chapter:SW-classes}).
\item Intersection forms of compact, connected, oriented 4-manifolds are even
(Proposition~\ref{proposition-int-form-spin-even}).
\item Noncompact connected 4-manifolds admit a smooth structure (Theorem~\ref{thm:smooth-outside-a-point}).
\item When the Kirby-Siebenmann invariant of a connected 4-manifold vanishes, both connected sum with copies of $S^2 \times S^2$ and taking the product with $\R$ yield smoothable manifolds (Theorem~\ref{thm:connect-sum-is-smooth}).
\item Transversality for submanifolds and for maps (Theorems~\ref{thm:TransvSubmanifolds} and~\ref{thm:TransveralityMaps}).
\item Codimension one and two homology classes can be represented by submanifolds (Theorem~\ref{thm:represent-homology-by-submanifolds}).
\item Classification of 4-manifolds up to homeomorphism with trivial and cyclic fundamental groups (Chapter~\ref{chapter:classification-simply-conn-4-mflds}).
\item Compact orientable manifolds that are homeomorphic are stably diffeomorphic (Theorem~\ref{thm:stably-diffeo-homeo-4-mflds} and Corollary~\ref{cor:stably-diffeomorphic-nonorientable-manifolds}).
\item Multiplicativity of signatures under finite covers (Theorem~\ref{thm:multiplicativity-of-signature}).
\item The definition of Reidemeister torsion for compact manifolds and some of its key technical properties (Section~\ref{section:reidemeister-torsion}).
\item Obstructions to concordance of knots and links (Theorem~\ref{theorem:fox-milnor}).
\item Poincar\'e duality for compact manifolds (with possibly non-empty boundary) with twisted coefficients (Theorems~\ref{thm:poincareduality} and~\ref{thm:poincareduality-chain-complex}).
\end{enumerate}

\begin{remark}
    We make a remark here about a notable \emph{omission} from this survey. We have not written in any great detail about non-compact $4$-manifolds. This is itself a rich topic, exhibiting some of the most exciting and distinctive features of dimension 4. However, it was decided to be beyond the scope of the current survey.
\end{remark}

\subsection*{Conventions.}
\bnm
\item Given a subset $A$ of a topological space $X$ we denote the interior by $\int A$.
\item  For $n\in \N_0$, $x\in \R^n$ and $r\geq 0$ we write $D^n_r(x)=\{ y\in \R^n \mid  \|y-x\|\leq r\}$. We write $D^n_r=D^n_r(0)$ and we write $D^n=D^n_1(0)$ for the closed unit ball in $\R^n$.
We refer to $\int D^n=\{ x\in \R^n \mid  \|x\|<1\}$ as the open $n$-ball.
\item Unless indicated otherwise $I$ denotes the interval $I=[0,1]$.
\item All maps between topological spaces are understood to be continuous.
\item A topological space $X$ is called \emph{simply connected} if it is nonempty, path-connected and if the fundamental group is trivial.
\item On several occasions we use cup and cap products and we cite several results from~\cite{Dold95,Br93,Hat02,Fr23}. Different books on algebraic topology often have different sign conventions for cup and cap products, but in all statements that we give, the sign conventions are irrelevant, so it is not a problem to mix results from different sources.
\enm

\subsection*{Acknowledgments.}
Thanks to Steve Boyer, Anthony Conway, Jim Davis, Mauricio G\'{o}mez L\'{o}pez, Fabian Hebestreit, Jonathan Hillman, Min Hoon Kim, Alexander Kupers, Markus Land, Tye Lidman, Chuck Livingston, Ciprian Manolescu, Erik Pedersen, George Raptis, Arunima Ray, and Eamonn Tweedy for very helpful conversations.

We are particularly grateful to Gerrit Herrmann for providing us with detailed notes which form the basis of  Appendix~\ref{appendix:poincare-duality}.

Jim Davis brought the Ranicki quotation above to our attention, and discussed its significance with us. He also provided crucial assistance with the proof of Theorem~\ref{thm:ks-basics}.

We wish to thank the referee for an extraordinarily thorough and helpful report.

SF was supported by the SFB 1085 `Higher Invariants' at the University of Regensburg, funded by the Deutsche Forschungsgemeinschaft (DFG).
MN gratefully acknowledges support by the SNSF Grant~181199.
MP was supported by an NSERC Discovery Grant, EPSRC New Investigator grant EP/T028335/2 and EPSRC New Horizons grant EP/V04821X/2. SF  wishes to thank the Universit\'e du Qu\'ebec \`a Montr\'eal, Durham University, Glasgow University, and Cal Poly for hospitality.

\tableofcontents
\chapter{Manifolds}\label{ch:manifold}
In this chapter we introduce the very basic notions of a manifold, of a submanifold,
of locally flat embeddings and of immersions.
We also state two fundamental results, namely the
Collar Neighbourhood Theorem~\ref{thm:collar}
and the Isotopy Extension Theorem~\ref{thm:isotopy-extension-theorem}.

\section{Definition of manifolds}
In the literature  the notion of a ``manifold'' gets defined differently, depending on the preferences of the authors. Thus we state here what we mean by a manifold.

\begin{definition}\label{def:manifold}
Let~$X$ be a topological space.
\begin{enumerate}[leftmargin=1cm,font=\normalfont]
\item We say that~$X$ is \emph{second countable} if there exists a countable basis for the topology.
\item An \emph{$n$-dimensional chart for~$X$ at a point $x\in X$} is a homeomorphism $\Phi\colon U\to V$ where~$U$ is an open neighbourhood of $x$ and \label{def:charts}
\begin{enumerate}[leftmargin=0.8cm]
\item[(i)]~$V$ is an open subset of $\R^n$ or
\item[(ii)]~$V$ is an open subset of the half-space $H_n=\{ (x_1,\dots,x_n)\in \R^n \mid  x_n\geq 0\}$ and $\Phi(x)$ lies on $E_{n-1}=\{ (x_1,\dots,x_n)\in \R^n \mid  x_n= 0\}$.
\end{enumerate}
In the former case we say that $\Phi$ is a \emph{chart of type (i)}; in the latter case we say that $\Phi$ is a \emph{chart of type (ii)}.
\item We say that~$X$ is an \emph{$n$-dimensional manifold} if~$X$ is second countable and Hausdorff, and if for every $x\in X$ there exists an $n$-dimensional chart $\Phi\colon U\to V$ at~$x$.
\item We say that a point $x$ on a manifold\label{def:boundary-point} is a \emph{boundary point} if $x$ admits a chart of type (ii). (A point cannot admit charts of both types~\cite[Theorem~2B.3]{Hat02}.) We denote  the set of all boundary points of~$X$ by $\partial X$.
\item An \emph{atlas} for a manifold $X$ consists of a family of charts such that the domains cover all of $X$. An atlas $\{\Phi_i\colon U_i\to V_i\}_{i\in I}$  is \emph{smooth} if all transition maps $\Phi_i\circ \Phi_j^{-1}\colon \Phi_j(U_i\cap U_j)\to \Phi_i(U_i\cap U_j)$  are smooth. A \emph{smooth manifold} is a manifold together with a smooth atlas. Usually one suppresses the choice of a smooth atlas from the notation.
\end{enumerate}
\end{definition}

To avoid misunderstandings we want to stress once again that what we call a ``manifold'' is often referred to as a ``topological manifold''.

\begin{definition}
An \emph{orientation} of an $n$-manifold $M$ is a choice of generators $\alpha_x\in H_n(M,M\sms \{x\};\Z)$ for each $x\in M\sms \partial M$ such that for every $x\in M\sms \partial M$ there exists an open neighbourhood $U\subseteq M\sms \partial M$ of $x$ and a class $\beta\in H_n(M,M\sms U;\Z)$ such that  $\beta$ projects to $\alpha_y$ for each $y\in U$.
\end{definition}

Using the cross product one can show that the product of two oriented manifolds admits a natural orientation. Furthermore, the boundary of an oriented manifold also comes with a natural orientation.
The proof of the latter statement is slightly delicate; we refer to \cite[Chapter~28]{Greenberg-Harper} or to \cite[Chapter~125.5]{Fr23} for details.

The following theorem \cite[Theorem~2B.3]{Hat02} is one of the foundational results on $n$-manifolds.

\begin{theorem}\textbf{\textup{(Invariance of Domain Theorem)}}\label{thm:invariance-of-domain}
If $U\subseteq \R^n$ is an open subset and if $h\colon U\to \R^n$ is an injective map,
then $h(U)$ is an open subset of $\R^n$.
\end{theorem}

We conclude this section with the following lemma.

\begin{lemma}\label{lem:manifolds-transitive}
	Let $M$ be a connected manifold of dimension $n\geq 2$. Then for any two
 sets of pairwise disjoint points $\{x_1,\dots,x_m\},\{y_1,\dots,y_m\}\in M\setminus \partial M$, there exists a homeomorphism $f\colon M\to M$ with $f(x_i)=y_i$ for $i=1,\dots,m$.
\end{lemma}

\begin{proof}
	Since $M$ is  connected we  see that $M\setminus \partial M$ is  path-connected.
Thus there exist points $x_1 = t_0, t_1, \ldots, t_{k+1} = y_1$  in $M\sms \partial M$ such that there are charts $(U_i, \psi_i\colon U_i\to \Int D^n)$ for $i=0, \ldots, k$ and both $t_i$, $t_{i+1}$ are contained in $U_i$.
It is elementary to show that given any two points $a,b\in \Int D^n$ there exists a homeomorphism $f\colon \Int D^n\to \Int D^n$ with $f(a)=b$ and which is the identity outside of a compact subset. It is now clear that one can find a homeomorphism $f\colon M\to M$ with $f(x_1)=y_1$ such that $f$ is the identity outside of a compact subset.
We now consider the image of the remaining points in $M\sms \{y_1\}$ and we restart the engine.
\end{proof}

\section{Definition of submanifolds}
We move on to the definition of a submanifold. Again there are many different definitions in the literature, so let us define carefully what we mean by a submanifold.

\begin{definition}\label{defn:ProperSubmanifold}
Let $M$ be an $n$-dimensional manifold. We say a subset $X\subseteq M$ is a \emph{$k$-dimensional submanifold} if given any $x\in X$ one of the following holds:\label{def:submanifold}
\bnm
\item[($\alpha$)] there exists a chart $\Phi\colon U\to V$ of type (i) for $M$ and $x$ such that
\[ \Phi(U\cap X)\,\,\subseteq \,\,\{(0,\dots,0,x_1,\dots,x_k) \mid x_1,\dots,x_k\in \R\},\]
\item[($\beta$)]   there exists a chart $\Phi\colon U\to V$ of type (ii) for $M$  and $x$ such that $\Phi(x)$ lies in $E_{n-1}$ and  \[ \hspace{1.15cm} \Phi(U\cap X)\,\,\subseteq\,\, \{(0,\dots,0,x_1,\dots,x_k)\in \R^n \mid x_k\geq 0\},\]
\item[($\gamma$)]   there exists a chart $\Phi\colon U\to V$ of type (i) for $M$  and $x$ such that $\Phi(x)$ lies in $E_{n-1}$ and  \[ \hspace{1.15cm} \Phi(U\cap X)\,\,\subseteq\,\, \{(0,\dots,0,x_1,\dots,x_k)\in \R^n \mid x_k\geq 0\}.\]
\enm
If for every $x\in X$ we can find charts as in ($\alpha$) and ($\beta$), and $X\subseteq M$ is a closed subset, then we call $X$ a \emph{proper} submanifold.
\end{definition}

\begin{figure}[h]
\begin{center}
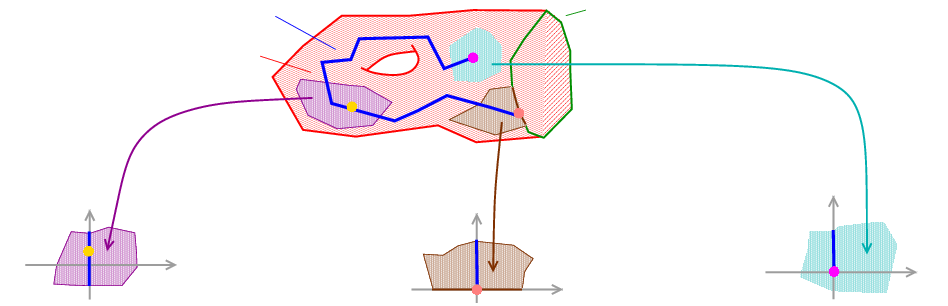
\caption{Definition of submanifolds.}
\label{fig:submanifolds}
\end{center}
\end{figure}

The following proposition is a straightforward consequence of the definitions. The proposition often makes it possible to reduce arguments about manifolds with boundary to the case of closed manifolds.

\begin{proposition}\label{lem:double-of-topological-manifold}~
\bnm
\item
	Let $N$ be an $n$-manifold, possibly disconnected. Let $A$ and $B$ be collections of components of $\partial N$ such that $A\cap B=\emptyset$. Let $f\colon A\to B$ be a homeomorphism. Then the quotient $N/\hspace{-0.05cm}\sim$ under the relation $a\sim f(a)$ is an $n$-manifold with boundary $\partial (N/\hspace{-0.05cm}\sim)=\partial N \sms (A\cup B)$. Moreover, the image of $A$ in $N/\hspace{-0.05cm}\sim$ is a submanifold.
	\item
Let $M$ be an $n$-manifold. Its double $DM:=M\cup M$, where the boundaries are identified via $\id_{\partial M}$, is an $n$-manifold with empty boundary. Moreover, $M\subseteq DM$ is a submanifold.
\enm
\end{proposition}

\begin{definition}\label{def:LocallyFlat}
A map $f\colon X\to M$ from a $k$-manifold to an $m$-manifold $M$ is  called a (proper) \emph{locally flat embedding} if $f$ is a homeomorphism onto its image and if the image is a (proper) submanifold of $M$.
\end{definition}

\begin{remark}\leavevmode
\bnm
\item Note that if $M$ is a $k$-manifold and $U$ is an open subset of $\R^k$, then it follows from the Invariance of Domain Theorem~\ref{thm:invariance-of-domain} that the image of any injective map $f\colon U\to M$ is an open subset of $M$. In particular $f(U)$ is a submanifold of $M$. Put differently, $f$ is locally flat.
\item
In point set topology, one often defines a \emph{topological embedding} to be a
map~$f \colon X\to Y$ of topological spaces that is a homeomorphism to its image.
The image of a topological embedding is not necessarily a submanifold and such an image is sometimes called  \emph{wild} due to the bizarre properties that such objects can exhibit.  For example, the famous Alexander horned sphere~\cite{Alexander24} is not a submanifold of $S^3$ under Definition~\ref{def:submanifold}, but it is the image of a wild topological embedding $S^2 \to S^3$.
\item In the literature a compact subset $F$ of 4-manifold is often called a \emph{locally flat surface}
if $F$ is homeomorphic to a compact 2-dimensional manifold with $\partial F=F\cap \partial M$ and if $F$ has the following  properties.
\bnmt
\item Given any $x\in F\sms \partial F$ there exists a topological embedding
$\varphi\colon D^2\times D^2\to M\sms \partial M$ with $\varphi(D^2\times D^2)\cap F=\varphi(D^2\times \{0\})$ and with $x\in \varphi(D^2\times \{0\})$.
\item Given any $x\in \partial F$ there exists a topological embedding
	$\varphi\colon D^2_{\geq 0}\times D^2\to M$ such that
$\varphi(D^2_{\geq 0}\times D^2)\cap F=\varphi(D^2_{\geq 0}\times \{0\})$,
$\varphi(D^2_{\geq 0}\times D^2)\cap \partial M=\varphi(\partial_{y=0} D^2_{\geq 0}\times D^2)$,
and with $x\in \varphi(D^2_{\geq 0}\times \{0\})$.
Here, we used the following abbreviations $D^2_{\geq 0}=\{(x,y)\in D^2 \mid y\geq 0\}$ and $\partial_{y=0} D^2 = \{(x,0)\in D^2\}$.
\enm
It follows easily from the definitions that $F\subseteq M$ is a locally flat surface if and only if $F$ is proper 2-dimensional submanifold of $M$.
\enm
\end{remark}


The following proposition gives examples of embeddings $D^2\to D^4$ that are not locally flat.

\begin{proposition}\label{prop:cone-locallyflat-iff-K-unknot}
Given a knot $K\subseteq S^3$ the corresponding cone
\[\op{Cone}(K):=\{ r\cdot Q\mid Q\in K\mbox{ and }r\in [0,1]\}\,\,\subseteq\,\,D^4.\]
is locally flat if and only if $K$ is the unknot.
\end{proposition}

\begin{proof} Consider the specific unknot $U$ that is the equator of the equator $U=S^1\subseteq S^2\subseteq S^3=\partial D^4$. Taking the cone radially inwards to the origin of $D^4$ exhibits $\cone(U)$ as a locally flatly (properly) embedded disc. Any other unknotted $K\subseteq S^3$ is related to $U$ by a homeomorphism of $S^3$. By the Alexander trick \ref{lem:alexander-trick}(1), this homeomorphism extends radially inwards to a homeomorphism of $D^4$ fixing the origin. Thus the cone on any other unknot $K$ is locally flatly embedded, as we obtain a chart as in Definition \ref{def:submanifold}(1) at the origin of $D^4$.

Conversely, suppose $K\subseteq S^3$ is a knot such that $C:=\op{Cone}(K)$ is locally flat.
This implies that there is a chart
$\Phi\colon U\to D^4$ where $U$ is an open neighbourhood of the cone point $0$,
such that $\Phi(0)=0$ and such that $\Phi(D^4\cap C)=D^2\times \{0\}$. We set $\Psi:=\Phi^{-1}\colon D^4\to U$.
We introduce the following  notation.
\bnm
\item[(i)] Given $J\subseteq [0,1]$ we write $D_J:=\{v\in D^4\mid \|v\|\in J\}$.
\item[(ii)] Given $J\subseteq [0,1]$ we write $N_J:=\Psi(D_J)$.
\enm
An elementary argument shows that there exist $s_1<t_1<s_2<t_2<s_3$
such that $D_{[0,s_1]}\subseteq N_{[0,t_1]}\subseteq D_{[0,s_2]}\subseteq N_{[0,t_2]}\subseteq D_{[0,s_3]}$.
We make the following observations:
\bnm
\item For any $J\subseteq [0,1]$ we have homeomorphisms $D_J\sms C\cong (S^3\sms K)\times J$ and $N_J\sms C\xrightarrow{\Phi} D_J\sms \op{Cone}(U)\cong (S^3\sms U)\times J$.
\item For any inclusion $J\subseteq J'$ of intervals the inclusion induced map $D_J\to D_{J'}$ is a  homotopy equivalence.
\enm
We consider  the following commutative diagram where all maps are induced by inclusions
\[ \xymatrix{   \pi_1(S^3\sms K)\cong \pi_1(D_{\{s_2\}}\sms C)\ar[dr]_\cong \ar[rr] && \pi_1(N_{[t_1,t_2]}\sms C)\cong \Z\ar[dl] \\ & \pi_1(S^3\sms K)\cong \pi_1(D_{[s_1,s_3]}\sms C).}\]
Since the inclusion $D_{\{s_2\}}\sms C\to D_{[s_1,s_3]}\sms C$ is a homotopy equivalence we see that the left diagonal map is an isomorphism.
Thus we see that we have an automorphism of $\pi_1(S^3\sms K)$ that factors through $\Z$. Since the abelianisation of $\pi_1(S^3\sms K)$ is isomorphic to $\Z$ we see that $\pi_1(S^3\sms K)\cong \Z$. It follows from the Loop Theorem that $K$ is in fact the unknot~\cite[Theorem 4.B.1]{Ro90}.
\end{proof}

We conclude this section with the following lemma, which   provides us with examples of locally flat embeddings:

\begin{lemma}
	Let $M$ be a connected manifold. Then any two points $x\ne y\in M\setminus \partial M$ are connected by a locally flat embedded arc.
\end{lemma}

\begin{proof}
By Lemma~\ref{lem:manifolds-transitive} we only have to deal with the case that $x$ and $y$ lie in a subspace that is homeomorphic to an open $n$-ball. But this case is trivial.
\end{proof}

\section{Immersions}\label{section:immersions}

We define immersions and generic immersions in the topological category cf.~\cite[Section~2]{KPRT}.
In the smooth category immersions are required to have injective derivative at each point. In the topological category we cannot make such a definition, but instead define immersions as follows.


\begin{definition}\label{defn:immersion}
A continuous map $F \colon \Sigma^k \to M^n$ between manifolds of dimensions $k\leq n$ is an \emph{immersion} if for each $p \in \Sigma$ there is a codimension zero submanifold $U \subseteq \Sigma$ containing $p$ such that $F|_{U} \colon U \to M$ is a locally flat embedding.
\end{definition}

Recall that a continuous map is said to be \emph{proper} if the inverse image of every compact set in the codomain is compact. With this notion we can define proper immersions.

\begin{definition}\label{defn:proper-immersion}
A proper continuous map $F \colon \Sigma^k \to M^n$ between manifolds of dimensions $k\leq n$ is a \emph{proper immersion} if for each $p \in \Sigma$ there is a codimension zero submanifold $U \subseteq \Sigma$ containing $p$ such that $U \subseteq \Sigma$ is an open subset, $F|_{U} \colon U \to M$ is a locally flat embedding, and every for $u\in U$ we can find a chart for $F(u)$ in $M$ as in ($\alpha$) and ($\beta$) of Definition~\ref{defn:ProperSubmanifold}.
\end{definition}

%

We now consider surfaces in 4-manifolds, that is we restrict to $k=2$ and $n=4$. We take $M$ to be a connected 4-manifold.
The \emph{singular set} of an immersion $F \colon \Sigma \to M$ is \[\mathcal{S}(F) := \{m \in M \mid |F^{-1}(m)| \geq 2\}.\]

\begin{definition}\label{def:gen_immersion}
Let $\Sigma$ be a surface, possibly noncompact.  A continuous, proper map $F \colon \Sigma \to M$ is said to be a \emph{generic immersion}, denoted $F\colon \Sigma\looparrowright M$, if it is a proper immersion and the singular set is a closed, discrete subset of $M$ consisting only of transverse double points, each of whose preimages lies in the interior of $\Sigma$. The requirement that all singular points be transverse double points means that whenever $m \in \mathcal{S}(F)$, there are exactly two points $p_1, p_2 \in \Sigma$ with $F(p_1)=m = F(p_2)$, and there are disjoint charts $\varphi_i$ around $p_i$, for $i=1,2$, where $\varphi_1$ and $\varphi_2$ are as in \eqref{diagram:above}, with respect to the same chart $\Psi$ around $m$ and the standard inclusions
\begin{align*}
\iota_1 \colon \R^2 &= \R^2 \times \{0\} \hookrightarrow \R^2 \times \R^2 = \R^4 \text{ and} \\
\iota_2 \colon \R^2 &= \{0\} \times \R^2 \hookrightarrow \R^2 \times \R^2 = \R^4.
\end{align*}
\begin{equation}\label{diagram:above}
 \xymatrix{
 \R^2 \ar[r]^{\iota_i} \ar[d]^{\varphi_i} & \ar[d]^{\Psi}  \R^4   \\
 \Sigma \ar[r]^F & M }
 \end{equation}
\end{definition}

Typically one prefers to work with generic immersions of surfaces in 4-manifolds than arbitrary immersions. As the name suggests we can always arrange by a homotopy that maps of surfaces are generic immersions; see Theorem~\ref{thm:generic-immersions-are-generic}.

\section{Collar neighbourhoods}\label{collar-neighbourhoods}

We discuss the existence of a collar neighbourhood of the boundary of a manifold. First we recall the definition of a neighbourhood.

\begin{definition}[Neighbourhood]
Let $X$ be a space.  A \emph{neighbourhood} of a subset $A \subseteq X$ is a set $U \subseteq X$ for which there is an open set $V$ satisfying $A \subseteq V \subseteq U$.
\end{definition}

Next we give our definition of a collar neighbourhood.

\begin{definition}[Collar neighbourhood]
Let $M$ be a  manifold and let $B$ be a union of components of $\partial M$.
A \emph{collar neighbourhood} is a  map $\Phi\colon B\times [0,r]\to M$ for some $r>0$ with the following three properties:
\bnm
\item $\Phi$   is a locally flat embedding,
\item  for all $x\in B$ we have $\Phi(x,0)=x$,
\item we have $\Phi^{-1}(B\times [0,r])\cap \partial M=B$.
\enm
Often, by a slight abuse of language,
we identify $B\times [0,r]$ with its image $\Phi(B\times [0,r])$ and we refer to $B\times [0,r]$ also as a collar neighbourhood.
\end{definition}

Note that a  collar neighbourhood of a union $B$ of components of  $\partial M$ is a neighbourhood of $B$.
Now we can state the Collar Neighbourhood Theorem in the formulation of \cite[Theorem~1]{Ar70}. The existence of collars is originally due to Brown~\cite{Brown62}, and there is another easier proof due to Connelly~\cite{Connelly}.

\begin{theorem}\textbf{\textup{(Collar Neighbourhood Theorem)}}\label{thm:topological-collar}\label{thm:collar}
Let $M$ be an $n$-manifold, let $C$ be a compact submanifold of $\partial M$
and let $f\colon C\times [0,2]\to M$ be a map with $f(x,0)=x$ for all $x\in C$.
We assume that we are in one of the following two settings:
\bnm
    \item $C$ is closed as a manifold and $f$ is a locally flat embedding;
    \item $C\subseteq \partial M$ is a codimension zero submanifold and the restriction of $f$ to
    $(\partial C\times [0,2])\cup (C\times \{2\})$ is a locally flat embedding.
   \enm
Then there exists a collar neighbourhood $g\colon \partial M\times [0,1]\to M$ with $g|_{C\times [0,1]}=f$.
\end{theorem}

An \emph{isotopy} of a space $Y$ is a continuous one parameter family of maps $H=\{H_t\}_{t\in [0,1]}\colon Y\times [0,1]\to Y$ such that $H_t \colon Y \to Y$ is a homeomorphism for all $t \in [0,1]$.
To formulate a uniqueness result for collar neighbourhoods it helps to introduce the following definition.

\begin{definition}
Let $f,g\colon X\to Y$ be two maps between topological spaces and let $Z$ be a subset of $X$.
We say $f$ and $g$ are \emph{ambiently isotopic rel.\ $Z$} if there exists an isotopy $H=\{H_t\}_{t\in [0,1]}\colon Y\times [0,1]\to Y$ such that $H_0=\id$, $H_t|_{Z}=\id_Z$ for all $t$, and such that  $H_1\circ f=g$.
\end{definition}

\begin{theorem}\label{thm:topological-collar-unique}
Let $M$ be a manifold.
Given two collar neighbourhoods $\Phi\colon \partial M\times [0,2]\to M$ and
$\Psi\colon \partial M\times [0,2]\to M$, their restrictions $\Phi|_{\partial M \times [0,1]}$ and $\Psi|_{\partial M \times [0,1]}$ are ambiently isotopic rel.\ $\partial M\times \{0\}$.
\end{theorem}

\begin{proof}
The theorem is due to \cite[Theorem~2]{Ar70}, although Armstrong comments that the proof was told to him by Lashof. See also~\cite[Essay~I,~Theorem~A.2,~p.~40]{KS77}.
\end{proof}

\section{The Isotopy Extension Theorem}\label{subsection:isotopy-extension-thm}
In the smooth setting the Isotopy Extension Theorem gets used frequently and often subconsciously. In the non-smooth setting the formulation of the Isotopy Extension Theorem requires some care.

\begin{definition}
Let $X$ be a $k$-dimensional manifold and let $M$ be a compact $m$-dimensional manifold. Let $h\colon X\times [0,1]\to M$ be a homotopy.
\bnm
\item We say $h$ is \emph{locally flat} if for every $(x,t)\in X\times [0,1]$ there exists a neighbourhood $[t_0,t_1]$ of $t$ and  level-preserving embeddings $\alpha\colon D^k\times [t_0,t_1]\to X\times [0,1]$ and $\beta\colon D^k\times D^{m-k}\times [t_0,t_1]\to M\times [0,1]$ to neighbourhoods of $(x,t)$ and $(h_t(x),t)$ respectively, such that the following diagram commutes:
\[ \xymatrix@R0.65cm@C2.3cm{ D^k\times \{0\}\times [t_0,t_1]\ar[d]^\alpha\ar@{^(->}[r] & D^k\times D^{m-k}\times [t_0,t_1]\ar[d]^\beta \\
X\times [0,1]\ar[r]^{(x,t)\mapsto (h_t(x),t)} & M\times [0,1].}\]
\item We say $h$ is \emph{proper} if for every $t\in [0,1]$ we have $h_t(X)\cap \partial M=h_t(\partial X)$.
\enm
\end{definition}

This definition allows us to formulate the following useful theorem \cite[Corollary~1.4]{KirbyEdwards1971} (see \cite[p.~530]{Lees1969} for a related result.)

\begin{theorem} \textbf{\emph{(Isotopy Extension Theorem)}}\label{thm:isotopy-extension-theorem}
Let $h\colon X\times [0,1]\to M$ be a locally flat proper isotopy of a compact manifold $X$ into a manifold $M$. Then $h$ can be covered by an ambient isotopy of $M$, i.e.\ there exists an isotopy $H\colon M\times [0,1]\to M$ such that $H_0=\id$ and $h_t=H_t\circ h_0$ for all $t\in [0,1]$.
\end{theorem}

Using the Isotopy Extension Theorem~\ref{thm:isotopy-extension-theorem}
we can now prove the following refinement of the
Collar Neighbourhood Theorem~\ref{thm:collar}.

\begin{theorem}\textbf{\textup{(Collar Neighbourhood Theorem for Proper Submanifolds)} }\label{thm:collar-boundary}
Let $M$ be a manifold and let $X\subseteq M$ be a proper submanifold.
There exists a collar neighbourhood $\partial M\times [0,1]$ such that
$(\partial M\times [0,1])\cap X$ is a collar neighbourhood for $\partial X\subseteq X$.
\end{theorem}

\begin{proof}
By the earlier Collar Neighbourhood Theorem~\ref{thm:collar}
we can pick a collar neighbourhood $\partial M\times [0,2]$ for $\partial M$ and we can also pick a collar neighbourhood $\partial X\times [0,2]$ for $\partial X$.
Given $t\in [0,1]$ we consider the obvious homeomorphisms
\[ f_t\colon M=(M\sms (\partial M\times [0,2)))\cup (\partial M\times [0,2])\,\to\,
(M\sms (\partial M\times [0,2)))\cup (\partial M\times [t,2])\]
and
\[ g_t\colon X=(X\sms (\partial X\times [0,2)))\cup (\partial X\times [0,2])\,\to\,
(X\sms (\partial X\times [0,2)))\cup (\partial X\times [t,2]).\]
Next we consider the following proper locally flat isotopy:
\[ \ba{rcl} h\colon X\times [0,1]&\to & M\\
(x,t)&\mapsto & \left\{ \ba{rl} (y,s)\in \partial M\times [0,t], &\mbox{ if }x=(y,s)\mbox{ with }y\in \partial X,s\in [0,t],\\
f_t(g_t^{-1}(x)), &\mbox{ otherwise.}\ea\right.\ea\]
Note that the collar neighbourhood $\partial M\times [0,1]$ is of the desired form for the proper submanifold $h_1(X)$. By the Isotopy Extension Theorem~\ref{thm:isotopy-extension-theorem} we can extend $h$ to a isotopy $H$ of $M$. Thus $H_1^{-1}(\partial M\times [0,1])$ is the desired collar neighbourhood for $M$.
\end{proof}

\begin{figure}[h]
\begin{center}
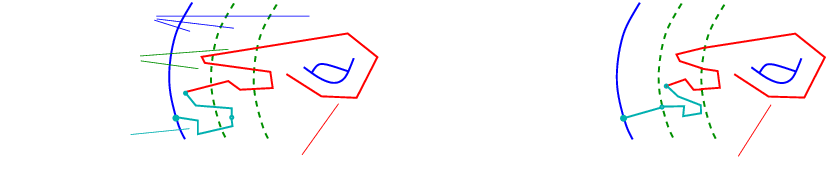
\caption{Illustration of the proof of the Collar Neighbourhood Theorem~\ref{thm:collar-boundary}.}\label{fig:collar-boundary}
\end{center}
\end{figure}

\chapter{CW structures, triangulations, and handle structures on manifolds}\label{chapter:CW structures}

In this chapter we will discuss the existence of various types of structures on compact manifolds.
The results of this chapter are summarised in the following diagram. We use the following colour code.
\begin{enumerate}[leftmargin=1cm]
\item A green arrow from $A$ to $B$ means that existence of structure $A$ implies the existence of structure $B$.
\item A dashed blue arrow from $A$ to $B$ means that existence of structure $A$ implies the existence of structure $B$, under the hypothesis that is written next to the blue arrow.
\end{enumerate}

\begin{figure}[h]
\begin{center}
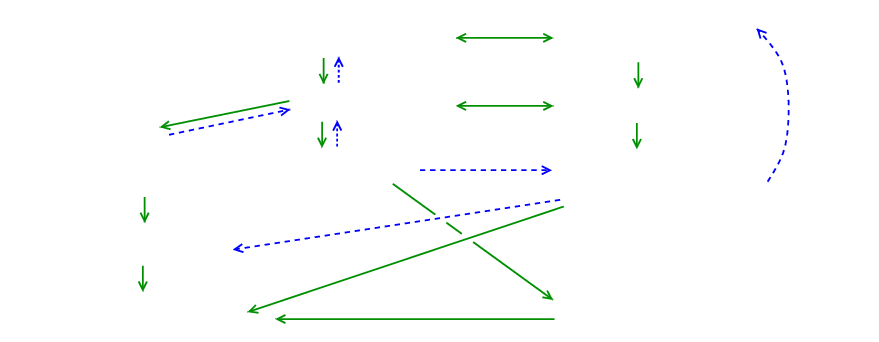
\caption{Existence of structures on a compact manifold.}
\label{fig:structures-on-manifolds}
\end{center}
\end{figure}

In the following we will explain some of the structures mentioned in the diagram and
we explain the various arrows.

\begin{definition}
An \emph{$n$-dimensional PL manifold} is a simplicial complex $X$ such that each $x\in X$ admits a neighbourhood that is PL homeomorphic to the standard PL $n$-ball.

Given a topological $n$-manifold $M$, a \emph{PL structure} is a triangulation of $M$ such that the resulting simplicial complex is an $n$-dimensional PL manifold.
\end{definition}




\begin{remark}
Note that Dedecker \cite{Dedecker1962} and Zeeman \cite{Zeeman1962} showed that  a  PL-manifold can also be defined as a topological manifold together with a ``piecewise linear'' atlas. We refer to the the above references for precise definitions.
\end{remark}

\begin{definition}
Let $\CAT=\TOP, \op{PL}$, or $\Diff$.
Let $M$ be a $\CAT$-manifold and let $W$ be a (possibly empty) union of components of $\partial M$. A \emph{$\CAT$-handle structure rel.~$W$}  is a $\CAT$-isomorphism to a $\CAT$-manifold that is obtained from $W\times [0,1]$ by iteratively attaching handles of dimension $0,1,2,\dots$ along $\CAT$-gluing maps.
\end{definition}

With these definitions we can now formulate the theorems which in particular contain all the results mentioned in the diagram in Figure~\ref{fig:structures-on-manifolds}.

\begin{theorem}\label{thm:smooth-implies-handles-pl}\mbox{}\label{thm:smooth-manifold-handle-pl}
\begin{enumerate}[leftmargin=1cm,font=\normalfont]
\item Every compact smooth manifold admits a smooth handle structure rel.~any union of boundary components.
\item Every smooth manifold admits a PL structure.
\item Every manifold that admits a smooth handle structure also admits a PL handle structure.
\end{enumerate}
\end{theorem}

\begin{proof}
\mbox{}
\begin{enumerate}[leftmargin=1cm]
\item In \cite[Section~3]{Milnor-Morse} and \cite[Section~6.4]{Hi76} it was shown that every compact smooth manifold admits a smooth handle decomposition.
\item In \cite[Theorem~10.6]{Munkres66} and  \cite[Chapter~IV.12]{Whitney57} it was shown that
every smooth manifold admits a PL structure.
\item Every manifold that admits a smooth handle structure is by definition smooth. It follows from (2) and
Theorem~\ref{thm:pl-manifolds} (1) that every smooth manifold admits a PL handle structure. One might expect that one can turn every smooth handle structure into a PL handle structure by suitable choices of PL structures on handles and by applying isotopies to the attaching maps. But it is not clear to us whether or not this approach can be made to work.
\qedhere
\end{enumerate}
\end{proof}

\begin{theorem}\label{thm:pl-manifolds}
\mbox{}
\begin{enumerate}[leftmargin=1cm,font=\normalfont]
\item Every compact PL manifold admits a PL handle structure rel.~any fixed union of boundary components.
\item Every compact PL manifold of dimension $\leq 7$ admits a smooth structure, which is unique up to isotopy in dimensions $\leq 6$.
\item Every manifold that admits a PL handle structure also admits a topological handle structure.
\end{enumerate}
\end{theorem}

 \begin{proof}
\mbox{}
\begin{enumerate}[leftmargin=1cm]
\item This statement is proved in \cite[Proposition~6.8]{RoSa-Block-III}.
\item The 0-dimensional case is clear and the 1-dimensional case can be proved fairly easily by hand. The 2-dimensional case can also done by hand: one picks obvious charts for the interiors of the 2-simplices and one can easily find charts that cover the interiors of the 1-simplices. For each 0-simplex one can pick a single chart by combining the various 2-simplices and adjusting the angle. The transition maps of such an atlas will all be smooth.
Alternatively \cite[Theorem~A]{Hatcher-Kirby-torus-trick} shows that every 2-di\-men\-sional topological manifold admits a smooth structure. Furthermore, \cite[Theorem~B]{Hatcher-Kirby-torus-trick}
says that every homeomorphism between smooth surfaces is isotopic to a diffeomorphism.

We turn to manifolds of dimension $3,4,5,6$, and $7$.
By smoothing theory for~$\PL$ manifolds, due to Munkres~\cite{Mun1959a,Munkres60} and Hirsch-Mazur~\cite[Part~II, Theorem~4.2]{Hirsch-Mazur74}, isotopy classes of smooth structures on a compact PL manifold $M$ are in one to one correspondence with homotopy classes of sections of a fibre bundle over $M$ with fibre $\PL/\O$. It follows from work of Munkres~\cite{Munkres-isotopies-60}, Cerf~\cite{Cer1968}, and Kervaire--Milnor \cite{KM1963} that $\PL/\O$ is 6-connected. It then follows from obstruction theory that every compact PL manifold of dimension $\leq 7$ admits a smooth structure, which is unique up to isotopy in dimensions $\leq 6$.

We remark that uniqueness of smooth structures in dimension 3 was proven earlier by Munkres \cite[p.~333]{Mun1959a}, \cite{Munkres-isotopies-60}, \cite[Theorems 6.2 and 6.3]{Munkres60}, and independently by Whitehead \cite[Corollary 1.18]{WhdJ1961a}. See  also  \cite[Chapter~3.10]{Thu1997} for a more detailed discussion.
%
%
\item This statement follows immediately from the definitions. \qedhere
\end{enumerate}
\end{proof}

\begin{theorem}\label{thm:manifold-low-dim}\label{thm:handle-dec}
\mbox{}
\begin{enumerate}[leftmargin=1cm,font=\normalfont]
\item Every manifold of dimension $\geq 5$ admits a topological handle structure rel.~any fixed union of boundary components.
\item
\bnmt\item Every compact manifold of dimension 1,2 or 3 admits a PL structure.
\item In dimension $\leq 4$ a manifold that admits a simplicial structure also admits a PL structure.
\item Given any $n\geq 4$ there exists a closed $n$-manifold that does not admit a simplicial structure.
\enm
\end{enumerate}
\end{theorem}

\begin{proof}
\mbox{}
\begin{enumerate}[leftmargin=1cm]
\item For manifolds of dimension $\geq 6$ this statement was proven by Kirby-Sieben\-mann~\cite[Essay~III, Theorem 2.1, p.~104]{KS77}.  Quinn~\cite[Theorem 2.3.1]{Qu82} extended this result to manifolds of dimension $5$.
\item
\bnmt
\item
Rad\'o~\cite{Ra26} showed in 1926 that every compact 2-manifold has a simplicial structure. (Uniqueness was proved by Papakyriakopoulos \cite{Pap1946} in 1946.)  Moise~\cite{Mo52,Mo77} proved the analogous result for 3-manifolds.  We refer to \cite[Theorems~6 and~8]{Bin1959}, \cite[Theorem~2]{Hamilton-triang-3-mflds} and \cite{Shal1984} for alternative proofs of the 3-dimensional case.
\item By (a) we only need to consider the case of a 4-dimensional manifold $M$.
We will defer the proof to Proposition~\ref{prop:PHM-main} below.
In fact, in Proposition~\ref{prop:PHM-main} we will show the stronger fact that every triangulation of a manifold $M$ of dimension $n$ at most $4$ is in fact a PL-structure.  The inductive nature of the proof of Proposition~\ref{prop:PHM-main} means that we need to consider all $n \leq 4$, even though only $n=4$ is logically required at this point.
%
%
%
%
%
\item Casson~\cite[p.~xvi]{AM90} showed in the 1980s that there exist closed 4-manifolds that do not have a simplicial structure.
It is now known that in every dimension $n \geq 5$, there exists a closed $n$-manifold that does not admit a simplicial structure. This question was reduced to a problem about homology $3$-spheres~\cite{Matumoto78, Stern80}, which was then solved by Manolescu~\cite{Ma16}.
Manolescu's examples are necessarily nonorientable.

We remark that orientable topological manifolds in dimension $n \geq 5$ that do not admit a PL structure were constructed much earlier by Siebenmann in \cite[Annex~C, Section 2]{KS77}. However each of Siebenmann's manifolds admits a simplicial structure.
\qedhere
\enm
\end{enumerate}
\end{proof}

In order to prove Theorem~\ref{thm:manifold-low-dim}~(2b), we will need some facts about polyhedral homology manifolds.  We are grateful to Arunima Ray for assistance with this proof.

\begin{definition}\label{Defn:PHM}
    A locally finite $n$-dimensional simplicial complex $P$ is an \emph{$n$-dimensional polyhedral homology manifold} if for every $x \in P$ and for any subdivision of $P$ with $x$ as a vertex,
    \[H_*(\link(x);\Z) \cong H_*(S^{n-1};\Z).\]
%
\end{definition}

\begin{remark}\leavevmode
\begin{enumerate}[leftmargin=1cm]
    \item It is immediate from the definition that if $P$ is an $n$-dimensional polyhedral homology manifold, then so is any subdivision.
 \item  There is an extensions of the definition for polyhedral homology manifolds with  nonempty boundary, but we will not give it here.
\end{enumerate}
\end{remark}

We need two facts about polyhedral homology manifolds. The first is that triangulated topological manifolds are examples of polyhedral homology manifolds.

\begin{proposition}[{\cite[Proposition~1.2]{Stern80}}]\label{prop:PHM-2}
    Let $M$ be a  topological manifold with empty boundary, with a simplicial structure. Then $M$ is a polyhedral homology manifold.
\end{proposition}

The second fact we need, which will drive the induction in the proof of Proposition~\ref{prop:PHM-main}, is that the polyhedral homology manifold property descends to links of vertices.

\begin{proposition}\label{prop:PHM-1}
    Let $P$ be an $n$-dimensional polyhedral homology manifold. Then for all vertices $v \in P$, the link $\link(v)$ is a polyhedral homology manifold.
\end{proposition}

\begin{proof}
If $n=1$, the link of a vertex consists of a collection of points, so is certainly a polyhedral homology manifold.  So from now on we let $n \geq 2$.
A lemma of Maunder~\cite[Lemma~5.4.7, p.~188]{Maunder} implies the following. Let $K$ be a simplicial complex, let $x \in K$ be a vertex, and let $L:= \link_K(x)$, and let $y \in L$ be a vertex. Let $z$ be the midpoint of the line segment $xy$. Then \[\wt{H}_*(\link_K(z);\Z) \cong \wt{H}_{*-1}(\link_{L}(y);\Z)\]
where $\link_K(z)$ is understood to be the link of $z$ in the simplicial complex obtained from $K$ by subdividing at $z$.

To apply this, let $K:= P$ and let $x \in P$. Then since $z\in P$ is a vertex in some subdivision of $P$,
\[\wt{H}_{*}(S^{n-2};\Z) \cong  \wt{H}_{*+1}(S^{n-1};\Z) \cong \wt{H}_{*+1}(\link_K(z);\Z) \cong \wt{H}_{*}(\link_{L}(y);\Z).\]
Since this holds for every vertex $y$ of $L$, it follows that $L = \link_P(x)$ is a polyhedral homology manifold.
\end{proof}

Now we prove the main result of our digression into polyhedral homology manifolds. In particular note that the case $n=4$ (b) implies Theorem~\ref{thm:manifold-low-dim}~(2b).

\begin{proposition}\label{prop:PHM-main}
    Let $n \leq 4$.
    \begin{enumerate}[leftmargin=1cm,font=\normalfont]
        \item[(a)]
        Let $P$ be a finite $(n-1)$-dimensional polyhedral homology manifold with $H_*(P;\Z) \cong H_*(S^{n-1};\Z)$. If $n=4$, suppose that $\pi_1(P)=\{1\}$.
        Then $P$ is a PL $(n-1)$-manifold and is PL homeomorphic to $S^{n-1}$.
        \item[(b)] Every triangulation of a topological $n$-manifold $M$ without boundary is a PL structure.
    \end{enumerate}
\end{proposition}

\begin{proof}
    We work inductively on $n$, starting with $n=1$. In this proof all homology groups are with $\Z$ coefficients, which will be omitted for brevity.

We also make the following observation before we begin.
Let $X$ be a simplicial complex such that for every vertex $x \in X$, the link $\link_X(x)$ is PL homeomorphic to the standard PL $(n-1)$-sphere. Then for each $x$, the star of $x$ is the cone on the link of $x$, and so the star is PL homeomorphic to the standard $n$-ball. Hence $X$ is a PL $n$-manifold without boundary. \\

\noindent \emph{The case $n=1$} (a).  This means that $P$ is dimension 0, so is a finite collection of points. Thus $P$ is certainly a PL manifold. Also $H_*(P) \cong H_*(S^0)$, so we deduce that $P$ consists of two points, that is $P \cong S^0$.\\

\noindent \emph{The case $n=1$} (b). By Proposition~\ref{prop:PHM-2}, $M$ is a 1-dimensional polyhedral homology manifold, so the link $\link(v)$ is dimension 0, with $H_*(\link(v)) \cong H_*(S^0)$ for every vertex $v \in M$.  Hence by the case $n=1$ (a), $\link(v) \cong S^0$, and so the triangulation of $M$ is a PL structure.  \\

\noindent \emph{The case $n=2$} (a). Let $P$ be a finite 1-dimensional polyhedral homology manifold with $H_*(P) \cong H_*(S^1)$. We have that $P$ is a graph. By definition of a polyhedral homology manifold, links of vertices have $H_*(\link(v)) \cong H_*(S^0)$, for every vertex $v \in P$. Hence vertices have valency two, so $P$ is a PL 1-manifold.  Since $H_*(P) \cong H_*(S^1)$, we must have that $P \cong S^1$, by the classification of compact PL $1$-manifolds. \\

\noindent \emph{The case $n=2$} (b). By Proposition~\ref{prop:PHM-2}, $M$ is a 2-dimensional polyhedral homology manifold, so by Proposition~\ref{prop:PHM-1} the link $\link(v)$ is a polyhedral homology manifold of dimension 1 with $H_*(\link(v)) \cong H_*(S^1)$ for every vertex $v \in M$.  Hence by the case $n=2$ (a), $\link(v) \cong S^1$, and so the triangulation of $M$ is a PL structure.  \\

\noindent \emph{The case $n=3$} (a).  Now $P$ is a 2-dimensional finite polyhedral homology manifold, with $H_*(P) \cong H_*(S^2)$.  By Proposition~\ref{prop:PHM-1}, $\link(v)$ is a polyhedral homology manifold and $H_*(\link(v)) \cong H_*(S^1)$ by the definition of a polyhedral homology manifold, for every vertex $v$ of $P$. Then by the case $n=1$~(a), $\link(v)$ is a PL manifold and $\link(v) \cong S^1$.  Hence $P$ is a PL manifold with $H_*(P) \cong H_*(S^2)$.  Thus $P \cong S^2$ by the classification of compact PL 2-manifolds.  \\

\noindent \emph{The case $n=3$} (b). By Proposition~\ref{prop:PHM-2}, $M$ is a 3-dimensional polyhedral homology manifold, so by Proposition~\ref{prop:PHM-1} the link $\link(v)$ is a polyhedral homology manifold of dimension 2 with $H_*(\link(v)) \cong H_*(S^2)$ for every vertex $v \in M$.  Hence by the case $n=3$ (a), $\link(v) \cong S^2$. Thus  the triangulation of $M$ is a PL structure, as desired.\\

\noindent \emph{The case $n=4$} (a).  Now $P$ is a 3-dimensional finite polyhedral homology manifold, with $H_*(P) \cong H_*(S^3)$.  In addition, since $n=4$ we have by the hypotheses of the proposition that $\pi_1(P)=\{1\}$. By Proposition~\ref{prop:PHM-1}, for every vertex $v$ of $P$, $\link(v)$ is a polyhedral homology manifold with $H_*(\link(v)) \cong H_*(S^2)$. Hence by the case $n=3$~(a), $\link(v)$ is a PL 2-manifold and $\link(v) \cong S^2$.  Hence $P$ is a PL 3-manifold.
Since $P$ is simply-connected, compact, and has the homology of $S^3$
the Poincar\'e conjecture in dimension three~\cite{Perelman:2002-1,Perelman:2003-1,Perelman:2003-2}, implies that $P \cong S^3$.\\

\noindent \emph{The case $n=4$} (b).
Let $M$ be a 4-dimensional topological manifold with a simplicial structure.
By Proposition~\ref{prop:PHM-2}, $M$ is a 4-dimensional polyhedral homology manifold, so by Proposition~\ref{prop:PHM-1} for every vertex $v \in M$, the link $\link(v)$ is a polyhedral homology manifold of dimension 3 with $H_*(\link(v)) \cong H_*(S^3)$.  Moreover, for a triangulation of a topological manifold of dimension $n \geq 3$, a standard argument, similar to the proof of Proposition~\ref{prop:cone-locallyflat-iff-K-unknot} (see \cite[Theorem~1.5]{Stern80} for a proof), shows that $\pi_1(\link(v))= \{1\}$ for every vertex $v$.  Hence we can apply the case  $n=3$ (a), to deduce that $\link(v)$ is a PL manifold and $\link(v) \cong S^3$. Thus the triangulation of $M$ is a PL structure, as desired.
\end{proof}


We continue by recalling the definition of a CW structure.

\begin{definition}
A \emph{CW complex}  is a topological space $X$ together with a filtration \[\emptyset = X_{-1} \subseteq X_0 \subseteq X_1 \subseteq X_2 \subseteq \cdots\]
such that $X= \colim X_n$ and such that for each $n \geq 0$, the space $X_n$ arises as a pushout
\[\xymatrix @C+1cm{ \coprod_{j \in \mathcal{J}_n} S^{n-1} \ar[r] \ar[d] & X_{n-1} \ar[d] \\ \coprod_{j \in \mathcal{J}_n} D^{n} \ar[r] & X_n}\]
where $\mathcal{J}_n$ indexes the discs $D^n$.  The interiors $\Int D^n$  of the discs are called the \emph{$n$-cells}.
For $n \geq 0$, a CW complex $X$ is said to be \emph{$n$-dimensional} if $X_n\setminus X_{n-1}\ne \emptyset$ and $X_{i} = X_{i+1}$ for all $i \geq n$.
We say a topological space~$X$ \emph{admits a CW structure} if $X$ admits such a filtration.
\end{definition}

The above results can be used to prove the following theorem on the existence of CW structures.

\begin{theorem}\label{thm:n-not-equal-4-CW structure}~
\bnm
\item Every  PL manifold \textup{(}and thus every smooth manifold\textup{)}  admits a CW structure.
\item For $n\leq 3$ every compact $n$-manifold admits the structure of a finite $n$-dimensional CW complex.
\item Let $n\geq 5$ and let $M$ be a compact $n$-manifold. Then $M$ is homeomorphic to the mapping cylinder of some map $f\colon \partial M\to X$, where $X$ is a finite CW complex.
\item For $n\geq 5$, every \emph{closed} $n$-manifold admits the structure of a finite $n$-dimensional CW complex.
\enm
\end{theorem}

\begin{proof}
\mbox{}
\bnm
\item  Every simplicial complex is evidently a CW complex.
The statement for PL manifold thus holds by definition.
For smooth manifolds this statement now follows from
Theorem~\ref{thm:smooth-manifold-handle-pl} (2).
Also note that it follows from
\cite[Theorem~3.5]{Milnor-Morse} and the existence of a handle decomposition
that every compact smooth $n$-dimensional manifold is homotopy equivalent to a compact $n$-dimensional CW complex. For most purposes it suffices to know that a compact smooth manifold has the homotopy type of a compact CW complex.
\item  In Theorem~\ref{thm:manifold-low-dim} we  saw that every compact $n$-manifold of dimension $\leq 3$ admits a PL structure. It follows from (1) that every such
manifold admits a CW structure.
\item Let $M$ be a manifold with (possibly empty) boundary of dimension $\geq 5$.
In Theorem~\ref{thm:handle-dec} we saw that $M$ admits a topological handle structure  rel.\ $\partial M \times I$.   Kirby-Siebenmann~\cite[Essay~III, Theorem 2.2, p.~107]{KS77} then says that $M$ is homeomorphic to the mapping cylinder of some map $f\colon \partial M\to X$, where $X$ is a finite CW complex. Thus (2) holds.
\item If $M$ is a closed manifold of dimension $\geq 5$, then it follows from (3)
that $M$ admits the structure of a finite $n$-dimensional CW complex.\qedhere
\enm
\end{proof}

It is not clear to us whether Theorem~\ref{thm:n-not-equal-4-CW structure} suffices to show that every compact high-dimensional manifold admits a CW structure. Put differently, to the best of our knowledge the following question is open for manifolds with nonempty boundary.

\begin{question}
Let $n\ge 5$. Does every compact  $n$-manifold have a CW structure?
\end{question}

The following question also seems to still be open, even in the closed case.

\begin{question}
Does every compact $4$-manifold have a CW structure?
\end{question}

In Theorem~\ref{thm:n-not-equal-4-CW structure} we showed that many compact manifolds admit a CW structure, but we saw that there are still open cases. In many applications it suffices to know that a topological space is homotopy equivalent to a finite CW complex. This leads us to the following theorem.

\begin{theorem}\label{thm:topological-manifold-CW complex}
\leavevmode
\bnm
\item \label{item:topcwcx1}
Every connected closed  nonempty $n$-manifold is homotopy equivalent to an $n$-di\-men\-sional finite CW complex.
\item \label{item:topcwcx2} Every connected compact nonempty $n$-manifold with nonempty boundary is homotopy equivalent to an $(n-1)$-dimensional finite CW complex.
\enm
\end{theorem}

In the proof of Theorem~\ref{thm:topological-manifold-CW complex} we will use the following theorem proven by Wall~\cite[Corollary 5.1]{Wa66}.

\begin{theorem}\label{thm:Wall} Let $X$ be a finite connected CW complex with fundamental group $\pi=\pi_1(X)$.  Suppose that there is an integer $n \geq 3$ such that $H^i(X;\Z[\pi]) = 0$  for all $i>n$. \textup{(}Here $H^i(X;\Z[\pi]) = 0$ denotes cohomology with twisted coefficients which we will introduce in Appendix~\ref{section:twisted-invariants}.\textup{)}
Then $X$ is homotopy equivalent to an $n$-dimensional finite CW complex.
\end{theorem}

\begin{proof}[Proof of Theorem~\ref{thm:topological-manifold-CW complex}]
\mbox{}
We start out with the following claim.

\begin{claim}
Every compact manifold is homotopy equivalent to a finite CW complex.
\end{claim}

\noindent We provide two different proofs of the claim.
\begin{enumerate}[leftmargin=1cm]
\item[(a)] Let $M$ be a compact  $n$-manifold.
If $n\geq 6$, then it  follows also from
Theorem~\ref{thm:n-not-equal-4-CW structure}~(3) and (4) together with the Cellular Approximation Theorem~\cite[Theorem~IV.11.4]{Br93} that $M$ is homotopy equivalent to a finite CW complex.
If $n<5$, then we just replace $M$ by $M\times D^6$ and apply the above argument.
(see also Kirby-Siebenmann~\cite[Section 1~(III), p.~744]{KS69}.)
\item[(b)] We now provide a second proof,  which is of a very different flavour.
Before we can provide this alternative argument we need to introduce one definition:
a space $X$ is called an \emph{absolute neighbourhood retract (ANR)} if $X$ is metrisable (recall that a space is metrisable if it admits the structure of a metric space inducing the given topology) and if whenever $X \subseteq Y$ is a closed subset of a metrisable space $Y$, then $X$ is a neighbourhood retract of $Y$.  That is, there is an open neighbourhood $U \subseteq Y$ containing $X$, with a map $r \colon U \to X$ such that the composition $X \to U \xrightarrow{r} X$ is equal to the identity on~$X$.

Now we can return to the actual proof of the claim:
It follows from the Dugundji Extension Theorem \cite[Theorem~6.1.1]{Sakai2013} and work of Hanner~\cite[Theorem~3.3]{Han51}  that every manifold (possibly noncompact) is an ANR, and West~\cite[Corollary~5.3]{We77} showed that every compact ANR is homotopy equivalent to a finite CW complex.
\end{enumerate}

With this claim it remains to prove the dimension statements of the theorem.
\bnm
\item  By Theorem~\ref{thm:n-not-equal-4-CW structure} we only need to  prove (\ref{item:topcwcx1}) in the case $n=4$.  But since the subsequent argument works for all $n\geq 4$ we also give it for all $n\geq 4$.
We follow the argument provided in \cite[Corollary~5.4]{We77}.

Let $M$ be a compact connected nonempty $n$-manifold. By the claim, $M$ is homotopy equivalent to a finite CW complex $X$.

First we consider the case that $M$ is orientable. Since $M$ is $n$-dimensional it follows from the Universal Poincar\'e Duality Theorem~\ref{thm:poincareduality} that for every $k>n$ and every $\Z[\pi_1(X)]$-module $\Lambda$ we have
\[ H^k(X;\Lambda)\,\,\cong\,\, H^k(M;\Lambda)\,\,\cong\,\, H_{n-k}(M,\partial M;\Lambda)\,\,=\,\,0.\]
By Theorem~\ref{thm:Wall}, $X$ is homotopy equivalent to an $n$-dimensional finite CW complex. Note that to apply Theorem~\ref{thm:Wall} we have used that $n\geq 3$.

If $M$ is nonorientable, then in the above one needs to apply Poincar\'e Duality for nonorientable manifolds. In the closed case, a proof can be found in \cite{Sun17}. The case for manifolds with boundary can be proved by combining the ideas of \cite{Sun17} and
the Universal Poincar\'e Duality Theorem~\ref{thm:poincareduality}.
\item
Now we turn to the proof of (\ref{item:topcwcx2}). Let $M$ be a compact connected $n$-manifold with nonempty boundary. We start with $n=1,2$, or $3$.
 We saw in Theorem~\ref{thm:manifold-low-dim} that in this dimension range every compact connected $n$-manifold admits a simplicial
 structure. It is well known that a compact connected $n$-manifold with nonempty boundary and a simplicial structure is homotopy equivalent to an $(n-1)$-dimensional simplicial complex: iteratively collapse top dimensional simplices starting with those that have a face on the boundary. In particular such a manifold is homotopy equivalent to an $(n-1)$-dimensional finite CW complex.

Now suppose that $n\geq 4$. By the claim $M$ is homotopy equivalent to  a finite CW complex $X$.
Let  $k\geq n$ and let $\Lambda$ be a $\Z[\pi_1(X)]$-module. By  the Universal Poincar\'e Duality Theorem~\ref{thm:poincareduality}   we have that
\[ H^k(X;\Lambda)\,\,\cong\,\,  H^k(M;\Lambda)\,\,\cong\,\, H_{n-k}(M,\partial M;\Lambda)\,\,=\,\,0.\]
Here the last conclusion is obvious for $k>n$. For $k=n$ the conclusion follows from the fact that $\partial M\ne \emptyset$, that $M$ is connected and the explicit calculation of 0-th twisted homology groups as given in \cite[Chapter~VI.3]{HS97}.
It follows from Theorem~\ref{thm:Wall} that $X$ is homotopy equivalent to an $(n-1)$-dimensional finite CW complex. Note here that to apply Theorem~\ref{thm:Wall} we used $n\geq 4$.\qedhere
\enm
\end{proof}

Theorem~\ref{thm:topological-manifold-CW complex} is strong enough to recover many familiar statements.

\begin{corollary}\label{cor:topological-mfd-homology}
Let~$M$ be a compact, connected, nonempty manifold.
\bnm
\item The group $\pi_1(M)$ is finitely presented.
\item All homology groups $H_k(M;\Z)$ are finitely generated abelian groups, in particular it makes sense to define the Euler characteristic
\[ \chi(M)\,\,:=\,\,\tmsum{n}{} (-1)^n\cdot b_n(M).\]
\item Let  $p\colon \wti{M}\to M$ be a finite covering. Then
\[\chi(\widetilde{M})\,\,=\,\,[\widetilde{M}:M]\cdot \chi(M).\]
\enm
\end{corollary}

\begin{proof}
The first two statements in the corollary are an immediate consequence of
Theorem~\ref{thm:topological-manifold-CW complex} and standard results on fundamental groups and homology groups of finite CW complexes.
We turn to the final statement. Let $X$ be a finite CW complex homotopy equivalent to $M$. Use the fact that the Euler characteristic is multiplicative for finite covers of finite CW complexes and use that a $k$-fold cover $\wti{M}$ of $M$ induces a $k$-fold cover~$\wti{X}$
of $X$ such that $\wti{M}$ and $\wti{X}$ are homotopy equivalent, to deduce the result.

Here is an alternative approach to proving the first two statements.
Let $M$ be a compact, connected manifold.
Borsuk's Theorem that $M$ is a Euclidean Neighbourhood Retract shows that $M$ is a retract of a finite CW complex; see~\cite[Appendix~A,~Corollary~A.9]{Hat02}, \cite[Appendix~E]{Br93}, \cite{Floris:ENRs}, and \cite[Theorem~124.3 and Proposition~124.7]{Fr23}. This fact is nontrivial but it is much easier to prove than Theorem~\ref{thm:topological-manifold-CW complex}.
Borsuk's Theorem implies immediately that the homology groups of $M$ are finitely generated and that
the fundamental group of $M$ is finitely generated.
In fact using a group theoretic lemma as in \cite[Lemma~1.3]{Wall65} or \cite[Theorem~3.1]{FR01}, one actually obtains that $\pi_1(M)$ is finitely presented.
But it is not clear how Borsuk's Theorem  can  be used to prove the third statement.
\end{proof}

\begin{remark}
As pointed out above, every compact smooth manifold admits the structure of a finite CW complex.
One can combine this fact with Theorem~\ref{thm:connect-sum-is-smooth} below to obtain an alternative proof of Corollary~\ref{cor:topological-mfd-homology} (1) and (2).
More precisely: Theorem~\ref{thm:connect-sum-is-smooth} says that for any compact  4-manifold $M$ there is a
closed orientable simply connected  $4$-manifold $N$ such that the connected sum $M \# N$ admits a smooth structure. Using the well-known behaviour of the fundamental group and the homology groups under the connected sum operation one can now fairly easily provide
 an alternative proof of Corollary~\ref{cor:topological-mfd-homology}~(1) and~(2).
\end{remark}

\chapter{The Annulus Theorem and the Stable Homeomorphism Theorem}\label{chapter:annulus}

The Annulus Theorem  and the Stable Homeomorphism Theorems are two (basically equivalent) fundamental results in the development of the theory of manifolds.
For example, in high dimensions,  the Stable Homeomorphism Theorem
is an essential ingredient in the proof of the Product Structure Theorem~\ref{theorem:product-structure}~\cite[Essay~I, Theorem~5.1, p.~31]{KS77}, which itself underpins all the results of \cite{KS77}. We state the Product Structure Theorem in Section~\ref{section:product-structure theorem}.
In dimension four, the Annulus Theorem is one of the many consequences of Quinn's controlled $h$-cobordism theorem~\cite[Chapter~7]{FQ90}.  In dimension 4 this theorem is used in the proofs of smoothing theorems (Chapter~\ref{chapter:smoothing}), existence and uniqueness of normal vector bundles (Chapter~\ref{chapter:tubular}), and transversality (Chapter~\ref{chapter:transversality}). We discuss these developments in the later chapters indicated. Later in this section (Section~\ref{section:connected-sum-well-defined}), we will discuss an application of the Annulus Theorem: showing that connected sum is a well-defined operation on connected, topological manifolds that are either oriented, or at least one of which is nonorientable.
Here is the Annulus Theorem.

\begin{theorem}\textbf{\textup{(Annulus Theorem)}}\label{thm:annulus}
Let $n\in \N_0$ and let $f,g\colon D^n\to \R^n$ be two orientation-preserving locally flat embeddings. If $f(D^n)\subseteq \int(g(D^n))$, then $g(D^n)\sms \int(f(D^n))$ is homeomorphic to $S^{n-1} \times [0,1]$.
\end{theorem}

For $n=0,1$ the Annulus Theorem is basically trivial.
For $n=2,3$ the Annulus Theorem follows from the work of Rad\'o~\cite{Ra26} and Moise~\cite{Mo52,Mo77} (see also \cite[p.~247]{Ed84}).
The Annulus Theorem was proved for dimensions $\geq 5$ by Kirby~\cite{Ki69}, with a little help from Siebenmann,
and in dimension~$4$ by Quinn~\cite[p.~506]{Qu82}, making use of the main results of \cite{Freedman-82}; see also~\cite[p.~247]{Ed84}.

The known proofs of the Annulus Theorem~\ref{thm:annulus} deduce it from the Stable Homeomorphism Theorem.
In the next chapter we will state the Stable Homeomorphism Theorem~\ref{thm:SHT}
and we will explain the argument, provided in \cite{Brown-Gluck}, showing that the Annulus Theorem~\ref{thm:annulus} can be deduced from the Stable Homeomorphism Theorem~\ref{thm:SHT}.

\section{The  Stable Homeomorphism Theorem}

We reduce the Annulus Theorem to the Stable Homeomorphisms Theorem stated
in Theorem~\ref{thm:SHT}. This follows from work
of Brown and Gluck~\cite{Brown-Gluck}, but since it requires some work to find this deduction in \cite{Brown-Gluck}, we give the details here.

\begin{definition}
Let $n\in \N_0$. A homeomorphism $f \colon \R^n \to \R^n$ is said to be \emph{stable} if there is a sequence of homeomorphisms $f_1,\dots,f_m \colon \R^n \to \R^n$ such that $f_m \circ \cdots \circ f_1 =f$ and such that for each $i$, the homeomorphism~$f_i$ is \emph{somewhere the identity}, which means that there is an open nonempty set $U \subseteq \R^n$ such that $f_i|_U$ is the identity on $U$.
\end{definition}

The key ingredient to the subsequent discussion is the following theorem.

\begin{theorem}\textbf{\textup{(Stable homeomorphism Theorem)}}\label{thm:SHT}
Let $n\in \N_0$.
  Every orientation preserving homeomorphism from $\R^n$ to itself is stable.
\end{theorem}

For $n\geq 5$ this was proven by Kirby~\cite[p.~575]{Ki69} using the torus trick.
Slightly more precisely, Kirby~\cite{Ki69} showed that the Stable Homeomorphism Theorem in dimensions at least five is a consequence of the surgery theoretic classification of $\PL$ homotopy tori,  which was worked out around the same time by Wall~\cite{Wall-hom-tori}, \cite[Section~15A]{Wall-Ranicki:1999-1}
and independently by Hsiang-Shaneson~\cite[p.~688]{HS69}, both proofs building on the work of Browder, Novikov, and Wall, which culminated in Wall's book~\cite{Wall-Ranicki:1999-1}. See also~\cite{Hsiang-Wall-hom-tori-II}.
For $n=4$ the Stable Homeomorphism Theorem was proven by Quinn~\cite{Qu82}, see also \cite[p.~247]{Ed84}.

Before we discuss consequences of the Stable Homeomorphism Theorem~\ref{thm:SHT}
we recall the two versions of the Alexander trick.

\begin{lemma}\textbf{\textup{(Alexander trick)}}\label{lem:alexander-trick}
Let $n\in \N_0$.
\bnm
\item Every homeomorphism of $S^{n-1}$ can be extended radially to a homeomorphism of $D^n$ that sends 0 to 0.
\item Let $f$ and $g$ be two homeomorphisms of $D^n$. If the restrictions of $f$ and $g$ to $S^{n-1}$ are isotopic, then $f$ and $g$ are isotopic homeomorphisms of $D^n$.
\item The topological group $\operatorname{Homeo}_{\partial}(D^n)$ of homeomorphisms of $D^n$ fixing the boundary pointwise is contractible.
\enm
\end{lemma}

\begin{proof}
The extension in the first statement can be obtained by coning: $f(t\cdot x)=t\cdot f(x)$.

The proof of the second is an amusing exercise; see e.g.~\cite[Lemma~5.6]{Hansen} for a proof.
The idea is also similar to the proof of the third statement, which we give now.
To show that $\operatorname{Homeo}_{\partial}(D^n)$ is contractible it suffices to give a homotopy
\[F \colon \operatorname{Homeo}_{\partial}(D^n) \times [0,1] \to \operatorname{Homeo}_{\partial}(D^n)\]
with $F(f,0) = \Id_D^n$ and $F(f,1) = f$, for all $f \in \operatorname{Homeo}_{\partial}(D^n)$.
To do this, we define
\[F(f,t)(x) = \begin{cases} t \cdot f\big(\smfrac{1}{t} \cdot x\big) & \|x\| \leq t \\
x & t \leq \|x\| \leq 1.
\end{cases}\]
We omit the proof that this is well-defined and continuous.
\end{proof}

We can now prove the following almost immediate consequence of the Stable Homeomorphism Theorem~\ref{thm:SHT}.

\begin{corollary}\label{cor:homeo-sn}
Let $n\in \N_0$.
Every orientation preserving self-homeomorphism of $S^n$ is isotopic to the identity.
\end{corollary}

\begin{proof}
We identify $S^n$ with $\R^n\cup \{\infty\}$. Let $h$ be a self-homeomorphism of $S^n=\R^n\cup \{\infty\}$.
By the Stable Homeomorphism Theorem~\ref{thm:SHT} we know that $h$ is stable.
Thus we only have to consider the case that $h$ fixes an open subset of $\R^n\cup \{\infty\}=S^n$, since a composition of homeomorphisms isotopic to the identity is isotopic to the identity. After an isotopy (using Theorem~\ref{thm:isotopy-extension-theorem}) we can assume that $h$ fixes an open neighbourhood of~$\infty$, so in particular there exists $C>0$ such that  $h$ is the identity on $\{ x\in \R^n\mid \|x\|\geq C\}$.
It then follows from Lemma~\ref{lem:alexander-trick} (2) that $h$ is isotopic to the identity.
\end{proof}

Now we being showing how to deduce the annulus theorem from the stable homeomorphism theorem.

\begin{definition}
Let $n\in \N_0$.
We say that two elements~$f_0, f_1 \in \Emb(D^n, \R^n)$ are \emph{intertwined}
if there exists an $h \in \Homeo(\R^n, \R^n)$ with $h \circ f_0 = f_1$.
\end{definition}

 We will need the following straightforward technical lemma. See Definition~\ref{def:LocallyFlat} for the definition of a locally flat embedding.

\begin{lemma}\label{lem:extend-embedding-of-disc}
Let $n\in \N_0$.
Let $M$ be an $n$-dimensional manifold and let $f\colon D^n\to M$ be a locally flat embedding into $\int M=M\sms \partial M$. Then there exists a locally flat embedding $F\colon \R^n\to M$ such that the restriction of $F$ to $D^n$ equals $f$.
\end{lemma}

\begin{proof}
Let $f\colon D^n\to M$ be a locally flat embedding.
By definition $f(D^n)$ is a submanifold of $M$.
It is straightforward to see that $W:=M\sms f(\int D^n)$ is also a submanifold of $M$.
By the Collar Neighbourhood Theorem~\ref{thm:topological-collar} there exists a collar $f(S^{n-1})\times [0,1]$. The map
\[ \ba{rcl} F\colon \R^n&\to & M\\
x&\mapsto & \left\{ \ba{ll} f(x), &\mbox{ if }x\in D^n,\\
\big(f(y), \frac{2}{\pi}\arctan(t-1)\big)&\mbox{ if $x=t\cdot y$  with $t\in [1,\infty)$ and $y\in S^{n-1}$},\ea\right.\ea\]
is easily seen to be a locally flat embedding.
\end{proof}

Denote the set of locally flat embeddings of $D^n$ into $\R^n$ by $\Emb(D^n, \R^n)$.

\begin{lemma}\label{lem:intertwined}
Let $n\in \N_0$.
Any two elements~$f_0, f_1 \in \Emb(D^n, \R^n)$ are intertwined.
\end{lemma}

\begin{proof}
It suffices to show that any  $f \in \Emb(D^n, \R^n)$ is intertwined with the standard embedding~$D^n \subseteq \R^n$.
So let $f\in \Emb(D^n,\R^n)$.
 Apply Lemma~\ref{lem:extend-embedding-of-disc} to extend $f$ to a locally flat embedding  $F\colon D^n_{\frac{3}{2}}\to \R^n$.  Note that $F$ restricts to a locally flat embedding of $S^{n-1}\times [\frac{1}{2},\frac{3}{2}]$ into $S^n=\R^n\cup \{\infty\}$.
Let $\wti{D}^n$ be another copy of $D^n$. By the generalised Schoenflies Theorem~\cite[Theorem~5]{Brown60} there exists a homeomorphism $g\colon \wti{D}^n\to  S^n\sms f(\int {D}^n)$.
Using that the homeomorphisms of $\wti{D}^n$ act transitively on the interior of $\wti{D}^n$, arrange that $g(0)=\infty$.

Note that $g^{-1}\circ f\colon S^{n-1}\to S^{n-1}$ is a homeomorphism. By Lemma~\ref{lem:alexander-trick}~(1) this homeomorphism extends to a homeomorphism $\phi$ of $D^n$. Replace $g$ by $g\circ \phi$ if necessary to obtain that $f=g\colon S^{n-1}\to f(S^{n-1})$.
Identify $S^n=\R^n\cup \{\infty\}= D^n\cup \wti{D}^n$ in such a way that $0\in \wti{D}^n$ corresponds precisely to $\infty$.
Consider the map
\[ \ba{rcl} F\colon S^n=D^n\cup \wti{D}^n &\to & S^n\\
 x&\mapsto & \begin{cases}  f(x)  & x\in D^n\\
 g(x) & x\in \wti{D}^n.
 \end{cases}\ea\]
The maps $f$ and $g$ agree on the overlap, so the map is well-defined and is a homeomorphism. Note that $F$ restricts to a homeomorphism of $\R^n$ which  has the property that the restriction to $D^n$ equals $f$. This shows that $F \circ \id=f$, so  $f$ and the standard embedding are intertwined.
\end{proof}

We continue with the following definition from \cite[p.~19]{Brown-Gluck}.
\begin{definition}
Let $n\in \N_0$.
Let $f_0,f_1\in \Emb(D^n, \R^n)$.
\bnm
\item
We say $f_0$ and $f_1$ are  \emph{strictly annularly equivalent} if
 $f_0(D^n) \subseteq \Int f_1(D^n)$ and if there exists a map $F \colon S^{n-1} \times I \to  \R^n$ that is a homeomorphism onto its image
 such that $F(x,0) = f_0(x)$
and $F(x,1) = f_1(x)$ for all~$x \in S^{n-1}$.
\item
We say $f_0$ and $f_1$ are  \emph{annularly equivalent} if there exists  a sequence
$f_0=g_0,g_1,\dots,g_k=f_1$ of elements of $\Emb(D^n, \R^n)$
such that any two successive $g_i$ are strictly annularly equivalent.
\enm
\end{definition}

\begin{theorem}\label{thm:annuluar-iff-stable}
Let $n\in \N_0$.
Let  $f_0, f_1 \in \Emb(D^n, \R^n)$ with $f_0(D^n) \subseteq \Int f_1(D^n)$.
If $f_0$ and $f_1$  are orientation preserving, then they
are strictly annularly equivalent if and only if they are intertwined.
\end{theorem}

\begin{proof}
If two such elements are strictly annularly equivalent, then they are
intertwined by~\cite[Theorem 5.2]{Brown-Gluck}.

Now suppose that $f_0$ and $f_1$ are intertwined, that is there exists an $h \in \Homeo(\R^n, \R^n)$ with $h \circ f_0 = f_1$.
By the Stable Homeomorphism Theorem~\ref{thm:SHT} we know that $h$ is stable.
Thus we know from \cite[Theorem 5.4]{Brown-Gluck} that the embeddings are annularly equivalent, i.e.\
there  exist $h_0,\dots,h_k\in \Emb(D^n,\R^n)$ such that $h_0=f_0,h_k=f_1$ and for each $i$ the maps $h_i$ and $h_{i+1}$ are strictly annularly equivalent.
 Since $f_0(D^n) \subseteq \Int f_1(D^n)$, the embeddings of the boundary spheres~$f_0(\partial D^n)$ and~$f_1(\partial D^n)$ are disjoint. Therefore it follows from \cite[Theorem 3.5]{Bro64}
that $f_0$ and $f_1$ are not only annularly equivalent, but are moreover strictly annularly equivalent.
\end{proof}

Now we can easily prove the Annulus Theorem~\ref{thm:annulus}.

\begin{proof}[Proof of the Annulus Theorem~\ref{thm:annulus}]
Let $f_0,f_1\colon D^n\to \R^n$ be  two orientation-preserving locally flat embeddings
with  $f_0(D^n)\subseteq \int(f_1(D^n))$.
By Lemma~\ref{lem:intertwined} and
Theorem~\ref{thm:annuluar-iff-stable}
the two maps $f_0$ and $f_1$ are strictly annularly equivalent.
But this  implies that  $f_1(D^n)\sms \int(f_0(D^n))$ is homeomorphic to $S^{n-1} \times [0,1]$.
\end{proof}

\section{The connected sum operation}
\label{section:connected-sum-well-defined}

\begin{definition}
Let $M$ and $N$ be connected nonempty oriented $n$-manifolds.
Pick an orientation preserving locally flat embedding $\Phi_M \colon D^n \to M\sms \partial M$ of an $n$-ball into $M$ and an orientation reversing locally flat embedding $\Phi_N \colon D^n \to N\sms \partial N$ of an $n$-ball into $N$. Define the \emph{connected sum} $M \# N$
of $M$ and $N$ by
\[M \# N := (M \setminus \Phi_M(\int(D^n))) \cup_{\Phi_M(S^{n-1})=\Phi_N(S^{n-1})} (N \setminus \Phi_N(\int(D^n)))\]
where
we glue the left hand side to the  right hand side  via the map
\[ \Phi_N\circ \Phi_M^{-1}\colon \Phi_M(S^{n-1})\,\,\xrightarrow{\cong}\,\, \Phi_N(S^{n-1}).\]
It follows from the
Collar Neighbourhood Theorem~\ref{thm:topological-collar} that
the topological space~$M \# N$ inherits the structure of an $n$-manifold; see \cite[Proposition~6.6]{Le13} for details. Furthermore $M\# N$ can be oriented in such a way that $M \setminus \Phi_M(D^n)$ and $N \setminus \Phi_N(D^n)$ are oriented submanifolds.
\end{definition}

\begin{theorem}\label{thm:connected-sum-well-defined}
Let $n\in \N_0$.
The connected sum $M\# N$ of two connected oriented $n$-manifolds $M$ and~$N$ is independent of
the choice of embeddings of the $n$-balls.
\end{theorem}

\begin{remark}
\mbox{}
\bnm
\item
In
Proposition~\ref{prop:additivity-of-intersection-form}
we will see that the manifolds~$\C\text{P}^2 \# \overline{ \C\text{P}^2 }$ and
$\C\text{P}^2 \# \C\text{P}^2 $ have non-isometric intersection forms, so they are not homeomorphic. Thus connected sum is not well-defined on orientable $4$-manifolds, rather it depends on the choice of orientation.
\item If  at least one of the two manifolds involved is nonorientable, then  the connected sum is well-defined.
This follows from the fact that one can show, say using the orientation cover of a nonorientable manifold, that in a nonorientable connected $n$-manifold $M$ any two locally flat embeddings of $D^n\to M\sms \partial M$ are ambiently isotopic. In the smooth category this argument is worked out in  detail in \cite[Chapter~47.4]{Fr23}. Using the Annulus Theorem~\ref{thm:annulus} one can translate the smooth argument to a topological argument.
\item  As discussed in \cite{BDFHKLN19}, in contrast to the case of orientable 3-dimensional manifolds, orientable 4-dimensional topological manifolds do not admit a unique decomposition as a connected sum of irreducible 4-manifolds. For example $\CP^2 \# S^2 \times S^2$ and $\CP^2 \# \CP^2 \# \ol{\CP}^2$ are diffeomorphic.
\enm
\end{remark}

The proof of Theorem~\ref{thm:connected-sum-well-defined} relies on the following two lemmas. The elementary proof of the first lemma is left to the reader.

\begin{lemma}\label{lem:balls-are-the-same}
Let $D_r^n(x)$ and $D_s^n(y)$ be two Euclidean balls in $\R^n$. There exists an orientation-preserving homeomorphism $f\colon \R^n\to \R^n$ with $f(D_r^n(x))=D_s^n(y)$ such that $f$ is the identity outside of some compact set.
\end{lemma}

The next lemma is a consequence of the Annulus Theorem~\ref{thm:annulus}.

\begin{lemma}\label{lem:application-annulus-theorem}
Let $\varphi,\psi\colon D^n\to \R^n$ be two orientation-preserving locally flat embeddings. If $\varphi (D^n)\subseteq \int(\psi (D^n))$, then there exists an orientation-preserving homeomorphism
$f$ of $\R^n$ with $f(\varphi(D^n))=\psi(D^n)$ such that $f$ is the identity outside of some compact set.
\end{lemma}

\begin{proof}
By the Annulus Theorem~\ref{thm:annulus} and the Collar Neighbourhood Theorem~\ref{thm:collar} we can find a locally flat embedding
$\theta\colon S^{n-1}\times [-1,2]$ such that $\theta(S^{n-1}\times [-1,0])\subseteq \varphi(D^n)$ is an (interior) collar for
$\partial \varphi(D^n)$, such that $\theta(S^{n-1}\times [0,1])=\psi(D^n)\sms \varphi(\int D^n)$ and such that $\theta(S^{n-1}\times [1,2])\subseteq \R^n\sms \psi(\int D^n)$ is an (internal) collar for $\partial (\R^n\sms \psi(\int D^n))$.
It is now obvious that we can find a homeomorphism $f$ with $f(\varphi(D^n))=\psi(D^n)$
which is the identity outside of $\theta(S^{n-1}\times [-1,2])$.
\end{proof}

The subsequent proof is partly based on the sketch given in \cite[p.~42]{Ro90}.

\begin{proof}[Proof of Theorem~\ref{thm:connected-sum-well-defined}] We have to show that the connected sum is independent of the choice of $\Phi_M\colon D^n\to M$ and $\Phi_N\colon D^n\to N$.
In the following we show that the oriented homeomorphism type of the connected sum  is independent of the choice of $\Phi_M$.
Basically the same argument then shows that the oriented homeomorphism type of the connected sum  is independent of the choice of $\Psi_N$. Putting these two orientation-preserving homeomorphisms together gives  independence of all choices.

After this preamble we now show that the oriented homeomorphism type of the connected sum is independent of the choice of $\Phi_M$. So suppose we are given two orientation-preserving embeddings $\Phi_1\colon D^n\to M$ and $\Phi_2\colon D^n\to M$ and suppose we are given an orientation-reversing embedding $\Psi\colon D^n\to N$. For $i=1,2$ we introduce the following notation.
\bnm
\item Write $D_i :=\Phi_i(D^n)$.
\item  Let
$X_i := M\sms \Phi_i(\int D^n)$ and let $Y := N\sms \Psi(\int D^n)$,
\item  Denote the restriction of $\Phi_i$ to $S^{n-1}$ by $\varphi_i$ and denote the restriction of $\Psi$ to $S^{n-1}$ by~$\psi$.
\enm
Figure~\ref{fig:connected-sum-well-defined} hopefully makes it easier for the reader to internalise the notation.
We have to show that there exists a homeomorphism
\[  (X_1\cup Y)/\varphi_1(x)\sim \psi(x)\,\, \to \,\, (X_2\cup Y)/\varphi_2(x)\sim \psi(x)\]
where the gluing on both sides is given by taking $x\in S^{n-1}$.

\begin{figure}[h]
\begin{center}
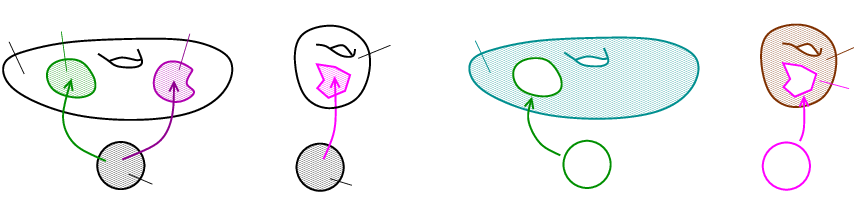
\caption{Illustration for the proof of
Theorem~\ref{thm:connected-sum-well-defined}.}
\label{fig:connected-sum-well-defined}
\end{center}
\end{figure}

\begin{claim}
There exists an orientation-preserving homeomorphism $h$ of $M$ so that $h(D_1)=D_2$.
\end{claim}

To prove the claim, first note that it  follows from  Lemmas~\ref{lem:extend-embedding-of-disc} and~\ref{lem:balls-are-the-same}, together with our hypothesis that $M$ is path connected,
that there exists an orientation-preserving homeomorphism $\mu$ of $M$  such that
$\mu(D_1)\subseteq \int D_2$.
Then apply Lemmas~\ref{lem:extend-embedding-of-disc} and~\ref{lem:application-annulus-theorem} to find an orientation-preserving homeomorphism $\nu$ of $M$ such that $\nu(\mu(D_1))=D_2$.
This concludes the proof of the claim.

After replacing $\varphi_1$ by $h\circ \varphi_1$  we can assume that  $\varphi_2^{-1}\circ \varphi_1$ is an orientation-preserving homeomorphism of $S^{n-1}$.
By Corollary~\ref{cor:homeo-sn} we know that
there exists an isotopy $H\colon S^{n-1}\times [0,1]\to S^{n-1}$
from $\varphi_2^{-1}\circ \varphi_1$ to the identity.

We write $C:=\Psi(S^{n-1})$.
By the Collar Neighbourhood Theorem~\ref{thm:collar} we can pick an (internal) collar $C\times [0,1]\subseteq Y$ for $\partial Y$.
It is straightforward to verify that
\[ \ba{rcl} (X_1\cup Y)/\varphi_1(x)\sim \psi(x)&\to & (X_2\cup Y)/\varphi_2(x)\sim \psi(x)\\
p&\mapsto & \left\{ \ba{ll} h(p), &\mbox{ if }p\in X_1,\\
\psi(H(\psi^{-1}(q),t))&\mbox{ if }p=(q,t)\in C\times [0,1]\\
p,&\mbox{ if }p\in Y\sms ( C\times [0,1])
\ea\right.\ea\]
is a well-defined map and is an orientation-preserving homeomorphism.
This shows that the connected sums defined using $\Phi_1$ and $\Phi_2$ give rise to manifolds of the same oriented homeomorphism type.
\end{proof}

\section{The Product Structure Theorem}\label{section:product-structure theorem}

The Product Structure Theorem \cite[Essay~I, Theorem~5.1, p.~31]{KS77}, is a key result for the development of topological manifold theory in high dimensions.  It is a consequence of the Stable Homeomorphism Theorem~\ref{thm:SHT}, together with a more sophisticated torus trick. The Product Structure Theorem is used in \cite{KS77} to deduce the existence of handle structures for manifolds of dimension $n \geq 6$, transversality and smoothing theory for $n \geq 5$, and the existence of a canonical simple homotopy type for all $n$. We will give some examples of the use of the Product Structure Theorem, for instance in Section~\ref{section:simple-homotopy-type} on the simple homotopy type of topological manifolds.
Even though the Product Structure Theorem a priori only concerns high dimensional manifolds, it still appears in the development of the theory of 4-manifolds.

The Product Structure Theorem will be stated for upgrading either a smooth or PL structure on $M \times \R$ to one on $M$.

A \emph{concordance} of (smooth, PL) structures $\Sigma, \Sigma'$ on a manifold $N$ is a (smooth, PL) structure $\Omega$ on $N \times I$ that restricts to $\Sigma$ on $N \times \{0\}$ and restricts to $\Sigma'$ on $N \times \{1\}$.

\begin{theorem}\textbf{\textup{(Product Structure Theorem)}}\label{theorem:product-structure}
  Let $M$ be a manifold of dimension $n \geq 5$. Let $\Sigma$ be a $($smooth, PL$)$ structure on $M \times \R^s$, with $s \geq 1$.  Let $U$ be an open subset of $M$ with a $($smooth, PL$)$ structure $\rho$ on $U$ such that $\rho \times \R^s = \Sigma|_{U \times\R^s}$. If $n=5$ then suppose that $\partial M \subseteq U$.

  Then there is a $($smooth, PL$)$ structure $\sigma$ on $M$ extending $\rho$, together with a concordance of (smooth, PL) structures from $\Sigma$ to $\sigma \times \R^s$, that is a product concordance in some neighbourhood of $U \times \R^s$ and that is a product near $M \times \R^s \times \{i\}$ for $i=0,1$.
\end{theorem}

\begin{remark}
The statement of the Product Structure Theorem was modelled on the Cairns-Hirsch Theorem~\cite[Essay~I, Theorem~5.3, p.~37]{KS77}, which was proven in the early 1960s, and provided the analogous upgrade from PL structures to smooth structures. See \cite{Hirsch-Mazur74} for a comprehensive treatment of smoothing theory for PL manifolds.  The Cairns-Hirsch Theorem tells us that if $M$ already has a PL structure $\varpi$, such that $\varpi \times \R^s$ is \emph{Whitehead compatible} (see the discussion below \cite[Essay~I, Theorem~5.3, p.~37]{KS77} for details) with a smooth structure $\Sigma$ on $M \times \R^s$, then the smooth structure $\sigma$ on $M$ produced by Theorem~\ref{theorem:product-structure} is Whitehead compatible with $\varpi$.
\end{remark}

In Section~\ref{section:simple-homotopy-type} on the simple homotopy type of a manifold we make use of the following stronger local version~\cite[Essay I, Theorem 5.2, p.~36]{KS77}.

\begin{theorem}\textbf{\textup{(Local Product Structure Theorem)}}\label{theorem:local-product-structure}
  Let $M$ be a manifold of dimension $n \geq 5$.
  \begin{enumerate}[leftmargin=1cm,font=\normalfont]
   \item[(i)] Let $W$ be an open neighbourhood of $M \times \{0\}$ in $M \times \R^s$, for some $s \geq 1$.
   \item[(ii)] Let $\Sigma$ be a $($smooth, PL$)$ structure on $W$.
   \item[(iii)] Let $C \subseteq M \times \{0\}$ be a closed subset such that there is a neighbourhood $N(C)$ of $C\subseteq W$ on which the $($smooth, PL$)$ structure $\Sigma$ is a product $\Sigma|_{N(C)} = \sigma \times \R^s$ for some $($smooth, PL$)$ structure $\sigma$ on $N(C)$.   If $n=5$ then suppose that $\partial M \subseteq C$.
   \item[(iv)]  Let $D \subseteq M \times \{0\}$ be another closed subset.
   \item[(v)] Let $V \subseteq W$ be an open neighbourhood of $D \setminus C$ in $M \times \R^s$.
    \end{enumerate}
  Then we have the following.
  \begin{enumerate}[leftmargin=1cm,font=\normalfont]
   \item A $($smooth, PL$)$ structure $\Sigma'$ on $W$ that equals $\Sigma$ on $(W \setminus V) \cup ((C \times \R^s) \cap W)$ and is a product $($smooth, PL$)$ structure $\rho \times \R^s$ on $(N(D) \times \R^s) \cap W$ for some neighbourhood $N(D)$ of $D$ and for some $($smooth, PL$)$ structure $\rho$ on $N(D)$.
   \item A concordance of $($smooth, PL$)$ structures from $\Sigma$ to $\Sigma'$, that is a product concordance on some neighbourhood of $(W \setminus V) \cup ((C \times \R^s) \cap W)$ and that is a product near $W \times \{i\}$ for $i=0,1$.
   \end{enumerate}
  \end{theorem}

Note that the Concordance implies Isotopy Theorem \cite[Essay~I,~Theorem~4.1, p.~25]{KS77} means that the concordances in Theorems~\ref{theorem:product-structure} and~\ref{theorem:local-product-structure} can be upgraded to isotopies of (smooth, PL) structures under the same hypotheses on dimensions, that is if  $n \geq 6$ \emph{or} if $n=5$ and the structures already agree on $\partial M$.

We start with a structure on $W$ that is a product structure over $C$. We obtain a concordance of structures to a product structure on $D$, supported in $V$, and that is a product concordance over $C$. We therefore have a product structure in some neighbourhood of $C \cup D$.

\chapter{Tubular neighbourhoods}\label{chapter:tubular}

Every smooth submanifold of a smooth manifold admits a normal vector bundle and, by the smooth Tubular Neighbourhood Theorem, also admits a tubular
neighbourhood~\cite[Sections 5~\&~6]{Hi76}, \cite[Chapter~2.5]{Wall16}. However, in the topological category submanifolds may not admit normal vector bundles, a general problem we discuss further below and in Chapter \ref{chapter:bundlestructures} once we have developed the necessary language. Curiously, in the special case of~$4$-manifolds these general problems do not exist, and familiar smooth results hold true using an appropriate notion of normal vector bundles (Definition \ref{def:extendable}).


\section{Tubular neighbourhoods: existence and uniqueness}
In the literature one can find many different definitions of tubular neighbourhoods for smooth submanifolds. We will give a definition for manifolds that is modelled on the definition provided by Wall~\cite{Wall16} for smooth manifolds. To do so we first need one extra definition.

\begin{definition}
Let $M$ be an $n$-dimensional manifold. We say a subset $W\subseteq M$ is a \emph{$k$-dimensional sub\-mani\-fold with corners} if given any $p\in W$ there exists a chart of the type (1), (2) or (3) as in Definition~\ref{def:submanifold} above,
or if
\bnm
\item[(4)] there exists a chart $\Phi\colon U\to V$ of type (ii) for $M$ such that
\[  \Phi(U\cap W)\,\,\subseteq \,\,\{(0,\dots,0,x_1,\dots,x_k)\mid x_i\in \R\mbox{ with }x_{k-1}\geq 0\mbox{ and } x_k\geq 0\}\]
and with $\Phi(p)\in \{(0,\dots,0,x_1,\dots,x_{k-2},0,0)\mid x_1,\dots,x_{k-2}\in \R\}$.
\enm
If $W$ is an $n$-dimensional submanifold with corners we write
\[ \partial_0W\,\,:=\,\, W\cap \overline{M\sms W},\quad \partial_1W\,\,:=\,\, W\cap \partial M,\]
and we note that
\[
\int W\,\,=\,\, W\sms \partial_0 W.
\]
\end{definition}

\begin{figure}[h]
\begin{center}
\includegraphics{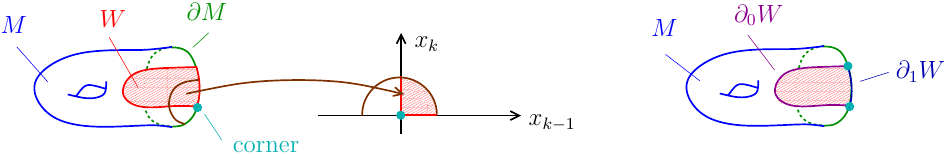}
\caption{Definition of $\partial_0$ and $\partial_1W$ of a submanifold.}
\label{fig:submanifold-with-corner}
\end{center}
\end{figure}

\begin{remark}\label{rem:complementcorrner}
The complement of the interior of a codimension 0 submanifold with corners is a submanifold with corners, with the obvious corners.
\end{remark}

\begin{definition}
Let $M$ be an $n$-manifold and let $X$ be a compact, proper, $k$-dimensional submanifold.
A \emph{tubular neighbourhood} for $X$ is a pair $(N,p\colon N\to X)$ with the following properties:
\bnm
\item $N$ is a neighbourhood of $X$;
\item $N$ is a codimension zero sub\-mani\-fold with corners of $M$;
\item the map $p\colon N\to X$ is a linear $D^{n-k}$-bundle such that $p(x)=x$ for all $x\in X$;
\item $\partial_1 N=p^{-1}(\partial X)$.
\enm
\end{definition}

Here linear means that there exists an atlas of local trivialisations such that the transition maps take values in $\op{O}(n-k)$ instead of $\operatorname{Homeo}(D^{n-k})$.

\begin{remark}\label{remark:Hirsch-no-tub-nbhd-example}
In the topological category, tubular neighbourhoods do not always exist. Indeed it is shown in~\cite[Theorem~4]{Hi68} that there exists a $4$-dimensional submanifold of $S^7$ that does not admit a tubular neighbourhood.
\end{remark}

Fortunately, for submanifolds of $4$-manifolds, tubular neighbourhoods exist and they are unique in the appropriate sense.

\begin{theorem}\textbf{\textup{(Tubular Neighbourhood Theorem)}}\label{thm:tubular-neighbourhood}
Every compact proper submanifold $X$ of a $4$-manifold $M$ admits a tubular neighbourhood.
\end{theorem}

\begin{theorem}\textbf{\textup{(Uniqueness of tubular neighbourhoods)}}\label{thm:uniqueness-of-general-tubular-neighbourhoods}
Let~$M$ be a $4$-manifold and let $X$ be a compact proper $k$-dimen\-si\-on\-al submanifold. Furthermore let $p_i\colon N_i\to X$, $i=1,2$  be two tubular neighbourhoods of $X$, with inclusion maps $\iota_i \colon N_i\to M$.
Then there exists an isomorphism $\Psi\colon N_1\to N_2$ of linear disc bundles such that
$\iota_2\circ \Psi\colon N_1\to M$ and $\iota_1\colon N_2\to M$ are ambiently isotopic rel.~$X$.
\end{theorem}

The proofs of the above two theorems rely on the existence and uniqueness results for normal vector  bundles in \cite[Section 9]{FQ90}, which we discuss further in Section \ref{subsec:normalvector}. Thus we postpone the proofs of the Theorems~\ref{thm:tubular-neighbourhood} and~\ref{thm:uniqueness-of-general-tubular-neighbourhoods} to Section~\ref{section:tubular-nhd-proofs}.
Right now, let us first observe some nice consequences of the existence and uniqueness of tubular neighbourhoods.

\begin{remark}
Let $X$ be a compact proper submanifold of a $4$-manifold $M$.
By Theorem~\ref{thm:tubular-neighbourhood} we can pick a tubular neighbourhood $p\colon N\to X$. We refer to $E_X:=M\sms \int N$ as the \emph{exterior of $X$}. Note $E_X\subseteq M$ is a submanifold with corners; cf.~Remark~\ref{rem:complementcorrner}.
By Theorem~\ref{thm:uniqueness-of-general-tubular-neighbourhoods} the homeomorphism type of the exterior is well-defined.
\end{remark}

\begin{lemma}\label{lem:exterior-deformation-retract}
Let $X$ be a compact proper submanifold of a $4$-manifold $M$.
The exterior $E_X$ of $X$ is a deformation retract of the complement $M\sms X$.
\end{lemma}

\begin{proof}
Let $p\colon N\to X$ be a tubular neighbourhood for $X$. Using the fact that $p$ is a \emph{linear} bundle, introduce compatible radial coordinates in the fibres and isotope $N\sm X$ radially outwards.
This implies that $\partial_0 N$ is a deformation retract of $N\sms X$. But this also implies that the exterior $E_X=M\sms N$  is a deformation retract of $M\sms X$.
\end{proof}

\begin{corollary}
Let $X$ be a proper submanifold of a compact $4$-manifold $M$.
If $X$ is compact, then the fundamental group of each component of $M\sms X$ is finitely generated, and $H_*(M\sms X)$ is finitely generated.
\end{corollary}

\begin{proof}
It follows from Lemma~\ref{lem:exterior-deformation-retract} that $M\sms X$ is homotopy equivalent to the exterior $E_X$ of $X$. Each connected component of $E_X$ is a  compact 4-manifold since we assume that $M$ is compact.  The corollary is now a consequence of  Corollary~\ref{cor:topological-mfd-homology}.
\end{proof}

\begin{proposition}\label{prop:trivial-normal-bundle}
Let $X \subseteq M$ be a compact, proper $2$-dimensional orientable submanifold of a compact, orientable $4$-manifold $M$, such that each connected component of $X$ has nonempty boundary. Then the tubular neighbourhood of Theorem~\ref{thm:tubular-neighbourhood} is homeomorphic to $X \times D^2$.
\end{proposition}

\begin{proof}
Connected surfaces with nonempty boundary are homotopy equivalent to wedges of circles. Every orientable linear $D^n$-bundle over a space $X$ is classified up to isomorphism by a homotopy class of maps $X\to \op{BSO}(n)$. As $\op{SO}(n)$ is connected, so $\op{BSO}(n)$ is simply connected, and any orientable linear disc bundle over $X$ is trivial. In particular, this is true when $n=2$.
\end{proof}

\section{Normal vector bundles}\label{subsec:normalvector}
The reader will be familiar with the definition of a normal vector bundle when working in the smooth category: if $X\subseteq M$ is a smooth submanifold of a smooth manifold, then the normal vector bundle is defined as  the quotient of the vector bundle $TM|_X$ by the subbundle~$TX$. This definition uses the smooth structure to ensure the existence of tangent vector bundles, and vector bundles are a strong enough bundle technology to ensure the existence
of quotient bundles. While some (weaker) canonical tangential structures do exist in the topological category (see Chapter~\ref{chapter:bundlestructures}), the idea of a `quotient bundle' no longer makes sense for them.

In the topological category, following \cite[Section 9]{FQ90}, we will use a definition of normal vector bundle that is much closer to the geometry of tubular neighbourhoods. We begin with a definition that is almost what we need but suffers from a slight technical problem, which we then remedy.

\begin{definition}
Let $M$ be an $n$-manifold and let $X$ be a proper $k$-dimensional submanifold. An \emph{internal linear bundle over $X$} is a pair $(E,p\colon E\to X)$ with the following properties.
\bnm
\item $E$ is a neighbourhood of $X$;
\item $E$ is a codimension zero sub\-mani\-fold of $M$;
\item the map $p\colon E\to X$ is an $(n-k)$-dimensional vector bundle such that $p(x)=x$ for all $x\in X$;
\item $\partial E=p^{-1}(\partial X)$.
\enm
\end{definition}

An internal linear bundle $(E,p\colon E\to X)$ is intended to mirror the notion, from the smooth category, of an \emph{open} tubular neighbourhood of $X$. As such, the definition as stands suffers from the potential technical problem that the closure of $E$ in $M$, which should be a closed tubular neighbourhood, may no longer be a submanifold; see Figure~\ref{fig:non-extendable}. As in \cite[p.~137]{FQ90}, we use the following additional idea to rule out this problem.

\begin{definition}\label{def:extendable}
Let $M$ be an $n$-manifold, let $X$ be a proper $k$-dimensional submanifold, and let $(E,p\colon E\to X)$ be an internal linear bundle over $X$. Suppose that given any $(n-k)$-dimensional vector bundle $(F,q\colon F\to X)$, any radial homeomorphism from an open convex disc bundle of $F$ to $E$ can be extended to a homeomorphism from the whole of $F$ to a neighbourhood of $E$. Then we say $(E,p\colon E\to X)$ is \emph{extendable}.
\end{definition}

\begin{figure}[h]
\begin{center}
\includegraphics{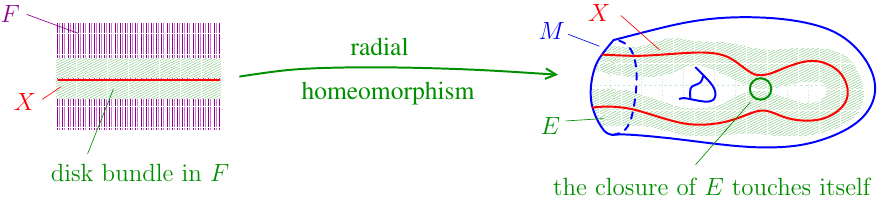}
\caption{Non-extendable internal linear bundle.}
\label{fig:non-extendable}
\end{center}
\end{figure}

Now we can define the notion of a normal vector bundle.

\begin{definition}
Let $M$ be a  $n$-manifold and let $X$ be a proper $k$-dimensional submanifold. A \emph{normal vector bundle} for $X$  is an internal linear bundle that is extendable.
(Note that the same concept is called a \emph{normal bundle} on
\cite[p.~137]{FQ90}, we prefer the name normal \emph{vector} bundle.)
\end{definition}

\begin{theorem}\textbf{\textup{(Existence of normal vector bundles)}}\label{thm:existnormal}
Every proper submanifold of a compact 4-manifold admits a normal vector bundle.
\end{theorem}

\begin{remark}
Generally, in high dimensions, the existence of normal vector bundles is peculiar to when the submanifold has low dimension or low codimension. We refer the reader to \cite[Section 9.4]{FQ90} for a discussion of the other known situations where these objects always exist. Here is a summary of the known cases.
A submanifold of dimension at most 3 in a closed manifold of dimension at least 5 has a normal vector bundle \cite[p.~150]{FQ90}. Codimension one submanifolds have normal vector bundles~\cite[Theorem~3]{Brown62}. That every codimension two submanifold of a manifold of dimension not equal to four has a normal vector  bundles was shown in~\cite{KS75}, and this was extended to include dimension four in \cite[Section~9.3]{FQ90}.
It is striking that, while among smooth manifolds dimension 4 exhibits worse than usual behaviour, in the topological category the existence of normal vector bundles seems to show that in this respect it is among the better behaved of the dimensions.
\end{remark}

For the proof of Theorem~\ref{thm:existnormal} we will essentially appeal to results of \cite{FQ90}.  We reproduce these results here for the benefit of the reader.

\begin{theorem}\label{thm:FQ-exist-normal}
  Let $N$ be a proper submanifold of a $4$-manifold $M$, with a closed subset $\partial N\subseteq K \subseteq N$ and a normal  vector bundle over some neighbourhood $U$ of $K$ in $N$.  Then there is a normal vector bundle over $N$ that agrees with the given one over some neighbourhood $V \subseteq U$ of $K$.  Moreover this extension is unique up to ambient isotopy relative to some neighbourhood $W \subseteq V$ of $K$.
  \end{theorem}

\begin{proof}[Proof of Theorem~\ref{thm:FQ-exist-normal}]
The existence statement of the theorem follows immediately from \cite[Theorem 9.3A]{FQ90} and
the uniqueness statement of the theorem follows immediately from ~\cite[Theorem 9.3D]{FQ90}.
\end{proof}

\begin{proof}[Proof of Theorem~\ref{thm:existnormal}]
Let $X$ be a proper submanifold of a compact 4-manifold~$M$.
The case that $X$ has no boundary follows immediately from
the existence statement of Theorem~\ref{thm:FQ-exist-normal} with $K=\emptyset$.
The case that $X$ has nonempty boundary follows also from Theorem~\ref{thm:FQ-exist-normal} if we apply more care. We sketch the argument.

First it is well-known that
given any pair $(X,\Sigma)$ where $X$ is a 3-manifold and $\Sigma$ is a proper submanifold,
then there exists a smooth structure on $X$ such that $\Sigma$ is a smooth submanifold.
Even though this fact is well-known and often used, it is hard to give complete references.
The existence of a smooth structure on $X$ follows from  \cite[p. 252 and 253]{Mo77}  and \cite[
Theorems 6.2 and 6.3]{Munkres60}. The extra complication of having a proper submanifold is
partly taken care of by
\cite[Theorem~XVIII.4.B]{Bin1983}.
Thus we can view the submanifold $\partial X\subseteq \partial M$ as a  smooth submanifold. Hence it has a smooth normal vector bundle, see  e.g.\ \cite[Chapter~III.2]{Kosinski} or \cite[Section~IV.5]{Lang}.

Next use the Collar Neighbourhood Theorem \ref{thm:collar-boundary} to obtain a collar $\partial M \times [0,1]\subseteq M$ that restricts to a collar $\partial X \times [0,1]$ for the boundary of $X$. Extend the smooth tubular neighbourhood of~$\partial X\subseteq \partial M$ into the collar by taking a product with $[0,1]$.

Finally, consider the 4-manifold without boundary $M':=M\sm (\partial M\times[0,\frac{1}{2}])$. What remains of $X$ is a submanifold $N:=X\sm (\partial X\times [0,\frac{1}{2}])$.
The submanifold $N$ already has a preferred normal vector bundle on the closed subset
$K := \partial X \times (1/2,1]$.
Now apply Theorem~\ref{thm:FQ-exist-normal} to  the triple~$(M', N, K)$ to obtain a normal vector bundle $E\to N$ agreeing with the given one on $K$. The normal vector bundles over $N$ and $\partial X\times [0,1]$ agree on the overlap $K$. Thus they define a normal vector bundle on all of $X$.
\end{proof}

\begin{figure}[h]
\begin{center}
\includegraphics{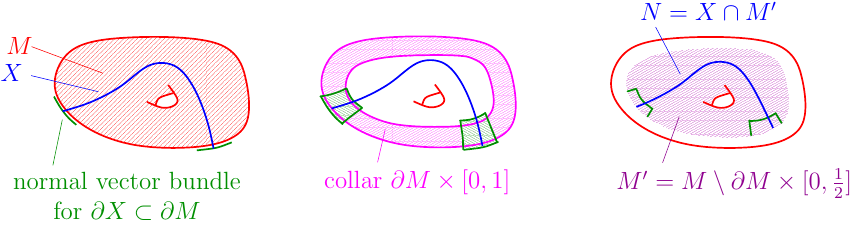}
\caption{Illustration of the proof of Theorem~\ref{thm:existnormal}.}
\label{fig:existnormal}
\end{center}
\end{figure}

Next we turn to the uniqueness of normal vector bundles.

\begin{theorem}\textbf{\textup{(Uniqueness of normal vector bundles)}}\label{thm:uniquenormal}
Let $M$ be a compact 4-manifold and let $X$ be a proper submanifold of $M$.
Suppose we are given two normal vector bundles $p_i\colon E_i\to X$, $i=1,2$ over $X$.
For $i=1,2$ let $\iota_i\colon E_i\to M$ be the inclusion map.
Then there exists a bundle isomorphism $f\colon E_1\xrightarrow{\cong} E_2$ such that $\iota_2\circ f$ and $\iota_1$ are ambiently isotopic rel.\  $X$.
\end{theorem}

\begin{proof}
If $X$ has no boundary, then the theorem is an immediate consequence of Theorem~\ref{thm:FQ-exist-normal}.
Now suppose that~$X$ has nonempty boundary.

First we claim that any normal vector bundle of $X$ is obtained by the construction outlined in the proof of Theorem~\ref{thm:existnormal}. To see this, let $p\colon E\to X$ be a normal vector bundle. Pick a collar neighbourhood $\partial X\times [0,2]\subseteq X$.  Since $p$ is extendable, we can view $p$ as the interior of a disc bundle $q\colon F\to X$ in $M$. Write $C:=q^{-1}(\partial X) \subseteq \partial M$. The disc bundle $q\colon q^{-1}(\partial X\times [0,2])\to \partial X\times [0,2]$ defines a collar neighbourhood $C\times [0,2]$ for the compact submanifold  $C$ of $\partial M$. By the Collar Neighbourhood Theorem~\ref{thm:topological-collar} we can extend the collar neighbourhood $C\times [0,1]$  of $C$ to a collar neighbourhood $\partial M\times [0,1]$. With this choice of collar neighbourhood of $\partial M$, the construction in the proof of Theorem~\ref{thm:existnormal}, with further appropriate choices, gives rise to the normal vector bundle $p\colon E\to X$.  This completes the proof of the claim.

After this long preamble it suffices to prove the theorem for  any two normal vector bundles obtained as in the proof of Theorem~\ref{thm:existnormal}.
Uniqueness follows by arguing that each step in the proof of existence of normal vector bundles was essentially unique. The proofs of uniqueness in the three steps make use of the following ingredients.

First, apply the uniqueness statement for normal vector bundles of submanifolds of smooth manifolds to $\partial X \subseteq \partial M$ e.g.\ \cite[Chapter~III.2]{Kosinski} or \cite[Section~IV.5]{Lang}.

Next use the uniqueness of collar neighbourhoods as formulated in Theorem~\ref{thm:topological-collar-unique}, applied to the two collar neighbourhoods of $\partial M$ subordinate to the given normal  vector bundles of~$X$.

Finally apply the full relative version of Theorem~\ref{thm:FQ-exist-normal} to extend the normal vector bundle uniquely over the rest of ~$X$.
\end{proof}

\section{Tubular neighbourhoods: proofs}\label{section:tubular-nhd-proofs}
Now we will use the results from the previous section to prove
Theorems~\ref{thm:tubular-neighbourhood}
and~\ref{thm:uniqueness-of-general-tubular-neighbourhoods}, i.e.\ we will prove the existence and uniqueness of tubular neighbourhoods.
First we show how one can obtain tubular neighbourhoods from normal vector bundles.

\begin{definition}
\mbox{}
\bnm
\item A \emph{form} over a real vector space $V$ is an $\R$-bilinear symmetric map $g\colon V\times V\to \R$. It is called \emph{positive definite} if for every $v\in V\sms \{0\}$ we have $g(v,v)>0$.
\item
Let $p\colon E\to X$ be a vector bundle over a topological space $X$. Given $x\in X$, write $E_x:=p^{-1}(x)$.
A \emph{positive definite form} $g=\{g_x\}_{x\in X}$ consists of a positive definite form $g_x$ over every  $E_x$ such that $g_x$ changes continuously with $x$.
\enm
\end{definition}

\begin{lemma}\label{lem:tubular-nhds-in-vector-bundles}
Let $X$ be a compact manifold and let  $p\colon E\to X$ be  an $n$-dimensional  vector bundle.
Then the space of positive definite forms on $E$ is nonempty and convex.
Furthermore, let $g=\{g_x\}_{x\in X}$ be a positive definite form on $E$ and consider the map
\[  p\colon E(g)\,:=\, \bigcup\limits_{x\in X}{} \{ v\in E_x\mid g_x(v,v)\leq 1\}\,\,\to \,\, X.
\]
This map  has the following properties.
\bnm
\item The  map $p|_{E(g)}\colon E(g)\to X$ is a linear $D^n$-bundle.
\item Given two different positive definite forms $g$ and $h$ on $E$ there
exists a continuous map $H\colon E\times [0,1]\to E$ with the following properties.
\bnmt
\item We have $H_0=\id$.
\item The map $H_1$ restricts to an isomorphism $H_1\colon E(g)\to E(h)$ of linear $D^n$-bundles.
\enm
\enm
\end{lemma}

\begin{proof}
\mbox{}
\bnm
\item This statement can be proved easily using the observation that the set of positive definition $\R$-bilinear symmetric forms on a real vector space is nonempty and convex.
\item  The proof of this statement follows from the same argument as \cite[Lemmas~2.5.2 and 2.5.4]{Wall16}.\qedhere
\enm
\end{proof}

Let $X$ be a compact manifold and let  $p\colon E\to X$ be  a  vector bundle. Given a positive definite form $g$ on $E$ we refer  to
\[  p\colon  E(g)\,\,:=\,\,  \bigcup\limits_{x\in X}{} \{ v\in E_x\mid g_x(v,v)\leq 1\}\,\,\to \,\, X
\]
as a \emph{corresponding disc bundle}. It follows from Lemma~\ref{lem:tubular-nhds-in-vector-bundles} that for most purposes the precise choice of $g$ is irrelevant.

We can now prove the existence of tubular neighbourhoods.

\begin{proof}[Proof of the Tubular Neighbourhood Theorem~\ref{thm:tubular-neighbourhood}]
Let $X$ be a compact proper submanifold of a $4$-manifold $M$.
By Theorem~\ref{thm:existnormal} there exists a normal vector bundle $p\colon N\to X$ for $X$.
By Lemma~\ref{lem:tubular-nhds-in-vector-bundles} (1)
there exists a corresponding disc bundle. Using the uniqueness statement of
Lemma~\ref{lem:tubular-nhds-in-vector-bundles} (2)
locally one can show that this disc bundle  is a submanifold with corner and a tubular neighbourhood.
\end{proof}

The uniqueness proof for tubular neighbourhoods also requires us to associate a normal vector bundle to a tubular neighbourhood.

\begin{lemma}\label{lem:tubular-implies-normal-vector-bundle}
Let~$M$ be a compact $4$-manifold and let $X$ be a compact proper $k$-dimen\-si\-on\-al submanifold. Let $p\colon N\to X$ be a tubular neighbourhood for $X$. There exists a normal vector bundle $q\colon E\to X$ and a positive definite form $g$ such that $N=E(g)$ and $p\colon N\to X$ equals $q\colon E(g)\to X$.
\end{lemma}

We call $q \colon E \to X$ a \emph{corresponding normal vector bundle}.

\begin{proof}
Let $p\colon N\to X$ be a tubular neighbourhood for $X$. Recall that we have $\int N=N\sms \partial_0 N$. Consider
$W:=M\sms \int N$. This is a compact  $4$-manifold.
Pick a collar neighbourhood $\partial W\times [0,1]$ and set $E:=N\cup (\partial_0 N\times [0,\frac{1}{2}))$. We have an obvious projection map $q\colon E\to X$ turning $q$ into a bundle map where the fibre is given by the open $(4-k)$-ball of radius $\tmfrac{3}{2}$. We leave it to the reader to turn $q\colon N\to X$ into an internal linear bundle, to show that it is in fact extendable (at this point one has to use that in the definition of $E$ we only used ``half'' of the collar neighbourhood $\partial_0 N \times [0,1]$),  and to equip $N$  with a positive definite form $g$ such that $N=E(g)$.
\end{proof}

We conclude the chapter with the proof of the uniqueness theorem for tubular neighbourhoods.
\begin{proof}[Proof of Theorem~\ref{thm:uniqueness-of-general-tubular-neighbourhoods}]
Let~$M$ be a $4$-manifold and let $X$ be a compact proper $k$-dimen\-si\-on\-al submanifold. Furthermore let $p_i\colon N_i\to X$, $i=1,2$   be two tubular neighbourhoods of $X$.
For $i=1,2$, let $q_i\colon E_i\to X$ be two corresponding normal vector bundles and let $g_i$ be the positive definite forms provided by Lemma~\ref{lem:tubular-implies-normal-vector-bundle}.
It follows from Theorem~\ref{thm:uniquenormal} that  there exists a bundle isomorphism $f\colon E_1\xrightarrow{\cong} E_2$ such that $\iota_2\circ f$ and $\iota_1$ are ambiently isotopic rel.~$X$. It follows from the definitions that $N_2$ is equivalent to the
disc bundle defined by $f^*g_2$ on $E_1$.
It follows from the
Lemma~\ref{lem:tubular-nhds-in-vector-bundles} (2)
together with the
Isotopy Extension Theorem~\ref{thm:isotopy-extension-theorem} that $f^*g_2$ and $g_1$ define equivalent tubular neighbourhoods. (Strictly speaking we did not formulate the Isotopy Extension Theorem~\ref{thm:isotopy-extension-theorem} for submanifolds with corner, but it is not difficult to prove a generalization.)
\end{proof}

\chapter{Background on bundle structures}\label{chapter:bundlestructures}
In this chapter we recall the bundle technologies we will need to use in later chapters. The three standard manifold categories smooth ($\Diff$), piecewise linear ($\PL$) and topological ($\TOP$) each have corresponding bundle types, with fibre $\R^n$. We first discuss the topological groups $\TOP(n)$ and $\O(n)$ which are structure groups for fibre bundles with fibre $\R^n$, corresponding to the $\TOP$ and $\Diff$ categories. There are natural topologies on $\TOP(n)$ and $\O(n)$, so that it is relatively straightforward to discuss the classification of these bundle types. We next discuss the $\R^n$-bundles that correspond to the $\PL$ category. In contrast to $\TOP(n)$ and $\O(n)$ there is no obvious appropriate topology on the group of piecewise-linear homeomoprhisms of $\R^n$ that fix the origin. For this reason, in the $\PL$ category it appears there is no choice but to delve into a more sophisticated approach to classify the bundles of interest, and we approach this via simplicial groups. We finish the section with a discussion of topological \emph{microbundles}.

\section{Topological, smooth, and piecewise linear $\R^n$-bundles}
Before we turn to the different flavours of $\R^n$-bundles let us first recall some general facts about bundles.

\begin{definition}
Let $G$ be a topological group $G$. A \emph{universal principal $G$-bundle} is a principal $G$-bundle $p\colon \up{E}G\to\up{B}G$ such that the following two conditions are satisfied.
\bnm
    \item The space $\up{B}G$ is a CW complex.
    \item Given any principal $G$-bundle $q\colon F\to C$, where $C$ is a topological space that is homotopy equivalent to a CW complex,
there exists a map $f\colon C\to \up{B}G$, unique up to homotopy, such that $q$ is isomorphic to the pullback bundle $f^*\up{E}G$.
\enm
The base space $\up{B}G$ is called a \emph{classifying space for $G$}.
\end{definition}

\begin{proposition}\label{prop:bg}
Given a topological group $G$, there exists a
universal principal $G$-bundle  $p\colon \up{E}G\to\up{B}G$.
This principal bundle is unique up to fibre homotopy equivalence.
\end{proposition}

\begin{proof}
The  Milnor join construction~\cite{Miln1956c}, \cite{Hus1993} or \cite{Stas1971}, or alternatively the geometric bar construction~\cite{May-classifying-spaces}, gives an explicit principal $G$-bundle $p\colon E\to B$
which has the universal property for all numerable principal $G$-bundles. Since principal $G$-bundles over CW complexes are numerable, the Milnor bundle has the universal property for principal $G$-bundles over CW complexes.

The topological space $B$ is not necessarily a CW complex.
Thus we apply the CW approximation theorem to get a CW complex $\wti{B}$ and a weak homotopy equivalence $\varphi\colon \wti{B}\to B$. The pullback bundle $\varphi^*E$ over $\wti{B}$ has the desired properties.

The uniqueness of universal principal $G$-bundles over CW complexes follows from a standard argument about universal objects.
\end{proof}

\begin{definition}\label{defn:TOP}
Given $n\in \N_0$ let $\TOP(n)$ be the subgroup of homeomorphisms of $\R^n$
that fix the origin, topologised using the compact open topology. A principal $\TOP(n)$-bundle
has an associated fibre bundle with fibre $\R^n$ and a preferred 0-section.
Call such a bundle a \emph{topological $\R^n$-bundle}.
Let $\TOP$ be the colimit $\colim \TOP(n)$, in the category of topological groups, under the inclusions
\[ \ba{rcl}
\TOP(n) &\to &\TOP(n+1)\\
(f\colon \R^n\to \R^n)&\mapsto & (f\times \Id_{\R}\colon \R^{n}\times \R\to \R^{n}\times \R).
\ea\]
We obtain the corresponding classifying spaces $\BTOP(n)$ and $\BTOP$.
\end{definition}

\begin{definition}
Given $n\in \N_0$ let $\O(n)$ be the orthogonal homeomorphisms of $\R^n$ that fix the origin, topologised in the standard way as a subspace of a vector space
$\O(n)\subseteq \op{M}(n\times n,\R)\cong\R^{n^2}$.
A principal $\O(n)$-bundle
has an associated fibre bundle with fibre $\R^n$ and a preferred 0-section, and such a bundle is in particular a vector bundle.
Define $\O$, $\operatorname{BO}(n)$ and $\operatorname{BO}$ analogously to Definition~\ref{defn:TOP}.
\end{definition}

We also introduce the following structure group, for comparison, and use in Chapter~\ref{chapter:SW-classes}.

\begin{definition}
Given $n\in \N_0$ let
    $\Diff(n)$ be the subgroup of diffeomorphisms of $\R^n$
that fix the origin, topologised using the weak $C^\infty$ topology~\cite[\textsection~2.1]{Hi76} (also called the (weak) \emph{Whitney} topology). A principal $\Diff(n)$-bundle
has an associated fibre bundle with fibre $\R^n$ and a preferred 0-section.
Call such a bundle a \emph{$\Diff$ $\R^n$-bundle}. Define $\Diff$, $\operatorname{BDiff}(n)$ and $\operatorname{BDiff}$ analogously to Definition~\ref{defn:TOP}.
\end{definition}

\begin{remark}\label{rem:vectorbundle}~
\bnm
\item\label{item:first}
The pullback topology on $\Diff(n)$ under the inclusion map $\Diff(n)\to\TOP(n)$ is by definition the coarsest topology such that $\Diff(n)\to\TOP(n)$ is continuous. It equals the compact open topology on the set $\Diff(n)$ (equivalent to the weak $C^0$ topology), not the weak~$C^\infty$ topology. In other words,  one does not topologise $\Diff(n)$ as a subspace of $\TOP(n)$.
\item It is nevertheless the case that $\Diff(n)\to\TOP(n)$ is a continuous map with respect to the  weak~$C^\infty$ topology on $\Diff(n)$ and the compact open topology on $\TOP(n)$.   There are induced maps $\operatorname{BDiff}(n) \to \BTOP(n)$ for each $n$, and $\operatorname{BDiff} \to \BTOP$.
\item
On the other hand, we now argue that \emph{both} inclusion maps $\op{GL}(n,\R)\to \TOP(n)$ and $\op{GL}(n,\R)\to \Diff(n)$ induce the standard topology on $\op{GL}(n,\R)$ via pullback (which is a little surprising at first glance). In the former case, simply note that the standard topology on $\op{GL}(n,\R)$ as a vector subspace of $\R^{n^2}$ is equivalent to the compact open topology, or equivalently the weak $C^0$ topology. In the latter case, note that as the maps in $\op{GL}(n,\R)$ are linear, and the weak $C^r$ topologies for $0\leq r <\infty$ are defined in terms of partial derivatives~\cite[\textsection~2.1]{Hi76}, they are all equivalent topologies on $\op{GL}(n,\R)$. As the weak $C^\infty$ topology is the limit of the $C^r$ topologies, this proves the statement.
\item The map
\[ \ba{rcl} \Diff(n)\times [0,1] &\to &\Diff(n)\\
(f,t)&\mapsto & \left(\ba{rcl} \R^n&\to &\R^n\\
x&\mapsto & \left\{ \ba{ll} \frac{1}{t}\cdot f(tx), &\mbox{ if }t\ne 0,\\
\op{D}f_0\cdot x\ea \right.\ea\right)\ea\]
is a deformation retraction from  $\Diff(n)$ to $\op{GL}(n,\R)$. Finally note that the Gram-Schmidt process can be used to determine a homotopy equivalence $\op{GL}(n,\R)\simeq \O(n)$. Thus $\R^n$-bundles with structure group $\Diff(n)$ always admit a (unique) structure group reduction to ordinary vector bundles whose transition functions lie in~$\O(n)$.
    \enm
\end{remark}

The following proposition shows that $\BTOP$ and $\operatorname{BDiff}$ classify $\CAT$ $\R^n$-bundles when $\CAT$ is $\TOP$ or $\Diff$.

\begin{proposition}\label{prop:universal-bundle}
For $\CAT=\TOP$ or $\Diff$, and for each $n \geq 0$, the space $\BCAT(n)$ is homotopy equivalent to a CW complex, and there exists a $\CAT$ $\R^n$-bundle~$\gamma_n^{\CAT}$ over $\BCAT(n)$, such that for every $\CAT$ $\R^n$-bundle $\xi$ over a CW complex, there is a $\CAT$ $\R^n$-bundle map $F \colon \xi \to \gamma_n^{\CAT}$, unique up to homotopy of $\CAT$ $\R^n$-bundle maps. In particular $F^*(\gamma_n^{\CAT}) \cong \xi$.
\end{proposition}

\begin{proof}
    Let $X$ be a space homotopy equivalent to a CW complex. There is a 1:1 correspondence between isomorphism classes of $\CAT(n)$ $\R^n$-bundles and isomorphism classes of principal $\CAT(n)$-bundles over~$X$. For more details, see e.g.~\cite[Proposition~11.22]{MR1886843}.

    Applying this to the universal principal $\CAT(n)$-bundle of Proposition~\ref{prop:bg}, we obtain a $\CAT$ $\R^n$-bundle~$\gamma_n^{\CAT}$ over $\BCAT(n)$. The desired universal property for~$\gamma_n^{\CAT}$ is inherited from the universal property of the universal principal $\CAT(n)$-bundle over~$\BCAT(n)$
\end{proof}

We finally move on to the piecewise linear category. For background on piecewise linear topology see \cite{Rourke-Sand-book}.

\begin{definition}
  A continuous map $f\colon K \to L$ between two simplicial complexes $K$ and $L$ is \emph{piecewise linear} ($\PL$) if there are subdivisions $K'$ of $K$ and $L'$ of $L$ such that $f \colon K' \to L'$ is a simplicial map.
\end{definition}

\begin{definition}\label{defn:PL-Rn-bundle}
%
A \emph{$\PL$ $\R^n$-bundle} is a topological $\R^n$-bundle $p\colon E\to B$, where both $E$ and $B$ are simplicial complexes, where $p\colon E\to B$ and the $0$-section are both $\PL$ maps, and where for every simplex $\Delta\subseteq B$ there exists a $\PL$ homeomorphism $\varphi$ such that the composition $p^{-1}(\Delta)\xrightarrow{\varphi}\Delta\times\R^n\xrightarrow{\operatorname{pr}_1}\Delta$ is equal to the projection $p$.

\end{definition}

The definition of the classifying space for $\PL$ $\R^n$-bundles is a little more involved than the constructions for $\O(n)$ and $\TOP(n)$, using the technology of semi-simplicial groups. For a gentle introduction to simplicial sets, see \cite{Friedman-SS}. The canonical reference for classifying spaces constructed using simplicial groups is~\cite{May-classifying-spaces}.

\begin{remark}
    When navigating the various references we use below, the reader should be aware that the terminology for simplicial groups and simplicial sets has changed over the years and there are some clashes between current usage and previous usage. What we are calling a \emph{simplicial set} is, for example, what is defined in \cite[Definition~3.2]{Friedman-SS}. In particular, note that this has both \emph{face} and \emph{degeneracy} maps. This object has historically been called a \emph{complete semi-simplicial set (c.s.s.)}, which has unfortunately in the past been abbreviated to just ``semi-simplicial set''. The modern terminology reserves \emph{semi-simplicial set} for a simplicial set without degeneracy maps as part of the data. To add to the confusion, a simplicial set, but without degeneracies, has also historically been called a $\Delta$-set, although in modern terminology this is usually reserved for the semi-simplical version of a simplicial \emph{complex}.
\end{remark}
Let $\Delta^k$ be the standard $k$-simplex. Recall a \emph{simplicial group} is a simplicial object in the category of groups, that is a contravariant functor from the simplicial category to the category of groups $\Delta\to\mathbf{\operatorname{\textbf{Group}}}$.

\begin{definition}\label{def:ss}
Given $n\in \N_0$ let $\PL(n)_{\bullet}$ be the simplicial group defined as follows.
  \begin{enumerate}[leftmargin=1cm]
    \item[(i)]\label{item:ss1} The group $\PL(n)_k$ assigned to the $k$-simplex is the group of $\PL$ $\R^n$-bundle isomorphisms of the trivial bundle $f \colon \R^n \times \Delta^k \to \R^n \times \Delta^k$. That is, $f$ is a $\PL$ homeomorphism preserving the 0-section and commuting with the projection to $\Delta^k$.
  \item[(ii)] A morphism $\lambda\colon\Delta^{\ell}\to\Delta^k$ is sent to the morphism $\lambda^{\#}$ which assigns to $f \colon \R^n \times \Delta^k \to \R^n \times \Delta^k$ the map such that the diagram below commutes
  \[
  \begin{tikzcd}
    \R^n \times \Delta^{\ell} \ar[r, "\lambda^{\#}(f)"]\ar[d,"\lambda\times 1"] &\R^n \times \Delta^{\ell}\ar[d,"\lambda\times 1"]\\
    \R^n \times \Delta^{k} \ar[r, "f"] &\R^n \times \Delta^{k}.
  \end{tikzcd}
  \]
\end{enumerate}

Define $\PL(n)$ as the topological group~\cite[Theorem~3]{Milnor-geom-real} realising the simplicial group $\PL(n)_{\bullet}$.
Then define $\BPL(n)$ using Proposition~\ref{prop:bg}.
Define $\PL$ and $\BPL$ as colimits, analogously to Definition~\ref{defn:TOP}.
\end{definition}

Equivalently, one can define $\BPL(n)$ by first using the geometric bar construction level-wise on $\PL(n)_{\bullet}$ to obtain a simplicial space $\BPL(n)_{\bullet}$, and then geometrically realising to obtain a space $\BPL(n)$; see e.g.~\cite[\textsection 1.2]{MR3995026} for the equivalence to the previous definition.

\begin{remark}
Consider the subgroup of $\TOP(n)$ consisting of $\PL$ homeomorphisms (note this is \emph{not} the definition of what we called $\PL(n)$ above. In particular $\PL(n)$ is a much larger set). One might think that this group, together with the subspace topology from $\TOP(n)$, is a realistic way to circumvent the construction above with simplicial groups. This would certainly produce a topological group that classifies \emph{some} category of $\R^n$ bundles, but we do not know whether it classifies the $\PL$ $\R^n$-bundles defined above (it seems unlikely, cf.~Remark~\ref{rem:vectorbundle}~\eqref{item:first}).
\end{remark}

The following proposition can now be viewed as an analogue of
Proposition~\ref{prop:bg}.

\begin{proposition}\label{prop:universal-bundlePL}
For each $n \geq 0$, the space $\BPL(n)$ is homotopy equivalent to a CW complex, and there exists a $\PL$ $\R^n$-bundle~$\gamma_n^{\PL}$ over $\BPL(n)$, such that for every $\PL$ $\R^n$-bundle $\xi$ over a CW complex, there is a $\PL$ $\R^n$-bundle map $F \colon \xi \to \gamma_n^{\PL}$, unique up to homotopy of $\PL$ $\R^n$-bundle maps. In particular $F^*(\gamma_n^{\PL}) \cong \xi$.
\end{proposition}

\begin{proof}
In~\cite[p.~24]{milnor4}, Milnor constructed a simplicial group $(\PL_n)_\bullet$. It is defined similarly to $\PL(n)_\bullet$, but differs from what we have done in Definition~\ref{def:ss}(i) by specifying that the maps $f$ are rather $\PL$ microbundle isomorphisms
of the trivial $\PL$ microbundle over the simplex (in particular are \emph{germs} of the types of maps $f$ we are using). The simplicial space $(\BPL_n)_\bullet$ that results from the level-wise geometric bar construction is a classifying space for simplicial principal $(\PL_n)_\bullet$-bundles. In~\cite[\textsection 5]{milnor4}, Milnor showed that the geometric realisation $\BPL_n:=|(\BPL_n)_\bullet|$ is homotopy equivalent to a locally finite simplicial complex (so in particular a CW complex), and that there is a universal rank $n$ $\PL$ microbundle over $\BPL_n$.

Kuiper-Lashof~\cite[Theorem~1]{MR216507} proved that each rank $n$ $\PL$ microbundle is the underlying microbundle of a unique $\PL$ $\R^n$-bundle. Apply this to the universal rank $n$ $\PL$ microbundle over $\BPL_n$, to obtain a universal $\PL$ $\R^n$-bundle over $\BPL_n$. The map $g\colon \PL(n)_\bullet\to (\PL_n)_\bullet$, given by simplex-wise taking the germs of maps $f$ as in Definition~\ref{def:ss}(i), is a homotopy equivalence of simplicial sets~\cite[Lemma~1.6(f)]{MR216507}. This induces a homotopy equivalence $\BPL(n)\xrightarrow{\simeq} \BPL_n$. Use this latter homotopy equivalence to pull back the universal rank $n$ $\PL$ microbundle over $\BPL_n$ to the desired base space~$\BPL(n)$.
\end{proof}

\begin{remark}
\mbox{}
\bnm
\item
The simplicial method used above for $\PL$ $\R^n$-bundles can be used in the smooth and topological categories as well, giving a uniform treatment. The resulting classifying spaces for the smooth and topological categories are homotopy equivalent to the spaces $\BO(n)$ and $\BTOP(n)$, defined earlier, by the universal property.
\item
An alternative uniform proof of
Propositions~\ref{prop:universal-bundle} and~\ref{prop:universal-bundlePL}, for all three categories $\TOP$, $\PL$ and $\Diff$ simultaneously,
was given by Kirby-Siebenmann~\cite[Essay~IV, Proposition~8.1, p.~181]{KS77}. Instead of a simplicial approach, they use E.~H.~Brown's theory of representability to obtain universal $\CAT$ \emph{microbundles} (cf.~Section~\ref{sec:microbundles}) over classifying spaces which are locally finite simplicial complexes. The theorems of Kister (Theorem~\ref{thm:kistmaz}) and Kuiper-Lashof~\cite[Theorem~1]{MR216507}, together with the analogue for smooth microbundles, means that the Kirby-Siebenman universal $\CAT$ microbundles can then be upgraded to universal $\CAT$ $\R^n$-bundles with the same base space, similarly to how we worked at the end of the proof of Proposition~\ref{prop:universal-bundlePL}.
The Kirby-Siebenmann proof also includes a statement relative to a closed subspace of the base.
\enm
\end{remark}

\section{Microbundles}\label{sec:microbundles}

All smooth manifolds have tangent vector bundles and all smooth submanifolds have normal vector bundles. This is one reason that vector bundles, corresponding to the structure group $\O(n)$ are the de facto bundle technology in the smooth category. A general difficulty we will face when talking about manifold transversality in Chapter \ref{chapter:transversality} is that we will need to use some well-defined notion of normal structure for a submanifold and, outside of the smooth category, submanifolds do not necessarily admit normal vector bundles. However, various weaker bundle technologies have been developed, which replace this crucial concept in the topological category.

This subsection is devoted to a discussion of \emph{microbundles}, which were introduced by Milnor in~\cite{Milnor63}. The existence and uniqueness of tangent and (stable) normal microbundles leads to the existence and uniqueness of tangent and (stable) normal $\TOP$ $\R^n$-bundles, via Kister's Theorem (Theorem~\ref{thm:kistmaz}). Source material on microbundles is not hard to find in the literature, but has been included here for the convenience of the reader, in order for this survey to be more self-contained.

The interaction between the weaker structure groups $\PL(n)$ and $\TOP(n)$ for tangent and (stable) normal $\R^n$-bundles, and the topological/$\PL$/smooth structures on the manifold itself are the topic of \emph{smoothing theory}, to which we turn in Chapter \ref{chapter:smoothing}.

\begin{definition}
An $n$-dimensional \emph{microbundle} $\xi$ consists of a base space~$B$ and a total space~$E$ sitting in a diagram
\[ B \xrightarrow{i} E \xrightarrow{p} B, \]
such that $p \circ i = \id_B$, and that is \emph{locally trivial} in the following sense: for every point~$b \in B$, there exists an open neighbourhood~$U$ of $B$, an open neighbourhood~$V$ of $i(b)$ and a homeomorphism~$\phi_b \colon V \to U \times \R^n$ such that
\[\xymatrix@R0.6cm{
& V \ar[rd]^-{p} \ar[dd]^-{\phi_b} &\\
U \ar[ur]^-{i} \ar[dr]_-{u\mapsto (u,0)} & & U\\
&U \times \R^n \ar[ur]_-{(u,v)\mapsto u} & \\
}
\]
commutes. We refer to $i$ as the \emph{inclusion map} of the microbundle and we refer to $p$ as the \emph{projection}.
\end{definition}

Note that we only require  neighbourhoods of the points $i(b)$
to be trivial, and not all of the fibre~$p^{-1}(b)$. In fact, we only care about
neighbourhoods~$i(B) \subseteq E$, and declare two microbundles $B \xrightarrow{i} E \xrightarrow{p} B$ and $B \xrightarrow{i'} E' \xrightarrow{p'} B$
to be \emph{equivalent}, if $i(B)$ and $i'(B)$ have homeomorphic neighbourhoods such that the homeomorphism commutes with both the inclusion map and the
restriction of the projection map.

\begin{definition}
Let $B\xrightarrow{i}E\xrightarrow{p}  B$ be a microbundle~$\xi$ and let $f \colon A \to B$ be a map.
The \emph{pullback} of $\xi$ under $f$ is the microbundle~$f^*\xi$ with total space
\[ f^*E = \big \{ (a,e) \in A \times E \mid f(a) = p(e) \big \}, \]
projection~$(f^*p) (a,e) = a$, and injection~$(f^*i) (a) = \big(a, i( f(a) ) \big)$.
In the case that $f$ is an inclusion, also consider the
microbundle~$\xi|_A$, which has total space~$p^{-1}(A) \subseteq E$, and projection~$p|_{p^{-1}(A)} \colon p^{-1}(A) \to A$ and injection~$i_A \colon A \to p^{-1}(A)$ are both the restrictions of $p$, $i$. In this case, the map of total spaces~$(a,e) \mapsto e$ gives a preferred isomorphism~$f^*\xi$ to~$\xi|_A$.
\end{definition}

A topological $\R^n$-bundle clearly has an underlying microbundle.
Kister proved the surprising result that \emph{every} microbundle over a manifold
is equivalent to such an underlying microbundle~\cite[Theorem 2 and Corollary 1]{Kister64}.

\begin{theorem}\textbf{\textup{(Kister's Theorem)}}\label{thm:kistmaz}
Let $B$ be a manifold and $B\xrightarrow{i}E\xrightarrow{p} B$ be an $n$-dimensional microbundle $\xi$. Then there exists an open set $F\subseteq E$ containing $i(B)$ such that $p|_F\colon F\to B$ is the projection map of a topological $\R^n$-bundle, whose $0$-section is $i$ and whose underlying microbundle is $\xi$. Moreover, if $F_1$ and $F_2$ are any two topological $\R^n$-bundles over $B$ such that the underlying microbundles are equivalent, then $F_1$ and $F_2$ are isomorphic as topological $\R^n$-bundles.
\end{theorem}

Every manifold admits a tangent microbundle.

\begin{definition}
The \emph{tangent microbundle} of an $n$-dimensional manifold $M$ is
the microbundle $M \xrightarrow{\Delta}M\times M \xrightarrow{(x,y)\mapsto x}  M$ where $\Delta$ is the diagonal map. Kister's theorem implies this corresponds to a unique \emph{topological tangent bundle} $\tau_M\colon M\to \BTOP(n)$, with corresponding \emph{stable topological tangent bundle} $\tau_M\colon M\to \BTOP$.
\end{definition}

More subtle is the concept of a normal microbundle.

\begin{definition}
A \emph{normal microbundle} of a submanifold~$S$ of a manifold $M$ is
a  microbundle $S \to E \to S$ such that $E$ is a neighbourhood of $S$ in $M$ and such that $S\to E$ is the inclusion.
\end{definition}

It is immediate from the definition of normal microbundle that the local flatness in the definition of a submanifold $S$ is a necessary condition for the existence of a normal microbundle. For example
wild knots and the Alexander horned sphere do not admit normal microbundles. Indeed, it is generally far from straightforward to prove the existence of normal microbundles at all. Here is an existence and uniqueness result due to Stern~\cite[Theorem~4.5]{Stern75}. See also~\cite{Hi66}, \cite[p.~65]{Hi68}, and~\cite[Essay~IV,~Appendix~A, p.~203]{KS77}.

\begin{theorem}\label{thm:microexist}
  Let $M^{n+q}$ be a manifold, and let $N^{n} \subseteq M^{n+q}$ be a proper submanifold of codimension $q$. Suppose that $n \leq q+1+j$ and $q \geq 5+j$ for some $j=0,1,2$.  Then $N$ admits a normal microbundle restricting to a normal microbundle of $\partial N \subseteq \partial M$.

If in addition  $n \leq q+j$, then this normal microbundle is unique up to isotopy.
\end{theorem}

\begin{remark}[Unique up to isotopy]
	For a submanifold $N \subseteq M$ we say a normal microbundle $N\xrightarrow{i}\nu(N)\xrightarrow{p} N$ is \emph{unique up to isotopy} if whenever there is another normal microbundle $N\xrightarrow{i'}\nu'(N)\xrightarrow{p'}N$, there exists a microbundle equivalence $f$ between $\nu(N)$ and $\nu'(N)$ such that $p' \circ f$ is isotopic to $p$ relative to $N$.	
\end{remark}

We exploit these theorems to define a stable topological normal structure on any closed manifold, that will play an important role in Chapter \ref{chapter:smoothing}.
Every closed smooth $n$-manifold $M$ can be embedded in $\R^k$, for some $k$. The stable class of a normal vector bundle gives rise to  \emph{stable normal vector bundle of $M$}, denoted $\mu_M \colon M \to \BO$. For $k$ large enough, the embedding of $M$ in $\R^k$ is unique up to isotopy, and using this one can show that the stable normal bundle is uniquely determined up to isomorphism.   We seek to adapt this construction to the topological category.

Consider that any closed $n$-manifold $M$ can be embedded as a submanifold $M\subseteq \R^m$ for large $m$ (this follows e.g.\ from \cite[Corollary A.9]{Hat02} together with \cite[Theorem~5]{MillR1972}). For large enough $m$, any two such embeddings are isotopic. For large enough $m$, Theorem~\ref{thm:microexist} implies there is a normal microbundle $\xi$. After possibly increasing $m$ further, the last sentence of Theorem \ref{thm:microexist} implies this normal microbundle $\xi$ is unique. By Kister's Theorem this defines a unique topological $\R^{m-n}$-bundle. We remove the dependence on $m$ by passing to the stable bundle $\TOP(m-n)\subseteq \TOP$. Thus the process described gives a well-defined classifying map $\nu_M\colon M\to \BTOP$. Summarising, we have the following.

\begin{definition}\label{def:topnormal} Given any closed $n$-manifold, the topological $\R^{\infty}$-bundle  $\nu_M\colon M\to \BTOP$, described above, is called the \emph{stable topological normal bundle}. It is well-defined and unique.
\end{definition}


The next example shows that outside the hypotheses of Theorem \ref{thm:microexist}, we should expect that normal microbundles can be very badly behaved.

\begin{example}\label{ex:notnormal}
Normal microbundles do not necessarily exist.
Rourke and Sanderson~\cite[Example 2]{RoSa67} construct $S^{19}$ as a submanifold of a certain
$28$-dimensional $\PL$ manifold $M$ in such a way that it does not admit a topological normal microbundle. The embedding is even piecewise linear.

Hirsch's example of a 4-submanifold of $S^7$ from \cite[Theorem~4]{Hi68} mentioned in Remark~\ref{remark:Hirsch-no-tub-nbhd-example} also does not admit a normal microbundle. Such a normal microbundle would contain a $\TOP(3)$ bundle, and every $\TOP(3)$ bundle can be improved to an $\O(3)$ bundle, which in turn contains a tubular neighbourhood. So it follows that there can be no normal microbundle in Hirsch's example. Note that historically this deduction was not possible until Hatcher~\cite{Hatcher-Smale-conj} proved that $\operatorname{BO}(3) \to \BTOP(3)$ is a homotopy equivalence.

Even when topological normal microbundles exist, they are not always unique: Rourke and Sanderson consider the smooth standard embedding $S^{18}\subseteq S^{27}$ \cite[Theorem 3.12]{RoSa-Block-III} and construct a certain normal microbundle $\xi$ of $S^{18}\subseteq S^{27}$. The construction of $\xi$ is such that if $\xi$ were concordant to the trivial normal microbundle, this concordance would induce a normal microbundle structure back on the embedding $S^{19}\subseteq M^{28}$ of the previous paragraph. As this is not possible, $\xi$ is nontrivial. Note that the normal vector bundle~$\nu S^{18}$ of the standard embedding is trivial, so $S^{18}\subseteq S^{27}$ admits at least two different normal microbundles.
\end{example}

The following theorem ensures the issues of the previous example are not seen in dimension 4.

\begin{theorem}\label{thm:normalmicroexist} Let $X$ be a proper submanifold of a $4$-manifold $M$. Then $X$ admits a normal microbundle. Moreover, if $\xi$ is a normal microbundle of $X$, it is the underlying microbundle to a normal vector bundle.
\end{theorem}

\begin{proof} The existence of normal microbundles in ambient dimension 4 is an immediate consequence of the existence of normal vector bundles (Theorem \ref{thm:existnormal}), and this is the only proof of which we are aware for this fact. (It would be interesting to know of a more elementary proof.)

We denote by $n$ the codimension of $X$ in $M$.
Given a normal microbundle $\xi$, we apply Kister's theorem \ref{thm:kistmaz} to obtain an embedded $\R^n$-bundle with underlying microbundle $\xi$. For $n\leq 3$, the homotopy fibre $\TOP(n)/\O(n)$ for the forgetful map $\BO(n)\to \BTOP(n)$ is contractible; see Proposition~\ref{prop:homotopy-type-TOPn-mod-On} for the relevant citations.
Using these facts, and checking the obstructions in each of the cases $n=0, 1, 2, 3, 4$, we see in each case the embedded topological $\R^n$-bundle can be upgraded to an embedded vector bundle. Choose such a vector bundle refinement. By restricting to an open disc bundle and rescaling we can ensure this internal linear bundle is extendable and thus is a normal vector bundle in the sense of Definition \ref{def:extendable}.
\end{proof}

We will make use of our discussion of normal microbundles in Chapter~\ref{chapter:transversality} on topological transversality.

\chapter{Stiefel-Whitney classes}\label{chapter:SW-classes}

The well-known treatment of Stiefel-Whitney classes in Milnor-Stasheff~\cite{MS74} is for vector bundles, that is bundles over a space $B$ with linear transition functions. In this chapter we discuss the analogous characteristic classes for $\TOP$ and $\PL$ $\R^n$-bundles.  These arise, in particular, for tangent bundles of $\PL$ and topological manifolds respectively.
Throughout this chapter we will use the terminology \emph{$\CAT$ $\R^n$-bundle}, where
\[
\CAT \in \{\Diff, \PL,\TOP\}.
\]
When $\CAT =\TOP$, this will mean  an $\R^n$-bundle with structure group $\TOP(n)$. When $\CAT=\PL$, this will mean $\PL$ $\R^n$-bundles in the sense of Definition~\ref{defn:PL-Rn-bundle}.  When $\CAT=\Diff$, we will mean an $\R^n$-bundle with structure group $\O(n)$, taking advantage of the structure group reduction described in Remark~\ref{rem:vectorbundle}.

Generalising work of Thom~\cite{Thom-espaces-fibre} in the case of orthogonal structure group,
Fadell~\cite[Definition 6.1]{Fadell} (see also Stern~\cite{Stern75}) defined Stiefel-Whitney classes of $\CAT$ $\R^n$-bundles, where $\CAT \in \{\Diff, \PL, \TOP\}$.
As we shall see, the definitions of Stiefel-Whitney classes in the $\PL$ and $\TOP$ cases are analogous to the classical definition for vector bundles.

\begin{remark}
Even more generally, one can also develop Stiefel-Whitney classes for spherical fibrations, using an analogous definition. Passing to the underlying $S^{n-1}$-bundle of a $\CAT$ $\R^n$-bundle, the spherical fibration definition will recover all definitions we give below. However, we have chosen not work in this level of generality.
\end{remark}

Let $\xi = (p \colon E \to B)$ be a  $\CAT$ $\R^n$-bundle over some topological space $B$.  Let $E_0$ denote the complement in the total space $E$ of the zero section. Recall that the Thom isomorphism~\cite{Thom-espaces-fibre} is an isomorphism
\[\Phi \colon H^i(B;\Z/2) \xrightarrow{\cong} H^{n+i}(E,E_0;\Z/2).\]
Fadell~\cite[Theorem~5.2]{Fadell} checked that the Thom isomorphism holds in the present context, with $\PL(n)$ or $\TOP(n)$ structure group.
The Steenrod squares are homomorphisms
\[\Sq^i \colon H^n(E,E_0;\Z/2) \to H^{n+i}(E,E_0;\Z/2).\]
Here, let us recall that given a pair of spaces $(X,A)$, for each $i \geq 0$ and for each $n \geq 0$ the Steenrod square $\Sq^i$ is a homomorphism $\Sq^i \colon H^n(X,A;\Z/2) \to H^{n+i}(X,A;\Z/2)$.  The Steenrod squares  have the following properties, which we will use; see e.g.\ \cite{Steenrod-Epstein}.
\begin{enumerate}[leftmargin=1cm]
    \item The $\Sq^i$ are natural with respect to maps of pairs $f \colon (X,A) \to (Y,B)$, for all $i \geq 0$.
    \item $\Sq^0 = \Id$;
    \item $\Sq^i(x) = 0$ for $x \in H^n(X,A;\Z/2)$ with $n < i$;
    \item $\Sq^i(x) = x \cup x$ for $x \in H^i(X,A;\Z/2)$.
\end{enumerate}

Let $1 \in H^0(B;\Z/2)$ denote the unit of the cohomology ring $H^*(B;\Z/2)$.

\begin{definition}[Stiefel-Whitney classes~{{\cite[Definition~6.1]{Fadell}}, \cite[p.~262]{Stern75}}]\label{defn:SW-classes-Stern}
  The $i$th Stiefel-Whitney class of $\xi$ is
  \[w_i(\xi) := \Phi^{-1} \circ \Sq^i \circ \Phi(1) \in H^i(B;\Z/2).\]
\end{definition}

The definition uses the sequence of maps
\[H^0(B;\Z/2) \xrightarrow{\Phi,\cong} H^n(E,E_0;\Z/2) \xrightarrow{\Sq^i} H^{n+i}(E,E_0;\Z/2) \xrightarrow{\Phi^{-1},\cong} H^i(B;\Z/2).\]
Since $\Sq^j \colon H^n(E,E_0;\Z/2) \to H^{n+j}(E,E_0;\Z/2)$ is the zero map for $j >n$, by the third property of Steenrod squares, it follows that $w_j(\xi) = 0$ for $j>n$.

We now restrict ourselves to bundles over spaces that are homotopy equivalent to CW complexes. In this context it was shown by \cite{Fadell} and ~Stern~\cite[Theorem~2.0]{Stern75} that the Stiefel-Whitney classes $w_i(\xi)$ for $0 \leq i \leq n$ satisfy the following properties:

\begin{proposition}\label{prop:Stern-SW-axioms}
Let $B$ be a space homotopy equivalent to a CW complex and let $\xi = (p \colon E \to B)$ be a  $\CAT$ $\R^n$-bundle over $B$.
Define the total Stiefel-Whitney class \[w(\xi) := \sum_{i=0}^n w_i(\xi) \in H^*(B;\Z/2).\]
\begin{enumerate}[leftmargin=1cm,font=\normalfont]
    \item For a $\CAT$ $\R^n$-bundle map $f=(f_E,f_B) \colon \xi \to \eta$, which consists of maps \[\begin{tikzcd}
    E(\xi) \ar[d,"p(\xi)"] \ar[r,"f_E"] & E(\eta) \ar[d,"p(\eta)"] \\ B(\xi) \ar[r,"f_B"] & B(\eta),
\end{tikzcd}
\]
we have that $f_B^*(w(\eta)) = w(\xi)$.
\item If $\xi = \eta^q \oplus \varepsilon^{n-q}$, where $\eta$ is a $\CAT$ $\R^q$-bundle over $B$ and $\varepsilon^{n-q}$ is a trivial $\CAT$ $\R^{n-q}$-bundle over $B$, then $w(\xi)= w(\eta)$.
\item For each $n$ there is a $\CAT$ $\R^n$-bundle such that $w_n(\xi)\neq 0$.
\end{enumerate}
\end{proposition}

Under the assumption that either $\CAT \in\{\PL,\Diff\}$ or $\CAT = \TOP$ and $q \neq 4,5$, Stern also proved that the properties in Proposition~\ref{prop:Stern-SW-axioms} characterise the Stiefel-Whitney classes. We are not sure whether the assumption that $q \neq 4,5$ can now be removed in the $\TOP$ case using Quinn's work.

Recall the universal $\CAT$ $\R^n$-bundle $\gamma_n^{\CAT}$ over $\BCAT(n)$ from Propositions~\ref{prop:universal-bundle} and~\ref{prop:universal-bundlePL}.  Here we are abusing notation, and using the deformation retract $\Diff(n)\simeq \O(n)$ (Remark~\ref{rem:vectorbundle}) to conflate these structure groups. We denote the \emph{universal Stiefel-Whitney classes} by \[\ol{w}_k^{\CAT} := w_k(\gamma_n^{\CAT}) \in H^k(\BCAT(n);\Z/2),\] for some $n \geq k$. We also write $\ol{w}_k^{\CAT} \colon \BCAT(n) \to K(\Z/2,k)$ for the corresponding map to the Eilenberg-Maclane space.  For $F \colon B \to \BCAT(n)$, classifying a $\CAT$ $\R^n$-bundle $\xi$, we have by Proposition~\ref{prop:Stern-SW-axioms} that
\[w_k(\xi) = F^*(\ol{w}_k^{\CAT}) \in H^k(B;\Z/2).\]

The definitions for each value of $\CAT$ are compatible in the following sense.

\begin{proposition}\label{prop:compatibility-SW-classes-defns}
Let $\xi$ be a $\CAT$ $\R^n$-bundle over $B$, for some $\CAT \in  \{\Diff,\PL\}$.  Let $\xi_{\TOP}$ be the underlying $\TOP$ $\R^n$-bundle.  Then $w_k(\xi) = w_k(\xi_{\TOP})$ for every $k \geq 0$.
In fact, the diagram
\[\begin{tikzcd}
B \ar[r,"\xi"] \ar[d,"\xi_{\TOP}"'] & \BCAT \ar[d,"\ol{w}_k^{\CAT}"] \ar[dl] \\  \BTOP \ar[r,"\ol{w}_k^{\TOP}"'] & K(\Z/2,k).
\end{tikzcd}\]
commutes up to homotopy.
\end{proposition}

\begin{proof}
    This is a consequence of the fact that the definitions in all three cases are directly analogous; cf.~\cite[Th\'{e}or\`{e}me~III.8]{Thom-espaces-fibre} and \cite[Theorem~6.10]{Fadell}, which consider the case of tangent bundles.
\end{proof}

From now on we will often take the previous proposition to heart, and omit the $\CAT$ superscript, writing $\ol{w}_k$ instead of $\ol{w}_k^{\CAT}$. For the rest of this chapter this will in any case only be used with $\CAT=\TOP$.

Recall that $\TOP(n)/\O(n)$ is by definition the homotopy fibre of $\BO(n) \to \BTOP(n)$.

\begin{proposition}\label{prop:homotopy-type-TOPn-mod-On}
  $\TOP(n)/\O(n)$ is contractible for $n \leq 3$, while for $n \geq 4$ there is a 5-connected map $\TOP(n)/\O(n) \to K(\Z/2,3)$.
\end{proposition}

\begin{proof}
    According to \cite[Essay V,~Section~5.0,~p.~246]{KS77}, the homotopy fibre $\TOP(2)/\O(2)$ is contractible,   $\pi_i(\TOP(3)/\O(3)) =0$ for $i \geq 4$, and $\pi_i(\TOP(3)/\O(3)) \cong \pi_i(\Diff(D^3,\partial D^3))$ for $i \geq 4$.  The latter group is trivial for all $i$, by Hatcher's theorem~\cite{Hatcher-Smale-conj}.  Thus all the homotopy groups of $\TOP(3)/\O(3)$ vanish, and therefore $\BO(3) \to \BTOP(3)$ induces an isomorphism on homotopy groups $\pi_i(\BO(3)) \to \pi_i(\BTOP(3))$ for all $i$.
    Since $\BO(3) \to \BTOP(3)$ is a map between spaces homotopy equivalent to a CW complex, we deduce that this map is a homotopy equivalence by Whitehead's theorem. Thus  $\TOP(3)/\O(3)$ is the homotopy fibre of a homotopy equivalence and so is contractible.

The reference \cite[Essay V,~Section~5.0, p.~246]{KS77} also includes the statement that for $n \geq 5$ and $i \leq 7$ we have $\pi_i(\TOP/\O,\TOP(n)/\O(n)) =0$.  In addition $\pi_i(\TOP/\O) \cong \pi_i(K(\Z/2,3))$ for $i \leq 6$, with the isomorphism induced by the map $\TOP/\O \to \TOP/\PL \simeq K(\Z/2,3)$, which is therefore a 5-connected map.  The homotopy equivalence $\TOP/\PL \simeq K(\Z/2,3)$ is from~\cite[Essay~IV,~Section 10.12, p.~200]{KS77}. Since the composition of two 5-connected maps is 5-connected, it follows that there is a $5$-connected map $\TOP(n)/\O(n) \to K(\Z/2,3)$ for $n \geq 5$, as claimed.

It remains to consider $n=4$. For this we appeal to \cite[Theorem~8.7A]{FQ90}, which states that $\TOP(4)/\O(4) \to \TOP/\O$ is 5-connected.  Combined with the fact already discussed that there is a 5-connected map $\TOP/\O \to K(\Z/2,3)$, we obtain the sought-for 5-connected map $\TOP(4)/\O(4) \to K(\Z/2,3)$.
\end{proof}


\begin{proposition}\leavevmode
\begin{enumerate}[leftmargin=1cm,font=\normalfont]
    \item  For $n \geq 2$, $\pi_1(\BTOP(n)) = \Z/2$.
    \item The corresponding unique homotopically nontrivial map $\BTOP(n) \to K(\Z/2,1)$ is the universal first Stiefel-Whitney class $\ol{w}_1$.
    \item    We have that $\pi_2(\BTOP(2)) \cong \Z$.
    \item  For $n \geq 3$, $\pi_2(\BTOP(n)) \cong \Z/2$.
    \item The corresponding unique homotopically nontrivial map $\BTOP(n) \to K(\Z/2,2)$ is the universal second Stiefel-Whitney class $\ol{w}_2$.
\end{enumerate}
 \end{proposition}

\begin{proof}
  For $n \leq 3$ we have $\TOP(n) \simeq \O(n)$ by Proposition~\ref{prop:homotopy-type-TOPn-mod-On}, and so the homotopy groups of $\TOP(n)$ are isomorphic to those of $\O(n)$. For $n \geq 4$ we have the long exact sequence in homotopy groups:
  \begin{align*}
   \pi_2(\TOP(n)/\O(n)) &\to \pi_2(\BO(n)) \to \pi_2(\BTOP(n)) \to  \\
    \pi_1(\TOP(n)/\O(n)) &\to \pi_1(\BO(n)) \to \pi_1(\BTOP(n)) \to \{\ast\}.
   \end{align*}
  Since $\pi_i(\TOP(n)/\O(n))=0$ for $i=1,2$ by Proposition~\ref{prop:homotopy-type-TOPn-mod-On}, we deduce that $\pi_i(\BO(n)) \cong \pi_i(\BTOP(n))$ for $i=1,2$, with the map induced by the canonical forgetful map.  Since $\pi_i(\BO(n)) \cong \pi_{i-1}(\O(n)) \cong \Z/2$ for $i=1,2$, the result follows from this and Proposition~\ref{prop:compatibility-SW-classes-defns}.
\end{proof}

We see that $\TOP(n)$ has two connected components, which are homotopy equivalent because $\TOP(n)$ is a topological group, and $\pi_1(\TOP(n),\Id) \cong \Z/2$ for $n \geq 3$.

\begin{definition}
  We define $\STOP(n)$ to be the subgroup of $\TOP(n)$ consisting of orientation preserving homeomorphisms.
We define $\TOPSpin(n)$ to be the universal cover of $\STOP(n)$. Define $\STOP$ and $\TOPSpin$ as corresponding colimits.
\end{definition}

This definition is analogous to the definition of $\SO(n)$ as the subgroup of $\O(n)$ of orientation preserving orthogonal matrices, and of $\Spin(n)$ as the connected double cover of $\SO(n)$; this is the universal cover for $n \geq 3$.

\begin{theorem}
    The topological group $\STOP(n)$ is the connected component of $\TOP(n)$ containing the identity.
\end{theorem}

\begin{proof}
We saw above that $\TOP(n)$ has two connected components. For a homeomorphism $f \colon \R^n \to \R^n$, and another such homeomorphism $g$, if $f$ and $g$ are isotopic, then $f$ and $g$ are either both orientation preserving (o.p.) or both orientation reversing (o.r.). The map $\pi_0(\TOP(n)) \to \{\text{o.p.}, \text{o.r.}\}$ is a surjective map from a set with two elements to another set with two elements, hence is a bijection.

(Alternatively, the theorem can be seen as a consequence of the Stable Homeomorphism Theorem~\ref{thm:SHT}, which says that every orientation-preserving homeomorphism of $\R^n$ is stable. Using that every homeomorphism of $\R^n$ that is the identity on some subset is isotopic to the identity, via the (inverted) Alexander trick, as in Corollary~\ref{cor:homeo-sn}, we deduce the result. The computations of the homotopy type of $\TOP(n)/\O(n)$ also used the Stable Homeomorphism Theorem, so this alternative proof is not independent of the first.)
    \end{proof}

We have the following commutative diagram of classifying spaces, together with the universal Stiefel-Whitney classes.

\[\begin{tikzcd}
    \BSpin(n) \ar[r] \ar[d] & \BTOPSpin(n) \ar[d] & \\
    \BSO(n)  \ar[r] \ar[d] & \BSTOP(n) \ar[r,"\ol{w}_2"] \ar[d] & K(\Z/2,2) \\
    \BO(n)  \ar[r]  & \BTOP(n) \ar[r,"\ol{w}_1"] & K(\Z/2,1)
  \end{tikzcd}\]
The horizontal maps in the bottom row induce isomorphisms on $\pi_1$, while the horizontal maps in the middle row induce isomorphisms on $\pi_2$.
Thus up to homotopy equivalence, $\BSTOP(n)$ is the 1-connected cover of $\BTOP(n)$, and $\BTOPSpin(n)$ is the 2-connected cover.
In other words, $\BSTOP(n)$ is the homotopy fibre of $\ol{w}_1$ and $\BTOPSpin(n)$ is the homotopy fibre of $\ol{w}_2$.

The analogous statements hold in the case of $\O(n)$.

\begin{definition}
Let $M$ be a space homotopy equivalent to a CW complex and let $\xi$ be a $\TOP$ $\R^n$-bundle over $M$ classified by a map which we also denote $\xi \colon M \to \BTOP(n)$.  An \emph{orientation} on $\xi$ is a lift $M \to \BSTOP(n)$ of $\xi$, and two orientations are equivalent if the lifts are homotopic over $\BSTOP(n) \to \BTOP(n)$.  If an orientation exists then we say $\xi$ is \emph{orientable}.
\end{definition}

\begin{proposition}\label{prop:orientation}
    An orientation for an $n$-manifold $M$ is equivalent to an orientation on the topological tangent bundle $\tau_M\colon M\to \BTOP(n)$. In particular $M$ is orientable if and only if $\tau_M$ is orientable.
\end{proposition}

\begin{proof}
    An orientation for $M$ is equivalent to a homology orientation for $M$, i.e.\ a coherent choice of generators of $H_n(M,M \sm \{x\};\Z)$, for $x \in M$. In turn, a homology orientation for $M$ is equivalent to a coherent system of orientations of the fibre of the tangent microbundle $M \xrightarrow{\Delta} M \times M \xrightarrow{\pr_1} M$, rel.\ the zero section $\Delta(M)$. This is because both are, by definition, a coherent system of generators of $H_n(M,M \sm\{x\};\Z)$, for each $x \in M$.  See \cite[Lemma~11.6]{MS74} for details.

Next, the latter notion is equivalent, via Kister's theorem~\cite{Kister64} and excision, to a coherent choice of generators of $H_n(F_x,F_x \sm \{0\};\Z)\cong \Z$. Here, $F_x$ is the fibre over $x$ in the topological tangent $\R^n$-bundle of $M$; see \cite[Lemma~11.7]{MS74}, but replace the exponential map with the embedding from Kister's theorem.

 To see that a coherent choice of isomorphisms $H_n(F_x,F_x \sm \{0\};\Z)\cong\Z$ is equivalent to a choice of lift $M \to \BSTOP(n)$, recall that the two connected components of $\TOP(n)$ correspond to whether a homeomorphism preserves or changes a fixed generator of $H_n(\R^n,\R^n \sm \{0\};\Z)$, so the structure group reduces to $\STOP(n)$ if and only if there is a coherent choice of generators for the homology groups $H_n(F_x,F_x \sm \{0\};\Z)$, for $x\in M$.
\end{proof}

\begin{definition}
For an $n$-manifold $M$, with topological tangent bundle $\tau_M$, we define $w_i(M) := w_i(\tau_M)$.
\end{definition}

\begin{proposition}
Let $M$ be a space homotopy equivalent to a CW complex and let $\xi$ be a $\TOP$ $\R^n$-bundle over $M$. The bundle $\xi$ is orientable if and only if $w_1(\xi)=0$.
In particular, by Proposition~\ref{prop:orientation}, a manifold $M$ is orientable if and only if $w_1(M)=0$.
\end{proposition}

\begin{proof}
Since $\BSTOP(n) \to \BTOP(n) \xrightarrow{\ol{w}_1} K(\Z/2,1)$ is a fibration sequence, we have an exact sequence of pointed sets
\[[M,\BSTOP(n)] \to [M,\BTOP(n)] \to [M,K(\Z/2,1)] \cong H^1(M;\Z/2).\]
The bundle $\xi$ is orientable if and only if the classifying map $\xi \in [M,\BTOP(n)]$ is homotopic to a map in the image of  $[M,\BSTOP(n)]$. The latter is equivalent to $\ol{w}_1 \circ \xi \in [M,K(\Z/2,1)]$ being null-homotopic, using the sequence.
Translating to cohomology groups this is equivalent to $w_1(\xi) = \xi^*(\ol{w}_1) =0$.  Here we used that $w_1(\xi)$ is equal to the pullback of the universal bundle along the classifying map for $\xi$.
%
 \end{proof}

\begin{definition}
Let $M$ be a space homotopy equivalent to a CW complex and let $\xi$ be a $\TOP$ $\R^n$-bundle over $M$ classified by a map which we also denote $\xi \colon M \to \BTOP(n)$.  Suppose that $w_1(M)=0$, so $\xi$ is orientable.  A \emph{spin structure} on $\xi$ is a lift $M \to \BTOPSpin(n)$ of $\xi$, and two spin structures are equivalent if the lifts are homotopy equivalent.  If a spin structure exists then we say $\xi$ is \emph{spin}.

Now suppose that $M$ is a topological $n$-manifold with topological tangent bundle classified by $\tau_M \colon M \to \BTOP(n)$.  Suppose that $w_1(M)=0$, so $M$ is orientable.  A \emph{spin structure} on $M$ is a spin structure on $\tau_M$. If a spin structure on $\tau_M$ exists then we say $M$ is \emph{spin}.
\end{definition}

\begin{proposition}
Let $M$ be a space homotopy equivalent to a CW complex and let $\xi$ be a $\TOP$ $\R^n$-bundle over $M$. Suppose that $\xi$ is orientable.
   Then $\xi$ is spin if and only if $w_2(\xi)=0$.
In particular, an orientable manifold $M$ is spin if and only if $w_2(M)=0$.
\end{proposition}

\begin{proof}
    This holds because $\BTOPSpin(n) \to \BSTOP(n) \xrightarrow{\ol{w}_2} K(\Z/2,2)$ is a fibration sequence. Thus a lift of $\xi \colon M \to \BSTOP(n)$ to $\BTOPSpin(n)$ exists if and only if $\ol{w}_2 \circ \tau_M$ is null-homotopic, i.e.\ if and only if $w_2(\xi)=0$.
    Here we are again using that a fibration sequence gives rise to an exact sequence of sets \[[M,\BTOPSpin(n)] \to [M,\BSTOP(n)] \xrightarrow{} [M,K(\Z/2,2)] \cong H^2(M;\Z/2),\]
and hence a lift to $\BTOPSpin(n)$ exists if and only if $w_2(\xi)= \xi^*(\ol{w}_2) = 0$.
\end{proof}

We now recall the Wu classes and the Wu formulae for the Stiefel-Whitney classes of compact manifolds.  The key point here is that while the treatment in \cite[Chapter~11]{MS74} is presented for smooth manifolds, in fact it uses only Poincar\'{e} duality and algebraic topology, so works just as well for topological manifolds.  This has been observed by Thom~\cite{Thom-espaces-fibre} and Fadell~\cite{Fadell}.
We proceed to summarise the treatment in Milnor-Stasheff.

Let $M$ be a compact $n$-manifold with $\Z/2$-fundamental class $[M,\partial M] \in H_n(M,\partial M;\Z/2)$.
Consider the homomorphism
\begin{align*}
   \theta\colon  H^{n-k}(M,\partial M;\Z/2) &\to \Z/2 \\
    x &\mapsto \langle \Sq^k(x), [M,\partial M]\rangle.
\end{align*}
By Poincar\'e duality, there is a unique class $v_k(M) \in H^k(M;\Z/2)$ with \[\langle v_k(M) \cup x, [M,\partial M]\rangle = \theta(x)\] for all $x \in H^{n-k}(M,\partial M;\Z/2)$.  In fact $v_k(M) \cup x = \Sq^k(x) \in H^n(M,\partial M;\Z/2)$ for every $x \in H^{n-k}(M,\partial M;\Z/2)$.

\begin{definition}
The class $v_k(M)$ is the \emph{$k$th Wu class of $M$}.
\end{definition}

\begin{proposition}\label{prop:Wu-formulae}
  Let $M$ be a compact $n$-manifold. Then the Wu formulae
  \[w_k(M) = \sum_{i+j=k} \Sq^i(v_j)\]
  hold for $k=0,\dots,n$.
\end{proposition}

\begin{proof}
    The proof in \cite[Theorem~11.11,~Lemma~11.13,~and~Theorem~11.14]{MS74}, which relies on \cite[Theorem~11.11~and~Lemma~11.13
]{MS74}, uses only Poincar\'{e} duality and products from algebraic topology, and so proceeds exactly as in Milnor-Stasheff.
\end{proof}


Here is a sample, and often used, application for the second Stiefel-Whitney class.

\begin{definition} Let $M$ be a compact, oriented $n$-manifold and $\Sigma^{n-2}\subseteq M$ a proper submanifold. We say $\Sigma$ is \emph{characteristic} if $\PD(w_2(M))=j_*([\Sigma])\in H_{n-2}(M,\partial M;\Z/2)$, where $j_*\colon H_{n-2}(\Sigma,\partial \Sigma;\Z/2)\to H_{n-2}(M,\partial M;\Z/2)$ is the inclusion-induced map.
\end{definition}

\begin{proposition}\label{prop:characteristicsurface}
Let $M$ be a compact $n$-manifold and $\Sigma\subseteq M$ a proper submanifold. Then $w_2(M\setminus \Sigma)=0$ if and only if $\Sigma$ is chracteristic.
\end{proposition}

\begin{proof}
Throughout the proof, $\Z/2$-coefficients are understood. By the Collar Neighbourhood Theorem (\ref{thm:topological-collar}), we may take a boundary collar on $M$ that restricts on $\Sigma$ to a boundary collar on $\Sigma$.
Set $K:=M\sm(\partial M\times[0,1))$, the complement of the open collar on $\partial M$; note $K$ is compact. Write $\nu\Sigma$ for a open tubular neighbourhood of $\Sigma$. Tubular neighbourhoods were proved to exist in codimension 2 when $n\neq 4$ by Kirby-Siebenmann~\cite{KS75} and for $n=4$ by Freedman-Quinn (Theorem~\ref{thm:FQ-exist-normal}). Next we draw a diagram, then define the maps and justify that it has exact rows, is commutative, and that the variously claimed isomorphisms indeed are so.
\[
\begin{tikzcd}
H^2(M, M\sm \nu\Sigma)\ar[d,"\cong"']\ar[r, "\beta"] &H^2(M)\ar[r, "\alpha"]\ar[d,"\cong"'] &H^2(M\sm\nu \Sigma)\ar[d,"\cong"']\\
H^2(K, K\sm \nu\Sigma)\ar[r, "\beta' "]\ar[d, "{-\cap\mathcal{O}}", "\cong"'] &H^2(K)\ar[r, "\alpha' "] \ar[d, "-\cap\mathcal{O}", "\cong"']&H^2(K\sm\nu \Sigma)\\
H_{n-2}(\partial M\cup\nu \Sigma,\partial M)\ar[r,"\text{incl}"]&H_{n-2}(M,\partial M)&\\
H_{n-2}(\nu \Sigma, \partial M\cap \nu\Sigma	)\ar[u,"\cong", "\text{incl}"']&H_{n-2}(\Sigma,\partial \Sigma)\ar[u,"j_*"']\ar[l,"\text{incl}","\cong"']&
\end{tikzcd}
\]
The top row is a section of the long exact sequence of the pair $(M,M\sm \nu \Sigma)$, and the central row is a section of the long exact sequence of the pair $(K,K\sm\nu\Sigma)$, thus both are exact. The downwards maps from the top row to the middle row are by definition the maps induced by inclusion, which is a homotopy equivalence, justifying these isomorphisms and the fact that the subdiagram consisting of the top two rows commutes.

The class $\mathcal{O}$ is the $\Z/2$-orientation class; see ~\cite[\textsection VI.8]{Br93}.
Replacing $M$ by $M \sm \partial M$, we obtain a (non-compact) 4-manifold $M'$ with empty boundary.
By~\cite[\textsection VI Theorem~8.3]{Br93}, capping with $\mathcal{O}$ induces isomorphisms \[H^2(K,L) \xrightarrow{\cong} H_{n-2}(M' \sm L,M' \sm K),\] both for $L:=K\sm\nu\Sigma$ and for $L := \emptyset$. Thus, we obtain isomorphisms
\[H^2(K, K\sm \nu\Sigma)) \xrightarrow{\cong} H_{n-2}(\partial M \times (0,1) \cup \nu \Sigma,\partial M \times (0,1)) \cong H_{n-2}(\partial M \cup \nu \Sigma,\partial M)\]
and
\[H^2(K) \xrightarrow{\cong} H_{n-2}(M',\partial M \times (0,1)) \cong H_{n-2}(M,\partial M).\]
These define the middle vertical maps labelled $-\cap \mathcal{O}$.  By naturality of the cap product, the left middle square commutes.
The commutativity of the bottom left square is clear, as is the isomorphism of the left-bottom arrow.
We also note that the composite of the top two central downwards arrows is the Poincar\'{e}-Lefschetz duality isomorphism $PD \colon H^2(M) \xrightarrow{\cong} H_{n-2}(M,\partial M)$.
Having established the relevant properties of the diagram, we now prove the lemma.

First note that by naturality of Stiefel-Whitney classes, we have $\alpha(w_2(M))=w_2(M\setminus \nu \Sigma)$. Note as well that $H_{n-2}(\partial M\cup \overline{\nu} \Sigma, \partial M;\Z/2)$ is isomorphic to $\Z/2$ and generated by the image of~$[\Sigma]$.

For one direction of the lemma, assume that $w_2(M\setminus \Sigma)=0$. By exactness of the top row, this implies $w_2(M)$ lies in the image of $\beta$. But as $H^2(M, M\sm \nu\Sigma))\cong \Z/2$ and $w_2(M)\neq 0$, this implies the generator of $\Z/2$ is sent to $w_2(M)$ by $\beta$. By commutativity of the diagram, we get $\PD(w_2(M))\equiv j_*([\Sigma])\in H_{n-2}(M,\partial M;\Z/2)$.

Conversely, assume that $\PD(w_2(M))\equiv j_*([\Sigma])\in H_{n-2}(M,\partial M;\Z/2)$. By commutativity of the diagram, this implies $w_2(M)$ is in the image of $\beta$.
Thus $w_2(M\setminus \nu\Sigma)= \alpha(w_2(M)) = 0$ by exactness of the top row.
\end{proof}

\begin{remark}
    We note that the previous proposition is valid when none of the manifolds involved is orientable and also $\Sigma$ is closed.
    We also note that the proof given above is fairly robust and could be easily adapted to other Stiefel-Whitney classes, provided the tubular neighbourhood conditions are met.
    For example, the same idea shows that the complement of a codimension 1 submanifold is orientable if and only if that submanifold is Poincar\'{e} dual to the first Stiefel-Whitney class.
\end{remark}

\chapter{Intersection forms and smooth 4-manifolds}\label{chapter:intersection-form}
In this chapter we introduce and study one of the most interesting invariants of 4-manifolds, namely the intersection form.
Later, in Theorem~\ref{thm:classn-simply-connected-4-mfld} we will see that any unimodular symmetric form over $\Z$ occurs as the intersection form of a closed oriented 4-manifold. In contrast we will see in this chapter that  not all  unimodular
symmetric form  over~$\Z$ can be realised as the intersection forms of closed oriented \emph{smooth} 4-manifolds.

\section{Intersection forms}
We start out with the definition of the intersection form.

\begin{definition}\label{defn:intersection-form}
\mbox{}
\bnm
\item Given a finitely generated abelian group $H$ we write $\tf H:=H/\mbox{torsion subgroup}$.
\item
Given a compact oriented $n$-manifold~$M$  we denote by $[M]\in H_{n}(M,\partial M;\Z)$ its fundamental class. Given a decomposition $\partial M=A\cup B$ where $A$ and $B$ are compact codimension-zero submanifolds of $\partial M$ with $A\cap B=\partial A=\partial B$ we denote the Poincar\'e duality isomorphism by
\[ \begin{array}{rcl} \PD\colon H^l(M,A;\Z)&\to & H_{n-l}(M,B;\Z)\\
{} \phi&\mapsto & \varphi\acap [M].\end{array}\]
\item
Given a compact oriented 4-manifold~$M$ we refer to the map
\[ \ba{rcl} Q_M\colon \tf H_2(M;\Z)\times \tf H_2(M;\Z)&\to& \Z\\
(a,b)&\mapsto &Q_M(a,b):=\langle \PD_M^{-1}(a)\acup \PD_M^{-1}(b),[M]\rangle\ea\]
as the \emph{intersection form}. (Here $\langle-,-\rangle$ denotes the Kronecker pairing.) Using Poincar\'e Duality one can easily show that if $M$ is closed, then  $Q_M$ is nonsingular.
\enm
\end{definition}

Let $E_8$ denote the even $8\times 8$ Cartan matrix of the eponymous exceptional Lie algebra; that is,
\[
E_8=\begin{pmatrix}
    2 & 1 &0 &0&0&0&0&0 \\
    1& 2 & 1 & 0&0&0&0&0 \\
  0& 1 & 2& 1 &0 &0 &0&0 \\
  0 &0 &1 & 2& 1 & 0&0&0 \\
 0&0&0&1&2&1& 0 & 1 \\
 0&0&0&0&1&2&1&0 \\
 0&0&0&0&0&1&2&0 \\
 0&0&0&0&1&0&0&2
   \end{pmatrix}. \]
Note that this is a symmetric integral matrix with determinant one.

\begin{example}
  Here are some important closed, smooth 4-manifolds.
  \begin{enumerate}[leftmargin=1cm]
    \item The 4-sphere $S^4$. This is simply connected and has $H_2(S^4;\Z)=\{0\}$.
    \item The complex projective plane $\cp^2$, which comes with a canonical orientation.  The same underlying manifold with the opposite orientation is $\overline{\cp^2}$. They are simply-connected manifolds with $H_2(\cp^2;\Z) \cong \Z$. The intersection form of $\cp^2$ is represented by the $1\times 1$-matrix  $(1)$ and the intersection form of $\overline{\cp^2}$ is represented by the $1\times 1$-matrix $(-1)$.
 \item  The manifold $S^2 \times S^2$ is simply-connected and $H_2(S^2 \times S^2;\Z) \cong \Z \oplus \Z$. The intersection form of $S^2 \times S^2$ is represented by the standard hyperbolic form
     $H:= \begin{pmatrix}
       0 & 1 \\ 1 & 0
     \end{pmatrix}$.
 \item The $K3$ surface or Kummer surface
\[ K3\,\,:=\,\, \big\{ [z_1:z_2:z_3:z_4]\in \cp^3\,\big|\, z_1^4+z_2^4+z_3^4+z_4^4=0\big\}\]
  This is a simply connected, smooth, spin, closed 4-manifold with $H_2(K3;\Z) \cong \Z^{22}$. As is shown in
  \cite[Theorem~1.3.8]{GS99} or alternatively \cite[p.~176]{McDuff-Salamon}, the intersection form of K3 is isometric to $E_8 \oplus E_8 \oplus H \oplus H \oplus H$.
          \end{enumerate}
\end{example}

The next proposition shows that the intersection form is well behaved under the connected sum operation.

\begin{proposition}\label{prop:additivity-of-intersection-form}
Let $M$ and $N$ be two oriented compact $4$-manifolds. Then
there is an isomorphism $H_2(M)\oplus H_2(N)\to H_2(M\# N)$ that induces an isometry of
 $Q_M\oplus Q_N$ and $Q_{M\# N}$.
\end{proposition}

\begin{proof}
The usual tools of algebraic topology, namely  a Mayer-Vietoris argument and the excision theorem, show that there exists
an isomorphism $\Theta\colon H_2(M)\oplus H_2(N)\to H_2(M\# N)$.

The statement that this isomorphism  $\Theta$ induces an isometry between  $Q_M\oplus Q_N$ and $Q_{M\# N}$
can be deduced from the functoriality of the cup and cap products \cite[Theorem~VI.5.2.(4)]{Br93} for maps between pairs of topological spaces.  Full details are provided in \cite[Proposition~153.12]{Fr23}.

In the smooth case the statement that the isomorphism $\Theta$ induces an isometry of forms
follows immediately from the fact that any class in second homology can be represented by an embedded oriented submanifold \cite[Proposition~1.2.3]{GS99} and the fact that one can calculate the intersection form in terms of algebraic intersection numbers of embedded oriented surfaces \cite[Theorem VI.11.9]{Br93}.
To apply this approach to general manifolds, one needs to use topological transversality, which holds, as discussed in Chapter~\ref{chapter:transversality}.
\end{proof}

\section{Intersection forms of spin manifolds}
Using the results from the previous chapter we can prove the following proposition.

\begin{proposition}\label{proposition-int-form-spin-even}
  Let $M$ be a compact, connected, oriented 4-manifold. If $M$ is spin then the intersection form of $M$ is even.
\end{proposition}

\begin{proof}
In this proof we use the properties of the Steenrod squares introduced just before Definition~\ref{defn:SW-classes-Stern}.

By definition the $k$th Wu class $v_k\in H^k(M;\Z/2)$ satisfies $\Sq^k(a)=v_k\cup a$, for every class $a\in H^{4-k}(M,\partial M;\Z/2)$. Hence if $a\in H^2(M,\partial M;\Z/2)$ then $v_2\cup a=\Sq^2(a)=a\cup a$.
By Proposition~\ref{prop:Wu-formulae}, the $n$th Stiefel-Whitney class of $M$ is given by $w_n=\sum_i\op{Sq}^i(v_{n-i})$.
Since $M$ is oriented, we have \[0=w_1=\op{Sq}^0(v_1) + \Sq^1(v_0) =v_1.\] Since $M$ is spin, we have that \[0=w_2=\op{Sq}^0(v_2)+\op{Sq}^1(v_1) + \Sq^2(v_0)=v_2.\] So for any $a\in H^2(M,\partial M;\Z/2)$, we have $a\cup a=0\cup a= 0\in H^4(M,\partial M;\Z/2) =\Z/2$. But this implies that for any $x\in FH_2(M;\Z)$ we have that $Q_M(x,x)=\langle \PD^{-1}(x)\cup \PD^{-1}(x),[M,\partial M]\rangle \equiv 0\pmod 2$. In other words $Q_M$ is an even form.
\end{proof}

\begin{proposition}\label{prop:divisible-by-8}
Let $M$ be a closed, oriented, connected, spin 4-manifold. Then the signature $\op{sign}(M)$ is divisible by 8.
\end{proposition}

\begin{proof}
%
In Proposition \ref{proposition-int-form-spin-even} we just proved that the intersection form $Q_M$ is even.  Since $M$ is closed we know that  that $Q_M$ is nonsingular.
Finally note that it is  an algebraic fact, shown for example in~\cite[Theorem~5.1]{MH}, that for any symmetric nonsingular, bilinear, even form $Q$, the signature of $Q$ is divisible by~8.
\end{proof}

\section{Twisted intersection forms and twisted signatures}\label{chap:twisted-int-form}

In this section we introduce twisted intersection forms for topological manifolds and discuss some properties of the corresponding twisted signatures.

Let $M$ be a compact, orientable, connected $4m$-dimensional manifold. We write $\pi := \pi_1(M)$. Let $\alpha\colon \pi\to U(k)$ be a unitary representation. We view the elements of  $\C^k$ as row vectors. Given $g\in \pi$ and $v\in \C^k$, define $v\cdot g:=v\cdot \alpha(g)$.
Thus we can view $\C^k$ as a right $\Z[\pi]$-module. Denote this module by $\C^k_\alpha$.
Define the \emph{twisted intersection form} of $(M,\alpha)$ to be the form
\[ \ba{rcl} Q_M\colon H_{2m}(M;\C^k_\alpha)\times H_{2m}(M;\C^k_\alpha)\,\xrightarrow{\op{PD}^{-1}\times \op{PD}^{-1}} & H^{2m}(M,\partial M;\C^k_\alpha)\times H^{2m}(M,\partial M;\C^k_\alpha)\\[0.1cm]
& \hspace{13pt}\downarrow \acup \\[0.1cm]
& H^{4m}(M,\partial M;\C^k_\alpha\otimes \C^k_\alpha)\\[0.1cm]
&\hspace{23pt}\downarrow {\langle\,,\,\rangle}\\[0.1cm]
 & H^{4m}(M,\partial M;\C) \\[0.1cm]
&\hspace{23pt}\downarrow \PD  \\[0.1cm]
& \hspace{10pt} H_0(M;\C)=\C.\ea\]
Here the first and the last map are given by the isomorphisms from the Poincar\'e Duality Theorem~\ref{thm:poincareduality} and the second map is given by Lemma~\ref{lem:cup-product}.
Note that in the bottom we view $\C$ as a trivial $\Z[\pi]$-module. The third map is induced by the following homomorphism of right $\Z[\pi]$-modules:
\[ \ba{rcl} \C^k_\alpha \otimes \C^k_\alpha&\to & \C\\
(v,w)&\mapsto & \langle v,w\rangle=\ol{v}{w}^T.\ea\]
It follows easily from the definitions that $Q_M$ is sesquilinear, namely $\C$-conjugate linear in the first entry and $\C$-linear in the second entry.
The usual proof for the (anti-) symmetry of the cup product e.g.\ \cite[Theorem~3.14]{Hat02}, can be modified to show that $Q_M$ is  hermitian, that is for every $v,w\in H^{2m}(M;\C^k_\alpha)$ we have $Q_M(v,w)=\ol{Q_M(w,v)}$.
Since $Q_M$ is hermitian, its signature is defined as the difference in the number of positive and negative eigenvalues. We refer to the signature of $Q_M$ as the \emph{twisted signature} $\sigma(M,\alpha)$.

For a group homomorphism $\gamma\colon \pi_1(M)\to \Gamma$, denote the corresponding $L^2$-signature by $\sigma^{(2)}(M,\gamma)$, as defined in say~\cite{At76,Lu02} and \cite[Chapter~5]{COT03}.

\begin{theorem}\label{thm:multiplicativity-of-signature}
Let $M$ be a closed, oriented,  connected  $4$-manifold.
\bnm
\item For every finite cover $p\colon \wti{M}\to M$ we have $\sigma(\wti{M})=[\wti{M}:M]\cdot \sigma(M)$.
\item For every unitary representation $\alpha\colon \pi_1(M)\to U(k)$ we have $\sigma(M,\alpha)=k\cdot \sigma(M)$.
\item For every group homomorphism $\gamma\colon \pi_1(M)\to \Gamma$ we have $\sigma^{(2)}(M,\gamma)=\sigma(M)$.
\enm
\end{theorem}

\begin{remark}\leavevmode
\bnm
\item
The same statement does not hold for $4$-dimensional Poincar\'e complexes in general. More precisely, Wall~\cite[Corollary~5.4.1]{Wa67} gave examples of $4$-dimensional Poincar\'e complexes for which the signature is not multiplicative under finite covers.
\item Alternative proofs for the first  and the
third statement are provided by  Schafer~\cite[Theorem~8]{Schafer70}
and L\"uck-Schick~\cite[Theorem~0.2]{LS01}. The approach taken in
L\"uck-Schick~\cite{LS01}
and Teleman \cite{Teleman1984} should also provide a proof of the second statement.
In fact these papers are also valid for manifolds of any dimension $4m$.
\enm
\end{remark}

\begin{proof}
First we give references for these three statements  for smooth manifolds.
\bnm
\item This  statement is a consequence of the Hirzebruch Signature Theorem (see e.g.\ \cite{MS74}).
\item This  statement was proven in \cite{APS75} (in fact the second statement contains the first statement as a special case).
\item This  statement was proven in \cite[p.~44]{At76}.
\enm
We now turn to manifolds that are not necessarily smooth.
We will prove the second statement of the theorem. The other statements can be proved in a  similar fashion. We refer to \cite[Lemma~5.9]{COT03} for a proof of (3).

So let $M$ be a closed oriented  connected $4$-manifold
and let $\alpha\colon \pi_1(M)\to U(k)$ be a  unitary representation.
By Theorem~\ref{thm:connect-sum-is-smooth}, there exists a closed orientable simply-connected 4-manifold $N$ such that $M\# N$ is smooth. We have $\pi_1(M\# N)=\pi_1(M)*\pi_1(N) \cong \pi_1(M)$ since $\pi_1(N)=\{1\}$.
Let $\beta\colon \pi_1(N)\to U(k)$ be the trivial representation.
We also write $\alpha*\beta \colon \pi_1(M\# N)=\pi_1(M)\to U(k)$ for the representation uniquely determined by $\alpha$ on $\pi_1(M)$.

By  Proposition~\ref{prop:additivity-of-intersection-form}, we have $\sigma(M\# N)=\sigma(M)+\sigma(N)$. Furthermore a slight generalisation of  Proposition~\ref{prop:additivity-of-intersection-form} shows that $\sigma(M\# N,\alpha*\beta)=\sigma(M,\alpha)+\sigma(N,\beta)$. Finally, we have $\sigma(N,\beta)=k\cdot \sigma(N)$.
The desired statement follows from these equalities and from the formula for twisted signatures of the closed smooth manifold $M\# N$.
\end{proof}

\section{Intersection forms of smooth 4-manifolds}
In Theorem~\ref{thm:classn-simply-connected-4-mfld} we will see that any unimodular symmetric form occurs as the intersection form of a closed oriented 4-manifold.
In the following we survey results on intersection forms of closed oriented smooth 4-manifolds. As we will see, the results in the smooth setting differ dramatically from the results in the topological setting.

In Proposition~\ref{prop:divisible-by-8} we saw that the signature of any closed, oriented, connected, spin 4-manifolds is divisible by 8.
The Rochlin Theorem~\cite{Rohlin} gives an extra restriction on the signatures of intersection forms of spin 4-manifolds that admit a smooth structure.

\begin{theorem}\textbf{\textup{(Rochlin Theorem)}}\label{thm:rokhlin}
Let~$M$ be a closed, oriented, connected, spin, smooth 4-manifold. Then the signature $\op{sign}(M)$ is divisible by 16.
\end{theorem}

\begin{remark}\label{rem:evenimpliesspin}
Let $M$ be a closed oriented 4-manifold with an even intersection form and such that $H_1(M;\Z)$ has no 2-torsion. This implies $H^2(M;\Z/2)\cong \Hom(H_2(M;\Z),\Z/2)$ and that the mod 2 reduction of $Q_M$ is isomorphic to the pairing $(a,b)=\langle a\cup b,[M]\rangle$ on $H^2(M;\Z/2)$. As $Q_M$ is even, this implies that $(a,a)=0\in \Z/2$ for any $a\in H^2(M;\Z/2)$. But we saw in the proof of Proposition~\ref{prop:divisible-by-8}
that $a\cup a=v_2\cup a$, so we must have that $v_2=0$ as this pairing is nondegenerate. We also saw in the proof of
Proposition~\ref{prop:divisible-by-8}  that $v_2=w_2$ when $M$ is oriented, so in fact $w_2=0$ and $M$ admits a spin structure.

It is not true that simply having an even intersection form implies $M$ is spin. Indeed, it is possible to construct a closed oriented $4$-manifold $M$ that has $Q_M=0$ (which is in particular an even form), but has nonvanishing $w_2$ \cite[Exercise 5.7.7(a)]{GS99}. In a similar spirit, by \cite{habegger82,FS84} there exists a closed oriented 4-dimensional smooth manifold $M$ with an even intersection form $Q_M$ that satisfies $\op{sign}(M)=8$. Hence this must also fail to be spin, now by the
Rochlin Theorem~\ref{thm:rokhlin}.
\end{remark}

\begin{theorem}[Freedman~\cite{Freedman-82}]
There exists a  closed orientable connected 4-manifold that does not admit a smooth structure.
\end{theorem}

\begin{proof}
By Theorem~\ref{thm:classn-simply-connected-4-mfld} there exists a simply connected closed oriented 4-manifold $M$ with $Q_M\cong E_8$.
By the Rochlin Theorem~\ref{thm:rokhlin} this manifold does not admit a smooth structure.
\end{proof}

In a remarkable twist, shortly after Freedman proved Theorem~\ref{thm:classn-simply-connected-4-mfld}, Donaldson \cite[Theorem~A]{Donaldson83} \cite[Theorem~1]{donaldson87}, proved the following result regarding intersection forms of  smooth 4-manifolds.

\begin{theorem}\textbf{\textup{(Donaldson's Theorem)}}\label{thm:donaldson}
Let~$M$ be a closed oriented connected smooth 4-manifold. If $Q_M$ is positive-definite, then $Q_M$ can be represented by the identity matrix.
\end{theorem}

To understand the significance of Donaldson's Theorem it is helpful to consider the following table from \cite[p.~28]{MH}, which basically says that there are lots of isometry types of nonsingular positive definite forms.
\[
\ba{rcccccc}
\mbox{Dimension:}&\,\,\,\,8\,\,\,\,&\,\,16\,\,&\,\,24\,\,&\,\,32\,\,&\,\,40\,\,\\
\hline
\ba{r}\mbox{Number of isometry types of nonsingular}\\
\mbox{positive definite even symmetric forms:}\ea&1&2&24&\geq 10^7& \geq 10^{51}\\
\ea
\]

\begin{remark}
Note that if
$M$ be a closed oriented smooth 4-manifold such that  $Q_M$ is negative-definite,
then $Q_{-M}=-Q_M$ is positive-definite. Thus we see that Donaldson's Theorem implies that $Q_M$ is represented by $-\id$.
\end{remark}

It follows from  \cite[Theorem II.5.3]{MH} that every nonsingular indefinite \emph{odd} symmetric form is isometric to $k\cdot (1)\oplus \ell\cdot (-1)$. These are realised by $k\cdot \cp^2\# \ell\cdot \ol{\cp}^2$.
Therefore we only need to discuss the realisability of nonsingular indefinite \emph{even} symmetric forms.
Again by  \cite[Theorem II.5.3]{MH}, every nonsingular \emph{even} indefinite symmetric form is isometric to $n\cdot E_8 \oplus  m\cdot H$ for some $(m,n) \in \N_0 \times \Z \sms \{(0,0)\}$.
The following theorem, proven by Furuta \cite{furuta01}, gives some restrictions on the possible values of $m$ and $n$.

\begin{theorem}\textbf{\textup{(Furuta's 10/8 Theorem)}}
If $M$ is a closed oriented connected smooth 4-manifold with \emph{indefinite} even intersection form, then
\[ b_2(M)\,\,\geq\,\, \tmfrac{10}{8}\cdot |\op{sign}(M)|+2.\]
In particular $Q_M\cong n\cdot E_8 \oplus m\cdot H$ for some $n \in 2\Z$ and $m \in \N$ with $m\geq |n|+1$.
\end{theorem}

Furuta's 10/8 Theorem  does not quite close the gap between the forms we can realise by smooth manifolds and the forms  we can exclude.
More precisely, it follows from the calculation of the intersection form of the K3 surface and of $S^2\times S^2$ that for any $n = 2p \in \Z$ and every $m\geq 3|p|$ there exists a closed oriented simply connected 4-dimensional smooth manifold with intersection form isometric to $n\cdot E_8\oplus m\cdot H$. In other words, we have
\[ \mbox{intersection form of $p\cdot \textup{K}3\,\# \,(m-3|p|)\cdot (S^2\times S^2)$}\,\,\cong\,\, 2p\cdot E_8\oplus m \cdot H.\]
The following conjecture predicts that this result is optimal.

\begin{conjecture}\textbf{\textup{(11/8-Conjecture)}}
If $M$ is a closed oriented connected smooth 4-manifold with indefinite even intersection form, then
\[ b_2(M)\,\,\geq\,\, \tmfrac{11}{8}\cdot |\op{sign}(M)|.\]
Equivalently, if $Q_M\cong 2p\cdot E_8 \oplus m\cdot H$ with $p\ne 0$, then $m\geq 3|p|$.
\end{conjecture}

\begin{remark}
\leavevmode
\bnm
\item
A proof of the 11/8-Conjecture would imply, by Freedman's Theorem~\ref{thm:classn-simply-connected-4-mfld}, that any closed oriented simply connected  smooth 4-manifold is \emph{homeomorphic} to either a connected sum of the form
$k\cdot \cp^2\# \ell\cdot \ol{\cp}^2$ or to a connected sum of the form $n\cdot \textup{K}3\# m \cdot (S^2\times S^2)$.
\item Currently the best known result in the direction of the 11/8-Conjecture is \cite[Corollary~1.13]{hopkins2018intersection}, which says that if $M$ is a closed oriented simply-connected 4-manifold
 that is not homeomorphic to $S^4$, $S^2\times S^2$ or the K3 surface
 and whose intersection form is indefinite and even, then
 $b_2(M)\geq \frac{10}{8}\cdot |\op{sign}(M)|+4$.
\enm
\end{remark}

\chapter{Smoothing 4-manifolds}\label{chapter:smoothing}
In this chapter we present three theorems which associate a smooth manifold to a given 4-manifold. Often these theorems can be used to reduce proofs about 4-manifolds to the case of smooth 4-manifolds, where the standard tools of differential topology are available.

\section{Smoothing noncompact 4-manifolds}
The first of our smoothing theorems~\cite[Corollary~2.2.3]{Qu82}, \cite[p.~116]{FQ90}, which is due to Freedman and Quinn, says that noncompact connected 4-manifolds admit a smooth structure.

\begin{theorem}\label{thm:smooth-outside-a-point}
Every connected, noncompact 4-manifold is smoothable. Thus every $4$-manifold $M$ has a smooth structure in the complement of
any closed set that has at least one point in each compact component of $M$.
\end{theorem}

There are some related statements in the literature on smoothing 4-manifolds in the complement of a point, that appeared prior to Freedman's work~\cite{Freedman-82} and prior to \cite{Qu82}.  We discuss them briefly here.
For the case of $\PL$ structures on noncompact 4-manifolds, given a lift of the (unstable) tangent microbundle classifying map $M \to \BTOP(4)$ to $M \to \BPL(4)$ (see Chapter~\ref{chapter:bundlestructures}), the result can be found in~\cite[p.~54]{Lashof-immersion-approach-2} and~\cite[Essay~V,~Addendum~1.4.1, p.~222]{KS77}. The analogous result for smooth bundle structures and smooth structures on noncompact manifolds was stated in~\cite[p.~156]{Lashof-immersion-approach-1}.
Alternatively, \cite{Hirsch-Mazur74}, \cite[Theorem~8.3B]{FQ90} apply to improve a $\PL$ structure to a smooth structure, unique up to isotopy, for any manifold of dimension at most six.
Again, in ~\cite{Lashof-immersion-approach-1} Lashof assumes a lift of the (unstable) tangent microbundle classifying map $M \to \BTOP(4)$ to a map $M \to \BO(4)$.  For noncompact connected 4-manifolds, such a lift always exists, as was later shown by Quinn~\cite{Qu82,Quinn-smooth-structures},~\cite[p.~116]{FQ90} using the full disc embedding theorem~\cite{Freedman-82}, and giving rise to Theorem~\ref{thm:smooth-outside-a-point}.

Due to the seminal nature of Freedman's Fields medal winning paper~\cite{Freedman-82}, it is well worth clarifying the details of some citations therein.
In the proof of Corollary~1.2, in the proof of Theorem~1.5 on page~369, in the proof of Theorem~1.6, and at the start of Section 10, Freedman uses that smoothing theory is available for noncompact 4-manifolds. In particular, smoothing for noncompact contractible 4-manifolds plays a vital r\^{o}le in Freedman's proof of the topological 4-dimensional Poincar\'{e} conjecture~\cite[Theorem~1.6]{Freedman-82}. Freedman cites~\cite{KS77} for this fact, however  \cite[Essay~V,~Remarks~1.6~(A), p.~230]{KS77} specifically excludes smooth structures (but for a stronger result). Nevertheless, as mentioned above, Lashof~\cite[p.~156]{Lashof-immersion-approach-1} proved the smooth version of \cite[Essay~V,~Addendum~1.4.1, p.~222]{KS77}, or one can use $\PL$ smoothing theory \cite{Hirsch-Mazur74}, \cite[Theorem~8.3B]{FQ90} to improve a PL structure from \cite[Essay~V,~Addendum~1.4.1, p.~222]{KS77} to a smooth structure, essentially uniquely.

Freedman only applies smoothing theory in cases, such as for contractible $M$, that he can ensure the existence of a lift of $\tau_M \colon M \to \BTOP(4)$ to $\BO(4)$.
Later, Quinn~\cite[Corollary~2.2.3]{Qu82} showed that such a lift always exists for connected noncompact 4-manifolds. In fact, he showed that the map $\TOP(4)/\O(4) \to \TOP/\O$ is 5-connected~\cite[Theorem~8.7A]{FQ90}, where only 3-connected is needed for Theorem~\ref{thm:smooth-outside-a-point}. In other words, it was shown prior to Freedman's work that homotopy 4-spheres admit a smooth structure in the complement of a point, so the results that Freedman required were indeed known. However, smoothing in the complement of a point was not known for general connected, compact 4-manifolds until after the work of Quinn in 1982.  Further discussion can also be found in Quinn~\cite{Quinn-smooth-structures} and Lashof-Taylor~\cite{Lashof-Taylor-84}.

Below we will give applications of Theorem~\ref{thm:smooth-outside-a-point};
see e.g.\ the proof of Theorem~\ref{thm:represent-homology-by-submanifolds}.

\section{The Kirby-Siebenmann invariant and stable smoothing of 4-manifolds}

The formulation of the other two statements on smoothing 4-manifolds that we will give (Theorems~\ref{thm:smooth-outside-a-point} and~\ref{thm:connect-sum-is-smooth}) make use of the Kirby-Siebenmann invariant.
The Kirby-Siebenmann invariant $\op{ks}(M)\in \Z/2$ of a compact 4-manifold is defined in~\cite[Section~10.2B]{FQ90}, or alternatively by \cite[p.~318]{KS77} or \cite[Definition~3.4.2]{Rudyak16}, and we describe the construction now.

The homotopy fibre $\TOP/\PL$ of the forgetful map $\BPL \to \BTOP$ has the homotopy type of a $K(\Z/2,3)$~\cite[Essay~IV, \S 10, p.~194]{KS77} 
and has the structure of a loop space, permitting the construction of the delooping $\B(\TOP/\PL)$~\cite[Theorem~C]{Boardman-Vogt-68}, \cite{Boardman-Vogt-book73}, which is an Eilenberg-Maclane space of type $K(\Z/2,4)$. A connected topological 4-manifold has a unique smooth structure on its boundary. Using the homotopy fibre sequence
\[\TOP/\PL \to \BPL \to \BTOP \to \B(\TOP/\PL),\]
the unique obstruction to a lift of the classifying map $\tau_M \colon M \to \BTOP$ of the stable topological tangent bundle  to $\BPL$ is therefore a homotopy class  in
\[[(M,\partial M),(\us{=K(\Z/2,4)}{\ub{\B(\TOP/\PL)}},*)] \cong H^4(M,\partial M;\Z/2) = \Z/2.\]
Here, we used again that 4-manifolds have the homotopy type of a CW complex (Theorem~\ref{thm:topological-manifold-CW complex}).
We refer to the corresponding element of $\Z/2$ as the  \emph{Kirby-Siebenmann  invariant}, $\op{ks}(M)$,  of the compact, connected manifold $M$.
For disconnected compact 4-manifolds, $M = \bigsqcup_{i=1}^n M_i$, define \[\ks(M) := \sum_{i=1}^n \ks(M_i) \in \Z/2.\]

For comparision, we note that in every dimension, a compact $n$-manifold with a $\PL$ structure on its boundary has a Kirby-Siebenmann invariant, which is a homotopy class of maps in $[(M,\partial M),(\B(\TOP/\PL),*)]$ determined by its topological tangent bundle.
In the following theorem we summarise some key properties of the Kirby-Siebenmann invariant.

\begin{theorem}\label{thm:ks-basics}
Let $M$ and $N$ be compact 4-manifolds.
\bnm
\item\label{item-ks-basics-1}
If $M\times \R$ admits a smooth structure $($e.g.\ if $M$ admits a smooth structure$)$, then $\op{ks}(M)=0$.
\item\label{item-ks-basics-4} The Kirby-Siebenmann invariant gives rise to a surjective homomorphism $\Omega_4^{\TOP} \to \Z/2$. In particular for $M$ a closed 4-manifold that bounds a compact 5-manifold, $\ks(M)=0$.
\item\label{item-ks-basics-2} The Kirby-Siebenmann invariant is additive under the connected sum operation.
\item\label{item-ks-basics-6} There is a short exact sequence
\[
0 \to \Omega^{\op{Spin}}_4 \to \Omega_4^{\op{TOPSpin}} \xrightarrow{\sigma/8} \Z/2 \to 0,
\]
with the first map the forgetful map and last map given by the signature divided by 8, modulo 2. This sequence does not split, so $\Omega_4^{\op{TOPSpin}} \cong \Z$. Moreover, for a closed spin manifold, the signature divided by 8, modulo 2, is equal to the Kirby-Siebenmann invariant, i.e.\ the map which takes $\sigma/8$ modulo $2$ equals the composition $\Omega_4^{\op{TOPSpin}} \to \Omega_4^{\op{STOP}} \xrightarrow{\op{ks}} \Z/2$.
\item\label{item-ks-basics-5}
If $S\subseteq \partial M$ and $T\subseteq \partial N$ are compact codimension zero submanifolds with a homeomorphism $S \cong T$, then \[\op{ks}(M\cup_{S\cong T} N) = \op{ks}(M)+\op{ks}(N).\]
\item\label{item-ks-basics-3}
If there exists a compact 5-manifold with $\partial W=M\cup_{\partial M\cong \partial N} N$, for some homeomorphism $\partial M \cong \partial N$, then $\op{ks}(M)=\op{ks}(N)$.
\enm
\end{theorem}

While they are certainly well-known to the experts, and frequently used, we could not find explicit proofs of these facts in the literature, so we give some details.

%
%
%

\begin{proof}[Proof of Theorem~\ref{thm:ks-basics}]
Let us prove (\ref{item-ks-basics-1}).
The topological tangent bundle of $M \times \R$ is isomorphic to $\tau_M \oplus \varepsilon$, where $\tau_M$ is the tangent microbundle of $M$ and $\varepsilon$ denotes a rank one trivial bundle over $M$.  If $M \times \R$ admits a smooth structure, then there is a lift $\tau_{M \times\R}^{\Diff} \colon M \times \R \to \BO(5)$, the smooth tangent bundle to $M \times \R$. Let $p \colon \BO(5) \to \BTOP(5)$ be the canonical map. Then $\tau_M \oplus \varepsilon = p \circ \tau_{M \times\R}^{\Diff}$.  Passing to the stable classifying spaces, we obtain a lift $M \to \BO$ whose composition with the canonical map $\BO \to \BTOP$ agrees with $\tau_M \oplus \varepsilon^{\infty}$, the stable tangent microbundle of $M$.  Since the map $\BO \to \BTOP$ factors through $\BPL \to \BTOP$, we have a stable lift of $\tau_M$ and so $\op{ks}(M)=0$.  This completes the proof of (\ref{item-ks-basics-1}).


%

Now to prove~(\ref{item-ks-basics-4}),
suppose that a closed 4-manifold $M = \bigsqcup_{i=1}^k M_i$ bounds a compact 5-manifold $W'$. Perform 0 and 1-surgeries on $W'$ to obtain a path connected, simply connected, compact 5-manifold $W$ with $\partial W = M$.
We prove that \[\ks(M) := \sum_{i=1}^k \ks(M_i) =0.\]
Consider the diagram
\[\xymatrix@R+0.5cm @C+0.5cm{M_i \ar[r] \ar[ddrrr]_{\ks(M_i)} & M \ar[r] \ar[ddrr]^(0.3){\ks(M)} & W \ar[r]^{\tau_W} \ar[ddr]^{\ks(W)} & \BTOP  \ar[dd] &  \\ &&&& \\
&&&\B(\TOP/\PL) \ar[r]^{\simeq}  & K(\Z/2,4). }\]
The restriction $M_i \to W \to \BTOP$ equals the stable tangent microbundle of $M$, since $M$ has a collar $M \times [0,1] \subseteq W$ by Theorem~\ref{thm:collar}.
Therefore the diagram commutes.  It follows that the top left horizontal map in the next diagram sends $\ks(W)$ to $(\ks(M_1),\dots,\ks(M_k))$, so the map $H^4(W;\Z/2) \to \Z/2$ sends $\ks(W)$ to $\sum_{i=1}^k \ks(M_i) = \ks(M)$.
\[\xymatrix{H^4(W;\Z/2) \ar[r] \ar[dd]^-{\cong}_-{\PD} & H^4(M;\Z/2)  \ar[r]^-{\cong} \ar[dd]^-{\cong}_-{\PD} & \bigoplus_{i=1}^k H^4(M_i;\Z/2) \ar[dr]^{\cong} \ar[dd]^-{\cong}_-{\PD} &  & \\
& & & \bigoplus_{i=1}^k \Z/2 \ar[r]^-{(1,\dots,1)} &\Z/2 \\
H_1(W,M;\Z/2) \ar[r]  & H_0(M;\Z/2) \ar[r]^-{\cong} & \bigoplus_{i=1}^k H_0(M_i;\Z/2) \ar[ur]^-{=} & &
 }\]
The left square of this diagram commutes by Poincar\'{e}-Lefschetz duality.  The middle square and the triangle commute trivially.  But since $W$ is connected and simply connected, every element of $H_1(W,M;\Z/2)$ can be represented by a (possibly empty) union of arcs with boundary on $M$. Thus the image of $\ks(W)$ in $\bigoplus_{i=1}^k H_0(M_i;\Z/2)$ is nonzero in evenly many summands, and therefore its image in $\Z/2$ on the far right is zero.  By commutativity of the diagram it follows that $\ks(M)=0$, as desired.

Now (\ref{item-ks-basics-4}) follows. First note that the addition on $\Omega_4^{\TOP}$ is by disjoint union, so $\ks$ is additive by definition. We have just shown that the map $\ks \colon \Omega_4^{\TOP} \to \Z/2$ is well-defined, since for $M$ a closed 4-manifold that bounds a compact 5-manifold, $\ks(M)=0$.  Therefore $\ks \colon \Omega_4^{\TOP} \to \Z/2$ is a homomorphism as desired.

Freedman~\cite[Theorem~1.7]{Freedman-82}  showed that there exists
a closed 4-manifold with intersection pairing that is isometric to $E_8$.
This manifold is called the \emph{$E_8$ manifold}.
Since the $E_8$ form has signature $8$, $\sigma(E_8)=8$, so $\sigma(E_8)/8=1$. Since the $E_8$ form is even, and $H_1(E_8;\Z)=0$, $E_8$ is spin.
The  construction of the $E_8$ manifold was a key step in the proof of the Classification Theorem~\ref{thm:classn-simply-connected-4-mfld}.
We will now show  that  $\ks(E_8) =1$.  To see this note that $E_8$ cannot be smoothed, even after adding copies of $S^2 \times S^2$, by the Rochlin Theorem~\ref{thm:rokhlin} that every closed spin smooth 4-manifold has signature divisible by 16. Whereas if $\ks(E_8)=0$, then $E_8$ would be stably smoothable by Theorem~\ref{thm:connect-sum-is-smooth}.  Therefore $\ks \colon \Omega_4^{\TOP} \to \Z/2$ is surjective.

Now we can prove (\ref{item-ks-basics-2}) easily.
Observe that a disjoint union $M \sqcup N$ is cobordant to $M \# N$ via the cobordism
\[(M \times I \sqcup N \times I) \cup_{S^0 \times D^4} (D^1 \times D^4),\]
with $\{-1\} \times D^4$ embedded in the interior of $M \times \{1\}$, and $\{1\} \times D^4$ embedded in the interior of $N \times \{1\}$. Then we have just shown that the Kirby-Siebenmann invariant vanishes on $M \# N \sqcup M \sqcup N$ and therefore $\ks(M\# N) = \ks(M) + \ks(N) \in \Z/2$.\\



 To prove~(\ref{item-ks-basics-6}), we consider the following diagram. The maps between bordism groups are structure forgetting maps, so the diagram commutes.
 \[\xymatrix@R0.76cm{ 0 \ar[r] & \Omega_4^{\op{Spin}}  \ar[r] \ar[d]^{\cdot 16} & \Omega_4^{\op{TOPSpin}}  \ar[r]^-{\ks} \ar[d] & \Z/2 \ar[r] \ar[d]^{=} & 0 \\
  0 \ar[r] & \Omega_4^{\op{SO}}  \ar[r] & \Omega_4^{\op{STOP}}  \ar[r]^-{\ks}  & \Z/2 \ar[r]  & 0 \\
   }\]
 Recall that $\Omega_4^{\op{SO}} \cong \Z$ given by the signature and generated by $\mathbb{CP}^2$. The signature provides a splitting homomorphism, so $\Omega_4^{\op{STOP}} \cong \Z \oplus \Z/2$. Also $\Omega_4^{\op{Spin}} \cong \Z$ given by the signature divided by 16 and generated by the $K3$ surface, so the forgetful map $\Omega_4^{\op{Spin}} \to \Omega_4^{\op{SO}}$ becomes, on identifying domain and codomain with $\Z$, multiplication by $16$.

Both sequences are exact: a smooth manifold has vanishing $\ks$ invariant, and vanishing $\ks(M)$ implies smoothable after adding copies of $S^2 \times S^2$ by Theorem~\ref{thm:connect-sum-is-smooth} below. Since $M \# (S^2 \times S^2)$ is (spin) bordant to $M$, the sequences are exact at their middle terms. The maps labelled $\ks$ are surjective because the $E_8$ manifold is spin and has $\ks(E_8)= 1$, as discussed in the proof of (\ref{item-ks-basics-4}).  Finally, after surgery to make it 1-connected,  a topological null bordism of a compact smooth 4-manifold can be smoothed by high dimensional smoothing theory, so the left hand maps are injective.

 We claim that the sequence in the upper row does not split. Consider the $K3$ surface generating $\Omega_4^{\op{Spin}} \cong \Z$. By the down-then-left route, $[K3]$ maps to $(16,0) \in \Z \oplus \Z/2 \cong \Omega_4^{\op{STOP}}$.  On the other hand the $E_8$ manifold represents a class in $\Omega_4^{\op{TOPSpin}}$ and maps to $(8,1) \in \Z \oplus \Z/2 \cong \Omega_4^{\op{STOP}}$.

Since $\ks$ is a homomorphism by (\ref{item-ks-basics-4}), we see that $2\cdot [E_8]$, which equals $[E_8 \#E_8]$ by \eqref{item-ks-basics-2}, maps to $0 \in \Z/2$ and so has trivial $\ks$ invariant. By exactness of the top row it lies in the image of $\Omega_4^{\op{Spin}}$. Let $N$ be a closed spin smooth 4-manifold $\op{TOPSpin}$-bordant to $E_8 \# E_8$.  Since $\sigma(E_8 \# E_8) = 16 = \sigma(K3)$, we have $[N] = [K3] \in \Omega_4^{\op{Spin}}$. It follows that $K3$, the generator of $\Omega_4^{\op{Spin}} \cong \Z$, maps to $2 \cdot [E_8] \in \Omega_4^{\op{TOPSpin}}$.  Thus we have a diagram with exact rows:
\[\xymatrix @R+0.25cm @C+0.85cm { 0 \ar[r] & \Z  \ar[r]^-{\cdot 2} \ar[d]^{=} & \Z  \ar[r] \ar[d]^-{1 \mapsto [E_8]} & \Z/2 \ar[r] \ar[d]^-{=} & 0 \\
  0 \ar[r] & \Z  \ar[r]^-{1 \mapsto 2\cdot[E_8]} & \Omega_4^{\op{TOPSpin}}  \ar[r]^-{\ks}  & \Z/2 \ar[r]  & 0. \\
   }\]
Since $\ks(E_8)=1$, the diagram commutes. Then by the five lemma, $\Omega_4^{\op{TOPSpin}} \cong \Z$, generated by $E_8$, and the sequence does not split, as claimed.
For a topological spin, compact 4-manifold, $\sigma/8$ is an integer, by Proposition~\ref{proposition-int-form-spin-even}.
By Proposition~\ref{prop:additivity-of-intersection-form}
we know that the signature is additive. It follows from this observation
and  the fact that $\Omega_4^{\op{TOPSpin}}$ is generated by $E_8$ that $M \mapsto \sigma(M)/8$ gives rise to the isomorphism  $\Omega_4^{\op{TOPSpin}} \cong \Z$.

The diagram
\[\xymatrix@C1.5cm@R0.76cm{\Omega_4^{\op{TOPSpin}} \ar[r]^{\cong}_{\sigma/8} \ar[d] & \Z \ar[d]_-{1 \mapsto (8,1)} \ar@{->>}[dr] & \\
 \Omega_4^{\op{STOP}} \ar[r]^-{\cong}  & \Z \oplus \Z/2  \ar[r]^-{\pr_2} & \Z/2,}\]
which commutes by computing on the generator $E_8$ of $\Omega_4^{\op{TOPSpin}} \cong \Z$, shows that $\ks(M) = \sigma(M)/8\in \Z/2$ for $\op{TOPSpin}$ manifolds $M$.  This completes the proof of (\ref{item-ks-basics-6}).

To prove (\ref{item-ks-basics-5}), it was suggested by Jim Davis to consider the exact sequence
\[\Omega_4^{\O} \to \Omega_4^{\op{\TOP}} \to \Omega_4^{\{{\O} \to \op{\TOP}\}} \to \Omega_3^{\O} =0.\]
Here elements  of $\Omega_4^{\{{\O} \to \op{\TOP}\}}$ are represented by compact topological 4-manifolds with smooth boundary, considered up to 5-dimensional cobordism relative to a smooth cobordism on the boundary. That is, 4-manifolds with boundary $(M,\partial M)$ and $(N,\partial N)$ are equivalent if there is a compact 5-manifold $W$ with boundary
\[\partial W = M \cup_{\partial M} \partial_{\op{vert}} W \cup_{\partial N} N,\]
for some smooth 4-dimensional cobordism $\partial_{\op{vert}} W$ with boundary $\partial M \sqcup \partial N$.

By the exact sequence, $\Omega_4^{\{{\O} \to \TOP\}}$ is isomorphic to the cokernel of $\Omega_4^{\O} \to \Omega_4^{\op{\TOP}}$. We claim that this cokernel is isomorphic to $\Z/2$ via the Kirby-Siebenmann invariant.  To see this, by (\ref{item-ks-basics-4}) there is a surjective homomorphism $\ks \colon \Omega_4^{\TOP} \to \Z/2$.
If $\ks(M) =0$ then $M$ is stably smoothable by Theorem~\ref{thm:connect-sum-is-smooth}, so $M$ is bordant to a smooth manifold and therefore lies in the image of $\Omega^{\O}_4$.  If $M$ is smooth, then $\ks(M)$ is zero, so the sequence $\Omega_4^{\O} \to \Omega_4^{\op{\TOP}} \xrightarrow{\ks} \Z/2 \to 0$ is exact, and we may identify this sequence with the given sequence.

To prove (\ref{item-ks-basics-5}), we therefore need that the disjoint union $M \sqcup N$ is bordant to $M \cup_{S\cong T} N$, where $S\subseteq \partial M$ and $T\subseteq \partial N$ are compact codimension zero submanifolds with a choice of homeomorphism $S \cong T$.
 Here is a  construction of such a bordism. For $I=[0,1]$, take
\[(M \times I) \,\sqcup \, (S \times I \times [1/2,1])\, \sqcup\, (N \times I),\]
identify
\[S \times \{0\} \times [1/2,1]\,\, \sim\,\,  S \times [1/2,1]\,\, \subseteq\,\, ( M \times [1/2,1]),\]
and, using the identification $S \cong T$, identify
\[S \times \{1\} \times [1/2,1]\, \sim\, T \times [1/2,1] \subseteq N \times [1/2,1].\]
Let $W$ be the result of this gluing and some rounding of corners.  The boundary of $W$ is
\[(M \sqcup N) \cup_{\partial M \,\sqcup\, \partial N} \partial_{\op{vert}} W \cup_{\partial(M \cup_{S\cong T} N)} M \cup_{S\cong T} N,\]
where
\begin{align*}
  \partial_{\op{vert}} W \,\,= \,\,&(\partial M \times [0,1/2]) \cup (\overline{\partial M \sm S} \times [1/2,1]) \\ \cup &( S \times I \times\{1/2\}) \cup (\partial S \times I \times [1/2,1]) \\ &(\partial N \times [0,1/2]) \cup  (\overline{\partial N \sm T} \times [1/2,1]).
  \end{align*}
This shows that $M \sqcup N$ and  $M \cup_{S\cong T} N$ are equal in $\Omega_4^{\{\O \to \TOP\}}$, and therefore have the same Kirby-Siebenmann invariants. Since $\ks(M \sqcup N) = \ks(M) + \ks(N)$, this completes the proof of (\ref{item-ks-basics-5}).

Finally we prove (\ref{item-ks-basics-3}).  If $M \cup_{\partial M=\partial N} N$ bounds a compact 5-manifold, then by (\ref{item-ks-basics-4}) we have that $\ks(M \cup_{\partial} N) =0$. By (\ref{item-ks-basics-5}), $\ks(M) + \ks(N) = \ks(M \cup_{\partial} N) \in \Z/2$. Therefore $\ks(M) = \ks(N)$ as required. This proves  (\ref{item-ks-basics-3}) and therefore completes the proof of Theorem~\ref{thm:ks-basics}.
\end{proof}

The following theorem says that the converse to Theorem~\ref{thm:ks-basics} (\ref{item-ks-basics-1}) holds for $M$ connected.

\begin{theorem}\label{thm:MxR}
If $M$ is a compact, connected 4-manifold with vanishing Kirby-Siebenmann invariant, then $M\times \R$ admits a smooth structure.
\end{theorem}

\begin{proof}
  The vanishing of the Kirby-Siebenmann invariant implies that there is a lift of $\tau_M \colon M \to \BTOP$ to a map $M \to \BPL$. Since $\PL/{\O}$ is 6-connected~\cite[Theorem~8.3B]{FQ90}, \cite[Proof~of~4.13]{Hirsch-Mazur74}, there is in fact a lift $\widetilde{\tau}_M \colon M \to \BO$. This corresponds to a lift $\widetilde{\tau}_M \oplus \varepsilon^n \colon M \to \BO(4+n)$, for some $n$. This in turn corresponds to a lift \[\wt{\tau}_{M \times \R^n} \colon M \times \R^n \to {\BO}(4+n)\]
  of the tangent microbundle $\tau_{M \times \R^n} \colon M \times \R^n \to \BTOP$.   By \cite[Essay~V, Theorem~1.4~p.~222]{KS77}, there exists a corresponding smooth structure on $M\times \R^n$.  Then apply the Product Structure Theorem~\ref{theorem:product-structure}~\cite[Essay~I, Theorem~5.1, p.~31]{KS77}, to deduce the existence of a smooth structure on $M \times \R$, using that the dimension of $M \times \R$ is at least five.
\end{proof}

\begin{example}
  Here is an application of Theorem~\ref{thm:MxR}.  By the classification of simply connected, closed 4-manifolds~\cite[Section~10.1]{FQ90} (see also our Theorem~\ref{thm:classn-simply-connected-4-mfld}), there is a simply connected, closed 4-manifold $N$ with intersection form $E_8 \oplus E_8$. As the form is even, the manifold is spin; see Remark~\ref{rem:evenimpliesspin}. Since this form is not diagonalisable over $\Z$,  by Donaldson's Theorem~\cite{Donaldson83} (Theorem~\ref{thm:donaldson})  this 4-manifold does not admit a smooth structure.  However the Kirby-Siebenmann invariant of $N$ vanishes, since for a closed simply connected 4-manifold $M$ with even intersection form, the Kirby-Siebenmann invariant $\ks(M)$ coincides with $\sigma(M)/8 \mod{2}$ (Theorem~\ref{thm:ks-basics}\eqref{item-ks-basics-6}), and $E_8 \oplus E_8$ is rank 16 and positive definite, with signature 16.  Therefore $N \times \R$ admits a smooth structure by Theorem~\ref{thm:MxR}, even though $N$ does not.
\end{example}

We give a straightforward consequence of Theorem~\ref{thm:ks-basics} in the following. We first recall a definition. Let $N$ be an integral homology $3$-sphere. As the $3$-dimensional spin bordism group is trivial $\Omega_3^{\Spin}=0$ \cite[Theorem 5.7.14]{GS99}, we may pick a smooth, compact, orientable, spin $4$-manifold $M$ with boundary $N$. The \emph{Rochlin invariant} $\mu(N)\in\Z/2$ is defined as the quantity $\sigma(M)/8\mod2$. This is well-defined, as a consequence of Novikov Additivity \cite[Theorem~5.3]{Ki89} and Theorem~\ref{thm:rokhlin}.

\begin{theorem}\label{thm:FQ165}
Let $M$ be a compact, oriented, spin manifold with boundary $N$ an integral homology $3$-sphere. Then $\op{ks}(M)=\sigma(M)/8+\mu(N)\in\Z/2$.
\end{theorem}

\begin{proof}
As $\Omega_3^{\Spin}=0$ \cite[Theorem 5.7.14]{GS99}, we may pick a smooth, compact, orientable, spin $4$-manifold $X$ with boundary $N$. Form the closed, spin $4$-manifold $Z=-X\cup_N M$. By Theorem~\ref{thm:ks-basics}~(\ref{item-ks-basics-5}), and using that $X$ is smooth, we have $\ks(Z)=\ks(X)+\ks(M)=\ks(M)\in\Z/2$. By Theorem~\ref{thm:ks-basics}~(\ref{item-ks-basics-6}), we have $\ks(Z)=\sigma(Z)/8\in\Z/2$. Using Novikov additivity and the facts so far, we have
\[
\ks(M)=\ks(Z)=\sigma(Z)/8=\sigma(X)/8+\sigma(M)/8= \mu(N)+\sigma(M)/8\in\Z/2. \qedhere
\]
\end{proof}

\begin{lemma}
There exists a unique closed $4$-manifold that is homotopy equivalent to $\cp^2$ and with nontrivial Kirby-Siebenmann invariant, and so which in particular is  not homeomorphic to $\cp^2$.
\end{lemma}

\begin{proof}
Let $K\subseteq \partial D^4$ be any knot with $\op{Arf}(K)=1\in\Z/2$ (for example, the trefoil). Attach a 2-handle $D^2 \times D^2$ to $D^4$ by identifying $S^1 \times D^2$ with a tubular neighbourhood of $K$, via the $+1$ framing of $K$, to obtain the $4$-manifold with boundary $X(K)$. The boundary $\partial X(K)$ is an integral homology sphere. Freedman proved that every integral homology sphere bounds a compact, contractible 4-manifold~\cite[Theorem~1.4$'$]{Freedman-82}, ~\cite[Corollary~9.3C]{FQ90}. Write $Y$ for this contractible $4$-manifold with $\partial Y\cong\partial X(K)$ and form the closed $4$-manifold $Z=-Y\cup X(K)$.

The manifold $Z$ is homotopy equivalent to $\cp^2$. Indeed, $H^2(Z;\Z)\cong\Z$ so we obtain a map $Z\to\cp^\infty=K(\Z,2)$ representing a generator. This map can now be homotoped to have domain $\cp^2\subseteq\cp^\infty$. This map will be a homology equivalence (this is clear on $H_1$, $H_2$, and $H_3$, and can be deduced on $H_4$ by considering that the cohomology ring of $Z$ agrees with that of $\cp^2$ by Poincar\'e duality). As $Z$ is simply connected, Whitehead's theorem now implies we have a homotopy equivalence.

We show that $\ks(Z)=1\in\Z/2$. The Rochlin invariant of $+1$ surgery on a knot in $S^3$ is equal to its Arf invariant (see e.g.~\cite[Example 2.5]{Saveliev2002}), and thus  $\mu(\partial X(K))=1\in\Z/2$. Applying Theorem~\ref{thm:FQ165} to the contractible manifold $Y$, we obtain $\ks(Y)=\mu(\partial X(K))=1\in\Z/2$. On the other hand, $\ks(X(K))=0$ because this manifold is smooth. By Theorem~\ref{thm:ks-basics}~(\ref{item-ks-basics-5}), $\ks(Z)=\ks(Y)+\ks(X(K))=1+0=1$.

By Freedman's classification (Theorem~\ref{thm:classn-simply-connected-4-mfld}), $Z$ is the unique closed manifold that is homotopy equivalent to $\cp^2$ with $\ks(Z)=1$.
\end{proof}

\begin{definition}\label{def:Chern}
The unique closed $4$-manifold that is homotopy equivalent to $\cp^2$ but not homeomorphic to $\cp^2$ (constructed above) is called the \emph{Chern manifold} and denoted~$\ast\cp^2$.
\end{definition}

\begin{remark}
The Chern manifold was first constructed in~\cite[p.~370]{Freedman-82}. It is not smoothable because $\ks(\ast\cp^2)=1$. For further discussion of star partners, see \cite[Section~10.4]{FQ90}, \cite{Stong-conn-sum}, and \cite{Teichner-star}.
\end{remark}

%

The following theorem says in particular that given any compact $4$-manifold~$M$ there exists a closed, orientable, simply-connected $4$-manifold $N$ such $M\# N$ is smoothable.

\begin{theorem}\label{thm:connect-sum-is-smooth}
Let $M$ be compact, connected $4$-manifold. There exists a closed, orientable, simply connected $4$-manifold $N$ such $M\# N$ admits a smooth structure.
If moreover the Kirby-Siebenmann invariant of $M$ is zero, then there exists a $k \in \mathbb{N}_0$ such that $M\#^k S^2\times S^2$ admits a smooth structure.
\end{theorem}

\begin{proof}
Let $M$ be compact $4$-manifold. Perform the connected sum with
an appropriate number of copies of $*\cp^2$, in order to obtain a manifold with every connected component having zero Kirby-Siebenmann invariant.
It follows from the discussion on \cite[p.~164]{FQ90} and
the Sum-Stable Smoothing Theorem~\cite[p.~125]{FQ90},
that performing the connected sum with enough copies of $S^2\times S^2$ produces a manifold that admits a smooth structure.
\end{proof}

\begin{remark}
  Given a lift of the classifying map of the (unstable) tangent microbundle of $M$ to $\BO(4)$, Lashof-Shaneson~\cite{Lashof-Shaneson-71} showed that there exists a $k \in \mathbb{N}_0$ such that $M\#^k S^2\times S^2$ admits a smooth structure.  The result quoted in the previous proof extended this to require only a lift of the corresponding \emph{stable} maps to deduce the same result.  The existence of a stable lift is significantly easier to verify.
\end{remark}

\chapter{Topological transversality}\label{chapter:transversality}

We turn to the subject of transversality in the topological category. Some discussion of this concept is in order. There are two important contexts for transversality: submanifold transversality and map transversality. In this book, map transversality will be deduced from submanifold transversality. Submanifold transversality when none of the manifolds involved has dimension 4 is due to Marin~\cite{Marin77}; cf.~\cite[Essay~III, Section~1, p.~83]{KS77}. Transversality in the remaining cases is due to Quinn~\cite{Qu82,Quinn88}; see also~\cite[Section 9.5]{FQ90}.

A naive definition of submanifold transversality in the topological category is that manifolds are \emph{locally transverse} if around any intersection point there is a chart in which the submanifolds appear as perpendicular planes.
On the other hand, there are examples (in the relative setting, in high dimensions) of submanifolds which cannot be made locally transverse via ambient isotopy; see Remark \ref{rem:Hudson}. Thus one cannot generally use this definition.

In light of this, in order to make general statements, one passes to some notion of \emph{global transversality}. Global transversality means that transversality statements are made with respect to a given choice of normal structure on one of the submanifolds involved. Of course, this forces one to engage with the question of existence and uniqueness of whatever normal structure is used, and the `correct' choice of normal structure is still not fully settled in the topological category. We refer the reader to~\cite[Sections 9.4, 9.6C]{FQ90} for a brief discussion of the competitors.

The most general statement of transversality \cite[Theorem]{Quinn88} uses  microbundles to describe normal structure, and this is the technology we will use. As discussed in Chapter~\ref{chapter:bundlestructures}, for general manifolds, tangent microbundles always exist but normal microbundles do not (see Example \ref{ex:notnormal}).

The case of dimension 4 is special, since here the normal vector bundles of Section~\ref{subsec:normalvector}, which are a stronger notion than normal microbundles, always exist. In fact, the results obtained for these normal vector bundles in dimension~$4$ are strong enough to ensure that submanifold transversality holds in ambient dimension 4 with the naive, local transversality definition discussed above. The reader may therefore wonder why we even introduce normal microbundles into a discussion primarily focused on 4-manifold transversality. The answer is that the `submanifold transversality implies map transversality' argument of Section \ref{subsec:maps} requires a bundle technology that works in all dimensions, and microbundles appear to be the most convenient.

\section{Transversality for submanifolds}\label{section:transversality-submanifolds}

\begin{definition}
Consider proper submanifolds~$X,Y$ of an ambient manifold and a normal microbundle~$\nu X$ for $X$,
with projection~$r_X \colon E(\nu X) \to X$. The proper submanifold~$Y$ is \emph{transverse to~$\nu X$}
if there exists a neighbourhood~$U \subseteq E(\nu X)$ of~$X$ such that $Y\cap U = r^{-1}_X(X \cap Y) \cap U$.
\end{definition}

\begin{figure}
\includegraphics{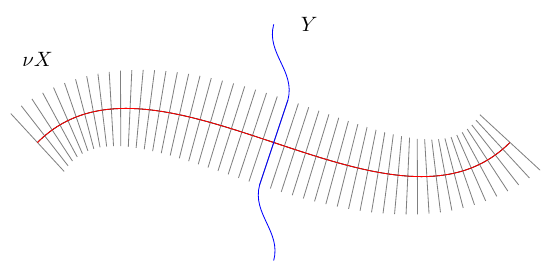}
\caption{Sketch of a transverse intersection of $Y$ to $\nu X$.}
\end{figure}

\begin{lemma}
Let $X,Y$ be submanifolds of  a manifold $M$.
Let $Y$ be transverse to a normal microbundle~$\nu X$ of $X$.
Then $X\cap Y$ is a submanifold of $Y$ with
normal microbundle~$(\nu X)|_{X\cap Y}$.
\end{lemma}

\begin{proof}
Let $X\xrightarrow{i}\nu X\xrightarrow{p} X$ be a normal microbundle of $X$.
Once we have established that $(\nu X)|_{X\cap Y}$ is a normal microbundle of~$X\cap Y$ in $Y$, the subspace~$X\cap Y$ will automatically be a submanifold since the trivialisation of the microbundle~$(\nu X)|_{X\cap Y}$ gives the required charts for~$X\cap Y$.

At least after shrinking the total space of $E(\nu X)$, each fibre~$p^{-1}(x)$ for $x \in X\cap Y$ will be contained in $Y$  by the definition of transversality. That is~$E((\nu X)|_{X\cap Y})$ is a subset of $Y$ and neighbourhood of $X\cap Y$. This shows that $(\nu X)|_{X\cap Y}$ is a normal microbundle of $X\cap Y \subseteq Y$.
\end{proof}

Transversality in high dimensions is due to Marin~\cite{Marin77}, cf.~\cite[Essay~III, Section~1, p.~83]{KS77}. The formulation below is from Quinn~\cite{Quinn88}.
Recall Definition~\ref{defn:ProperSubmanifold} of a proper submanifold.
Note that in the next theorem there is no restriction on dimensions. The manifold $M$ is allowed to be noncompact and have nonempty boundary.

\begin{theorem}\textbf{\textup{(Transversality Theorem for submanifolds)}}\label{thm:TransvSubmanifolds}
Let $X$ and $Y$ be proper submanifolds of a compact manifold~$M$.
Let $\nu Y$ be a normal microbundle for $Y$.
Let $C \subseteq M$ be a closed subset such that $X$
is transverse to $\nu Y$ in a neighbourhood of $C$.
Let $U$ be a neighbourhood of the set~$\big( M \sm C \big) \cap X \cap Y$.
Then there exists an isotopy of $X$ supported in~$U$ to a proper submanifold~$X'$
such that~$X'$ is transverse to~$\nu Y$.
\end{theorem}

\begin{proof}
	See Quinn~\cite{Quinn88} for all cases but $\dim M = 4$, $\dim X = 2$ and $\dim Y =2$.
	For the remaining case, first establish local transversality using~\cite[Section 9.5]{FQ90}.
	Note that $X\cap Y$ is a discrete collection of points. Therefore, the coordinate chart, witnessing local transversality, defines a normal neighbourhood of $Y$ near $X\cap Y$.
	This normal vector bundle can be extended to a normal vector bundle~$\nu Y'$ on all of $Y$ by~\cite[Theorem 9.3A]{FQ90}. The submanifold $X$ is now transverse to~$\nu Y'$,
	but (possibly) not to $\nu Y$. By Theorem \ref{thm:normalmicroexist}, our
	microbundle~$\nu Y$ comes from a normal vector bundle. By uniqueness of normal vector bundles
	(Theorem~\ref{thm:uniquenormal}), there is an isotopy from $\nu Y'$ to $\nu Y$. Apply this isotopy to~$X$. Now~$X$ is transverse to $\nu Y$.
\end{proof}

\begin{remark}\label{rem:Hudson}
The analogous statement to Theorem \ref{thm:TransvSubmanifolds} is false for local transversality. Examples of this failure even exist in the $\PL$ category: Hudson~\cite{Hud69} constructs, for certain large $n$, closed $\PL$ submanifolds $X,Y\subseteq \R^n$, that are topologically unknotted Euclidean spaces of codimension~$\geq 3$, in such a way that $X$ and $Y$ are $\PL$ locally transverse near a closed neighbourhood $K$ of infinity but also so that it is impossible to move $X$ and $Y$ by isotopy relative to $K$ to make them locally transverse everywhere.
\end{remark}

Although transversality for submanifolds (Theorem~\ref{thm:TransvSubmanifolds}) is only stated for a pair of submanifolds,
it can be used to make collections of submanifolds transverse.
\begin{lemma}
	Let $M$ be an $2m$-dimensional manifold for $m \geq 1$, and let $X_1, \ldots, X_n$ be $m$-dimensional compact submanifolds with normal microbundles~$\nu X_i$. Then the submanifolds~$X_i$ can be isotoped such that there are no triple intersection points and the submanifolds intersect $($pairwise$)$ transversely.
\end{lemma}

\begin{proof}
We give a proof by induction. When $n=1$, there is nothing to show, since every submanifold is embedded.
For the inductive step, denote $X_n$ by $Y$. The inductive hypothesis states that we can isotope any $n-1$ submanifolds $X_1,\dots,X_{n-1}$ so there are no triple points and that they intersect pairwise transversely.
We will prove that the submanifolds~$X_1,\dots,X_{n-1}$ can be further isotoped so that they are transverse to~$\nu Y$ and $Y$ is free of triple points. Note that having no triple points on $Y$ implies that there exists an open set~$U_Y$ such that: $X_i \cap X_j \subseteq U_Y$ for~$1 \leq i < j \leq n-1$, and $M \setminus U_Y$ is a neighbourhood of $Y$. To obtain the lemma apply the inductive hypothesis, picking all further isotopies to be supported in $U_Y$.

We proceed by showing the inductive step: we can isotope every~$X_i$ to be transverse to $\nu Y$ such that no triple points lie on $Y$.
For each $i = 1, \ldots, n-1$, apply Theorem~\ref{thm:TransvSubmanifolds} to arrange that $Y$ and $X_i$ intersect transversely. By compactness of the submanifolds, the subset
\[ T_Y = Y \cap (X_1 \cup \cdots \cup X_{n-1})\]
is compact. Pick disjoint open neighbourhoods~$V_y \subseteq Y$ around each point~$y \in T_Y$.
Pick a chart~$\phi$ of $Y$ around $\phi(0) = y$ contained in $V_y$, and a microbundle chart around $y \in Y$. In the local model, $Y$ corresponds to $\R^m \times \{ 0 \}$ and the~$X_i$ that intersect~$Y$ in $y$ will be mapped to $0 \times \R^m$.
\begin{figure}
\includegraphics{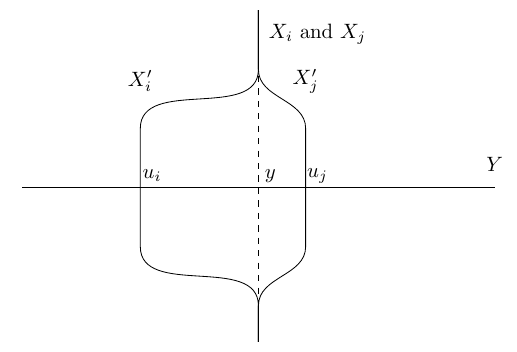}
\caption{Displacing a triple point~$y$ in a microbundle chart.}
\end{figure}
For those $X_i$, pick disjoint points~$u_i \in \R^m$ (here we use $m > 0$), and pick a continuous function~$\eta \colon \R_{\geq 0} \to [0,1]$ with $\eta(t) = 1$ for $0 \leq t \leq 1$ and $\eta(t) = 0$ for~$t \geq 2$. Replace $X_i$ in the chart with the image of
\begin{align*}
	\R^m &\to \R^m \times \R^m\\
	v &\mapsto \big( \eta\big( \| v \| \big) u_i , v \big).
\end{align*}
Call this new submanifold~$X_i'$. It agrees with $X_i$ outside the ball of radius~$2$, and is isotopic to $X_i$. In $V_y$, the submanifold~$X_i'$ intersects~$Y$ only in $\phi(u_i)$ and there it intersects~$Y$ transversely with respect to~$\nu Y$. The collection~$\{ X_i' \}$ has no triple intersection points in the set~$V_y$ anymore.
\end{proof}

Here is another result on submanifold transversality.
It might often happen that one can find a continuous map of, for example, a disc $D^2$ into a 4-manifold $M$, perhaps if fundamental group computations yield a null homotopy of a circle. Then this disc can be isotoped to a generic immersion.  If $M$ were smooth, this would be a consequence of standard differential topology, an observation that we leverage.

\begin{theorem}\label{thm:generic-immersions-are-generic}
  Let $\Sigma$ be a connected, compact 1 or 2 dimensional manifold and let $f \colon (\Sigma,\partial \Sigma) \to (M,\partial M)$ be a continuous map of $\Sigma$ into a connected 4-manifold ~$M$, such that $f$ is a smooth embedding near $\partial \Sigma$.
  Then there is an homotopy of $f$ rel.\ a (possibly smaller) neighbourhood of $\partial \Sigma$  to a generic immersion $f' \colon (\Sigma,\partial \Sigma) \to (M,\partial M)$.
\end{theorem}

\begin{proof}
We start out with the following claim.

\begin{claim}
The map $f\colon \Sigma\to M$ is homotopic to a map that misses a point $P\in M$.
\end{claim}

We pick $P\in M\sms \partial M$. Since $\dim(\Sigma)\leq 2$ we can equip  $\Sigma$ with a smooth structure. Using a chart we can equip an open neighbourhood $V$ of $P$ with a smooth structure.
We pick another open neighbourhood $U$ of $P$ with $\ol{U}\subseteq V$. Since $f^{-1}(V\sms \ol{U})$ is an open subset of $\Sigma$ we can find a compact submanifold $F\subseteq \Sigma$
with $f^{-1}(\ol{U})\subseteq \op{Int}(F)\subseteq F\subseteq f^{-1}(V)$. The map $f\colon F\to V$ is now a map between smooth manifolds. Thus, using Whitney approximation, we can find a homotopy rel.\ $\partial F$ from $f$
to a smooth map $g\colon F\to V$. Since $\dim(F)<\dim(U)$ we see that this map misses a point in $U$. Since $f|_{\partial F}=g|_{\partial F}$, and $\partial F$ is fixed throughout the homotopy, we can extend, by a constant homotopy, the homotopy from $f\colon F\to V$ to $g\colon F\to V$ to a homotopy from $f\colon \Sigma\to M$ to a map $g\colon \Sigma\to M$ such that
$f$ and $g$ agree outside of $F$.  Thus $g$ also misses the point $P$.
This concludes the proof of the claim.

Using Theorem~\ref{thm:smooth-outside-a-point} we smooth $M$ in the complement of that point $P$. Now by \cite[Theorem~2.2.6 and Theorem~2.2.12]{Hi76} we can isotope $f$ rel.\ $\partial \Sigma$ to a smooth immersion, which we can then isotope rel.\ $\partial \Sigma$ to a smooth generic immersion $f'$, i.e.\ a map that is  self-transverse by~\cite[Theorem~4.2.1]{Hi76}, \cite[Theorem~4.6.6]{Wall16}, with no triple points by general position~\cite[Theorem~4.7.7]{Wall16} or \cite[Chapter~III, Corollary~3.3]{GoGu}.
We have now shown that $f\colon \Sigma\to M$ is homotopic to a map $f'\colon \Sigma\to M\sms \{P\}$ which is a smooth generic immersion. But this implies that $f'\colon \Sigma\to M\sms \{P\}$ is in particular a generic  immersion. But then $f'\colon \Sigma\to M$ is also an generic immersion.
\end{proof}

No purely topological proof of this is known.
When manipulating generic immersions of surfaces, it is helpful to have control on the homotopies between them.

\begin{definition}
A \emph{generic homotopy} between generic immersions $F_0,F_1 \colon \Sigma \to M$ as in Definition~\ref{def:gen_immersion} is sequence of ambient isotopies, finger moves, Whitney moves, and cusp homotopies.
\end{definition}

We finish by quoting the following theorem, which was stated in \cite{FQ90}, and proven in \cite{PRT19}.

\begin{theorem}
    Every homotopy $H \colon \Sigma \times [0,1] \to M$ between generic immersions $F_0,F_1 \colon \Sigma \to M$ is homotopic rel.\ $\Sigma \times \{0,1\}$ to a generic homotopy.
\end{theorem}

\section{Transversality for maps}\label{subsec:maps}

\begin{definition}
Let $f \colon M \to N$ be a continuous map between two manifolds and let $X$ be a submanifold of
$N$ with normal microbundle~$\nu X$.
The map~$f$ is said to be \emph{transverse} to~$\nu X$ if $f^{-1}(X)$ is
a submanifold admitting a normal microbundle~$\nu f^{-1}(X)$
and
\begin{align*}
f \colon \nu f^{-1}(X) &\to f^*\nu X\\
m &\mapsto (r(m), f(m))
\end{align*}
is an isomorphism of microbundles.
\end{definition}

In the next theorem, we show how to reduce transversality for
maps to transversality for submanifolds.  Again, there are restrictions neither on dimensions nor codimensions.

\begin{theorem}\label{thm:TransveralityMaps}
Let $M$ and $N$ be a manifolds, let $Y \subseteq N$ be a proper submanifold with normal microbundle~$\nu Y$, let $f \colon M \to N$ be a map such that $f^{-1}(Y)$ is a submanifold of $M$,
and let $U$ be a neighbourhood of the
set
\[\ograph f \cap (M \times Y) \subseteq M \times N.\] Then
there exists a homotopy~$F \colon M\times I \to N$ such that
\begin{enumerate}[leftmargin=1cm,font=\normalfont]
\item\label{item:LeftSide} $F(m, 0) = f(m)$ for all $m \in M$;
\item\label{item:Transversality} $F_1
\colon m \mapsto F(m,1)$ is transverse to $\nu Y$; and
\item\label{item:Control} for $m \in M$ either
\begin{enumerate}[leftmargin=0.8cm,font=\normalfont]
\item $(m, f(m)) \notin U$, in which case $F(m,t) = f(m)$ for all $t \in I$, or
\item $(m, f(m)) \in U$, in which case $(m, F(m,t)) \in U$ for all $t \in I$.
\end{enumerate}
\end{enumerate}
\end{theorem}
\begin{figure}
\includegraphics[width=16cm]{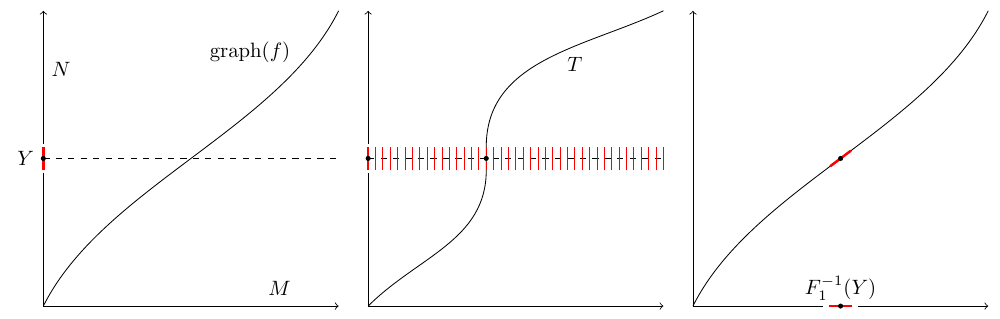}
\caption{Transversality for maps from transversality for submanifolds.}
\label{fig:TransMfd}
\end{figure}

\begin{proof}
Note that $M \times Y \subseteq M \times N$ is a proper submanifold with
normal microbundle~$M \times \nu Y = \op{pr}_Y^* \nu Y$.
Also $\ograph f$ is a proper submanifold of $M\times N$.
By Theorem~\ref{thm:TransvSubmanifolds}, there exists an
isotopy
\[G \colon \ograph f \times I \to M\times N,\]
supported in $U$, of the submanifold~$\ograph f$
to a submanifold~$T \subseteq M\times N$ such that $T$ is transverse to
$M \times \nu Y$ over $M \times Y$.

Define the map~$F$ as the composition
\[ F \colon M \times I \to \ograph f \times I \xrightarrow{G} M \times N \xrightarrow{\op{pr}_N} N.\]
Since the isotopy $G$ is supported in $U$, statement~(\ref{item:Control}) holds.
By construction, $F(x, 0) = \op{pr}_N (x, f(x)) = f(x)$, which proves statement~(\ref{item:LeftSide}).

Now we prove statement~(\ref{item:Transversality}).
Let $F_1 \colon M \to N$ be the map that sends $x \mapsto F(x,1)$.
We keep track of the preimages through the maps of the composition that defines $F_1$; see Figure~\ref{fig:TransMfd}.
By transversality of $T$ to $M\times \nu Y$, we see that $Z = T \cap( M\times Y) = \op{pr}_N^{-1}(Y)$ is a submanifold of $T$ with normal microbundle~$M\times \nu Y |_Z$, and that the projection to $N$ induces a microbundle isomorphism~$M \times \nu Y|_Z \xrightarrow{\sim} \op{pr}_N^* \nu Y$. By definition, $\op{pr}_N \colon T \to N$ is transverse to $\nu Y$.

We transport the submanifold~$Z$ back to $M$. Consider the commutative diagram
\[ \begin{tikzcd}
& & N\\
M \ar[r,"\cong"] \arrow[bend left=20]{urr}{F_1} \arrow[bend right=30,swap]{rr}{q} & \ograph f \times \{1\} \ar[r,"\cong"] & T \ar[u,"\op{pr}_N"]
\end{tikzcd},\]
where~$q$ is the composition, which is a homeomorphism. Now $F^{-1}_1(Y) = q^{-1}(Z)$ is a submanifold with normal microbundle
\[ q^* \big( M\times \nu Y|_Z \big) = q^* \op{pr}_N^*\nu Y = F^*_1 \nu Y.\]
That is $F_1 \colon M \to N$ is transverse to $\nu Y$.
\end{proof}

\section{Representing homology classes by submanifolds}

Our first goal in this section is to prove the following theorem.

\begin{theorem}\label{thm:represent-homology-by-submanifolds-all-dims}
Let $X$ be a compact orientable $n$-manifold. Let $k=n-2$ or $k=n-1$ and let $\sigma\in H_k(X,\partial X;\Z)$. Then the class $\sigma$ can be represented by a $k$-dimensional submanifold $Y$ with $\partial Y\subseteq \partial X$.
\end{theorem}

We will prove the theorem using purely topological methods, in particular the topological transversality arguments from Chapter~\ref{chapter:transversality}, in particular
Sections~\ref{section:transversality-submanifolds} and~\ref{subsec:maps}.
Then we will prove a refined version in dimension four that uses the trick of smoothing away from a point.  We will need the notion of a Thom class.


\begin{definition}
Let $\xi = S \xrightarrow{i} E \xrightarrow{p} S$ be a $k$-dimensional microbundle over $S$. For each $x\in S$ we write $E_x:=p^{-1}(\{x\})$.
A \emph{Thom class} of $\xi$ is a class~$\tau(\xi) \in H^k(E, E \setminus i(S); \Z)$ that restricts
to a generator~$H^k(E_x, E_x \setminus i(x); \Z) \cong H^k(\R^k ; \R^k \setminus \{0 \}; \Z) \cong \Z$
for all $x \in S$. The microbundle $\xi$ together with a Thom class is called an \emph{oriented microbundle}.

A Thom class of a topological $\R^n$ bundle over $S$ is by definition a Thom class of the underlying microbundle.
\end{definition}


\begin{remark}
As in the smooth case, consider the orientation
bundle~$\pi \colon \op{Or}(\xi) \to S$
with fibre over $x \in S$ the discrete
set
$\op{Or}(\xi)_x = \big\{ \text{primitive classes of }H_k(E_x, E_x \setminus i(x); \Z) \big\}$.
This is a $\Z/2$-principal bundle, and a Thom class~$\tau(\xi)$ determines
a global section~$s \in \Gamma(\op{Or}(\xi))$ by enforcing~$\langle \tau(\xi), s(x) \rangle = 1$ for every~$x \in S$.
By the same equation, a global section~$\Gamma(\op{Or}(\xi))$ determines
a Thom class.
\end{remark}

\begin{remark}\label{rem:OrientedNormal}
Let $X$ be an oriented manifold. Let $S$ be a submanifold with
normal microbundle~$\nu S$. A standard argument, similar to the construction of the fundamental class of an oriented manifold, shows that an orientation of $S$ determines
a unique Thom class  of $\nu S$
compatible with the ambient orientation
and vice versa.
\end{remark}

To prove Theorem~\ref{thm:represent-homology-by-submanifolds-all-dims}, we will consider a map~$f \colon X \to Y$ that is transverse to
an oriented submanifold~$S$ of $Y$. By Remark~\ref{rem:OrientedNormal}, $\nu S$ is an oriented microbundle carrying the Thom class~$\tau$.
Note as $f \colon \nu f^{-1}(S) \to f^* \nu S$ is an isomorphism, also $f^* \tau$ is a Thom
class of $\nu f^{-1}(S)$ and we orient $f^{-1}(S)$ accordingly. Before we proceed with
the proof, we recall the following compatibility between Thom classes and
Poincar\'e duality~\cite[Definition~VI.11.1, Corollary~VI.11.6]{Br93},
interpreted for microbundles.

\begin{lemma}\label{lem:ThomPD}
Let $X$ be a compact oriented $n$-manifold with fundamental class~$[X]\in H_n(X,\partial X)$, and let $i \colon S \to X$ be an oriented proper $k$-dimensional submanifold of $X$ with normal microbundle~$\nu S$.
The composition
\[ H^{n-k} \big(\nu S, \nu S \setminus i(S) \big) \us{\cong}{\xleftarrow{\op{Exc.}}} H^{n-k}(X, X \sm S)
\to H^{n-k}(X)  \xrightarrow{\PD_X} H_k(X,\partial X)\]
maps the Thom class~$\tau$ of $\nu S$, that is determined by the orientations of $X$ and $S$, to the fundamental class~$i_*[S]$.
\end{lemma}

\begin{proof}
We start out with a general piece of notation.
Given an oriented $n$-dimensional topological manifold $W$  and given a compact subset $K\subseteq W\subseteq W$ we write $[W]\in H_n(W,(W\sms K)\cup \partial E)$ for the unique element which, for each $x\in W\sms (K\sms \partial W)$ is sent to the generator of
$H_n(W,W\sms \{x\})$ given by  the orientation  of $W$.
Recall that  if $W$ is compact, then  Poincar\'e duality map $\PD_W$ is~$\acap [W]$.
(We refer to \cite[Section~VII.12]{Dold95} for the definition and precise nature of the cap product on relative (co-) homology.)

We set $d:=n-k$ and make a few preliminary observations.
\bnm
\item It suffices to prove the lemma for connected $S$.
\item We identify $S$ with $i(S)$.
\item
By Kister's Theorem~\ref{thm:kistmaz} there exists an open subset $E\subseteq \nu S$ containing $S$ such that $p\colon E\to B$ is a projection map of a topological $\R^d$-bundle whose $0$-section is $i$.
\item
Pick an $x \in S\sm\partial S$ and a closed $k$-disc~$U \subseteq S\sms \partial S$ containing~$x$
such that there is a local trivialisation~$\Phi\colon p^{-1}(U) \to  U \times \R^{d}$
 for the $\R^d$-bundle  in a neighbourhood of $x$. We can and will choose $\Phi$ such that
 it preserves the orientation of the fibres.
 \item We let $j\colon U\to U\times \{0\}$ and $k\colon U\times \R^d\to  \R^d$ denote the obvious maps.
\enm
Next we consider the following diagram.
\[ \xymatrix@C0.9cm@R0.65cm{
 H^{d}(X, X \sm S)\ar[rr]^{\acap [X]} \ar[d]^{\op{Exc.}}_\cong&& H_k(X,\partial X)\\
 H^{d} \big(\nu S, \nu S \setminus S \big) \ar[rr]^{\acap \, [\nu S]} \ar[d]^{\op{Exc.}}_\cong&&
  H_k(\nu S,\partial \nu S)\ar[u]\\
H^{d}(E,E\sms S)\ar[d]^{\operatorname{Id}}_{\cong} \ar[rr]^{\acap\, [E]}&&H_k(E,\partial E)\ar[u]\ar[d]&\ar[l]_{i_*}^\cong H_k(S,\partial S)\ar[d]^\cong\\
H^{d}(E,E\sms S)\ar[rr]^{\acap \, [E]}\ar[d]^{\op{Exc.}}_\cong&&H_k(E,E\sms p^{-1}(x))&\ar[l]_-{i_*}^-\cong
H_k(S,S\sms x)\\
H^{d}(p^{-1}(U),p^{-1}(U)\sms S)\ar[rr]^{\acap \, [p^{-1}(U)]}&&H_k(p^{-1}(U),p^{-1}(U\sms x))\ar[d]^{\Phi_*}_\cong
\ar[u]^{\op{Exc.}}_\cong
&\ar[l]_-{i_*}^-\cong
H_k(U,U\sms x)\ar[u]^{\op{Exc.}}_\cong\ar[dl]^{j_*}\\
\ar[u]^{\Phi^*}_\cong H^{d}(U\!\times\! \R^{d},U\!\times\! (\R^{d}\! \sms \!\{0\}))
\ar[rr]_{\cong}^-{\acap \, [U\times \R^{d}]}&& H_k(U\!\times\! \R^{d},(U\!\sms \! \{x\})\!\times\! \R^{d})&\\ \ar[u]^{k^*}_\cong
H^d(\R^d,\R^d\!\sms \!\{0\})
}
\]
The bottom vertical maps are given by the local trivialisation $\Phi$. All other vertical maps are the obvious maps of pairs of topological spaces.
The maps decorated with  ``Exc.''  are isomorphisms by the excision theorem.
The horizontal maps $i_*$ to the right are isomorphisms by the Serre spectral sequence (note that  a priori we do not know whether $i\colon S\to E$ is a homotopy equivalence.)
Note that bottom vertical map is an isomorphism since $\{x\}$ is a deformation retract of the disc $U$.

We make the following observations.
\bnm
\item By definition of the Thom class, it is sent, under the left vertical maps, to the standard generator $[\R^d]^*\in H^d(\R^d,\R^d\sms \{0\})$.
\item It follows from the compatibility of the cross product with the cap product~\cite[Section~VII.12.17]{Dold95} (see also \cite[Proposition~138.2]{Fr23}) that
\[k^*([\R^d]^*)\acap [U\times \R^d] = j_*([S]). \]
\enm
Hence the Thom class in $ H^{d}(\nu S, \nu S \setminus S )$ and the fundamental class $[S,\partial S] \in H_k(S,\partial S)$ have the same image in $H_k(U\!\times\! \R^{d},(U\!\sms \! \{x\})\!\times\! \R^{d})$.

Standard facts about the relative cap product show that the diagram commutes.  In particular note that the map $H_k(E,\partial E) \to H_k(E,E \sms p^{-1}(x))$ is an isomorphism.
It is then not too hard to chase the diagram to deduce that the Thom class and $[S,\partial S]$ also have the same image in $H_k(X,\partial X)$, and so the lemma holds.
\end{proof}


\begin{proof}[Proof of Theorem~\ref{thm:represent-homology-by-submanifolds-all-dims}]\label{pf:represent-homology-by-submanifolds-2}
First we let $k=n-1$.
Let $\alpha \in H^{1}(X;\Z)$ be the Poincar\'{e} dual to $\sigma \in H_{n-1}(X,\partial X)$. Recall that in
Theorem~\ref{thm:topological-manifold-CW complex} we showed that $X$ is homotopy equivalent to a CW complex.
Therefore we have the following correspondence between homotopy classes of maps to Eilenberg-Maclane spaces and cohomology classes of $X$:
\begin{align*}
[X, S^1] = [X, K(\Z,1)] & \xrightarrow{\cong} H^1(X; \Z)\\
f &\mapsto  f^* \theta,
\end{align*}
where $\theta$ is the Hom dual of the fundamental class of $S^1$. Note that we used here that $X$ is homotopy equivalent to a CW complex. Pick an arbitrary point~$\pt \in S^1$
and denote a tubular neighbourhood by $\nu (\pt)$.
Note that the Thom class~$\tau_{\pt}$ for $\nu (\pt)$ is mapped under $H^1(S^1, S^1 \setminus \pt) \to H^1(S^1)$ to $\theta = \PD_{S^1}^{-1} [\pt]$.
Let $f \colon X \to S^1$ be a map
corresponding to $\alpha$, so $f^* \theta = \alpha$.
Make $f$ transverse to a tubular neighbourhood
of $\pt \in S^1$ using Theorem~\ref{thm:TransveralityMaps}.
Consequently, $S := f^{-1}(\pt)$ is an $(n-1)$-dimensional submanifold of $X$.
By definition, $f$ induces a bundle isomorphism~$f \colon \nu S \to f^*\nu (\pt)$.
We have, as elements in $H^{1}(X; \Z)$, that
\[ \alpha = f^* \theta = f^* \PD_{S^1}^{-1} [\pt] = f^* \tau_{\pt}. \]
Note that  $f^* \tau_{\pt}$ is the image of the Thom class of $\nu S$ under the first two maps in the composition displayed in the statement of Lemma~\ref{lem:ThomPD}.  Lemma~\ref{lem:ThomPD} thus tells us the last equality of:
\[\alpha = f^*\tau_{\pt} = \tau_S = \PD_X^{-1} [S] \in H^1(X;\Z). \]
We have thus shown that $[S]=\op{PD}_X(\alpha)=\sigma$.

A similar proof works for the codimension two case, i.e.\ $k=n-2$.  Let $\alpha \in H^2(X;\Z)$ be Poincar\'{e} dual to $\sigma \in H_{n-2}(X,\partial X)$. Recall that
\begin{align*}
[X, \CP^\infty] = [X, K(\Z,2)] & \xrightarrow{\cong} H^2(X; \Z)\\
f &\mapsto  f^* \theta,
\end{align*}
where $\theta \in H^2(\CP^\infty;\Z)$ is the Hom dual of the fundamental class of $\CP^1 \subseteq \CP^\infty$. Since $X$ is homotopy equivalent to an $n$-dimensional CW complex we can homotope a given representing map $f$ to have image in $\CP^{m} \subseteq \CP^\infty$, where $m = \lfloor n/2 \rfloor$.  We abuse notation and denote the image of $\theta$ under the restriction map $H^2(\CP^\infty;\Z) \to H^2(\CP^m;\Z)$ also by~$\theta$.

Let $f \colon X \to \CP^m$ be a map
corresponding to $\alpha$, so $f^* \theta = \alpha$.
Make $f$ transverse to a normal microbundle
of $\CP^{m-1} \subseteq \CP^m$ using Theorem~\ref{thm:TransveralityMaps}.
The inverse image $S := f^{-1}(\CP^{m-1})$ is an $(n-2)$-dimensional submanifold of $X$.

Let $\tau_{\CP^{m-1}} \in H^2(\CP^m,\CP^m \sms \CP^{m-1})$ denote a Thom class for the normal bundle $\nu \CP^{m-1}$.
By Lemma~\ref{lem:ThomPD}, the map $H^2(\CP^{m},\CP^m \sms \CP^{m-1};\Z) \to H^2(\CP^m;\Z)$ sends~$\tau_{\CP^{m-1}}$ to~$\theta = \PD_{\CP^{m}}^{-1} [\CP^{m-1}]$.

By definition of $S$ and $\nu S$, the map $f \colon X \to \CP^m$ induces a bundle isomorphism~$f \colon \nu S \to f^*\nu \CP^{m-1}$.
We have, as elements in $H^{2}(X; \Z)$, that
\[ \alpha = f^* \theta = f^* \PD_{\CP^{m}}^{-1} [\CP^{m-1}] = f^* \tau_{\CP^{m-1}}. \]
Note that  $f^* \tau_{\CP^{m-1}}$ is the image of the Thom class of $\nu S$ under the first two maps in the composition displayed in the statement of Lemma~\ref{lem:ThomPD}.
Thus Lemma~\ref{lem:ThomPD} gives  the last equality of:
\[\alpha = f^*\tau_{\CP^{m-1}} = \tau_S = \PD_X^{-1} [S] \in H^2(X;\Z). \]
We have now shown that $[S]=\op{PD}_X(\alpha)=\sigma$.
\end{proof}

\begin{remark}
We have seen that $f^! [\pt] = \PD_X \circ f^* \circ \PD^{-1}_{S^1} [\pt] = [f^{-1} (\pt)]$, when $f$ is transverse to $\pt$, and similarly that $f^! [\CP^{m-1}] = f^{-1}(\CP^{m-1})$.
\end{remark}


Next we offer the following promised refinement of Theorem~\ref{thm:represent-homology-by-submanifolds-all-dims} in the 4-dimensional case, together with an alternative proof that uses smoothing away from a point.

\begin{theorem}\label{thm:represent-homology-by-submanifolds}
Let $X$ be a compact orientable $4$-manifold and let $A$ be a union of components of $\partial X$. Let $k=2$ or $k=3$ and let $\sigma\in H_k(X,A;\Z)$.
\bnm
\item The class $\sigma$ can be represented by a $k$-dimensional submanifold $Y$ with $\partial Y\subseteq A$.
\item In the case $k = 3$, the boundary of $Y$ can be specified: if $B\subseteq A$ is an oriented closed $2$-dimensional smooth submanifold contained in $A$ such that $\partial(\sigma)=[B]\in H_{2}(A;\Z)$, then $\sigma$ can be represented by an oriented compact $3$-dimensional submanifold $Y$ with $\partial Y=B$.
\enm
\end{theorem}

The submanifold $B$ can be assumed to be smooth, since $\partial X$ is a $3$-manifold and so has a unique smooth structure by \cite{Mo52}, \cite[p.~252--253]{Mo77}.

Note that Theorem~\ref{thm:represent-homology-by-submanifolds} also holds for $k=0$ and $k=1$. This is trivial for $k=0$. To see this for $k=1$, remove a point from each connected component to get a smooth 4-manifold by Theorem~\ref{thm:smooth-outside-a-point}. Note that  $H_1(X \sm \{\pt\},A;\Z) \cong H_1(X,A;\Z)$. Then by smooth approximation and general position, every 1-dimensional homology class can be represented by a $1$-dimensional submanifold of $X$.

\begin{example}\mbox{}
\bnm
\item If we apply the theorem to $A=\partial X$, we see that any homology class in $H_2(X,\partial X;\Z)$  and $H_3(X,\partial X;\Z)$ can be represented by a properly embedded submanifold. For $A=\emptyset$ we obtain the analogous statement for absolute homology groups.
\item Let $F$ be a properly embedded $2$-dimensional submanifold of $D^4$ and let~$S$ be a surface in $\partial D^4=S^3$ with $\partial S=\partial F$. Consider the $4$-manifold $X := D^4\sms \nu F$.
In the boundary of $X$ we have the surface $B=S\cup F\times \{1\}$. It follows from the long exact sequence of the pair $(X,\partial X)$ and Poincar\'e duality that the map $H_3(X,\partial X;\Z)\to H_2(\partial X;\Z)$ is an epimorphism.
It follows from Theorem~\ref{thm:represent-homology-by-submanifolds}, applied to $A=\partial X$, that there exists a 3-dimensional submanifold $Y$ of $D^4\sms \nu F$ with $\partial Y=B$.
This statement is folklore, and a proof using topological transversality for maps was written down by
Lewark-McCoy~\cite{LM15}.
\enm
\end{example}


For $n=4$ the statement of the following theorem is precisely the statement of Theorem~\ref{thm:represent-homology-by-submanifolds} in the smooth category.

\begin{proposition}\label{prop:represent-homology-by-smooth-submanifolds}
Let $X$ be a compact, orientable, smooth $n$-manifold and let $A$ be a union of components of $\partial X$. Let $\ell=1$  or $\ell=2$ and let $\sigma\in H_{n-\ell}(X,A;\Z)$. Then the following holds.
\bnm
\item The class $\sigma$ is represented by an $(n-\ell)$-dimensional smooth orientable submanifold~$Y$ with $\partial Y\subseteq A$.
\item Suppose $\ell=1$ and suppose we are given a closed, oriented $(n-2)$-dimensional smooth submanifold  $B$ of $A$ such that $\partial(\sigma)=[B]\in H_{n-2}(A)$.  Then $\sigma$ is represented by an oriented compact $(n-1)$-dimensional smooth submanifold $Y$ with $\partial Y=B$.
\enm
\end{proposition}

\begin{example}
Let $K\subseteq S^3$ be a knot. We write $X=S^3\sms \nu K$. Let $\lambda\subseteq \partial X$ be a longitude of $K$, i.e.\ $\lambda$ is a curve that represents a generator
of $\op{ker}(H_1(\partial X;\Z)\to H_1(X;\Z)$.   There exists a homology class $\sigma\in H_2(X,\partial X;\Z)$ with $\partial(\sigma)=[\lambda]\in H_1(\partial X;\Z)$. It follows from Proposition~\ref{prop:represent-homology-by-smooth-submanifolds} that there exists an orientable surface $F$ in $X$ with $\partial F=K$.
\end{example}

\begin{proof}
Let $X$ be a compact orientable smooth $n$-manifold and let $A$ be a union of components of $\partial X$.
First we prove statements (1) and (2) for the  case  $\ell=1$. Let $\sigma\in H_{n-1}(X,A;\Z)$.
\bnm
\item
Write $\wti{A}=\partial X\sms A$.
Let $\PD \colon H^1(X,\wti{A};\Z) \to H_{n-1}(X,A;\Z)$ be the Poincar\'e duality isomorphism.
We have $H^{1}(X,\wti{A};\Z)\cong [X/\wti{A}, S^1]$ and any such class can be represented by a continuous map $\varphi\colon X\to S^1$ that is constant on $\wti{A}$, and uniquely determined up to homotopy rel.\ $\wti{A}$.  We can and shall homotope $\varphi$ to a smooth map.
Furthermore, arrange that $-1 \in S^1$ is a regular value of~$\varphi$.
Then $Y:=\varphi^{-1}(-1)$ is an $(n-1)$-dimensional submanifold whose boundary lies on $\partial X\sms \wti{A}$, that is the boundary lies on $A$.
The manifold $Y= \varphi^{-1}(-1)$ satisfies $[Y]=\sigma$ (this follows from Lemma~\ref{lem:ThomPD}, as explained in the proof of Theorem~\ref{thm:represent-homology-by-submanifolds-all-dims}).
\item
Now suppose that we are given an oriented closed $(n-2)$-dimensional submanifold  $B$  of $A$ such that $\partial(\sigma)=[B]\in H_{n-1}(A)$.
Pick a collar neighbourhood $\partial X\times [0,1]$ and choose a continuous map $\varphi\colon X\sms \big(\partial X\times [0,1)\big) \to S^1$ as above. Also choose a tubular neighbourhood $B\times [-1/2,1/2]$ of $B$ in $A$. Consider the map sending $(b,t)$ to $e^{\pi i(t-1)}$, $b\in B, t\in [-1/2,1/2]$ and extend it to a smooth map $\psi\colon A\to S^1$ by sending all other points into $\{ e^{\pi i t} \in S^1 \mid t \in [-2/3, 2/3]\}$. Since $\partial(\sigma)=[B]\in H_{n-2}(A)\cong H^1(A;\Z)$, we see that the restriction of $\varphi$ to $A\times \{1\}=A$ is homotopic to $\psi\colon A\to S^1$.
Therefore, using this homotopy in the interval $[\frac{1}{2},1]$, we can extend $\varphi$ to a function on $X$ that restricts to $\psi$ on each $A\times \{s\}$ with $s\in [0,\frac{1}{2}]$. Finally, smoothen $\varphi$ without changing it on $A\times [0,\frac{1}{4}]$ to obtain a smooth map $X\to S^1$ in the same homotopy class. This is possible since the original $\varphi$ was already smooth on $A\times [0,\frac{1}{2}]$. Put differently, the new smooth map $\varphi\colon X\to S^1$ restricts to $\psi$ on $A=A\times \{0\}$.

Note that $-1$ is a regular value of $\psi$, and by changing $\varphi$ outside $A \times [0, \frac{1}{4}]$, we can also arrange $-1$ to be a regular value of $\varphi$.
The manifold $Y= \varphi^{-1}(-1)$ satisfies $[Y]=\sigma$ (as in (1) this follows from Lemma~\ref{lem:ThomPD} below) and $\partial Y=B\times \{0\} = B$.
\enm
For $\ell=2$ the argument is similar: we have to replace the argument using $S^1$ by the argument of \cite[Proposition~1.2.3]{GS99}. Recall from Theorem~\ref{thm:topological-manifold-CW complex} that $X$ is homotopy equivalent to a finite CW complex. Therefore, we represent a codimension~$2$ homology class $\sigma \in H_{n-2}(X,\partial X; \Z) \cong H^2(X; \Z)$ by a map $X \to \mathbb{CP}^{\infty}$, and homotope into the $k$-skeleton to a map $f \colon X \to \mathbb{CP}^k$ for $k \geq 2$.
Now arrange $f$ to be transverse to the codimension~$2$ submanifold $\mathbb{CP}^{k-1} \subseteq \mathbb{CP}^k$.
The desired submanifold is the preimage $Y = f^{-1}(\mathbb{CP}^{k-1})$. We leave further details to the reader. Again the argument is similar to that in the proof of Theorem~\ref{thm:represent-homology-by-submanifolds-all-dims}.
\end{proof}

\begin{lemma}\label{lem:contained-in-submanifold}
Let $W$ be a smooth $n$-manifold and let $C$ be a compact subset. There exists a compact smooth $n$-dimensional submanifold $X$ of $W$ that contains ~$C$.
\end{lemma}

\begin{proof}
By the Whitney Embedding Theorem (see e.g.\ \cite[Theorem~6.15]{Le13}), there exists a proper embedding $f\colon W\to \R^{2n+1}$. Recall that in this context proper means that the preimage of a compact set is compact. Pick a point $P\in \R^{2n+1}$ that does not lie in the image of $f$.
Denote the Euclidean distance to the point $P$ by $d\colon \R^{2n+1}\to \R_{\geq 0}$.
This map is smooth outside $P$, so in particular $d\circ f\colon W\to \R_{\geq 0}$ is smooth.  Since $C$ is compact, there exists an $r\in \R_{\geq 0}$ such that $(d\circ f)(C)\subseteq [0,r]$. By Sard's Theorem, there exists a regular value $x>r$. Then $X:=(d\circ f)^{-1}([0,x])$ has the desired properties.
\end{proof}

\begin{proof}[Proof of Theorem~\ref{thm:represent-homology-by-submanifolds}]
Let $M$ be a compact orientable connected $4$-manifold and let $A$ be a union of components of $\partial M$. Let $k=2$ or $k=3$ and let $\sigma\in H_k(X,A;\Z)$.

Pick a point $P\in M\sms \partial M$
and pick an open ball $B\subseteq M\sms \partial M$ containing $P$.
It follows from a Mayer-Vietoris argument applied to $M=(M\sms \{P\})\cup B$ that the inclusion induced map $H_k(M\sms \{P\},A)\to H_k(M,A)$ is an isomorphism for $k=2,3$.

Now let $\sigma\in H_k(M,A)$. By the previous paragraph we can view $\sigma$ as an element in $H_k(M\sms \{P\},A)$. By Theorem~\ref{thm:smooth-outside-a-point} the manifold $M\sms \{P\}$ is smooth.
There exists a compact subset $K$ of $M\sms \{P\}$ such that $\sigma$ lies in the image of $H_k(K,A)\to H_k(M\sms \{P\},A)$, since one can take the union of the images of the singular simplices in a singular chain representing $\sigma$.
By Lemma~\ref{lem:contained-in-submanifold}, there exists a compact 4-dimensional smooth submanifold $X$ of $M\sms \{P\}$ that contains the compact set $K\cup \partial M$. Note that $A$ is again a union of components of $\partial X$.
The desired statement of Theorem~\ref{thm:represent-homology-by-submanifolds}
is now an immediate consequence of Proposition~\ref{prop:represent-homology-by-smooth-submanifolds}~(1), with $\sigma$ the image of $\sigma \in H_k(K,A)$ under the inclusion induced map to $H_k(X,A)$.
\end{proof}



\chapter{Tubing of surfaces}\label{chapter:tubing}

As an example of the use of the technology we have discussed thus far, we show that one can tube together two locally flat embedded surfaces in a 4-manifold, to obtain an embedding of the connected sum.  This operation is standard in the smooth category, but as ever in the topological category one should take some care.

The following situation is by no means the most general such result possible.  We wish to illustrate two things. First, that operations on surfaces that can be performed in the smooth category can usually also be performed in general 4-manifolds with locally flat surfaces (although performing these operations in a parametrised way seems to be beyond current knowledge). Second, we want to show the level of detail required to demonstrate that such operations work.

\begin{proposition}\textbf{\textup{(Tubing Theorem)}}\label{prop:tubing}
  Let $S$ and $T$ be 2-dimensional proper submanifolds of a connected 4-manifold $M$, that is $S$ and $T$ are locally flat embedded surfaces.
Pick a point $P \in S \sm \partial S$ and $Q \in T \sm \partial T$.  Let $[\gamma] \in H_1(M,\{P,Q\};\Z)$ be a relative homology class.
There is a locally flat embedded arc $C$ joining $P$ and $Q$, satisfying the following.
 \begin{enumerate}[leftmargin=1cm,font=\normalfont]
 \item[(i)] We have $[C]= [\gamma] \in H_1(M,\{P,Q\};\Z)$.
 \item[(ii)] The interior of $C$ is disjoint from $S \cup T$.
\item[(iii)] The arc $C$ extends to a neighbourhood $C \times D^2$ embedded in $M$ such that $E_S:= \{P\} \times D^2 \subseteq S$ and $E_T:= \{Q\} \times D^2 \subseteq T$.
\item[(iv)] We have $(C \times D^2) \setminus (E_S \cup E_T) \subseteq M \sm(S \cup T)$.
\item[(v)] The intersection of $C \times D^2$ with a normal disc bundle $D(S)$ of $S$ is such that for every $d$, $(C \times \{d\}) \cap D(S)$ is a ray in a single fibre of $D(S)$, and similarly for $T$. Moreover there is a trivialisation of the normal vector bundle over $E_S$ as $E_S \times D^2$ such that for every $c \in C$ with $(\{c\} \times D^2) \cap D(S) \neq \varnothing$, we have that $\{c\} \times D^2 = E_S \times \{e\}$ for some $e \in D^2$, and all such $e$ that arise this way lie on a fixed ray from the origin of $D^2$.
 \end{enumerate}
\end{proposition}

These data allow us to perform tubing of surfaces ambiently.

\begin{proposition} \label{prop-tubing-gives-submanifold}
Given data $S$, $T$, $C \times D^2$, $E_S$ and $E_T$ as in Proposition~\ref{prop:tubing}, the subset
\[ (S\sm \tmfrac{1}{2} E_S)  \cup (T \sm \tmfrac{1}{2}E_T) \cup C \times \smfrac{1}{2}S^1\]
is a  2-dimensional submanifold abstractly homeomorphic to $S\# T$.
\end{proposition}

\begin{proof}
  The surfaces and the tube are locally flat by assumption, or by construction from Proposition~\ref{prop:tubing}.  The circles where the tube is glued to the surface are locally flat points. To see this observe that we have arranged a coordinate system in which this gluing is a completely standard attachment at angle $\pi/2$.
\end{proof}

\begin{figure}[h]
\begin{center}
\includegraphics{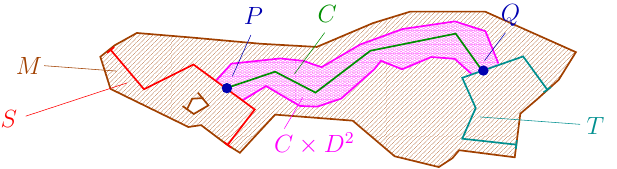}
\caption{Illustration of Proposition~\ref{prop:tubing}.}\label{fig:tubing}
\end{center}
\end{figure}

\begin{proof}[Proof of Proposition~\ref{prop:tubing}]
  Since $S$ and $T$ are proper submanifolds, they have normal vector bundles by Theorem~\ref{thm:existnormal}. Pick normal disc bundles $D(S)$ and $D(T)$, and remove the interiors of $\tsmfrac{1}{2}D(S)$ and $\tsmfrac{1}{2}D(T)$ i.e.\ smaller disc bundles inside the normal disc bundles.  We obtain a manifold with boundary
  \[X := M \sm \big(\Int\tmfrac{1}{2}D(S) \cup \Int\tmfrac{1}{2}D(T)\big)\]
  together with a collar neighbourhood of the boundary arising from $D(S) \sm \Int\tsmfrac{1}{2}D(S)$, and the same with $T$ replacing~$S$, extended using Theorem~\ref{thm:collar} to a collar neighbourhood for all of $\partial X$. Choose a closed disc neighbourhood $E_S$ of $P$ in $S$.
  We write $\partial_S X$ for the fibrewise boundary of $\tsmfrac{1}{2}D(S)$, $\partial_T X$ for the fibrewise boundary of $\tsmfrac{1}{2}D(T)$, and $\partial_1 X$ for  $\partial_S X \cup \partial_T X = \overline{\partial X \sm \partial M}$.

  Choose a trivialisation of the normal vctor bundle $\nu S$ in a neighbourhood $N(E_S)$ of $E_S$, as $N(E_S) \times D^2$.  A ray in $D^2$ from the origin to the boundary determines an embedding $E_S \times [0,1] \subseteq \tsmfrac{1}{2}D(S)$.  We obtain in particular a disc $E_S \times \{1\} \in N(E_S) \times \{\op{pt}\} \subseteq N(E_S) \times S^1$.  Choose a smooth structure on $\partial X$ (which we may do since $\partial X$ is a 3-manifold), and choose a smoothly embedded neighbourhood $F_S \cong D^3$ in $\partial_S X$ that contains $E_S \times \{1\}$ in its interior.

 Make the analogous set of choices and constructions for $T$, to obtain $E_T$, $N(E_T)$, $E_T \times [0,1] \subseteq \tsmfrac{1}{2}D(T)$, and $F_T \cong D^3$ in $\partial_T X$ that contains $E_T \times \{1\}$ in its interior.

Remove a point $r$ from $X$, and using Theorem~\ref{thm:smooth-outside-a-point} choose a smooth structure on $X \sm \{r\}$ extending the chosen smooth structure on $\partial X$.  Choose a smoothly embedded path $C_X \subseteq X$ between the centres of $E_S \times \{1\}$ and $E_T \times \{1\}$, such that $C_X$ extends along the previously chosen rays inside the normal vector bundles to a path $C$ between $P$ and $Q$ such that $[C]= [\gamma] \in H_1(M,\{P,Q\};\Z)$.
Extend $C_X$ to a codimension zero submanifold $N(C_X)$ homeomorphic to $I \times D^3$, with $I \times \{0\} \subseteq I \times D^3$ mapping to $C_X$, and such that $\{0\} \times D^3$ maps to $F_S \subseteq \partial_S X$ and $\{1\} \times D^3$ maps to $F_T \subseteq \partial_T X$.

\begin{figure}[h]
\begin{center}
\includegraphics{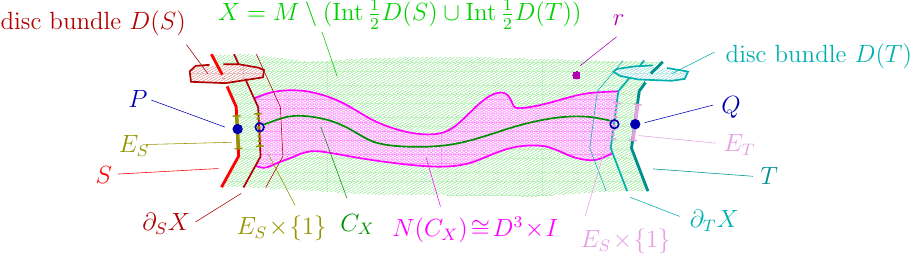}
\caption{Illustration for the proof of Proposition~\ref{prop:tubing}.}\label{fig:tubing-proof}
\end{center}
\end{figure}

Now, for small $\varepsilon$, $[0,\varepsilon] \times D^3$ and $[1-\varepsilon,1] \times D^3$ give rise to collar neighbourhoods of the closed subsets $F_S$ and $F_T$ of $\partial_1 X$.  Use Theorem~\ref{thm:collar} to extend this collar neighbourhood to a collar neighbourhood over all of $\partial X$.

We now have two collar neighbourhoods  of $\partial X$, the collar $\Psi_1 \colon \partial X \times [0,1] \hookrightarrow X$ we have just constructed which is compatible with $N(C_X)$, and the collar neighbourhood $\Psi_2 \colon \partial X \times [0,1] \hookrightarrow X$ constructed above from $D(S) \sm \Int\tsmfrac{1}{2}D(S)$ and $D(T) \sm \Int\tsmfrac{1}{2}D(T)$.
By Theorem~\ref{thm:topological-collar-unique}, there is an isotopy $H_t \colon M \to M$ starting from the identity, such that $H_1 \circ \Psi_1 = \Psi_2$, i.e.\ sends the first collar to the second.

We now obtain a codimension zero submanifold $C_X \times D^3$ homeomorphic to $I \times D^3$ such that, with respect to the collar neighbourhood $\Psi_2$, we have:
\begin{itemize}[leftmargin=0.7cm]
  \item For all $c \in C_X$ such that $\{c\} \times D^3 \cap \Psi_2(\partial X \times [0,1]) \neq \varnothing$, we have that $\{c\} \times D^3 \subseteq \Psi_2(\partial X \times \{t\})$ for some $t \in [0,1]$.
  \item For every $d \in D^3$, $(C \times \{d\}) \cap(\partial X \times [0,1]) = \Psi_2(\{x\} \times [0,1])$ for some $x$ in either $F_S$ or $F_T$.
      \end{itemize}
In addition, above we constructed two discs $E_S \subseteq F_S$ and $E_T \subseteq F_T$. Any two embedded discs in a 3-ball are ambiently isotopic: place this isotopy inside $C_X \times D^3$ to obtain a locally flat embedding $C_X \times D^2 \cong I \times D^2 \subseteq C_X \times D^3$.

Now consider $X \subseteq M$ and take the union
\[(E_S \times [0,1]) \cup (C_X \times D^2) \cup (E_T \times [0,1])\,\, \subseteq\,\, M\]
to obtain an embedding $C \times D^2 \cong I \times D^2$ whose intersection with $S$ equals $E_S$ and whose intersection with $T$ equals $E_T$. The core $C = C \times \{0\}$ is  a locally flat embedded path in $M$ from $P$ to $Q$ with interior in $M \sm (S \cup T)$ and with the correct relative homology class in $H_1(M,\{P,Q\};\Z)$. We may then perform the tubing $S\#  T := (S\sm E_S)  \cup (T \sm E_T) \cup C \times S^1$ as promised.
\end{proof}

\chapter{Classification results for 4-manifolds}\label{chapter:classification-simply-conn-4-mflds}

It is well-known (e.g.\ \cite[Theorem~5.1.1]{CZ}) that any finitely presented group is the fundamental group of a closed orientable smooth 4-manifold.
Markov \cite{Markov} used this fact to show that closed 4-manifolds cannot be classified up to homeomorphism.  To circumvent this group theoretic issue one aims to classify 4-manifolds with a given isomorphism type of a fundamental group.

In this chapter we present the known $4$-manifold classification results that have been obtained using the techniques of classical, or modified, surgery theory in the topological category, combined with Freedman's Disc Embedding Theorem~\cite{Freedman-82,FQ90,DETbook}. The use of this theorem requires the fundamental group of the $4$-manifold be ``good'' \cite[Part II, Introduction]{FQ90}, a condition that has a precise geometric description using the ``$\pi_1$-null disc property''. We will not reproduce that description here, but will instead note which groups are currently known to be good. Freedman showed that the infinite cyclic group and finite groups are good~\cite[pp.~658-659]{Freedman-82} (see also \cite[Section 5.1]{FQ90}). In addition, by \cite[Lemma~1.2]{Freedman-Teichner-95-I}  the class of good groups is closed under extensions, direct limits, subgroups and quotients.  It follows that solvable groups are good. Furthermore in \cite[Theorem~0.1]{Freedman-Teichner-95-I} and \cite{Krushkal-Quinn-00}
it was shown that groups with subexponential growth are good.

\section{Simply connected $4$-manifolds}

The following theorem was the first noteworthy result towards a classification of 4-dimensional manifolds.

\begin{theorem}\label{thm:milnor-whitehead}
Suppose $M$ and~$N$ are two closed oriented simply-connected 4-dimensional  manifolds. If the intersection forms are isometric, then $M$ and~$N$ are homotopy equivalent.
\end{theorem}

\begin{proof}
This theorem was proved for smooth manifolds by
Milnor~\cite[Theorem~3]{Milnor-58}, building on work of Whitehead~\cite{whitehead-49}. A proof that works in the general case is given in \cite[Chapter V, Theorem~1.5]{MH}.
\end{proof}

We state Freedman's classification for closed, simply connected 4-manifolds~\cite[Theorem~1.5]{Freedman-82}. We give the statement as in \cite[Theorem~10.1]{FQ90}.
We note that Freedman's original statement only applied to $4$-manifolds smoothable away from a point, and so the statement below requires the subsequent developments by Quinn. The last sentence comes from~\cite{Quinn-isotopy}.

\begin{theorem}\label{thm:classn-simply-connected-4-mfld}
  Fix a triple $(F,\theta,k)$, where $F$ is a finitely generated free abelian group, $\theta$ is a symmetric, nonsingular, bilinear form $\theta \colon F \times F \to \Z$, and $k \in \Z/2$.  If $\theta$ is even, that is $\theta(x,x) \in \Z$ is even for every $x \in F$, then suppose that $\sigma(\theta)/8 \equiv k  \in \Z/2$.

Then there exists a closed, simply connected, oriented $4$-manifold $M$ with $H_2(M;\Z) \cong F$, with intersection form isometric to $\theta$ and with Kirby-Siebenmann invariant equal to~$k$.

Let $M$ and $M'$ be two closed, simply connected, oriented $4$-manifolds and let $\phi \colon H_2(M;\Z) \xrightarrow{\cong} H_2(M';\Z)$ be an isometry of the intersection forms. Suppose that $\ks(M) = \ks(M')$.  Then there is an orientation preserving homeomorphism $M \xrightarrow{\cong} M'$ inducing $\phi$ on second homology. This homeomorphism is unique up to isotopy.
\end{theorem}

In other words, every even, symmetric, integral matrix with determinant $\pm 1$ is realised as the intersection form of a unique closed, simply connected, oriented $4$-manifold. For such matrices which are odd instead, there are precisely two closed, simply connected, oriented $4$-manifolds up to homeomorphism, exactly one of which has vanishing Kirby-Siebenmann invariant and is therefore stably smoothable. These two manifolds are homotopy equivalent by Theorem~\ref{thm:milnor-whitehead}.

In particular, the last paragraph with $M=M'$ implies that every automorphism of the intersection form of a closed, simply connected, oriented $4$-manifold is realised by a self-homeomorphism of $M$.

The following special case of  Theorem~\ref{thm:classn-simply-connected-4-mfld},  when $F=0$, is worth pointing out explicitly.

\begin{corollary}\textbf{\textup{(4-dimensional Poincar\'{e} conjecture)}}
  If $N$ is a $4$-manifold homotopy equivalent to $S^4$ then $N$ is homeomorphic to $S^4$.
\end{corollary}

\begin{proof}
Note $N$ and $S^4$ are closed, simply connected, and oriented. Furthermore, $H_2(N; \Z) = 0$ and the zero map $H_2(N; \Z) \to H_2(S^4; \Z)$ is an isometry (between zero forms).
By the last paragraph of Theorem~\ref{thm:classn-simply-connected-4-mfld}, there is a homeomorphism $N \cong S^4$ realising this isometry. Note that since $N$ has trivial and therefore even intersection form, $\ks(N) = \sigma(N)/8=0$ by Theorem~\ref{thm:ks-basics}~(\ref{item-ks-basics-6}).
\end{proof}

\section{Non simply-connected $4$-manifolds}

 We summarise known classification results for different types of nontrivial fundamental groups.

\subsection{Infinite cyclic group}
First, we present a classification result~\cite[Theorem~10.7A]{FQ90} for closed, oriented 4-manifolds with fundamental group $\Z$
which is quite similar to Theorem~\ref{thm:classn-simply-connected-4-mfld}. To state the theorem we need some extra definitions.

\begin{definition}
For a finitely generated free  $\Z[\Z]$ module $F$, a hermitian sesquilinear form $\theta \colon F \times F \to \Z[\Z]$  is called \emph{even} if there is a left $\Z[\Z]$-module homomorphism $q \colon F \to \overline{\Hom_{\Z[\Z]}(F,\Z[\Z])}$ with $q+ q^* \colon F \to \overline{\Hom_{\Z[\Z]}(F,\Z[\Z])}$ equal to the adjoint of $\theta$. Otherwise we call the form \emph{odd}.
\end{definition}

\begin{definition}Two homeomorphisms $h_0,h_1 \colon M \to N$ are \emph{pseudo-isotopic} if there is a homeomorphism $H \colon M \times I \to N \times I$ with $H|_{M \times \{i\}} = h_i \colon M \times \{i\} \to N \times \{i\}$ for $i=0,1$.
\end{definition}

An isotopy of homeomorphisms gives rise to a pseudo-isotopy. Perron and Quinn proved that the converse holds for compact simply connected 4-manifolds~\cite{Perron-isotopy}, \cite{Quinn-isotopy}).
Budney and Gabai~\cite{Budney-Gabai-23} showed that pseudo-isotopy does not in general imply isotopy for homeomorphisms between 4-manifolds with nontrivial fundamental groups.

\begin{theorem}\label{thm:classn-4-mfld-z}
  Fix a triple $(F,\theta,k)$, where $F$ is a finitely generated free $\Z[\Z]$-module, $\theta$ is a hermitian, nonsingular, sesquilinear form $\theta \colon F \times F \to \Z[\Z]$, and $k \in \Z/2$.
If $\theta$ is even, then suppose that $\sigma(\R \otimes \theta)/8 \equiv k  \in \Z/2$.

  Then there exists a closed, oriented $4$-manifold $M$ with $\pi_1(M) \cong \Z$, with $H_2(M;\Z[\Z])$  isomorphic to $F$, whose equivariant intersection form
  \[\lambda_M \colon H_2(M;\Z[\Z]) \times H_2(M;\Z[\Z]) \to \Z[\Z]\] is isometric to $\theta$, and with $\ks(M)=k$.

Let $M$ and $M'$ be two closed, oriented $4$-manifolds with $\pi_1(M) \cong \Z \cong \pi_1(M')$ and let $\phi \colon H_2(M;\Z[\Z]) \xrightarrow{\cong} H_2(M';\Z[\Z])$ be an isometry of the equivariant intersection forms. Suppose that $\ks(M) = \ks(M')$.  Then there is an orientation and basepoint  preserving homeomorphism $M \xrightarrow{\cong} M'$ inducing the given identification of the fundamental groups and inducing $\phi$ on $\Z[\Z]$ coefficient second homology. There are exactly two pseudo-isotopy classes of such homeomorphisms.
 \end{theorem}

The last sentence of this theorem is a correction to \cite[Theorem~10.7A]{FQ90} by Stong and Wang~\cite{Stong-Wang-2000}.\smallskip

\subsection{Baumslag-Solitar groups}
Here is another family of groups for which a complete classification of closed orientable 4-manifolds up to homeomorphism is known. This is the family of solvable Baumslag-Solitar groups
\[B(k) := \langle a,b \mid aba^{-1}b^{-k} \rangle. \]
Note that $B(0)=\Z$ and $B(1)=\Z^2$.
Baumslag-Solitar groups are solvable and, as we pointed out above, solvable groups are good. The next classification result was proven by Hambleton, Kreck, and Teichner in~\cite{HKT09}.

\begin{definition}
The \emph{$w_2$-type} of a closed, oriented 4-manifold $M$ with universal covering $\widetilde{M}$ is type I, II, III, as follows:  (I) $w_2(\wt{M}) \neq 0$; (II) $w_2(M)=0$; and (III) $w_2(M) \neq 0$ but $w_2(\wt{M})=0$.
 \end{definition}

\begin{theorem}\label{thm:BS-groups-classification}
  Let $B(k)$ be a solvable Baumslag-Solitar group and let $M$ and $N$ be closed, oriented 4-manifolds with fundamental group isomorphic to $B(k)$.
  Suppose that there is an isomorphism $\phi \colon H_2(M;\Z[B(k)]) \to H_2(N;\Z[B(k)])$ of $\Z[B(k)]$-modules such that:
  \begin{enumerate}[leftmargin=1cm,font=\normalfont]
  \item The map $\phi$ induces an isometry between the intersection form $\lambda \colon H_2(M;\Z[B(k)]) \times H_2(M;\Z[B(k)]) \to \Z[B(k)]$ and the corresponding intersection form on $H_2(N;\Z[B(k)])$.
  \item The Kirby-Siebenmann invariants agree $\ks(M)=\ks(N)$.
  \item The $w_2$-types of $M$ and $N$ coincide.
  \end{enumerate}
Then $M$ and $N$ are homeomorphic via an orientation preserving homeomorphism that induces $\phi \colon H_2(M;\Z[B(k)]) \to H_2(N;\Z[B(k)])$.
\end{theorem}

There is also a precise realisation result for these invariants~\cite[Theorem~B]{HKT09} and 4-manifolds with fundamental group $B(k)$.

\subsection{Finite cohomological dimension}
In the same paper as that discussed in the previous section~\cite{HKT09}, further classification results were given for 4-manifolds with geometrically 2-dimensional fundamental groups.

Some partial results towards a classification for 4-manifolds whose fundamental groups are good and have cohomological dimension 3 appear in Hambleton-Hildum~\cite{Hambleton-Hildum}.

Kasprowski-Land~\cite{Kasprowski-Land} studied 4-manifolds $M$ with 4-dimensional fundamental group, under the assumption that the classifying map $M \to B\pi$ is degree one, i.e.\ induces an isomorphism $H_4(M;\Z) \xrightarrow{\cong} H_4(B\pi;\Z)$.

\subsection{Finite groups}

Next, $4$-manifolds with finite fundamental groups were studied by Hambleton and Kreck in \cite{Hambleton-Kreck:1988, Hambleton-Kreck-93}.
Given a finitely generated abelian group $G$, let $TG$ be its torsion subgroup and let $\tf G:= G/TG$.

The most complete result was for 4-manifolds with finite cyclic fundamental group, given below.

\begin{theorem}\label{thm:classification-finite-cyclic}
  Let $G$ be a finite cyclic group and let $M$ and $N$ be closed, oriented 4-manifolds with fundamental group isomorphic to $G$.
  Suppose that there is an isomorphism $\phi \colon \tf H_2(M;\Z) \to \tf H_2(N;\Z)$ such that the following hold.
  \begin{enumerate}[leftmargin=1cm,font=\normalfont]
  \item The map $\phi$ induces an isometry between the intersection form $\lambda_M \colon \tf H_2(M;\Z) \times \tf H_2(M;\Z) \to \Z$ and the intersection form $\lambda_N \colon \tf H_2(N;\Z) \times \tf H_2(N;\Z) \to \Z$.
  \item The Kirby-Siebenmann invariants agree $\ks(M)=\ks(N)$.
  \item The $w_2$-types of $M$ and $N$ coincide.
 \end{enumerate}
Then $M$ and $N$ are homeomorphic via an orientation preserving homeomorphism that induces $\phi \colon \tf H_2(M;\Z) \to \tf H_2(N;\Z)$.
\end{theorem}

A full realisation result for the invariants in Theorem~\ref{thm:classification-finite-cyclic} is not known, however Hambleton-Kreck showed how to realise in the majority of cases.
The following relations between the invariants hold.
\begin{enumerate}[leftmargin=1cm]
    \item If $w_2(M)=0$, then $\ks(M) \equiv \sigma(M)/8 \in \Z/2$ and $\lambda_M$ is even.
    \item If $M$ is type I, then $\lambda_M$ is odd.
    \item If the order of $G$ is odd, then $w_2(\wt{M})=0$ implies $w_2(M)=0$, so there are no 4-manifolds with $w_2$-type III.
\end{enumerate}

We outline a construction that realises all configurations of the invariants, with the restriction that in $w_2$-type III, the intersection form $\lambda_M$ is even.
The key is a construction of rational homology 4-spheres. For every finite cyclic group $G$, \cite[Proposition~4.1]{Hambleton-Kreck-93} produces the following manifolds.
\begin{enumerate}[leftmargin=1cm]
    \item A rational homology sphere $\Sigma_G^{II}$ with $w_2$-type II and with fundamental group $G$. Note that $\ks(\Sigma_G^{II}) \equiv \sigma(\Sigma_G^{II})/8 =0$.
   \item A rational homology sphere $\Sigma_G^{III,0}$ with $w_2$-type III, trivial Kirby-Siebenmann invariant, and fundamental group $G$.
   \item  A rational homology sphere $\Sigma_G^{III,1}$ with $w_2$-type III, nontrivial Kirby-Siebenmann invariant, and fundamental group $G$.
\end{enumerate}

There can be no rational homology sphere with $w_2$-type I by \cite[Theorem~4.2]{Hambleton-Kreck:1988}; in the notation of that theorem, $Q''(\pi_1,0)$ gives rise to manifolds with nontrivial $H_2(-;\Q)$.
Now we describe the partial realisation of the invariants from Theorem~\ref{thm:classification-finite-cyclic}.

\begin{enumerate}[leftmargin=1cm]
  \item By taking the connected sum of $\Sigma_G^{II}$  with a closed, spin, simply connected manifold, we can realise any even, nonsingular, symmetric, bilinear form as the intersection form $\lambda_M \colon \tf H_2(M;\Z) \times \tf H_2(M;\Z) \to \Z$ of a closed, oriented 4-manifold $M$ with fundamental group $G$ and with $w_2$ type II. In this case $\ks(M)$ is determined by the signature of $\lambda_M$.
  \item  Likewise, taking connected sum of $\Sigma_G^{III,0}$ or $\Sigma_G^{III,1}$ with a  closed, spin, simply connected manifold,  we can realise every even $\lambda_M$ as the intersection form of a closed, oriented 4-manifold $M$ with fundamental group $G$ and with $w_2$ type III, with prescribed Kirby-Siebenmann invariant.
  \item Finally, by taking connected sum of $\Sigma_G^{III,0}$ or $\Sigma_G^{III,1}$  with a closed, oriented, simply connected 4-manifold, we can realise any odd, nonsingular, symmetric, bilinear form as the intersection form $\lambda_M \colon \tf H_2(M;\Z) \times \tf H_2(M;\Z) \to \Z$ of a closed, oriented 4-manifold $M$ with fundamental group $G$ and with $w_2$ type I, with prescribed Kirby-Siebenmann invariant.
\end{enumerate}

\begin{question}
Must the intersection form in $w_2$-type III be even?
If not, how can we realise all intersection forms and Kirby-Siebenmann invariants in $w_2$-type III?
\end{question}


In his survey paper, Hambleton~\cite[Theorem~5.2]{Hambleton-gokova} also outlined a homeomorphism classification for closed, spin 4-manifolds with finite odd order fundamental group.

The following result on 4-manifold with finite fundamental group from ~\cite[Theorem~B]{Hambleton-Kreck-93-b-cancellation} also deserves to be mentioned.

\begin{theorem}
    Let $M$ and $N$ be closed, oriented, topological 4-manifolds with finite fundamental group. Suppose that $M \#^r (S^2 \times S^2)$ and  $N\#^r (S^2 \times S^2)$ are homeomorphic for some $r \in \N$. Suppose that $X = X_0 \# (S^2 \times S^2)$. Then $X$ is homeomorphic to $Y$.
\end{theorem}

\subsection{Nonorientable 4-manifolds}
For nonorientable closed 4-manifolds, the homeomorphism classification results we are aware of are for fundamental group $\Z/2$ in~\cite{HKT-nonorientable} and for fundamental group $\Z$ in~\cite{Wang-nonorientable}.
For nonorientable closed 4-manifolds with fundamental group $\Z/2$, the paper~\cite{HKT-nonorientable} gives a complete list of invariants for distinguishing such manifolds up to homeomorphism~\cite[Theorem~2]{HKT-nonorientable}, and gives a list of the possible manifolds~\cite[Theorem~3]{HKT-nonorientable}.

\subsection{4-manifolds with nonempty boundary}

Simply-connected compact 4-manifolds with a fixed 3-manifold as boundary were classified by Boyer in \cite{Boyer-86,Boyer-93}, with an independent contribution by Stong~\cite{Stong-manifolds-boundary}.
Homeomorphisms of such 4-manifolds were classified up to isotopy by Orson-Powell~\cite{Orson-Powell-23}.  Since the statements are somewhat involved, we refer the reader to the original articles.

For compact 4-manifolds with fundamental group $\Z$, an analogous  classification was given by Conway-Powell~\cite{Conway-Powell-23} and Conway-Piccirillo-Powell~\cite{CPP22}, under the assumptions that $\pi_1(\partial M) \to \pi_1(M) \cong \Z$ is surjective, and that the homology $H_1(\partial M;\Z[\Z])$ of the corresponding $\Z$-cover is a $\Z[\Z]$-torsion module.

\chapter{Stable smoothing of homeomorphisms}\label{chapter:stablesmoothing}

Wall~\cite{Wall-stable-diff} proved that simply connected, closed, smooth 4-manifolds with isometric intersection forms are stably diffeomorphic. It follows that every pair of simply connected, closed, homeomorphic smooth 4-manifolds are stably diffeomorphic.
We shall discuss the analogous statement without the simply connected hypothesis.

\begin{definition}\leavevmode
\begin{enumerate}[leftmargin=1cm]
\item Let $M$ and $N$ be connected, smooth 4-manifolds.  We say that $M$ and $N$ are \emph{stably diffeomorphic} if there is an integer $k$ such that the connected sums $M \#^k (S^2 \times S^2)$ and $N \#^k (S^2 \times S^2)$ are diffeomorphic.
\item Let $M$ and $N$ be connected 4-manifolds.  We say that $M$ and $N$ are \emph{stably homeomorphic} if there is an integer $k$ such that the connected sums $M \#^k (S^2 \times S^2)$ and $N \#^k (S^2 \times S^2)$ are homeomorphic.
 \end{enumerate}
\end{definition}

The next theorem is due to Gompf~\cite{Gompf84}.

\begin{theorem}\label{thm:stably-diffeo-homeo-4-mflds}
Every homeomorphic pair of compact, connected, orientable, smooth 4-manifolds with diffeomorphic boundaries are stably diffeomorphic.

Moreover, let $f \colon M \to N$ be a homeomorphism between two such 4-manifolds, that restricts to a diffeomorphism $f| \colon \partial M \to \partial N$. Then $f|$ extends to a stable diffeomorphism.
\end{theorem}

One might imagine a stronger statement, that given a homeomorphism $f\colon M\to N$ we can smoothen it stably up to isotopy.  However such a statement is only known for simply connected 4-manifolds~\cite[Chapter~8]{FQ90}, and does not hold in general.

 For non simply-connected manifolds, one must consider the bundle map of stable tangent microbundles induced by $f$, and lift it to a bundle map between the stable tangent bundles.  Such a lift does not exist in general; there is a Casson-Sullivan obstruction in $H^3(M,\partial M;\Z/2)$ to its existence.  Cappell-Shaneson~\cite{Cappell-Shaneson-4D-surgery}, using unpublished work of R.~Lee, showed that there is a homeomorphism of $(S^1 \times S^3) \# (S^2 \times S^2)$ for which this obstruction is nontrivial. Hence this homeomorphism cannot be stably smoothed.
If a lift does exist, then for each lift there is a stabilisation by $\#^k (S^2 \times S^2)$, and then a pseudo-isotopy of the stabilised $F$ to a diffeomorphism~\cite[Chapter~8]{FQ90}. So even when the Casson-Sullivan invariant vanishes, one only has a stable smoothing up to pseudo-isotopy.

 The proof of Theorem~\ref{thm:stably-diffeo-homeo-4-mflds} that we shall give using Kreck's modified surgery~\cite{surgeryandduality} was outlined in Teichner's thesis~\cite[Theorem~5.1.1]{teichnerthesis}.
We think this proof is worth publicising with expanded details, because the method is arguably more conceptual than Gompf's original, and because it allows us to expand on Gompf's statement in the nonorientable case.

Gompf also proved that for every pair of compact, connected, nonorientable, smooth 4-manifolds $M$ and $N$ that are homeomorphic, $M \# S^2 \wt{\times} S^2$ and $N \# S^2 \wt{\times} S^2$ are stably diffeomorphic.  We shall slightly improve on this statement.

\begin{theorem}\label{thm:stable-diff-nonorientable}
  Let $M$ and $N$ be compact, connected, nonorientable, smooth 4-manifolds. Suppose that $M$ and $N$ are homeomorphic via a homeomorphism restricting to a diffeomorphism $\partial M \cong \partial N$.  If $w_2(\wt{M}) \neq 0 \neq w_2(\wt{N})$, that is the universal covers of $M$ and $N$ are not spin, then $M$ and $N$ are stably diffeomorphic via a stable diffeomorphism extending the given diffeomorphism $\partial M \cong \partial N$.
\end{theorem}

Gompf's statement \cite[p.~116]{Gompf84} for the nonorientable case, given in the next corollary, follows easily from Theorem~\ref{thm:stable-diff-nonorientable}. However note that Theorem~\ref{thm:stable-diff-nonorientable}  shows that for many nonorientable 4-manifolds, the extra summand given by the twisted bundle $S^2 \wt{\times} S^2$ is not necessary.

\begin{corollary}\label{cor:stably-diffeomorphic-nonorientable-manifolds}
   Let $M$ and $N$ be compact, connected, nonorientable, smooth 4-manifolds. Suppose that $M$ and $N$ are homeomorphic via a homeomorphism restricting to a diffeomorphism $\partial M \cong \partial N$.  Then $M \# S^2 \wt{\times} S^2$ and $N \# S^2 \wt{\times} S^2$ are stably diffeomorphic.
\end{corollary}

\begin{proof}
  Taking the connected sum of any 4-manifold with $S^2 \wt{\times} S^2 \cong \mathbb{CP}^2 \# \overline{\mathbb{CP}^2}$ gives rise to a 4-manifold whose universal cover is not spin. The corollary therefore follows from Theorem~\ref{thm:stable-diff-nonorientable}.
\end{proof}

The hypothesis in Theorem~\ref{thm:stable-diff-nonorientable} that $w_2(\wt{M}) \neq 0 \neq w_2(\wt{N})$ cannot be dropped in general. Cappell and Shaneson found an example of a smooth 4-manifold $R$ that is homotopy equivalent to the real projective space  $\mathbb{R}\textup{P}^4$ but that is not stably diffeomorphic to $\rp^4$~\cite{Cappell-Shaneson-4D-surgery, Cappell-Shaneson-new-4-manifolds}. When these papers were published, it was not possible to prove that the fake $\mathbb{R}\textup{P}^4$ manifold $R$ is homeomorphic to $\mathbb{R}\textup{P}^4$, but this was later established~\cite[p.~221]{ruberman-invariant-knots} as a consequence of the work of Freedman and Quinn~\cite{FQ90}, and the fact that the Whitehead group of $\Z/2$ is trivial.

Later, Kreck~\cite{Kreck-nonorientable} showed a much more general statement in this direction.   Let $K3:=\{[z_0:z_1:z_2:z_3]\in \cp^3 \mid z_0^4+z_1^4+z_2^4+z_3^4=0\}$  denote the Kummer surface.
As discussed in \cite[Chapter~1.3]{GS99}, this is a closed, smooth, spin $4$-manifold with signature 16, $b_2(K3)= 22$ and intersection form $3\cdot H\oplus 2\cdot E_8$. Here is Kreck's result from~\cite{Kreck-nonorientable}. These were the first known examples of exotic pairs of $4$-manifolds.

\begin{theorem}\label{thm:kreck-counterexamples}
  Let $\pi$ be a finitely presented group with a surjective homomorphism $w \colon \pi \to \Z/2$. Then there exists a closed, smooth, connected $4$-manifold $W$ with fundamental group~$\pi$ and orientation character~$w$, with the property that $W \# K3$ and $W \#^{11} (S^2 \times S^2)$ are homeomorphic $4$-manifolds that are not stably diffeomorphic.
\end{theorem}

One part of this is easy to see: if $W$ is nonorientable then there are homeomorphisms
\begin{align*}
  W \# K3\,\, &\cong \,\,W \# E_8 \# E_8 \#^3 (S^2 \times S^2) \,\,\cong\,\, W \# E_8 \# \overline{E}_8 \#^3 (S^2 \times S^2) \\  &\cong \,\,W \#^8 (S^2 \times S^2) \#^3 (S^2 \times S^2) \,\,\cong\,\, W \#^{11} (S^2 \times S^2). \end{align*}
Here we used Theorem~\ref{thm:classn-simply-connected-4-mfld} that simply connected closed 4-manifolds with Kirby-Sieben\-mann invariant vanishing are determined by their intersection forms, and we used that the connected sum $M \# N$ of an oriented manifold $M$ with a nonorientable manifold $N$ is homeomorphic to $\overline{M} \# N$.

In the following three sections we will prove Theorems~\ref{thm:stably-diffeo-homeo-4-mflds} and~\ref{thm:stable-diff-nonorientable}. To keep the notation manageable we will only provide a proof  for closed manifolds, and descirbe the case of nonempty boundary in Section~\ref{section:homeo-4-mflds-with-bdy-stable-homeo-vs-stable-diffeo}.

\section{Kreck's modified surgery}\label{section:modified-surgery}
Below we will state a theorem due to Kreck that relates stable diffeomorphisms of 4-manifolds with bordism theory. This came as a corollary of Kreck's  modified surgery theory~\cite{surgeryandduality}. First we need some definitions from~\cite{surgeryandduality}.

Recall that a topological space $A$ is \emph{$m$-connected} if $\pi_k(A)=0$ for $1\leq k\leq m$ and is \emph{$m$-coconnected} if $\pi_k(A)=0$ for $k\geq m$. A map of spaces $f\colon A\to B$ is \emph{$m$-connected} if the homotopy cofibre (i.e.~the mapping cone) is $m$-connected; equivalently $f_*\colon\pi_k(A)\to \pi_k(B)$ is an isomorphism for $k<m$ and is surjective for $k=m$. A map of spaces $f\colon A\to B$ is \emph{$m$-coconnected} if the homotopy fibre is $m$-coconnected; equivalently $f_*\colon\pi_k(A)\to \pi_k(B)$ is an isomorphism for $k>m$ and is injective for~$k=m$.

\begin{definition}
  A \emph{normal 1-type} of a closed, connected, smooth 4-manifold $M$ is a 2-coconnected fibration $\xi \colon B \to \BO$ for which there is a 2-connected lift $\wt{\nu}_M \colon M \to B$ of the stable normal vector bundle $\nu_M \colon M \to \BO$ such that $\xi \circ \wt{\nu}_M = \nu_M \colon M \to \BO$.  We call such a choice of lift  $\wt{\nu}_M \colon M \to B$ a \emph{normal 1-smoothing}.
\end{definition}

\begin{remark}\leavevmode
The data of a normal 1-type is $\xi \colon B \to \BO$. The existence of $\wt{\nu}_M$ is a condition on that data.
\end{remark}

\begin{definition}
  A \emph{normal 1-type} of a closed, connected 4-manifold $M$ is a 2-coconnected fibration $\xi^{\operatorname{TOP}} \colon B^{\operatorname{TOP}} \to \BTOP$ for which there is a 2-connected lift $\wt{\nu}_M \colon M \to B^{\operatorname{TOP}}$ of the stable topological normal bundle $\nu_M \colon M \to \BTOP$ (Definition \ref{def:topnormal}) such that $\xi^{\operatorname{TOP}} \circ \wt{\nu}_M = \nu_M \colon M \to \BTOP$.  We call such a choice of lift  $\wt{\nu}_M \colon M \to B^{\operatorname{TOP}}$ a \emph{normal $\TOP$ 1-smoothing}.
\end{definition}

Normal 1-types $\xi \colon B \to \BO$ of a closed, connected smooth 4-manifold are fibre homotopy equivalent over $\BO$, and using this we abuse notation and refer to \emph{the} normal 1-type of a smooth 4-manifold, and similarly for the topological version.  Here are some of the key examples in the oriented case.
We will give the details of the nonorientable case in Section~\ref{section:nonorientable}.

Write $\pi = \pi_1(M)$ and let $w_2 \in H^2(M;\Z/2)$ be the second Stiefel-Whitney class of $M$.  There are three main cases for the normal 1-types of oriented, closed smooth 4-manifolds. For more details, see~\cite[Sections~2~and~3]{KLPT-17}.

\begin{lemma}\label{lemma:smooth-normal-1-types}
Let $M$ be a closed, oriented, connected, smooth 4-manifold with universal covering $\wti{M}$.  We write $\pi:=\pi_1(M)$.
  \begin{enumerate}[leftmargin=1cm,font=\normalfont]
  \item\label{item:tot-non-spin-smooth} Suppose that we have $w_2(\wt{M}) \neq 0$. Then $\xi \colon B = \op{B\pi} \times \BSO \to \BO$ is the normal 1-type of $M$, with the map $\xi$ given by projection to $\BSO$ followed by the canonical map $\BSO \to \BO$.
  \item\label{item:spin-case-smooth}  Suppose that $w_2(M)=0$, i.e.\ $M$ is spin. Then $\xi \colon B = \op{B\pi} \times \BSpin \to \BO$ is the normal 1-type of $M$, with the map $\xi$ given by projection to $\BSpin$ followed by the canonical map $\BSpin \to \BO$.
 \item \label{item:almost-spin-case-smooth} Suppose that we have $w_2(M) \neq 0$ but $w_2(\wt{M}) = 0$. Then there is a model for $\op{B\pi}$ and a fibration $w_2  \colon \op{B\pi} \to K(\Z/2,2)$ that pulls back along $M \to \op{B\pi}$ to $w_2(M)$. The fibration $\xi \colon B \to \BO$ is obtained from pulling back $w_2$ along the universal class $\ol{w}_2 \colon \BSO \to K(\Z/2,2)$, to obtain the space $B$ and a fibration $B \to \BSO$, and then composing with $\BSO \to \BO$.  Since we have a fibration sequence $\BSpin \to \BSO \xrightarrow{\ol{w}_2} K(\Z/2,2)$, the pullback gives rise to a fibration sequence $\BSpin \to B \to \op{B\pi}$. Then $\xi \colon B \to \BO$ is the normal 1-type of~$M$.
\end{enumerate}
\end{lemma}



Recall $\pi_1(\STOP) \cong \pi_1(\TOP) \cong \pi_1(\O)\cong \Z/2$; see Chapter~\ref{chapter:SW-classes}.

\begin{lemma}\label{lemma:top-normal-1-types}
Let $M$ be a closed, oriented, connected 4-manifold with universal covering $\wti{M}$. We write $\pi:=\pi_1(M)$.
\begin{enumerate}[leftmargin=1cm,font=\normalfont]
  \item\label{item:tot-non-spin-homeo} Suppose that we have $w_2(\wt{M}) \neq 0$. Then $\xi^{\operatorname{TOP}} \colon B^{\operatorname{TOP}} = \op{B\pi} \times \BSTOP \to \BTOP$ is the normal 1-type of $M$,
  with the map given by projection to $\BSTOP$ followed by the canonical map $\BSTOP \to \BTOP$.
  \item\label{item:spin-case-homeo}  Suppose that $w_2(M)=0$, i.e.\ $M$ is spin. Then $\xi^{\operatorname{TOP}} \colon B^{\operatorname{TOP}} = \op{B\pi} \times \BTOPSpin \to \BTOP$ is the normal 1-type of $M$,
  with the map given by projection to $\BTOPSpin$ followed by the canonical map $\BTOPSpin \to \BTOP$.
 \item\label{item:almost-spin-case-homeo} Suppose that we have $w_2(M) \neq 0$ but $w_2(\wt{M}) = 0$.
Then there is a model for $\op{B\pi}$ and a fibration $w_2  \colon \op{B\pi} \to K(\Z/2,2)$ that pulls back along $M \to \op{B\pi}$ to $w_2(M)$.
The fibration $\xi^{\TOP} \colon B^{\TOP} \to \BTOP$ is obtained from pulling back $w_2$ along the universal class $\ol{w}_2 \colon \BSTOP \to K(\Z/2,2)$, to obtain the space $B^{\TOP}$ and a fibration $B^{\TOP} \to \BSTOP$, and then composing with $\BSTOP \to \BTOP$.  Since we have a fibration sequence $\BTOPSpin \to \BSTOP \xrightarrow{\ol{w}_2} K(\Z/2,2)$, the pullback gives rise to a fibration sequence $\BTOPSpin \to B \to \op{B\pi}$. Then $\xi^{\TOP} \colon B^{\TOP} \to \BTOP$ is the normal 1-type of~$M$.
%
%
\end{enumerate}
\end{lemma}

Here is the relevant theorem of Kreck~\cite[Theorem~C]{surgeryandduality}, which relates bordism over the normal 1-type to stable diffeomorphism.  We write $\Omega_4(B,\xi)$ for the group under disjoint union of closed 4-manifolds $M$ together with a lift  $\ol{\nu} \colon M \to B$ along $\xi$ of the stable normal vector bundle, considered up to bordism over $(B,\xi)$.

\begin{theorem}\label{thm:kreck}
  Two closed, connected, smooth 4-manifolds $M$ and $N$ with $\chi(M) = \chi(N)$ and normal 1-types both fibre homotopy equivalent to a fixed fibration $\xi \colon B \to \BO$ are stably diffeomorphic if and only if \[[(M,\wt{\nu}_M)] = [(N,\wt{\nu}_N)] \in \Omega_4(B,\xi)\]
   for some choices of normal 1-smoothings $\wt{\nu}_M$ and $\wt{\nu}_N$.
\end{theorem}

\begin{proof}[Sketch of the proof]
One direction is quite easy: one has to check that $M$ and $M \# (S^2 \times S^2)$ are bordant over the normal 1-type of $M$.

For the other direction, start with a 5-dimensional bordism $W$ over $(B,\xi)$ and perform surgery below the middle dimension~\cite[Section~3]{surgeryandduality} to arrange that the map~$W \to B$ is 1-connected. We outline the procedure next.
First  perform surgery on pairs of discs, extending the map to $B$ over the new copies of $D^1 \times S^4$ to make the map $\pi_1(W) \to \pi_1(B)$ surjective. Then represent normal generators of the kernel of $\pi_1(W) \to \pi_1(B))$ by framed circles using Chapters~\ref{chapter:tubular} and~\ref{chapter:transversality}. Since $\pi_1(W)$ and $\pi_1(B)$ are finitely presented, this can be done with finitely many circles. To prove that the normal vector bundles admit framings one must use the bundle data $\xi$; for this we refer to \cite[Section~3]{surgeryandduality}.   Perform surgery on the framed circles to make the map to~$B$ 1-connected.  This completes the surgery below the middle dimension step.

Now represent the elements of $\ker(\pi_2(W) \to \pi_2(B))$ by framed embedded spheres, and remove thickenings of these spheres.  Also, for each sphere, remove a tube $D^1 \times D^4$ connecting that sphere to either $M$ or $N$. Choose whether to tube to $M$ or $N$ so as to preserve the Euler characteristic equality.  This operation of removing copies of $S^2 \times D^3$, tubed to the boundary, has the effect of adding copies of $S^2 \times S^2$ to $M$ and $N$ giving rise to $M'$ and $N'$ respectively. The operation also converts $W$ to an $s$-cobordism $W'$. That $(W';M',N')$ is an $s$-cobordism means by definition that the inclusion maps $M' \to W'$ and $N' \to W'$ are simple homotopy equivalences. The stable $s$-cobordism theorem~\cite{Quinn-stable-s-cob} states that every $5$-dimensional $s$-cobordism becomes diffeomorphic to a product after adding copies of $(S^2 \times S^2) \times I$ along a smoothly embedded interval $I \subseteq W'$ with one endpoint on each of $M$ and $N$.  This completes the sketch proof of Theorem~\ref{thm:kreck}.
\end{proof}

The proof of the topological version is similar.

\begin{theorem}\label{thm:kreck-homeo}
  Two closed, topological 4-manifolds $M$ and $N$ with $\chi(M) = \chi(N)$ and
 normal 1-types both fibre homotopy equivalent to a fixed fibration $\xi^{\operatorname{TOP}} \colon B^{\operatorname{TOP}} \to \BTOP$ are stably homeomorphic if and only if
   \[[(M,\wt{\nu}_M)] \,=\, [(N,\wt{\nu}_N)]\,\, \in\, \Omega_4^{\operatorname{TOP}}(B^{\operatorname{TOP}},\xi^{\operatorname{TOP}})\] for some choices of normal 1-smoothings $\wt{\nu}_M$ and $\wt{\nu}_N$.
\end{theorem}

From now on, to ease notation, we will sometimes abbreviate $\Omega_4^{\operatorname{TOP}}(B^{\operatorname{TOP}},\xi^{\operatorname{TOP}})$ to $\Omega_4^{\operatorname{TOP}}(B,\xi)$.

\begin{proof}
  One direction is again quite easy: we need that homeomorphic manifolds are bordant over $B$, and that $M$ and $M \# (S^2 \times S^2)$ are bordant in $\Omega_4^{\TOP}(B,\xi)$.
  For the other direction, apply the same argument as above to improve a cobordism $W$ to an $s$-cobordism. The stable
  $s$-cobordism Theorem applies to topological $s$-cobordisms as well as to smooth $s$-cobordisms. This is not written in \cite{Quinn-stable-s-cob}, but the same proof applies, with the following additions (see the Exercise on \cite[p.~107]{FQ90}). First, $5$-dimensional cobordisms admit
  a topological handle structure~\cite[Theorem~9.1]{FQ90}.  The proof of \cite{Quinn-stable-s-cob} consists of simplifying a handle decomposition, and tubing surfaces in 4-manifolds around and into parallel copies of one another to remove intersections. This is possible in the topological category by using transversality (Theorem~\ref{thm:TransvSubmanifolds}) to arrange that intersections between surfaces are isolated points, and the existence of normal vector bundles (Theorem~\ref{thm:existnormal}) to take parallel copies using sections.
  \end{proof}

\section{Stable diffeomorphism of homeomorphic orientable 4-manifolds}\label{section:orientable}

Now we will explain the proof of Theorem~\ref{thm:stably-diffeo-homeo-4-mflds}. For the convenience of the reader, we recall the statement.

\begin{theorem*}\textbf{\textup{\ref{thm:stably-diffeo-homeo-4-mflds}.}}
Every homeomorphic pair of closed, connected, orientable, smooth 4-manifolds are stably diffeomorphic.
\end{theorem*}

The proof will rest on the following proposition.

\begin{proposition}\label{prop:forgetful-injective}
 Let $(B,\xi)$ be one of the oriented smooth normal 1-types from Lemma~\ref{lemma:smooth-normal-1-types}, and let $(B^{\operatorname{TOP}}, \xi^{\operatorname{TOP}})$ be the corresponding topological normal 1-type from Lemma~\ref{lemma:top-normal-1-types} obtained by replacing $\BSO$ with $\BSTOP$ or $\BSpin$ with $\BTOPSpin$ as appropriate.
The forgetful map
 \[F \colon \Omega_4 (B,\xi) \to \Omega_4^{\operatorname{TOP}}(B^{\operatorname{TOP}},\xi^{\operatorname{TOP}}) = \Omega_4^{\operatorname{TOP}}(B,\xi)\]
 is injective.
\end{proposition}

The combination of this proposition with Theorems~\ref{thm:kreck} and~\ref{thm:kreck-homeo} implies the following corollary, which is the closed version of Theorem~\ref{thm:stably-diffeo-homeo-4-mflds}, with a slightly more precise statement concerning orientations.

\begin{corollary}
Every pair of smooth, closed, connected, oriented 4-manifolds that are homeomorphic via an orientation preserving homeomorphism are stably diffeomorphic via an orientation preserving diffeomorphism.
\end{corollary}

\begin{proof}
  We prove the corollary assuming Proposition~\ref{prop:forgetful-injective}. Homeomorphic 4-manifolds are in particular stably homeomorphic and have the same normal 1-types. Therefore two homeomorphic smooth  4-manifolds as in the statement of the corollary are bordant over the normal 1-type, so give rise to equal elements in $\Omega_4^{\operatorname{TOP}}(B,\xi)$. By Proposition~\ref{prop:forgetful-injective}, they give rise to equal elements of $\Omega_4 (B,\xi)$. Then by Theorem~\ref{thm:kreck}, the two 4-manifolds are stably diffeomorphic, as asserted.
\end{proof}

\begin{proof}[Proof of Proposition~\ref{prop:forgetful-injective}]
  Let $S$ be $\SO$ in case \eqref{item:tot-non-spin-smooth} of the smooth list of 1-types given in
  Lemma~\ref{lemma:top-normal-1-types}, and let $S$ denote $\op{Spin}$ in cases \eqref{item:spin-case-smooth} and \eqref{item:almost-spin-case-smooth}.

  Let $ST$ be $\STOP$ in case \eqref{item:tot-non-spin-homeo} of the topological list of 1-types above, and let $ST$ denote $\TOPSpin$ in cases \eqref{item:spin-case-homeo} and \eqref{item:almost-spin-case-homeo}.

  The James spectral sequence~\cite[Theorem~3.1.1]{teichnerthesis}, \cite[Section~3]{KLPT-17} is of the form:
  \[E^2_{p,q} = H_p(\op{B\pi};\Omega^S_q) \, \Rightarrow \, \Omega_{p+q}(B,\xi).\]
We have that $\Omega^S_4 \cong \Z$, detected by the signature.  Indeed, the signature is a $\Z$-valued invariant that agrees for stably diffeomorphic 4-manifolds.
The signature of a 4-manifold with a normal 1-smoothing into $B$ gives rise to an element of $\Omega_{4}(B,\xi)$.
The $E^2$ term $H_0(\op{B\pi};\Omega_4^S) \cong \Z$ is computed using $\Omega^S_4 \cong \Z$.

\begin{claim}
 This term $H_0(\op{B\pi};\Omega_4^S)$ survives to the $E^{\infty}$ page.  That is, all differentials with this as codomain are trivial.
\end{claim}

Let us prove the claim. Since $\Omega^S_q$ is torsion for $q=1,2,3$, no terms from those $q$-lines can map to $H_0(\op{B\pi};\Omega_4^S)$ under a differential.

Aside from $H_0(\op{B\pi};\Omega_4^S)$, there is one other potentially infinite term on the 4-line of the  $E^{\infty}$ page, namely the subgroup of $H_4(\op{B\pi};\Omega_0^S)$ arising as the kernel of relevant differentials. Since there are no differentials with the $(4,0)$ term as codomain, this subgroup is a quotient of $\Omega_4(B,\xi)$.  The image of $[M,c] \in \Omega_4(B,\xi)$ is the image $c_*([M])$ of the fundamental class under the classifying map $c_* \colon H_4(M;\Z) \to H_4(\op{B\pi};\Z) \cong H_4(\op{B\pi};\Omega^S_0)$.

There could be a nontrivial differential $d_{5,0}^5 \colon H_5(\op{B\pi};\Omega_0^S) \to H_0(\op{B\pi};\Omega_4^S)$.  However if there were a nonzero differential, then only finitely many signatures would occur for 4-manifolds with normal 1-type $B$ and fixed invariant in $H_4(\op{B\pi};\Omega^S_0)$.  But we can add copies of the $K3$-surface, mapping to a point in $B$, to a given fixed element of $\Omega_{4}(B,\xi)$, keeping the normal 1-type and $c_*([M])$ the same, but changing the signature by $+16$ for each copy of $K3$.  This contradiction implies that $d_{5,0}^5$ is the zero map.
This completes the proof of the claim that the term $H_0(\op{B\pi};\Omega_4^S)$ survives to the $E^{\infty}$ page.


Since $H_0(\op{B\pi};\Omega_4^S)$ survives to the $E^{\infty}$ term, we have a short exact sequence:
\[0 \to \Omega_4^S \to \Omega_4(B,\xi) \to \wt{\Omega}_4(B,\xi) \to 0,\]
where $\wt{\Omega}_4(B,\xi)$ denotes the quotient. That is, there is a filtration with iterated graded quotients given by the $E^{\infty}$ page:
\[0 \subseteq E^{\infty}_{4,0} = \Omega^S_4 \subseteq \cdots \subseteq \Omega_4(B,\xi),\]
and it is the quotient by the $E^{\infty}_{4,0}$ subgroup that we denote $\wt{\Omega}_4(B,\xi)$.

Similarly, for the topological case, we have
\[0 \to \Omega_4^{ST} \to \Omega_4^{\TOP}(B,\xi) \to \wt{\Omega}_4^{\TOP}(B,\xi) \to 0. \]
The only difference in the proof from the smooth case is that
we also have to argue that the Kirby-Siebenmann invariant $\Z/2 \subseteq \Omega_4^{ST}$ survives to the $E^{\infty}$ page.  But the Kirby-Siebenmann invariant is additive, and realised on simply connected manifolds, by the $E_8$ manifold.  Thus there exist bordism classes  (i.e.\ stable homeomorphism classes) realising both trivial and nontrivial Kirby-Siebenmann invariants  within a normal 1-type, and so this $\Z/2$ cannot be killed by a differential.

Since the structure forgetting map $\Omega_q^S \to \Omega_q^{ST}$ is an isomorphism for $0 \leq q \leq 3$, we have an isomorphism
$\wt{\Omega}_4(B,\xi) \cong \wt{\Omega}_4^{\TOP}(B,\xi)$.   This uses that the differentials agree, by naturality of the James spectral sequence with respect to homology theories.  Indeed, note that the differentials depend only on the classifying space $\op{B\pi}$, and on the complex line bundle $E \to \op{B\pi}$ in case \eqref{item:almost-spin-case-smooth}.  Both are category independent.

Then there is a map of short exact sequences:
\[\xymatrix@R0.75cm{0 \ar[r] & \Omega_4^{S} \ar[r] \ar[d] & \Omega_4(B,\xi) \ar[r] \ar[d]& \wt{\Omega}_4(B,\xi) \ar[r] \ar[d]^-{\cong} & 0\phantom{.} \\
0 \ar[r] & \Omega_4^{ST} \ar[r] & \Omega_4^{\TOP}(B,\xi) \ar[r] & \wt{\Omega}_4^{\TOP}(B,\xi) \ar[r] & 0.}\]
The left vertical map is injective, either inclusion into the first summand $\Z \to \Z \oplus \Z/2$ for non-spin or $16\Z \to 8\Z$ in the spin case, when  $B = \op{B\pi} \times \BSpin$ and
$B^{\operatorname{TOP}} = \op{B\pi} \times \BTOPSpin$.
Since the left and right vertical maps are injective, it follows from a diagram chase that the central vertical map is also injective, as required.
\end{proof}

\section{Nonorientable 4-manifolds and stable diffeomorphism}\label{section:nonorientable}

For the convenience of the reader, we recall the statement of Theorem~\ref{thm:stable-diff-nonorientable}.

\begin{theorem*}\textbf{\textup{\ref{thm:stable-diff-nonorientable}.}}
  Let $M$ and $N$ be closed, connected, nonorientable, smooth 4-manifolds. Suppose that $M$ and $N$ are homeomorphic.  If $w_2(\wt{M}) \neq 0 \neq w_2(\wt{N})$, that is the universal covers of $M$ and $N$ are not spin, then $M$ and $N$ are stably diffeomorphic.
\end{theorem*}

Here is the normal 1-type for nonorientable manifolds with a certain $w_2$-type~\cite[Chapter~2]{teichnerthesis}, first in the smooth and then in the topological case.

\begin{lemma}\label{lemma:normal-1-types-nonorientable}
  Let $M$ be a nonorientable closed, connected smooth 4-manifold with $w_2(\wt{M}) \neq 0$. We set $\pi:=\pi_1(M)$. Then the normal 1-type of $M$ is $\xi \colon B= \op{B\pi} \times \BSO \to \BO$ with the map $\xi = w_1 \oplus \up{B}i$ given by the Whitney sum of a bundle on $\op{B\pi}$ determined by $w_1 \colon \pi \to \Z/2$
 and the canonical map $\up{B}i \colon \BSO \to \BO$ induced by the inclusion $i \colon \SO \to \O$.
\end{lemma}

\begin{lemma}\label{lemma:normal-1-types-nonorientable-topological}
  Let $M$ be a nonorientable closed, connected 4-manifold with $w_2(\wt{M}) \neq 0$. We set $\pi:=\pi_1(M)$. Then the normal 1-type of $M$ is $\xi \colon B= \op{B\pi} \times \BSTOP \to \BTOP$ with the map $\xi = w_1 \oplus \up{B}i$ given by the Whitney sum of a bundle on $\op{B\pi}$ determined by the orientation character $w_1 \colon \pi \to \Z/2$ and the canonical map $\up{B}i \colon \BSTOP \to \BTOP$ induced by the inclusion $i \colon \STOP \to \TOP$.
\end{lemma}

These normal 1-types give rise to a James spectral sequence (see \cite[Chapter~II]{Teichner93} for details) governing the bordism groups of $(B,\xi)$
\[E^2_{p,q} = H_p(\op{B\pi};\Omega_q^{w_1}) \,  \Rightarrow \, \Omega_{p+q}(B,\xi).\]
Note that the coefficients are twisted using $\Z^{w_1} \otimes \Omega_q$, where by definition, $g \in \pi$ acts on $\Z^{w_1}$ by multiplication by $(-1)^{w_1(g)}$.
The corresponding topological James spectral sequence is:
\[E^2_{p,q} = H_p(\op{B\pi};(\Omega_q^{\op{\STOP}})^{w_1}) \,  \Rightarrow \, \Omega_{p+q}^{\TOP}(B,\xi).\]
As in the previous section, here we abbreviate $\Omega_4^{\operatorname{TOP}}(B^{\operatorname{TOP}},\xi^{\operatorname{TOP}})$ to $\Omega_4^{\operatorname{TOP}}(B,\xi)$.

By Kreck's Theorem~\ref{thm:kreck}) and the argument in the proof of Theorem~\ref{thm:stably-diffeo-homeo-4-mflds}, in order to prove Theorem~\ref{thm:stable-diff-nonorientable} it suffices to prove the next injectivity statement, which is an analogue of
Proposition~\ref{prop:forgetful-injective}.

\begin{proposition}\label{proposition:injectivity-non-or-bordism}
  Let $(B,\xi)$ be one of the normal 1-types in Lemma~\ref{lemma:normal-1-types-nonorientable} and let $(B^{\TOP},\xi^{\TOP})$ be the corresponding topological normal 1-type over $\BTOP$.  The forgetful map
  \[F \colon \Omega_4(B,\xi) \to \Omega_4^{\TOP}(B^{\TOP},\xi^{\TOP}) = \Omega^{\TOP}(B,\xi)\]
  is injective.
\end{proposition}

\begin{proof}[Proof of Theorem~\ref{thm:stable-diff-nonorientable} assuming Proposition~\ref{proposition:injectivity-non-or-bordism}]
  Homeomorphic 4-manifolds have the same normal 1-types and are trivially $\TOP$ bordant over this normal 1-type, injectivity of $F$ implies that homeomorphic nonorientable, closed, connected, smooth 4-manifolds are smoothly bordant over their normal 1-type, and therefore by Theorem~\ref{thm:kreck} are stably diffeomorphic.
\end{proof}

\begin{proof}[Proof of Proposition~\ref{proposition:injectivity-non-or-bordism}]
The structure of the proof is very similar to that of the proof of Proposition~\ref{prop:forgetful-injective}. This proof is therefore somewhat terse.
In the smooth James spectral sequence computing $\Omega_4(B,\xi)$, we consider the term on the $E^2$ page $H_0(\op{B\pi};\Omega_4^{w_1}) \cong \Z/2$. This is detected by the Euler characteristic of the manifold modulo two.

Since we can always perform connected sum with a copy of $\mathbb{CP}^2$ (note that for nonorientable manifolds connected sum with $\cp^2$ and $\overline{\cp^2}$ is the same), both mod 2 Euler characteristics are realised by bordism classes over $(B,\xi)$. Also note that adding $\cp^2$ does not change the normal $1$-type when $w_2(\wt{M}) \neq 0$.  Therefore $H_0(\op{B\pi};\Omega_4^{w_1}) \cong \Z/2$ survives to the $E^{\infty}$ page.

In the topological case, the corresponding term in the James spectral sequence computing $\Omega_4^{\TOP}(B,\xi)$ is \[H_0(\op{B\pi};(\Omega_4^{\op{\STOP}})^{w_1}) \cong \Z/2 \oplus \Z/2.\]
We can add copies of $\cp^2$ and $*\cp^2$ to a given element of $\Omega_4^{\TOP}(B,\xi)$ to show that this term survives to the $E^{\infty}$ page.  The structure forgetting map $\Z/2 \cong H_0(\op{B\pi};\Omega_4^{w_1}) \to H_0(\op{B\pi};(\Omega_4^{\op{\STOP}})^{w_1}) \cong \Z/2 \oplus \Z/2$ is injective.

Therefore the filtrations of $\Omega_4(B,\xi)$ and $\Omega_4^{\TOP}(B,\xi)$ arising from the spectral sequence give rise to short exact sequences, that form the rows of the following commutative diagram:
\[\xymatrix{0 \ar[r] & \Z/2 \ar[r] \ar[d] & \Omega_4(B,\xi) \ar[r] \ar[d]& \wt{\Omega}_4(B,\xi) \ar[r] \ar[d]^-{\cong} & 0 \\
0 \ar[r] & \Z/2 \oplus \Z/2 \ar[r] & \Omega_4^{\TOP}(B,\xi) \ar[r] & \wt{\Omega}_4^{\TOP}(B,\xi) \ar[r] & 0.}\]
We noted above that the left vertical map is injective. Since $\Omega_q \to \Omega_q^{\op{\STOP}}$ is an isomorphism for $0 \leq q \leq 3$, the right vertical map is an isomorphism and is therefore injective.  It follows from a diagram chase that the central vertical map is also injective, as required.
\end{proof}

For the other normal 1-types of nonorientable 4-manifolds, the $E^2_{0,4}$ terms are given by the homology with twisted coefficients $H_0(\op{B\pi};(\Omega_4^{\op{\Spin}})^{w_1})$ and $H_0(\op{B\pi};(\Omega_4^{\op{\TOPSpin}})^{w_1})$. The forgetful map $H_0(\op{B\pi};(\Omega_4^{\op{\Spin}})^{w_1}) \to H_0(\op{B\pi};(\Omega_4^{\op{\TOPSpin}})^{w_1})$ is not injective, because $16\Z \cong \Omega_4^{\op{\Spin}} \to \Omega_4^{\op{\TOPSpin}} \cong 8\Z$ is not surjective.  Not only does the proof break down in these cases, but the examples in Theorem~\ref{thm:kreck-counterexamples} show that the hypothesis on the $w_2$-type cannot be removed.

\section{Homeomorphic 4-manifolds with boundary}\label{section:homeo-4-mflds-with-bdy-stable-homeo-vs-stable-diffeo}

In this section we briefly explain how to extend the proofs to manifolds with boundary.
Here is the relevant version of Kreck's theorem~\cite{surgeryandduality}.

\begin{theorem}\label{thm:kreck-manifolds-with-bdy}
    Let $M$ and $N$ be compact $($smooth$)$ 4-manifolds, let $f \colon \partial M \to \partial N$ be a diffeomorphism, and suppose that the normal 1-types of $M$ and $N$ are both fibre homotopy equivalent to the same fibration  $\xi^{\TOP} \colon B^{\TOP} \to \BTOP$ $(\xi \colon B \to \BO)$.

    Then $M$ and $N$ are stably homeomorphic $($diffeomorphic$)$ via a stable homeomorphism $($diffeomorphism$)$ extending $f$ if and only if $M$ and $N$ admit normal 1-smoothings $\ol{\nu}_M$ and $\ol{\nu}_N$ into $B^{\TOP}$ $(B)$ such that $\ol{\nu}_M = \ol{\nu}_N \circ f \colon \partial M \to B^{\TOP}$ $(B)$, for which the union $(M \cup_f N,\ol{\nu}_M \cup -\ol{\nu}_N)$ represents the trivial element of $\Omega_4(B^{\TOP},\xi^{\TOP})$ $(\Omega_4(B,\xi))$.
\end{theorem}

In Propositions~\ref{prop:forgetful-injective} and~\ref{proposition:injectivity-non-or-bordism} we proved injectivity of the forgetful maps in bordism $\Omega_4 (B,\xi) \to \Omega_4^{\operatorname{TOP}}(B^{\operatorname{TOP}},\xi^{\operatorname{TOP}})$ relevant to Theorems~\ref{thm:stably-diffeo-homeo-4-mflds} and \ref{thm:stable-diff-nonorientable}.

\begin{proof}[Proof of Theorems~\ref{thm:stably-diffeo-homeo-4-mflds} and \ref{thm:stable-diff-nonorientable}]
Let $M$ and $N$ be smooth, compact 4-manifolds with nonempty boundary. Let $(B^{\operatorname{TOP}},\xi^{\operatorname{TOP}})$ and $(B,\xi)$ be the relevant normal 1-types. Let $f \colon \partial M \to \partial N$ be a diffeomorphism that extends to a homeomorphism from $M$ to $N$. Then by Theorem~\ref{thm:kreck-manifolds-with-bdy} in the topological category, there exist normal 1-smoothings $\ol{\nu}_M^{\TOP} \colon M \to B^{\TOP}$ and $\ol{\nu}_N^{\TOP} \colon N \to B^{\TOP}$ such that $(M \cup_f N,\ol{\nu}_M^{\TOP} \cup -\ol{\nu}_N^{\TOP})$ represents the trivial element of $\Omega_4(B^{\TOP},\xi^{\TOP})$.
Since $M$ and $N$ are smooth, their stable normal vector bundle classifying maps  lift along $\BO \to \BTOP$, and we obtain normal 1-smoothings $\ol{\nu}_M \colon M \to B$ and $\ol{\nu}_N \colon N \to B$ such that $(M \cup_f N,\ol{\nu}_M \cup -\ol{\nu}_N) \in \Omega_4 (B,\xi)$ maps under the forgetful map to
$(M \cup_f N,\ol{\nu}_M^{\TOP} \cup -\ol{\nu}_N^{\TOP}) = 0 \in \Omega_4(B^{\TOP},\xi^{\TOP})$.
By injectivity of the forgetful maps
$\Omega_4 (B,\xi) \to \Omega_4^{\operatorname{TOP}}(B^{\operatorname{TOP}},\xi^{\operatorname{TOP}})$,
as proven in Propositions~\ref{prop:forgetful-injective} and~\ref{proposition:injectivity-non-or-bordism}, it follows that $(M \cup_f N,\ol{\nu}_M \cup -\ol{\nu}_N) =0 \in \Omega_4 (B,\xi)$. By Theorem~\ref{thm:kreck-manifolds-with-bdy}, we learn that indeed~$M$ and~$N$ are stably diffeomorphic via a stable diffeomorphism extending the given diffeomorphim $f \colon \partial M \to \partial N$.
\end{proof}

\chapter{Reidemeister torsion in the topological category}

\section{The simple homotopy type  of a manifold}\label{section:simple-homotopy-type}
In the following we need the notion of a \emph{simple} homotopy equivalence.
We will not give a definition, instead we refer to \cite[p.~40]{Turaev:2001-1} for details.
Roughly, a simple homotopy equivalence between CW complexes is a sequence of elementary expansions and collapses of pairs of cells whose dimension differs by one.

As we discussed in Chapter~\ref{chapter:CW structures}, it is not clear whether topological 4-manifolds have CW-structure. Fortunately the following definition allows us to define a simple homotopy type even for topological spaces which are not homeomorphic to a CW complex.

\begin{definition}\label{defn:SimpleHomotopyType}
	Let  $(W,V)$ be a pair of topological spaces. Consider tuples $(W,V,f, X,Y)$, where $(X,Y)$ is a finite CW complex pair with $Y \subseteq X$, and $f \colon W\to X$ and $f|_V \colon V \to Y$  homotopy equivalences. Two such tuples $(W,V,f,X,Y)$ and $(W,V,f',X',Y')$, with $(X',Y')$ another finite CW pair and $f'\colon (W,V) \to (X',Y')$, are \emph{equivalent} if there exists a simple homotopy equivalence of pairs~$s \colon (X,Y) \to (X',Y')$ such that $s \circ f$ is homotopic to $f'$ and $s|_Y \circ f|_V \colon V \to Y'$ is homotopic to $f'|_{V}$. Such an equivalence class of $(W,V,f,X,Y)$ is called a \emph{simple homotopy type} of $(W,V)$. In particular, a simple homotopy type of $(W,\emptyset)$ is called a \emph{simple homotopy type} of $W$.\label{def:simple-homotopy-type}
\end{definition}

Now consider a compact, connected $n$-manifold~$M$.
If $M$ admits a smooth structure, then
by Theorem~\ref{thm:n-not-equal-4-CW structure} we know that
$M$ admits in particular a CW structure, and we equip $M$ with the simple homotopy type given by $(M, \emptyset,\id,M,\emptyset)$.
By Chapman's Theorem~\cite[p.~488]{Chapman-74} below, this simple homotopy type is independent of the
choice of CW structure on~$M$.

\begin{theorem}\textbf{\textup{(Chapman's Theorem)}}
Let $W$ be a compact topological space. Any two CW structures on $W$ are simple homotopy equivalent.
\end{theorem}

As we pointed out in Chapter~\ref{chapter:CW structures}, it is unknown whether every compact manifold admits a CW structure. In the remainder of this section, we will nonetheless introduce the simple homotopy type of a compact manifold $M$ following \cite[Essay~III,~Section~4, p.~117]{KS77}. The first step is to construct a disc bundle~$D(M) \to M$ together with a $\PL$ structure on the total space~$D(M)$.
We will work with a compact $m$-dimensional manifold $M$ with boundary $\partial M$, and seek to construct the simple homotopy type of $(M,\partial M)$.

\begin{construction}\label{const:SimpleType}
	We deal with the case $\partial M = \emptyset$ first, and then later address the additional complications arising from having nonempty boundary.	

	As a first step to constructing the disc bundle $D(M) \subseteq \R^n$, we need an embedding of $M$ into $\R^{n-1}$ for some large integer~$n-1 > 2m+5$. For a closed $m$--manifold $M$ such an embedding is readily available \cite[Corollary A.9]{Hat02}.
	It follows from Theorem~\ref{thm:microexist} that for $n-1 > 2m+5$ all such embeddings of $M$ are isotopic, and that they admit a normal microbundle~$\nu_{\R^{n-1}}(M)$ that is unique up to isotopy.  By Theorem~\ref{thm:kistmaz} this normal microbundle~$\nu_{\R^{n-1}}(M)$ can be upgraded to a topological $\R^{n-1- m}$--bundle. By taking the product with $\R$, construct an embedding $M \subseteq \R^n$ whose normal microbundle is $\nu(M) = \nu_{\R^{n-1}}(M)\times \R$. Since we stabilised once, the normal microbundle $\nu(M)$ contains a normal disc bundle~$B(M)$ \cite[Essay~III,~Proposition~4.4, p.~120]{KS77}.
	
The next big step will be to upgrade $B(M) \subseteq \R^n$ from a submanifold to a $\PL$ submanifold. Since the interior is codimension~$0$, the interior of $B(M)$ is automatically also a $\PL$ submanifold. However, we have to arrange $\partial B(M)$ to be a $\PL$ submanifold of $\R^n$ itself. In the next paragraphs, we modify the $\PL$ structure on $\R^n$ such that $\partial B(M)$ becomes a $\PL$ submanifold and then isotope this new $\PL$ structure on $\R^n$ back to the standard $\PL$ structure.

Using the Collar Neighbourhood Theorem~\ref{thm:collar}, pick a collar $W_\partial = \partial B(M) \times (-1,1)$ and $D_\partial = \partial B(M) \times [-\frac{1}{2},\frac{1}{2}]$. The Local Product Structure Theorem~\ref{theorem:local-product-structure}~\cite[Essay~I,~Theorem~5.2, p.~36]{KS77}, applied with $W=W_\partial$ and $D=D_\partial$, gives a $\PL$ structure $\sigma_\partial$ on $\R^n$ such that $\partial B(M)$ is a $\PL$ submanifold and $\sigma_\partial$ is concordant to the standard $\PL$ structure $\sigma_\text{std}$.

\begin{figure}[h]
\begin{center}
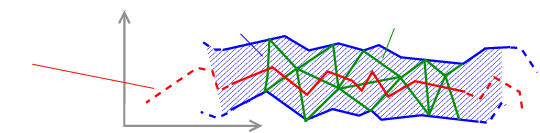
\caption{Illustration of $B(M)$.}\label{fig:simple-homotopy-type-closed-mfd}
\end{center}
\end{figure}

Now we will isotope the pair $\partial B(M)\subseteq B(M)$ so that they become $\PL$ submanifolds of $(\R^n, \sigma_\text{std})$. The $\PL$ structure~$\sigma_\partial$ is concordant to $\sigma_\text{std}$. Since concordance implies isotopy \cite[Essay~I,~Theorem~4.1, p.~25]{KS77} in dimension $m\geq 6$, there is an isotopy $\phi_t \in \op{Homeo}(\R^n)$ such that $\phi_0 = \id$, and $\phi_1^*\sigma_\text{std} = \sigma_\partial$. Consequently, $D(M) := \phi_1(B(M))$ and $D(\partial M) := \phi_1(B(\partial M))$ are $\PL$ submanifolds of $(\R^n, \sigma_\text{std})$, which defines a simple type of $M$, the tuple $(M, z, D(M))$, where $z \colon M \to D(M)$ is the zero section.

Having finished the case $\partial M \neq \emptyset$, next we discuss the procedure for a manifold $M$ with nonempty boundary.

Take the union of $M$ with an external open collar $\partial M \times [0,1)$ of its boundary. Write $M' := M \cup_{\partial M} \partial M \times [0,1)$.
Embed~$M'$ into $\R^n$ as in the closed case \cite[Corollary A.9]{Hat02}.
Note that $M'$ has empty boundary and so  it is properly embedded. As in the closed case, obtain a disc bundle $B(M')$ and let $B(M)$ be the restriction of this disc bundle to $M$.

Now we have to take much more care. Note that $\partial B(M)$ decomposes as $\partial B(M) = B(\partial M) \cup_X B_\partial (M)$. Here $B_\partial(M)$ denotes the fibrewise boundary and $X = \partial \big( B(\partial M) \big)$ denotes the intersection of $B(\partial M)$ and $B_\partial(M)$. As above, we will find a $\PL$ structure $\sigma_\partial$ of $\R^n$ such that each subset $B(\partial M)$, $X$  and $B_\partial (M)$ is $\PL$--submanifold of $\R^n$.

Our first goal is to modify the $\PL$ structure on $\R^n$ so that the corners~$X$ become a $\PL$ submanifold of $\R^n$.
Denote the standard $\PL$ structure on $\R^n$ by $\sigma_\text{std}$.
Pick a bicollar $\partial B(M) \times [-1,1] \subseteq \R^n$ of the boundary of the codimension~$0$ submanifold~$B(M)$.
Again by the Collar Neighbourhood Theorem~\ref{thm:collar}, we can pick a bicollar $X \times [-1,1] \subseteq \partial B(M)$. We consider the open set $W_X := X \times (-1,1)^2 \subseteq \partial B(M) \times (-1,1)\subseteq \R^n$, and $D_X = X \times [-\frac{1}{2},\frac{1}{2}]^2$.
The Local Product Structure Theorem~\ref{theorem:local-product-structure}~\cite[Essay~I,~Theorem~5.2, p.~36]{KS77}, applied with $W=W_X$ and $D=D_X$, gives a $\PL$ structure on $X$ and a $\PL$ structure~$\sigma_X$ on $\R^n$, which is concordant to $\sigma_\text{std}$ rel.\ $\R^n \sm \big( X \times (-\frac{2}{3}, \frac{2}{3})^2 \big)$. This $\PL$ structure~$\sigma_X$ has the property that it agrees with the product $\PL$ structure on $X \times (-1,1)^2$ in a neighbourhood of $D_X$. Thus $X$ is a $\PL$ submanifold of $(\R^n, \sigma_X)$.

\begin{figure}[h]
\begin{center}
\includegraphics{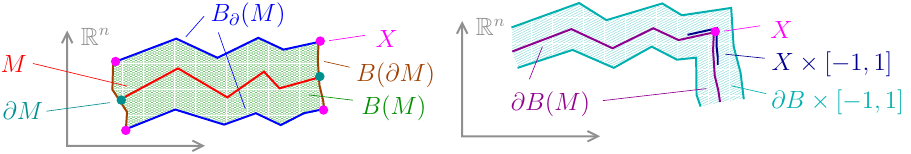}
\caption{Illustration of construction of $D(M)$ if $\partial M\ne \emptyset$.}\label{fig:simple-homotopy-type-mfd-with-boundary}
\end{center}
\end{figure}

Now we arrange the next stratum~$\partial B(M) \supseteq X$ to be a $\PL$ submanifold of $\R^n$. Near $D_X = X \times [-\frac{1}{2},\frac{1}{2}]$, the $\PL$ structure~$\sigma_X$ is the product $\PL$ structure, and therefore $\partial B(M) \cap \Int D_X = X \times (-\frac{1}{2},\frac{1}{2}) \times \{0\}$ is already a $\PL$ submanifold of $(\R^n,\sigma_X)$. Furthermore, $\sigma_X$ is a product along $(-\frac{1}{2},\frac{1}{2})$ near $X \times [-\frac{1}{3},\frac{1}{3}]$.
Pick $W_\partial = \partial B(M) \times (-1,1)$, $C_\partial = X \times [-\frac{1}{3},\frac{1}{3}] \times [-\frac{1}{3},\frac{1}{3}]$ and $D_\partial = \partial B(M) \times [-\frac{1}{2},\frac{1}{2}]$. As above, the Local Product Structure Theorem~\ref{theorem:local-product-structure}~\cite[Essay~I,~Theorem~5.2, p.~36]{KS77}, applied with $W=W_\partial$, $C=C_{\partial}$, and $D=D_\partial$, gives a $\PL$ structure $\sigma_\partial$ on $\R^n$ such that $\partial B(M)$ is a $\PL$ submanifold and $\sigma_\partial$ is concordant to $\sigma_X$ rel.\ $\big( \R^n \sm \partial B(M) \times (-\frac{2}{3},\frac{2}{3}) \big) \cup C_\partial$. Since $X \subseteq C_\partial$ the submanifold~$X$ is still a $\PL$ submanifold of $(\R^n, \sigma_\partial)$.

As in the closed case, use a concordance from $\sigma_\partial$ to $\sigma_\text{std}$ to obtain an isotopy $\phi_t \in \op{Homeo}(\R^n)$ such that $\phi_0 = \id$, and $\phi_1^*\sigma_\text{std} = \sigma_\partial$. Define $D(M) := \phi_1(B(M))$ and $D(\partial M) := \phi_1(B(\partial M))$, which are both $\PL$ submanifolds of $(\R^n, \sigma_\text{std})$.
We obtain a simple homotopy type $(M,\partial M,z,D(M),D(\partial M)$, where again $z$ is the zero section of the disc bundle.
This finishes the case where $M$ has nonempty boundary.

In both cases, $\partial M$ empty and nonempty, our construction involved many choices.
Let $D'(\partial M) \subseteq D'(M)$ be obtained by other choices. Following the discussion \cite[p.~123]{KS77}, we can suitably stabilise the bundles and find a commutative diagram of $\PL$ maps:
	\[ \begin{tikzcd}
			D(M) \times D^s \ar[r, "\cong"] & D'(M) \times D^r \\
			D(\partial M) \times D^s \ar[r, "\cong"] \ar[u] & D'(\partial M) \times D^r \ar[u],
	\end{tikzcd} \]
	where $D^k$ denotes the disc with its standard $\PL$ structure and the horizontal maps are $\PL$ isomorphisms that preserve the zero sections up to homotopy.
\end{construction}

\begin{definition}\label{defn:SimpleHomotopyManifold}
	The \emph{simple homotopy type} of a compact connected $n$-manifold $M$ is given by $(M,s)$, where $s \colon M \to D(M)$ is the inclusion of the $0$-section. The simple homotopy type of the pair~$(M, \partial M)$ is given by the square
	\[ \begin{tikzcd} M \ar[r,"s"] & D(M) \\ \partial M \ar[r, "s|_{\partial M}"] \ar[u] & D(\partial M), \ar[u]
	\end{tikzcd} \]
	where $D(\partial M) \subseteq D(M)$ are the disc bundles from Construction~\ref{const:SimpleType}, with CW structures arising from a choice of $\PL$ triangulations corresponding to the $\PL$ structures.
\end{definition}

By the commutative square at the end of Construction~\ref{const:SimpleType}, the simple homotopy type of $(M,\partial M)$ is well-defined.
Here we use that $\PL$ isomorphisms are simple: for any choice of triangulations underpinning the $\PL$ structures, the resulting homeomorphism is a simple homotopy equivalence. Also stabilising by $D^s$ does not change the simple homotopy type, since as $\PL$ manifolds $D^s \cong D^{s-1} \times [-1,1]$, and $D^{s-1} \times \{0\} \to D^{s-1} \times [-1,1]$ is a simple equivalence.

\begin{remark}
  Why is the simple homotopy type of $\partial M$ obtained in this way the same as that obtained by applying Construction~\ref{const:SimpleType} with $\partial M$ considered as a manifold without boundary?

  For suitably high $n$, we may assume  that the embedding of $(M,\partial M)$ into $\R^n$ is isotopic, and thus by Theorem~\ref{thm:isotopy-extension-theorem} ambiently isotopic, to an embedding with $i \colon \partial M \hookrightarrow \{\vec{x} \in \R^n \mid x_1=0\}\cong \R^{n-1}$ and an (interior) collar $\partial M \times [0,1]$ embedded as a product in $\{\vec{x} \in \R^n \mid 0 \leq x_1 \leq 1\} $ with $(x,t) \mapsto (i(x),t)$, as in Theorem~\ref{thm:collar-boundary}. Such an isotopy does not affect the simple homotopy type obtained, by the argument sketched above, which can also be found on \cite[p.~123]{KS77}.   The simple homotopy type of $\partial M$ obtained from Construction~\ref{const:SimpleType}, via an embedding of $\partial M$ into $\R^{n-1}$, uses a disc bundle $D(\partial M)$ that stabilises using the $x_1$ direction to a disc bundle $D'(\partial M)$, with fibre a disc of one dimension higher, for $\partial M$ embedded in $\R^n$. This latter disc bundle gives rise to the canonical simple homotopy type of $\partial M$ from Definition~\ref{defn:SimpleHomotopyManifold}.
\end{remark}

\begin{remark}
	If $M$ is a smooth manifold, then $M$ has an underlying $\PL$ structure, and with a bit more care in Construction~\ref{const:SimpleType}, we can arrange that the bundle $D(M)$ is a $\PL$ bundle. Note that this is stronger than just a $\PL$ structure on the total space.
	For $\PL$ bundles, the bundle projection $D(M) \rightarrow M$ is a simple homotopy equivalence. Indeed, for trivial bundles this is discussed above, and in general the projection is an $\alpha$-equivalence (a notion defined in \cite{Ferry77}) for any cover~$\alpha$ of $M$ and so is simple \cite[Corollary~3.2]{Ferry77}.
	It follows that the simple homotopy type defined by $(M, \id)$ agrees with the one of $(M,s)$, and the same holds for the relative simple homotopy type of the pair $(M, \partial M)$.

According to \cite[Essay~III,~Theorem~5.11, p.~123]{KS77}, if a manifold has a triangulation, then the simple homotopy type of the manifold agrees with the simple homotopy type of that triangulation.  It is not clear to us whether the analogous statement holds if $M$ has a CW structure not coming from a triangulation.
\end{remark}

\section{The cellular chain complex and Poincar\'{e} triads}\label{section:based-chain-complex}
Throughout this section let $M$ be a compact connected $n$-manifold.
Furthermore assume that we are given a decomposition $\partial M=R_-\cup R_+$ into codimension zero submanifolds such that $\partial R_-=R_-\cap R_+=\partial R_+$.

The following proposition follows from the argument of Construction~\ref{const:SimpleType}, applied with even more iterations to deal with corners of corners. See also the proof of \cite[Essay~III,~Theorem~5.13, p.~136]{KS77}.

\begin{proposition}\label{prop:CW triads}
There exists a finite CW complex triad $(X,X_-,X_+)$
and a  homotopy equivalence of triads $f\colon (M,R_-,R_+)\to (X,X_-,X_+)$ such that
the following two statements hold:
\bnm
\item The restrictions of $f$ to $M$, $R_\pm$ and $R_-\cap R_+$ give the simple homotopy types of these manifolds, as defined in Defintion~\ref{def:simple-homotopy-type}.
\item The restrictions of $f$ to the pairs $(M,\partial M)$, $(\partial M,R_\pm)$ and $(R_\pm,R_-\cap R_+)$ give the simple homotopy types of these pairs of manifolds, as defined in the previous section.
\enm
\end{proposition}

We continue with a general definition regarding CW complexes.

\begin{definition}
Let $(X,Y)$ be a pair of CW complexes such that $X$ is connected. We write $\pi=\pi_1(X)$
and we denote  the universal covering by $p\colon \wti{X}\to X$.
The group $\pi$ acts on the left on the cells of the CW complex $(\wti{X},p^{-1}(Y))$.
This equips $C_*^{\op{cell}}(\wti{X},p^{-1}(Y))$ with the structure of a left $\Z[\pi]$-module.
We define
\[\ba{rcl} C_*^{\op{cell}}(X,Y;\Z[\pi])&:=&
\Z[\pi]\otimes_{\Z[\pi]} C_*^{\op{cell}}(\wti{X},p^{-1}(Y))\\
C^*_{\op{cell}}(X,Y;\Z[\pi])&:=&\hom_{\opnormal{right-}\Z[\pi]}(\ol{C_*^{\op{cell}}(\wti{X},p^{-1}(Y))},\Z[\pi]).\ea\]
Here, given a left $\Z[\pi]$-module $M$ we denote  the right $\Z[\pi]$-module given by $m\cdot g:=g^{-1}\cdot m$ by $\ol{M}$.
Note that the  group $\pi$ acts freely on the left on the cells of the CW complex $(\wti{X},p^{-1}(Y))$.
For each cell in $X\sms Y$,  pick a lift to $\wti{X}$. This
turns
$C_*^{\op{cell}}(X,Y;\Z[\pi])$ and $C^*_{\op{cell}}(X,Y;\Z[\pi])$
into based left $\Z[\pi]$-module (co-) chain complexes.
\end{definition}

Now we can state the main theorem of this section.

\begin{theorem}\label{thm:simple-homotopy-equivalence-chain-complexes}
The finite CW complex triad $(X,X_-,X_+)$ from
Proposition~\ref{prop:CW triads}
is a simple Poincar\'e triad, meaning that
there is a chain level representative $\sigma\in  C_{n}^{\op{cell}}(X,X_-\cup X_+)$ of the fundamental class $[X] \in H_n(X,X_+ \cup X_-;\Z)\cong H_n(M,\partial M;\Z)$ such that
\[- \acap \sigma \colon C^{n-r}_{\op{cell}}(X,X_-;\Z[\pi_1(X)]) \,\,\to\,\, C_{r}^{\op{cell}}(X,X_+;\Z[\pi_1(X)])\]
is a simple chain homotopy equivalence.
\end{theorem}

The theorem is proved in \cite[Essay~III,~Theorem~5.13,~p.~136]{KS77}.
In the Universal Poincar\'e Duality Theorem~\ref{thm:poincareduality-chain-complex} we will prove that there exists a chain homotopy equivalence between the two chain complexes. But we will not prove that there
exists a \emph{simple} homotopy equivalence; for that the reader will need to consult \cite{KS77}.

\section{Reidemeister torsion}\label{section:reidemeister-torsion}
In this section we introduce  Reidemeister torsion invariants  for compact manifolds
and  discuss some of the key properties of these invariants.

Let $M$ be a compact connected $n$-manifold and write $\pi=\pi_1(M)$.
Let $R_-$ be a compact codimension 0 submanifold of $\partial M$. In many applications $R_-=\emptyset$ or $R_-=\partial M$. We write $R_+=\ol{\partial M\sms R_-}$.
Let $F$ be a field and let $\alpha \colon \pi \to \GL(d,F)$ be a  representation of the fundamental group of $M$.  With respect to this representation, we consider the twisted homology $H_k(M,R_-;F^d)$, as defined in
Section~\ref{section:twisted-invariants}.

\begin{assumption}\label{assumption}
  Suppose that $H_k(M,R_-;F^d)=0$ for all $k$.
\end{assumption}

Pick a homotopy equivalence of triads $f\colon (M,R_-,R_+)\to (X,X_-,X_+)$
as in Proposition~\ref{prop:CW triads}. We use the homotopy equivalence $f$ to make the identification $\pi_1(X)=\pi$.
By a serious abuse of notation,
we refer to the cellular chain complex $C_*^{\op{cell}}(X,X_-; \Z[\pi])$ of $(X,X_-)$ as the \emph{cellular chain complex $C_*^{\op{cell}}(M,R_-; \Z[\pi])$ of $(M,R_-)$}.
As in Section~\ref{section:based-chain-complex} we view $C_*^{\op{cell}}(M,R_-; \Z[\pi])$ as a based left $\Z[\pi]$-module chain complex.
Equip the $F$-module chain complex $C_*^{\op{cell}}(M,R_-; F^d) = F^d \otimes_{\Z[\pi]} C_*^{\op{cell}}(M,R_-; \Z[\pi])$ with the basing given by the tensor products of the $\Z[\pi]$-bases of $C_*^{\op{cell}}(M,R_-; \Z[\pi])$ and the canonical $F$-basis for $F^d$.

We write $\sim_\alpha$ for the equivalence relation on $F^{\times}:=F\sms \{0\}$ that is
given by the subgroup $\{ \pm \det(\alpha (g)) \mid  g \in \pi_1(M)\}\subseteq F^\times$.
We define $\tau(M,R_-,\alpha)\in F^{\times}/\sim_\alpha$ to be the Reidemeister torsion of the above  acyclic, based $F$-module chain complex. We refer to  \cite[Section~6]{Turaev:2001-1} for the definition of the Reidemeister torsion of an acyclic, based $F$-module chain complex.
It follows from a slight generalisation of~\cite[Theorem~9.1]{Turaev:2001-1} that
$\tau(M,R_-,\alpha)\in F^{\times}/\sim_\alpha$ is well-defined,  in that it is independent of the choice of the representative of the simple homotopy type of $(X,X_-,X_+)$ and it is independent of the choice of the lifts of the cells.  If $R_-=\emptyset$ then we write $\tau(M,\alpha):=\tau(M,\emptyset,\alpha)$.

The following two theorems give the two arguably most important properties of Reidemeister torsion.

\begin{theorem}\label{thm:torsion-multiplicative}
Let $M$ be a compact connected $n$-manifold,  let $R_-$ be a compact codimension zero submanifold of $\partial M$ and let $\alpha \colon \pi_1(M) \to \GL(d,F)$ be a representation.
Let $R_-^1,\dots,R_-^m$ be the components of $R_-$.
By abuse of notation we also write $\alpha$ for the composition $\alpha \colon \pi_1(R_-^i) \to \pi_1(M) \to \GL(d,F)$ defined  using a path from the base point of $M$ to  a base point of  $R_i^i$. If $H_*(R_-^1;F^d) =\cdots = H_*(R_-^m;F^d) = H_*(M;F^d)=0$,  then
\[\tau(M,\alpha)\,\, =\,\, \prod\limits_{i=1}^{m}\, \tau(R_-^i,\alpha) \cdot \tau(M,R_-,\alpha)\,\,\in\, F^\times/\sim_\alpha.\]
\end{theorem}

\begin{proof}
 We have the following short exact sequence of chain complexes with compatible bases:
  \[0\,\,\to\,\, C_*^{\op{cell}}(X_-;F^d) \,\,\to\,\, C_*^{\op{cell}}(X;F^d) \,\,\to\,\, C_*^{\op{cell}}(X,X_-;F^d) \,\,\to\,\, 0.\]
Given such a short exact sequence, the multiplicativity of the torsion is proven in \cite[Theorem~3.4]{Turaev:2001-1}.
\end{proof}

\begin{definition}
Let $F$ be a field with (possibly trivial) involution.
Given a representation $\alpha\colon \pi \to \GL(d,F)$ we denote the representation
$g\mapsto \ol{\alpha(g^{-1})}^T$ by $\alpha^\dagger$. We say that $\alpha$ is \emph{unitary} if $\alpha=\alpha^\dagger$.
\end{definition}

\begin{example}
Let $\phi\colon \pi\to \Z$ be a group homomorphism. Equip $\Q(t)$ with the usual involution given by $\ol{t}=t^{-1}$. The representation
$\alpha\colon \pi\to \GL(1,\Q(t))$ given by $g\mapsto t^{\phi(g)}$ is unitary.
\end{example}

\begin{theorem}\label{thm:torsion-PD}
Let $M$ be a compact $n$-manifold with $($possibly empty$)$ boundary. Assume that we are given a decomposition $\partial M=R_-\cup R_+$ into codimension zero submanifolds such that $\partial R_-=R_-\cap R_+=\partial R_+$. Furthermore let $F$ be a field with $($possibly trivial$)$ involution.
Let $\alpha \colon \pi_1(M) \to \GL(d,F)$ be a representation such that $H_*(\partial M;F^d) =0 = H_*(M;F^d)$. Then
\[\tau(M,R_-,\alpha)\,\, = \,\,\overline{\tau(M,R_+,\alpha^\dagger)}^{(-1)^{n+1}} \in F^\times/\sim_\alpha.\]
In particular, if $\alpha$ is unitary we have
\[\tau(M,R_-,\alpha)\,\, = \,\,\overline{\tau(M,R_+,\alpha)}^{(-1)^{n+1}} \in F^\times/\sim_\alpha.\hspace{0.2cm}\]
\end{theorem}

\begin{proof}
We write $\pi=\pi_1(M)$.
Write $C_*^\pm=C_*^{\op{cell}}(M,R_\pm; \Z[\pi])$, recalling the convention described below Assumption~\ref{assumption}.

It follows from Theorem~\ref{thm:simple-homotopy-equivalence-chain-complexes} that the torsion of the  based $F$-module chain complex
$F^d\otimes_{\Z[\pi]} C_*^{-}$ agrees with the torsion
of the based $F$-module chain complex
\[F^d\otimes_{\Z[\pi]}\hom_{\opnormal{right-}\Z[\pi]}(\ol{C_{n-*}^+},\Z[\pi]).\]
Consider the following isomorphism of based left $F$-module chain complexes
\[ \ba{rcl}F^d_{\alpha}\otimes_{\Z[\pi]}\hom_{\opnormal{right-}\Z[\pi]}(\ol{C_{n-*}^+},\Z[\pi])
&\to&\ol{\hom_{\opnormal{left-} F}(F^d_{\alpha^\dagger}\otimes_{\Z[\pi]} C_{n-*}^+,F)}\\
v\otimes \varphi&\mapsto & \left( \ba{rcl} F^d_{\alpha^\dagger}\otimes_{\Z[\pi]} \ol{C^+_{n-*}}&\to & F\\
(w\otimes \sigma)&\mapsto & \ol{v\alpha(\varphi(\sigma))\ol{w}^T}\ea\right)\ea\]
Using this isomorphism $\tau(M,R_-,\alpha)$ also equals the torsion of the chain complex on the right hand side.
It follows from algebraic duality for torsions~\cite[Theorem~1.9]{Turaev:2001-1} that the torsion of the based chain complex on the right hand side equals
$\overline{\tau(M,R_+,\alpha^\dagger)}^{(-1)^{n+1}}$.
\end{proof}

\chapter{Obstructions to being topologically slice}

\section{The Fox-Milnor Theorem}
In this section we provide an example of the use of many of the theorems described  in the previous chapters by applying them to obtain an obstruction for a knot to be topologically slice.

\begin{definition}
Let $Y$ be a  homology $3$-sphere that is the boundary of  an integral homology $4$-ball~$X$.
\bnm
\item We say a knot $K$ in  $Y$ is \emph{topologically slice in $X$} if  $K$ bounds a \emph{slice disc}, that is a proper submanifold of $X$ homeomorphic to a disc.
\item Suppose $X$ is equipped with a smooth structure, e.g.\ $X=D^4$.
We say a knot $K$ in  $Y$ is \emph{smoothly slice in $X$} if  $K$ bounds a \emph{smooth slice disc}, that is a proper smooth submanifold of $X$ diffeomorphic to a disc.
\enm
\end{definition}

There are many classical obstructions to a knot in $S^3$ being smoothly slice in $D^4$. For example, there are obstructions based on the Alexander polynomial \cite{Fox-Milnor:1966-1} and the Levine-Tristram signatures~\cite{Tristram:1969-1,Le69} and there are the more subtle Casson-Gordon~\cite{Casson-Gordon:1978-1,Casson-Gordon:1986-1} obstructions. It is not hard to see that these results also apply if we replace $S^3$ by any integral homology 3-sphere and if we replace $D^4$ by any smooth homology 4-ball.

Even though these results, having appeared prior to the work of Freedman and Quinn, were formulated as obstructions to being smoothly slice, it has been understood for many years that the original proofs can be modified to prove that these are in fact obstructions to being topologically slice.

In this section we will prove a sample theorem on the Alexander polynomial $\Delta_K(t)$ of a knot $K$. (On page~\pageref{def:alexander-polynomial} we recall the definition of the Alexander polynomial of a knot.)
The following theorem, which in the smooth setting was first proved by Fox-Milnor~\cite{Fox-Milnor:1966-1},
is arguably the most basic obstruction to a knot being  topologically slice knot.

\begin{theorem}\textbf{\textup{(Fox-Milnor)}}\label{theorem:fox-milnor}
Suppose that $K$ is a knot in a homology $3$-sphere $Y$ that bounds an integral homology $4$-ball~$X$. If $K$ is topologically slice in $X$, then the Alexander polynomial $\Delta_K(t)$ of $K$ factors as $\Delta_K(t) = \pm t^k\cdot f(t)\cdot  f(t^{-1})$ for some $k \in \Z$ and for some $f(t) \in \Z[t,t^{-1}]$ such that $f(1) = \pm 1$.
\end{theorem}

Even though this result is very well known we want to provide a detailed proof. In particular we want to highlight where some of the results discussed in this book are used.
The reader is encouraged to go through the above  papers
\cite{Tristram:1969-1,Le69,Casson-Gordon:1978-1,Casson-Gordon:1986-1}
and to modify the proofs to deal with topologically slice knots.

\section{A proof of the Fox-Milnor Theorem}
For the proof of the Fox-Milnor Theorem~\ref{theorem:fox-milnor} we adopt the following notation.
\bnm
\item  Let $Y$ be a  homology $3$-sphere   bounding some integral homology $4$-ball~$X$.
\item Given a knot $K$ in  $Y$, denote its zero framed surgery by $N_K$.
\item Given an oriented knot $K$ let $\mu_K$ be an oriented meridian.
\item For a slice disc $D$ in $X$, let $N(D)$ be a tubular neighbourhood provided by Theorem~\ref{thm:tubular-neighbourhood}.
We refer to $W_D=\ol{X\sms N(D)}$ as the \emph{exterior of $D$}.
\item The ring of integral Laurent polynomials in one variable is denoted $\Z[t, t^{-1}]$ or $\Z[t^{\pm 1}]$.
\enm
Many topological slicing obstructions, such as knot signatures \cite{Tristram:1969-1}, the Fox-Milnor condition \cite{Fox-Milnor:1966-1}, the Blanchfield form~\cite{Ke75}, Casson-Gordon invariants~\cite{Casson-Gordon:1978-1,Casson-Gordon:1986-1}, $L^2$-signature defects~\cite{COT03}  and $L^{(2)}$-von Neumann $\rho$-invariants~\cite{CT07}, rely implicitly and explicitly  on the next three propositions  or slight variations thereof.

\begin{proposition}\label{prop:exterior-slice-disc-h1}
Let $K$ be an oriented knot in $Y$ and let $D$ be a slice disc in $X$.
\bnm
\item We have $\partial W_D=N_K$.
\item The inclusion map $\mu_K\to W_D$ induces a $\Z$-homology equivalence.
\enm
\end{proposition}

In the remainder of this section, given an oriented knot, we use $\phi\colon \pi_1(N_K)\to \langle t\rangle$ and  $\phi\colon \pi_1(W_D)\to \langle t\rangle$ to denote the unique homomorphisms that send the oriented meridian to~$t$. These homomorphisms allow us to view $\Z[t^{\pm 1}]$ and $\Q(t)$ as a $\Z[\pi_1(N_K)]$-module and a $\Z[\pi_1(W_D)]$-module.

\begin{proof}
First note that it follows from Proposition~\ref{prop:trivial-normal-bundle} (or more directly, the fact that $D$ is contractible) that the tubular neighbourhood of $D$ is trivial, thus we can identify it with $D\times D^2$.
\bnm
\item We have to check that the framing of $K$ induced by the unique trivialisation $N(D)\cong D^2 \times D$ is the $0$--framing. Consider the double $DX=X\cup_YX$ .
Note that $Y$ splits  $DX$ into two copies of $X$. Let $D$ be contained in  one copy of $X$, and push a Seifert surface~$\Sigma$ into the other copy of $X$. By picking a collar neighbourhood for $Y=\partial X$ we obtain a bicollar neighbourhood $Y \times [-1,1]\subseteq DX$. Using Theorem~\ref{thm:collar-boundary} we can arrange that $D \cap (Y\times [-1,1]) = K \times [-1,0]$, and $\Sigma \cap (Y\times [0,1])= K \times [0,1]$. Let $F = \Sigma \cup -D \subseteq DX$. We compute the Euler number $e(F) \in \Z$ in two ways. First, note that $e(F) = [F] \cdot [F] = 0$, since $H_2(DX; \Z) = 0$.
	On the other hand, the number $e(F)$ is also the difference between the induced framings of $N(\Sigma)|_K$ and $N(D) |_K$. Consequently, the two framings agree and $N(D)$ induces the $0$--framing, which by definition is the framing induced by $N( \Sigma)|_K$.
\item
Let $\mu_K \to W_D$ be the inclusion of the meridian $\mu_K$ of $K$.
Then we have $H_*(W_D,*\times \mu_K;\Z)=H_*(W_D,D\times S^1;\Z)=H_*(X,D\times D^2)=0$  by excision and the hypothesis that $X$ be a homology $4$-ball. By the homology long exact sequence for the pair $(W_D,\mu_K)$, the meridional map $\mu_K\to W_D$ induces a homology equivalence, so $W_D$ is a homology circle.\qedhere
\enm
\end{proof}

\begin{proposition}\label{prop:exterior-slice-disc-mfd}
The exterior $W_D$ of a slice disc $D$ is homotopy equivalent to a finite $3$-dimensional CW complex. In particular the homology groups \[ H_*(W_D;\Z[t^{\pm 1}]),\,\, H_*(W_D,N_K;\Z[t^{\pm 1}]) \text{ and } H_*(N_K;\Z[t^{\pm 1}])\] are all finitely generated.
\end{proposition}

\begin{proof}
Note  that  $W_D$ is a compact 4-manifold with nonempty boundary.
It follows from Theorem~\ref{thm:topological-manifold-CW complex} that $W_D$ is homotopy equivalent to a 3-dimensional CW complex. The statements regarding the homology groups follow from
Proposition~\ref{prop:twisted-homology-fg}.
\end{proof}

\begin{proposition}\label{prop:alexander-module-w-torsion}\mbox{}
\bnm
\item For any knot $K$ in a homology 3-sphere the modules $H_*(N_K;\Z[t^{\pm 1}])$  are $\Z[t^{\pm 1}]$-torsion.
\item If  $D$ is a slice disc, then all the modules $H_*(W_D;\Z[t^{\pm 1}])$ are $\Z[t^{\pm 1}]$-torsion.
\enm
\end{proposition}

\begin{proof}
We start out with the proof of the second statement.
Let $P \subseteq \Z[t^{\pm 1}]$ be the multiplicative subset of Laurent polynomials that augment to $\pm 1$, that is $p(1) = \pm 1$ if and only if $p \in P$.
We shall prove the slightly stronger statement, that  $H_k(W_D;\Z[t^{\pm 1}])$ is $P$-torsion for $k>0$.  Since $H_0(W_D;\Z[t^{\pm 1}]) \cong \Z[t^{\pm 1}]/(t-1)$ is $\Z[t^{\pm 1}]$-torsion, the result will follow.
We write $\pi=\pi_1(W_D)$.
Let \[Q:= P^{-1}\Z[t^{\pm 1}]\] be the result of inverting the polynomials in $P$.
By Proposition~\ref{prop:finite-chain-complex} there exists a chain complex $C_*$ of finite length consisting of finitely generated free left $\Z[\pi]$-modules such that for any  ring $R$ and any $(R,\Z[\pi])$-bimodule $A$ we have
\[ H_k(W_D,\mu_K;A)\,\,\cong\,\,H_k(A\otimes_{\Z[\pi]} C_*).\]
By Proposition~\ref{prop:exterior-slice-disc-h1} we know that
$H_k(\Z\otimes_{\Z[\pi]} C_*)=H_k(W_D,\mu_K;\Z)=0$.
Since $C_*$ is  a chain complex of finite length consisting of finitely generated free left $\Z[\pi]$-modules
we obtain from  chain homotopy lifting \cite[Proposition~2.10]{COT03}, see also \cite[Lemma~3.1]{Nagel-Powell-17}, that
$H_k(Q\otimes_{\Z[\pi]} C_*)=0$. A straightforward calculation shows that
 $H_*(S^1,\pt;Q)=0$. It follows that $H_*(W_D,\pt;Q) =0$, so that $H_k(W_D;\Z[t^{\pm 1}])$ is $P$-torsion for $k >0$.

The first statement is very well known. One of the many proofs would be to use the above argument and the fact that $S^1\to Y\sms \nu K$ is a homology equivalence to show that the modules  $H_*(Y\sms \nu K;\Z[t^{\pm 1}])$ are torsion.
A basic Mayer-Vietoris argument then shows that the modules $H_*(N_K;\Z[t^{\pm 1}])$ are also torsion.
\end{proof}

We want to recall the definition of the Alexander polynomial of a knot.
To do so we need the notion of the order of a module.

\begin{definition}
Let $H$ be a finitely generated free abelian group and let $M$ be a finitely generated $\Z[H]$-module. By~\cite[Corollary~IV.9.5]{Lang02} the ring $\Z[H]$ is Noetherian which implies that
$M$ admits a free resolution
\[ \Z[H]^r\,\,\xrightarrow{\cdot A}\,\,\Z[H]^s\,\,\to\,\, M\,\,\to\,\,0.\]
Without loss of generality we can assume that $r>s$.
Since $\Z[H]$ is unique factorisation domain, see \cite[Lemma~IV.2.3]{Lang02}, the order $\ord(M)$ is defined as the greatest common divisor of the $s\times s$-minors of $A$. By \cite[Lemma~4.4]{Turaev:2001-1} the order is well-defined, i.e.\ independent of the choice of the free resolution, up to multiplication by a unit in $\Z[H]$.
\end{definition}

The fact that  the order is only well-defined up to multiplication by a unit leads us to the following notation:

\begin{notation}
Let $H$ be a free abelian group. Given $p,q\in \Z[H]$   we write $p\doteq q$ if
$p=\pm h\cdot q$ for some $h\in H$.
\end{notation}

In the proof of the Fox-Milnor Theorem we will need the following lemma, collecting basic facts about orders of finitely generated $\Z[H]$-modules.

\begin{lemma}\label{lem:orders}
Let $H$ be a finitely generated free abelian group.
\bnm
\item If $0\to A\to B\to C\to 0$ is a short exact sequence of finitely generated $\Z[H]$-modules, then $\op{ord}(B)\doteq \op{ord}(A)\cdot \op{ord}(C)$.
\item If $0\to C_k\to \dots \to C_0\to 0$ is an exact sequence of
 finitely generated $\Z[H]$-torsion modules, then the alternating product of the orders is a unit in $\Z[H]$.
\item For any finitely generated $\Z[H]$-module $A$ we have $\op{ord}(\ol{A})\doteq \ol{\op{ord}(A)}$.
Here we write $f\mapsto \ol{f}$ for the natural involution on $\Z[H]$ given by $h\mapsto h^{-1}$ for $h\in H$, and given a $\Z[H]$-module $A$ we write $\ol{A}$ for the $\Z[H]$-module given by
$f\cdot a:=\ol{f}\cdot a$.
\item
 For any finitely generated torsion $\Z[H]$-module $A$ we have $\op{ord}(\ext^1_{\Z[H]}(A,\Z[H]))\doteq \op{ord}(A)$.
\enm
\end{lemma}

\begin{proof}
Statement  (1) is proven for $H\cong \Z$ in \cite[Lemma~5]{Le67}. The general case follows from  \cite[Theorem~3.12]{Hillman}. Note that (2) is an immediate consequence of (1), by separating the long exact sequence into short exact sequences such as $0 \to \im C_{j} \to C_{j-1} \to \im C_{j-1} \to 0$, applying (1), and performing substitutions using the resulting equations involving orders.

Next (3) follows immediately from the definition.
Finally (4) is well-known to the experts, but we could not find a reference, therefore we sketch the key ingredients in the proof.
We introduce the following notation.
\bnm
\item[(a)] Given any prime ideal $\mathfrak{p}$ of $\Z[H]$, let $\Z[H]_{\mathfrak{p}}$ be the localisation at $\mathfrak{p}$, that is we invert all elements that do not lie in $\mathfrak{p}$. We view $\Z[H]$ as a subring of $\Z[H]_{\mathfrak{p}}$.
\item[(b)] Given a ring $R$ and $f,g\in R$ we write $f\doteq_R g$ if $f$ and $g$ differ by multiplication by a unit in $R$.
\enm

Now we sketch the proof of (4). We will use the following five observations.
\bnm
\item[(i)] Since $\Z[H]$ is a unique factorisation domain, for any prime element $p\in \Z[H]$ the ideal $(p)$ is a prime ideal.
\item[(ii)] Being a unique factorisation domain and being Noetherian are preserved under localisation~ \cite[Theorem~7.53]{Peskine}, \cite[Corollary~8.8']{Rowen}. In particular each  $\Z[H]_{\mathfrak{p}}$ is a Noetherian unique factorisation domain. This allows us, by the same definitions as above, to define the order of a finitely generated module over  $\Z[H]_{\mathfrak{p}}$.
\item[(iii)] Localisation is flat \cite[Proposition~XVI.3.2]{Lang02}. It follows that for any finitely generated $\Z[H]$-module $M$ and any prime element $p\in \Z[H]$ one has $\ord(M)\doteq_{\Z[H]_{(p)}} \ord(\Z[H]_{(p)}\otimes_{\Z[H]}M)$ and
 $\Z[H]_{(p)}\otimes_{\Z[H]} \ext_{\Z[H]}^1(M,\Z[H]) \cong \ext_{\Z[H]_{(p)}}^1(\Z[H]_{(p)}\otimes_{\Z[H]}M,\Z[H]_{(p)})$ as $\Z[H]_{(p)}$-modules.
\item[(iv)] By \cite[Corollary~A.14]{Osborne} every commutative ring with the property that every prime ideal is principal, is a PID. It follows easily that for each prime element $p$, the localisation $\Z[H]_{(p)}$ is a PID.
\item[(v)] Let $L$ be a torsion $\Z[H]_{(p)}$-module. Since $\Z[H]_{(p)}$ is a PID every two elements have a greatest common divisor. We can therefore perform row and column operations to find a resolution for $L$ such that the presentation matrix is diagonal.  From this observation one easily deduces that  $L \cong \ext_{\Z[H]_{(p)}}^1(L,\Z[H]_{(p)})$ as left $\Z[H]_{(p)}$-modules.
    Since $L$ is torsion the presentation matrix is injective and so its transpose presents the $\ext$ group. To convert the $\ext$ group to a left module, we use the trivial involution, which we may do since $\Z[H]_{(p)}$ is a commutative ring.
\item[(vi)] Suppose that $f$ and $g$ are in $\Z[H]$.  If $f\doteq_{\Z[H]_{(p)}} g$ for all prime elements  $p\in \Z[H]$, then
since $\Z[H]$ is a unique factorisation domain we must have  $f\doteq_{\Z[H]} g$.
\enm

Now with $L = \Z[H]_{(p)} \otimes_{\Z[H]} A$ a finitely generated $\Z[H]$-torsion module, we have \[\Z[H]_{(p)} \otimes_{\Z[H]} A \cong \ext_{\Z[H]_{(p)}}^1(\Z[H]_{(p)} \otimes_{\Z[H]} A,\Z[H]_{(p)})\]
for every prime element $p$, by (iv).  On the other hand, again for each prime element $p$, we have
\[\ord(A)\doteq_{\Z[H]_{(p)}} \ord(\Z[H]_{(p)}\otimes_{\Z[H]} A )\]
by (iii).  Combining these two observations yields
\[\ord(A)\doteq_{\Z[H]_{(p)}} \ord(\ext_{\Z[H]_{(p)}}^1(\Z[H]_{(p)} \otimes_{\Z[H]} A,\Z[H]_{(p)})). \]
By the second part of (iii) we have that
\[\ord(\ext_{\Z[H]_{(p)}}^1(\Z[H]_{(p)} \otimes_{\Z[H]} A,\Z[H]_{(p)})) \doteq_{\Z[H]_{(p)}} \ord(\Z[H]_{(p)}\otimes_{\Z[H]} \ext^1_{\Z[H]}(A,\Z[H])).   \]
By the first part of (iii) again we have
\[\ord(\Z[H]_{(p)}\otimes_{\Z[H]} \ext^1_{\Z[H]}(A,\Z[H]) ) \doteq_{\Z[H]_{(p)}} \ord(\ext^1_{\Z[H]}(A,\Z[H])).  \]
Thus combining the last three equalities we have
\[\ord(A)\doteq_{\Z[H]_{(p)}} \ord(\ext^1_{\Z[H]}(A,\Z[H]))\]
for all prime elements $p$.  Now (4) follows by applying (vi).
\end{proof}

We use the notion of order to define the Alexander polynomial of a knot in a homology $3$-sphere.

\begin{definition}
The \emph{Alexander polynomial $\Delta_K(t)$} of a knot $K$ is defined as the order of the Alexander module $H_1(N_K;\Z[t^{\pm 1}])$. Note that this polynomial is only well-defined up to units in $\Z[t^{\pm}]$.\label{def:alexander-polynomial}
\end{definition}

After these preparations we turn to the actual proof of the Fox-Milnor Theorem~\ref{theorem:fox-milnor}.
We need the following elementary lemma.

\begin{lemma}\label{lem:switch-to-lambda-chain-complex}
Let $\pi$ be a group, let $C_*$ be a chain complex of left free $\Z[\pi]$-modules and let $\phi\colon \pi\to \langle t\rangle$ be a homomorphism.
The map
\[ \ba{rcl}  \hom_{\op{right-}\Z[\pi]}(\ol{C_*},\Z[t^{\pm 1}])&\to &
  \ol{\hom_{\op{left-}\Z[t^{\pm 1}]}(\Z[t^{\pm 1}]\otimes_{\Z[\pi]} C_*, \Z[t^{\pm 1}])}\\
f&\mapsto & (p\otimes \sigma\mapsto p\cdot \ol{f(\sigma)})
\ea
\]
is well-defined and is an isomorphism of left $\Z[t^{\pm 1}]$-cochain complexes.\qed
\end{lemma}

\begin{proof}[First proof of the Fox-Milnor Theorem~\ref{theorem:fox-milnor}]
In this proof we abbreviate $\Lambda := \Z[t^{\pm 1}]$.
We start out with the following three observations.
\bnm
\item[(a)] We have $H_0(W_D;\Lambda)\cong H_0(N_K;\Lambda)\cong \Lambda/(t-1)$.
\item[(b)] We have $H_0(W_D,N_K;\Lambda)=0$.
\item[(c)] By Proposition~\ref{prop:alexander-module-w-torsion} and Proposition~\ref{prop:exterior-slice-disc-mfd} we  know that for all $k$
\[ \Ext^0_{\Lambda}(H_k(W_D,N_K;\Lambda),\Lambda)\,\,=\,\,\hom_{\Lambda}(H_k(W_D,N_K;\Lambda),\Lambda)\,\,=\,\,0.\]
\enm

\begin{claim}
For any $i\in \N$ we have
\[ \ba{rcl} H_i(N_K;\Lambda)&\cong& \ol{\Ext^1_{\Lambda}(H_{2-i}(N_K;\Lambda),\Lambda)}\\
 H_i(W_D;\Lambda)&\cong& \ol{\Ext^1_{\Lambda}(H_{3-i}(W_D,N_K;\Lambda),\Lambda)}.\ea \]
\end{claim}

We prove the second statement of the claim. The proof of the first statement is almost identical.
By the Poincar\'e Duality Theorem~\ref{thm:poincareduality} we have an isomorphism
$H_i(W_D;\Lambda) \cong H^{4-i}(W_D,N_K;\Lambda)$ of $\Lambda$-modules.
By Lemma~\ref{lem:switch-to-lambda-chain-complex}, applied to $C_*=C_*(W_D,N_K; \Z[\pi])$,
we know that
\[ H^{4-i}(W_D,N_K;\Lambda)\,\,\cong\,\, \ol{H_{4-i}(\hom_{\Lambda}(\Lambda\otimes_{\Z[\pi]} C_*(W_D,N_K; \Z[\pi]),\Lambda))}.\]
Finally we apply the universal coefficient spectral sequence \cite[Theorem~2.3]{Le77}
to the
$\Lambda$-module chain complex $C_*(W_D,N_K;\Lambda)$.
It follows from the above observations (b) and (c) that the spectral sequence collapses and that we have an isomorphism
\[ H_{4-i}\big(\hom_{\Lambda}(\Lambda\otimes_{\Z[\pi]} C_*(W_D,N_K;\Lambda)\big)\,\,
\cong\,\, \ext^1_{\Lambda}(H_{3-i}(W_D,N_K;\Lambda),\Lambda).\]
This concludes the proof of the claim.

Next we consider the long exact sequence of the pair $(W_D,N_K)$ of twisted homology with $\Lambda$-coefficients:
\[
\cdots
 \to  H_2(W_D;\Lambda) \to H_2(W_D,N_K;\Lambda)  \to H_1(N_K;\Lambda) \to H_1(W_D;\Lambda) \to H_1(W_D,N_K;\Lambda) \to \cdots\]
It follows from Propositions~\ref{prop:exterior-slice-disc-mfd} that all the above modules are finitely generated. Thus it makes sense to consider their orders.
Also note that in Proposition~\ref{prop:alexander-module-w-torsion} we saw that the modules for $N_K$ and $W_D$ are all $\Lambda$-torsion. It follows from the long exact sequence that the relative homology groups $H_*(W_D,N_K;\Lambda)$ are also $\Lambda$-torsion. By  Lemma~\ref{lem:orders} (3) the alternating product of the orders equals $\pm t^k$.

By the above claim and Lemma~\ref{lem:orders} (3) and (4) the orders are anti-symmetric around $H_1(N_K;\Lambda)$. More precisely, we have
\[ \ord(H_2(W_D,N_K;\Lambda))\,\,\doteq\,\,\ord\big(\ol{\ext_{\Lambda}^1(H_1(W_D;\Lambda),\Lambda)}\big)\,\,\doteq\,\, \ol{\ord(H_1(W_D;\Lambda))},\]
and the same type of relation holds as we progress further from the middle term $H_1(N_K;\Lambda)$ in the above long exact sequence.
But this implies that there exist nonzero polynomials $f,g\in \Lambda$ with
$f\cdot \ol{f}\doteq \Delta_K(t)\cdot g\cdot \ol{g}$.
It follows easily from the fact that  $\Lambda=\Z[t^{\pm 1}]$ is a UFD that there exists an $h\in \Lambda$ with $h\cdot \ol{h}\doteq \Delta_K(t)$.
\end{proof}


A knot $K$ in $Y$ is \emph{homotopy ribbon} if there is a slice disc $D$ in $X$ such that $\pi_1(Y \sm \nu K) \to \pi_1(W_D)$ is surjective.

\begin{corollary}
    Let $D \subseteq X$ be a homotopy ribbon disc for $K \subseteq Y$. Let $f(t) = \ord H_1(W_D;\Lambda)$. Then $\Delta_K(t) \doteq f(t)f(t^{-1})$.
\end{corollary}

\begin{proof}
    Since $\pi_1(S^3 \sm \nu K) \to \pi_1(W_D)$ factors through $\pi_1(N_K)$, the map $\pi_1(N_K) \to \pi_1(W_D)$ is surjective.  Hence $\pi_1(N_K)^{(1)} \to \pi_1(W_D)^{(1)}$, the map on commutator subgroups is surjective. The respective homology groups with $\Lambda$ coefficients are the abelianisations of the commutator subgroups, and so $H_1(N_K;\Lambda) \to H_1(W_D;\Lambda)$ is surjective.  Hence $H_1(W_D,N_K;\Lambda)=0$, and so
    \[1 = \ord (H_1(W_D,N_K;\Lambda)) = \ol{\ord(H_2(W_D;\Lambda)}.\]
    The pervious proof then implies that \[\Delta_K(t) \doteq \ord(H_1(N_K;\Lambda)) \doteq \ord H_1(W_D;\Lambda) \cdot \ol{\ord H_1(W_D;\Lambda)} = f(t)f(t^{-1})\]
    as required.
\end{proof}

We conclude with an alternative argument for the Fox-Milnor Theorem in the topological category using Reidemeister torsion.
The advantage of the Reidemeister torsion invariant is that proofs are often easier, and it has in general a smaller indeterminacy than the order of homology, although this will not manifest itself in the upcoming proof.

\begin{proof}[Second proof of Theorem~\ref{theorem:fox-milnor}]
We continue with the notation introduced above.
As before we have a homomorphism $\alpha \colon \pi_1(W_D) \to H_1(W_D;\Z) \toiso \Z$, sending an oriented meridian of $K$ to $1 \in \Z$.  As usual $\Q(t)$ denotes the field of fractions of the Laurent polynomial ring $\Z[t,t^{-1}]$.  We take $d=1$, and so obtain a representation $\phi \colon \pi_1(W_D) \to \GL(1,\Q(t))$, that sends $g \mapsto (t^{\alpha(g)})$.
In the previous proof we had already seen that the modules $H_*(N_K;\Z[t^{\pm 1}])$,
$H_*(W_D;\Z[t^{\pm 1}])$ and $H_*(N_K;\Z[t^{\pm 1}])$ are $\Z[t^{\pm 1}]$-torsion.
Since $\Q(t)$ is flat over  $\Z[t^{\pm 1}]$ it follows that the corresponding twisted homology groups with $\Q(t)$-coefficients are zero.

By the discussion in Section~\ref{section:reidemeister-torsion} we can consider the Reidemeister torsions $\tau(W_D,\phi)$,  $\tau(N_K,\phi)$ and $\tau(W_D,N_K,\phi)$.  By Theorem~\ref{thm:torsion-PD}, $\tau(W_D,N_K,\phi) \doteq  \overline{\tau(W_D,\phi)}^{(-1)^5} \doteq \overline{\tau(W_D,\phi)}^{-1}$.  Since the torsion is multiplicative in short exact sequences by Theorem~\ref{thm:torsion-multiplicative}, we have that
\[\tau(W_D,\phi) \,\,\doteq\,\, \tau(N_K,\phi) \cdot \tau(W_D,N_K,\phi) \,\,\doteq\,\, \tau(N_K,\phi) \cdot \overline{\tau(W_D,\phi)}^{-1}.\]
By \cite[Theorem~14.12]{Turaev:2001-1} the torsion of the zero surgery of a knot is equal to $\Delta_K(t)/((t-1)(t^{-1}-1))$. It follows that $\Delta_K(t)$ is a norm as claimed.
\end{proof}

\begin{remark} The two proofs presented above avoid the use of the smooth category, and so are in keeping with the spirit of this book. However, one can give a further alternative proof by allowing smooth techniques. First one can use Theorem~\ref{thm:connect-sum-is-smooth} to find a simply connected 4-manifold $W'$ such that $W:= W_D \# W'$ is smoothable.  Then one can triangulate~$W$ and apply Reidemeister torsion machinery without appealing to~\cite[Essay~III]{KS77}.  The disadvantage of this approach is that typically $H_2(W';\Z)$ will be nontrivial, so that $W$ is not acyclic over $\Q(t)$.  One can proceed by choosing a self-dual basis for homology, so that one can still obtain a torsion invariant that is well-defined up to norms.  Apply~\cite[Theorem~2.4]{Cha-Friedl:2013-1}, and argue that since the intersection form of $W$ is nonsingular, the contribution of $W'$ to the torsion is a norm.
\end{remark}

\begin{appendix}

\chapter{Poincar\'e Duality with twisted coefficients}\label{appendix:poincare-duality}

Surveying the literature, we felt it would be beneficial to have a more detailed proof of Poincar\'{e} duality with twisted coefficients for manifolds with boundary, but without a smooth or $\PL$ structure, so we offer one in this appendix.  One can find other proofs of Poincar\'{e} duality for some subsets of these conditions,
e.g.\  Sun \cite{Sun17} and Kwasik-Sun \cite{KwasikSun18} provide a proof in the closed case. A detailed discussion of twisted (co-) homology and Poincar\'e duality can also be found in \cite[Part XXIV]{Fr23}.

\section{Twisted homology and cohomology groups}\label{section:twisted-invariants}
We start out with the following notation.

\begin{notation}
Given a group $\pi$ and a left $\Z[\pi]$-module $A$,
write $\ol{A}$ for the right $\Z [\pi]$-module that has the same underlying abelian group but for which the right action of $\Z[\pi]$ is defined by $a\cdot g:=g^{-1}\cdot a$ for $a\in A$ and $g\in \pi$.  The same notation is also used with the r\^{o}les of left and right reversed and $g \cdot a := a \cdot g^{-1}$.
\end{notation}

We recall the definition of twisted homology and cohomology groups.

\begin{definition}
Let $X$ be a connected topological space that admits a universal cover $p\colon \wti{X}\to X$.
Write $\pi=\pi_1(X)$.
Let $Y$ be a subset of $X$ and let $A$ be a right $\Z[\pi]$-module.
Let $\pi$ act on $\wti{X}$ by deck transformations, which is naturally a left action. Thus, the singular chain complex  $C_*(\wti{X},p^{-1}(Y))$ becomes a left $\Z[\pi]$-module chain complex.
Define the \emph{twisted chain complex}
\[ C_*(X,Y;A)\,\,:=\,\,\big( A\otimes_{\Z[\pi]} C_*(\wti X, p^{-1}(Y)),\id\otimes \partial_*\big).\]
The corresponding \emph{twisted homology groups} are $H_k(X,Y;A)$.
With~$\delta^k = \hom(\partial_k,\id)$ define the \emph{twisted cochain complex} to be
\[ C^*(X,Y;A)\,\,:=\,\,\big(\hom_{\opnormal{right-}\Z[\pi]}\big(\ol{C_*(\wti{X},p^{-1}(Y)}), A\big), \delta^*\big).\]
The corresponding \emph{twisted cohomology groups} are $H^k(X,Y;A)$.
\end{definition}

Note that if $R$ is some ring (not necessarily commutative) and if  $A$ is an $(R,\Z[\pi])$-bimodule, then the above twisted homology and cohomology groups are naturally left $R$-modules.

Given a CW complex one can similarly define twisted cellular (co-)
chain complexes and twisted cellular (co-) homology groups. The following
proposition implies  that twisted singular (co-) homology groups
are isomorphic to twisted cellular (co-) homology groups.

\begin{proposition}\label{prop:sing-vs-cell}
Let $(X,Y)$ be a CW complex pair and write $\pi=\pi_1(X)$.
The singular chain complex $C_*^{\op{sing}}(X,Y; \Z[\pi])$ and  the cellular chain complex $C_*^{\op{cell}}(X,Y; \Z[\pi])$ are
chain homotopy equivalent as chain complexes of left $\Z[\pi]$-modules.
\end{proposition}

The proof of Proposition~\ref{prop:sing-vs-cell} relies on the following very useful lemma.

\begin{lemma}\label{lem:whiteheadtrick}\label{lem:algebraiclem}
Let $\map{f}{C_*}{D_*}$ be a chain map of chain complexes of  free left $\Z[\pi]$-modules $($here chain complexes are understood to start in degree 0$)$ that induces an isomorphism on homology. Then $f$ is a chain equivalence.
\end{lemma}

\begin{proof}
	Since $f$ induces an isomorphism of homology groups we know that the mapping cone $\cone(f)_\ast$ is acyclic. By assumption $C_\ast$ and $D_\ast$ are free left $\Z[\pi]$-modules. It follows that  $\cone(f)_\ast$ is also a chain complex of free left $\Z[\pi]$-modules. But this guarantees the existence of a chain homotopy $\id_{\cone{(f)}_\ast}\simeq_P 0$, since we can view $\cone(f)_*$ as a free resolution of $0$ and any two such resolutions are chain homotopic.
	Recall that chain homotopy means
	\begin{align}\label{eq:chainhomotopy} \partial^{\cone{(f)}_\ast} \circ P + P\circ \partial^{\cone{f}_\ast} = \id_{\cone{(f)}_\ast} \end{align}
	If we write $P$ as a matrix
	\[
	P_n=\begin{pmatrix}
	P^{11}_n & P^{12}_n \\
	P^{21}_n & P^{22}_n
	\end{pmatrix}\colon
	C_{n-1} \oplus  D_n \to
	C_{n} \oplus  D_{n+1}
	\]
	then one easily verifies using Equation~\eqref{eq:chainhomotopy}, that $P^{12}_*\colon D_*\to C_*$ is a chain homotopy inverse of $f_*$, where the chain homotopies are given by $P^{11}_*$ and $P^{22}_*$.
\end{proof}

\begin{proof}[Proof of Proposition~\ref{prop:sing-vs-cell}]
Given a CW complex $A$ we consider  the intermediate chain complex
$C_*^{\op{int}}(A):=\ker(C_n(A^n)\xrightarrow{\partial}C_{n-1}(A^n)\to C_{n-1}(A^n,A^{n-1}))$. Given a subcomplex $B$ of $A$ we set
$C_*^{\op{int}}(A,B):=\coker(C_n^{\op{int}}(A)\to
C_n^{\op{int}}(B))$.

Let  $p \colon \wti X \to X$ denote the universal cover. We write $\wti{Y}:=p^{-1}(Y)$.
In  \cite[p.~303]{Schubert} (see also  \cite[Lemma~4.2]{Lu98}) it is shown
that the natural maps
$\iota\colon C_*^{\op{int}}(\wti{X},\wti{Y})\to C_*(\wti{X},\wti{Y})$ and
$\pi\colon C_*^{\op{int}}(\wti{X},\wti{Y})\to C_*^{\op{cell}}(\wti{X},\wti{Y})$
induce isomorphisms of homology groups.

Note that $C_*(\wti{X},\wti{Y})$ and $C_*^{\op{cell}}(\wti{X},\wti{Y})$ are free $\Z[\pi]$-left modules.
But it is not clear whether each $C_*^{\op{int}}(\wti{X},\wti{Y})$ is a  free $\Z[\pi]$-left module. But it  is straightforward to show that there exists a chain complex $F_*$ consisting of free $\Z[\pi]$-left modules and a chain map
$\varphi\colon F_*\to C_*^{\op{int}}(\wti{X},\wti{Y})$ of left $\Z[\pi]$-modules which induces isomorphisms of homology groups.

It follows from Lemma~\ref{lem:whiteheadtrick}
that $\iota\circ \varphi\colon F_*\to
C_*(\wti{X},\wti{Y})$ and $\pi\circ \varphi\colon
F_*\to C_*^{\op{cell}}(\wti{X},\wti{Y})$ are chain homotopy equivalences of left $\Z[\pi]$-modules. It follows that
$C_*(\wti{X},\wti{Y})$ and $C_*^{\op{cell}}(\wti{X},\wti{Y})$
are chain homotopy equivalent as chain complexes of let $\Z[\pi]$-modules,
which is equivalent to  $C_*^{\op{sing}}(X,Y; \Z[\pi])$ and    $C_*^{\op{cell}}(X,Y; \Z[\pi])$  being
chain homotopy equivalent as chain complexes of left $\Z[\pi]$-modules.
\end{proof}

\begin{proposition}\label{prop:finite-chain-complex}
Let $M$ be a compact $n$-manifold and let $N\subseteq M$ be a subspace that is a compact manifold in its own right. Write $\pi=\pi_1(M)$.
There exists a chain complex $C_*$ of finite length consisting of finitely generated free left $\Z[\pi]$-modules such that for any  ring $R$, for any $(R,\Z[\pi])$-bimodule $A$ and for any $k\in \N_0$ we have left $R$-module isomorphisms
\[ H_k(M,N;A)\,\,\cong\,\,H_k(A\otimes_{\Z[\pi]} C_*)\]
and
\[ H^k(M,N;A)\,\,\cong\,\,H_k(\Hom_{\Z[\pi]}(\overline{C_*},A)).\]
\end{proposition}

\begin{remark}
Note that we do \emph{not} demand that $N$ be a submanifold of $M$. For example $N$ could be a union of boundary components of $M$, or $N$ could be a submanifold of the boundary. Evidently $N$ could also be the empty set.
\end{remark}

\begin{proof}
By Theorem~\ref{thm:topological-manifold-CW complex} the manifolds $M$ and $N$ are homotopy equivalent to finite CW complexes $X$ and $Y$ respectively.
Let $i\colon N\to M$ be the inclusion map. By the Cellular Approximation Theorem there exists a cellular map  $j\colon Y\to X$ such that the following diagram commutes up to homotopy:
\[ \xymatrix@C1.42cm@R0.65cm{ N \ar[r]^i \ar[d]_(0.45)\simeq & M \ar[d]^(0.45)\simeq \\
Y\ar[r]^j & X.}\]
Next we replace $M$ and $X$ by the mapping cylinders of $i$ and $j$ respectively, to create cofibrations.
Given a map $f\colon U\to V$ between topological spaces let $\op{cyl}(f)$ be the mapping cylinder. We view $U$ as a subset of $\op{cyl}(f)$ in the obvious way. With this notation we have
\[ H_k(M,N;A)\,\cong\,H_k(\op{cyl}(i\colon N\to M),N;A)\,\cong\, H_k(\op{cyl}(j\colon Y\to X),Y;A).\]
The mapping cylinder $Z:=\op{cyl}(j\colon Y\to X)$ admits the structure of a finite CW complex such that $Y$ is a subcomplex.
Thus we can compute the twisted homology groups $ H_k( \op{cyl}(j\colon X\to Y);A)$
using the relative twisted cellular chain complex, and similarly for cohomology. Put differently,
$C_*= C_*^{\op{cell}}(Z,Y; \Z[\pi])$ has the desired properties.
\end{proof}

In order to give a criterion for twisted homology modules to be finitely generated, we need the notion of a Noetherian ring.

\begin{definition}
A ring $R$ is said to be \emph{left Noetherian} if for any descending chain
\[ R\,\,\supseteq \,\, I_1\,\,\supseteq \,\, I_2\,\,\supseteq \,\, I_3\,\,\supseteq \,\,\dots\]
of left $R$-ideals the inclusions eventually become equality. If $R$ is commutative, then we just say Noetherian.
\end{definition}

\begin{example}
The following rings are left Noetherian:
\bnm
\item The ring $\Z$ is Noetherian.
\item
Any (skew) field is left Noetherian.
\item If $A$ is a commutative Noetherian ring, then the multivariable Laurent polynomial ring $A[t_1^{\pm 1},\dots,t_k^{\pm 1}]$ is also Noetherian \cite[Corollary~IV.9.5]{Lang02}.
\enm
\end{example}

The following theorem is often implicitly  used.

\begin{proposition}\label{prop:twisted-homology-fg}
Let $M$ be a compact $n$-manifold, let $N\subseteq M$ be a subspace that is a compact manifold in its own right, let $R$ be a ring and let $A$ be an $(R,\Z[\pi])$-bimodule.
If $R$ is left Noetherian and if $A$ is finitely generated as a left  $R$-module, then all the twisted homology modules $H_*(M,N;A)$ are finitely generated left $R$-modules.
\end{proposition}

In the proof of Proposition~\ref{prop:twisted-homology-fg} we will need the following lemma; cf.\
\cite[Proposition~1.21]{Lam91} or \cite[Proposition~X.1.4]{Lang02}.

\begin{lemma}\label{lem:module-over-noetherian}
Let $R$ be a left Noetherian ring. If $P$ is a finitely generated left $R$-module, then any left submodule of $P$ is also a finitely generated left $R$-module.
\end{lemma}

\begin{proof}[Proof of Proposition~\ref{prop:twisted-homology-fg}]
By Proposition~\ref{prop:finite-chain-complex}, there exists a chain complex $C_*$ of finite length consisting of finitely generated free left $\Z[\pi]$-modules such that
\[ H_k(M,N;A)\,\,\cong\,\,H_k(A\otimes_{\Z[\pi]} C_*).\]
Given $k\in \N_0$ we denote the rank of $C_k$ as  a free left $\Z[\pi]$-module  by ~$r_k$.
Then we have $A\otimes_{\Z[\pi]} C_k \cong A\otimes_{\Z[\pi]} \Z[\pi]^{r_k}\cong A^{r_k}$.
In particular $H_k(M,N;A)$ is isomorphic to a quotient of a submodule of a finitely generated $R$-module.  The desired statement follows from
Lemma~\ref{lem:module-over-noetherian}.
\end{proof}

\subsection{Cup and cap products on twisted (co-) chain complexes}
Throughout this section let $X$ be  a connected topological space admitting a universal cover, and write $\pi=\pi_1(X)$.
We want to introduce the cup product and the cap product on twisted (co-) chain complexes.
Given an $n$-simplex $\sigma$, define the $p$-simplices $\sigma\lfloor_{p}$ and $\sigma\rfloor_{p}$ by
\begin{align*}
\sigma\rfloor_{p}(t_0,\ldots,t_p)&:=\sigma(t_0,\ldots,t_p,0,\ldots,0),\\
\sigma\lfloor_{p}(t_0,\ldots,t_p)&:=\sigma(0,\ldots,0,t_0,\ldots,t_p).
\end{align*}
Throughout this section let $A$ and $B$ be
two right $\Z[\pi]$-modules $A$. We view  $A\otimes_\Z B$ as a right $\Z[\pi]$-module via the diagonal action of $\pi$.

First we introduce the cup product on twisted cohomology. The following lemma can be verified easily by hand, say along the lines of the proof of \cite[Lemma~3.6]{Hat02}.

\begin{lemma}\label{lem:cup-product}
Let $Y$ be a subset of $X$.
For all $p,q\in \N_0$ We consider the map
\begin{align*}
{ \acup}\colon\coCh{p}{X,Y}{A} \times \coCh{q}{X,Y}{B}&\longrightarrow \coCh{p+q}{X,Y}{{A}\otimes_\Z B}\\
(\phi,\psi) &\longmapsto \big(\sigma\mapsto \varphi(\sigma\lfloor_p)\otimes_\Z  \psi(\sigma\rfloor_{k-p}) \big).
\end{align*}
$($Note that  the right-hand side is indeed  a $\Z[\pi]$-homomorphism, i.e.\ it defines an element
$\coCh{p+q}{X}{{A}\otimes_\Z B}$.$)$
 Furthermore the map descends to a well defined map
\[ \acup \colon H^p(X,Y;A) \times H^q(X,Y;B)\,\,\longrightarrow\,\, H^{p+q}(X,Y;A\otimes_\Z B).\]
We refer to this map as the \emph{cup product}.
\end{lemma}

Next we introduce the cap product. As with cup product, first we define it on the chain level.

\begin{lemma}\label{lem:cap-product}
Let $S,T\subseteq X$ be subsets. We write $C_k(X,\{S,T\})=C_k(X)/(C_k(S)+C_k(T))$.
The map
\begin{align*}
\acap\colon\coCh{p}{X,S}{A} \times \Ch{k}{X,\{S,T\}}{A}&\longrightarrow \Ch{k-p}{X,T}{{A}\otimes_\Z B}\\
(\psi, b\otimes_{\Z[\pi]} \sigma) &\longmapsto (\psi(\sigma\lfloor_p)\otimes_\Z b) \otimes_{\Z[\pi]} \sigma\rfloor_{k-p}.
\end{align*}
is well-defined. We refer to this map as the \emph{cap product}.
\end{lemma}

\begin{proof}
We verify that the given map respects the tensor product. Thus let $\psi\in C^p(X;A)$,  $\sigma\in C_k(\wti{X})$, $\gamma\in \pi$ and $b\in B$.
We calculate that
\[ \ba{rll}
\psi\acap b\otimes_{\Z[\pi]} \gamma\sigma &=\,\, \left(\psi(\gamma\sigma\lfloor_p)\otimes_\Z b\right) \otimes_{\Z[\pi]} \gamma\sigma\rfloor_{k-p}
&=\,\,(\gamma\psi(\sigma\lfloor_p)\otimes_\Z b)\cdot\gamma \otimes_{\Z[\pi]} \sigma\rfloor_{k-p}\\[0.1cm]
&=\,\,(\gamma^{-1} \gamma\psi(\sigma\lfloor_p)\otimes_\Z b\gamma) \otimes_{\Z[\pi]} \sigma\rfloor_{k-p}&=\,\,\psi\acap b\gamma\otimes_{\Z[\pi]} \sigma.
\ea
\]
It follows easily from the definitions that the cap product descends to the given quotient (co-) chain complexes.
\end{proof}

\begin{lemma}\label{lem:boundary-cap}\label{lem:capboundary}
Let $f \in C^p(X;A)$ and let $c \in C_k(X;B)$.  We have
\[\partial (f\acap c)\,\,=\,\,(-1)^p\cdot (-\delta(f)\acap c + f\acap\partial c) \in C_{k-1}(X;A \otimes_{\Z} B). \]
\end{lemma}

\begin{proof}
The lemma follows from a calculation using the definition of the cap product and the boundary maps, see e.g.\ \cite[Lemma~192.9]{Fr23} for details. Note that the precise signs differ from similar formulas in some textbooks in algebraic topology since there are many different sign conventions in usage.
\end{proof}

\begin{corollary}\label{cor:cap-well-defined}
Let  $S,T\subseteq X$ be  subsets, let $R$ be a ring and let $A$ be an $(R,\Z[\pi_1(M)])$-bimodule.
For any cycle $\sigma\in C_n(X,\{S,T\};\Z)$  the cap product
\[\ba{rcl} \acap [\sigma] \colon H^k(X,S;A)&\to & H_{n-k}(X,T;A)=H_{n-k}(X,T;A\otimes_{\Z} \Z)\\
{}[\varphi]&\mapsto & [\varphi\acap \sigma]\ea\]
is well-defined. Furthermore this map only depends on the homology class $[\sigma]\in H_n(X,S\cup T;\Z)$.
\end{corollary}

\section{The Poincar\'e Duality Theorem}
The following theorem is a generalisation of the familiar Poincar\'e duality for untwisted coefficients to the case of twisted coefficients.

\begin{theorem}\label{thm:poincareduality}
\textbf{\textup{(Twisted Poincar\'e Duality Theorem)}}
Let $M$ be a   compact, oriented, connected $n$-dimensional manifold. Let $S$ and $T$ be codimension 0 compact submanifolds of $\partial M$ such that $\partial S=\partial T=S\cap T$ and $\partial M=S\cup T$.
Let $[M]\in H_n(M,\partial M;\Z)$ be the fundamental class of $M$. If $R$ is a ring and if  $A$ is an  $(R,\Z [\pi_1(M)])$-bimodule, then the map
\[-\acap [M]\colon H^k(M,S;{A})\,\,\rightarrow \,\,H_{n-k}(M,T;A)\]
 defined by Lemma~\ref{lem:boundary-cap} is an isomorphism of left $R$-modules.
\end{theorem}

We also have the following Poincar\'e Duality statement on the (co-) chain level.

\begin{theorem}\textbf{\textup{(Universal Poincar\'e Duality Theorem)}} \label{thm:poincareduality-chain-complex}
Let $M$ a compact, oriented, connected $n$-dimensional manifold. Let $S$ and $T$ be codimension 0 compact submanifolds of $\partial M$ such that $\partial S=\partial T=S\cap T$ and $\partial M=S\cup T$.
Let $\sigma\in C_n(M,\{S,T\};\Z)$ be a representative of the fundamental class of $M$. If $R$ is a ring and if  $A$ is an  $(R,\Z [\pi_1(M)])$-bimodule, then the map
\[-\acap \sigma\colon C^k(M,S;\Z[\pi_1(M)])\rightarrow C_{n-k}(M,T;\Z[\pi_1(M)])\]
 defined by Lemma~\ref{lem:boundary-cap} is a chain homotopy equivalence of left $R$-chain complexes.
\end{theorem}

Note that Theorem~\ref{thm:simple-homotopy-equivalence-chain-complexes} can be used to give an alternative proof that the chain complexes of the theorem are chain homotopy equivalent.

\begin{proof}[Proof of Theorem~\ref{thm:poincareduality-chain-complex}
using Theorem~\ref{thm:poincareduality}]
The Universal Poincar\'e Duality Theorem~\ref{thm:poincareduality-chain-complex} follows immediately from
the Twisted Poincar\'e Duality Theorem~\ref{thm:poincareduality}
together with Lemma~\ref{lem:whiteheadtrick}.
\end{proof}

In the following sections  we will provide a proof of the Twisted Poincar\'e Duality Theorem~\ref{thm:poincareduality}.
But just for fun we would like to show that
the Universal Poincar\'e Duality Theorem~\ref{thm:poincareduality-chain-complex} also implies
the Twisted Poincar\'e Duality Theorem~\ref{thm:poincareduality}, in other words, the two theorems are equivalent:

\begin{proof}[Proof of Theorem~\ref{thm:poincareduality}
using Theorem~\ref{thm:poincareduality-chain-complex}]
Let $M$ be a compact, oriented, connected $n$-dimensional manifold. To simplify the discussion we just deal with the case that $S=\partial M$ and $T=\emptyset$. We pick a representative $\sigma$ for $[M]$ and we write $\pi=\pi_1(M)$.
Let $A$ be an  $(R,\Z [\pi])$-bimodule. Given a chain complex $D_*$ of right $\Z[\pi]$-modules we consider the cochain map
\[ \ba{rcl} \Xi \colon A\otimes_{\Z[\pi]} \hom_{\opnormal{right-}\Z[\pi]}(D_*;\Z[\pi]))&\to & \hom_{\opnormal{right-} \Z[\pi]}(D_*,A)\\
a\otimes f&\mapsto & (\sigma \mapsto a\cdot f(\sigma)).\ea\]
Note that $\Xi$ is an isomorphism if each $D_k$ is a finitely generated free $\Z[\pi]$-module. But in general $\Xi$ is not an isomorphism.

Furthermore we consider the following diagram
\[ \xymatrix@C2.1cm@R0.65cm{ H^k(A\otimes_{\Z[\pi]} C^*(M,\partial M;\Z[\pi]))\ar[r]^-{\id_A\otimes (\acap \sigma)} \ar[d]^(0.45){\Xi_*}&  H_{n-k}(A\otimes_{\Z[\pi]} C_*(M; \Z[\pi]))=H_{n-k}(M;A)\\
H^k(C^*(M,\partial M;A))=H^k(M,\partial M;A).\ar[ur]_{\acap [M]}}\]
One easily verifies that the diagram commutes. The top horizontal map is an isomorphism by
the Universal Poincar\'e Duality Theorem~\ref{thm:poincareduality-chain-complex}. It remains to show that the vertical map is an isomorphism. As we had pointed out above, on the chain level $\Xi$ is in general not an isomorphism.

As in the proof of Proposition~\ref{prop:finite-chain-complex} we can use Theorem~\ref{thm:topological-manifold-CW complex} to find a pair $(X,Y)$ of finite CW complexes and a homotopy  equivalence $f\colon (X,Y) \to (M,\partial M)$.
By Proposition~\ref{prop:sing-vs-cell}  there exists a homotopy equivalence $\Theta\colon
C_*^{\op{cell}}(X,Y; \Z[\pi])\to C_*(X,Y; \Z[\pi])$ of $\Z[\pi]$-chain complexes.
 We consider the following diagram where all tensor products and homomorphism are over $\Z[\pi]$:
\[ \xymatrix@C0.6cm@R0.65cm{
 A\otimes C^*_{\op{cell}}(X,Y;\Z[\pi])\ar[d]^(0.45){\Xi_*}\ar[r]^{\Theta^*} &
 A \otimes C^*(X,Y;\Z[\pi])\ar[d]^(0.45){\Xi_*}\ar[r]^{f^*}&
 A \otimes C^*(M,\partial M;\Z[\pi])\ar[d]^(0.45){\Xi_*}\\
C^*_{\op{cell}}(X,Y;A)\ar[r]^{\Theta^*}&C^*(X,Y;A)\ar[r]^{f^*}&C^*(M,\partial M;A).}\]
One easily verifies that the diagram commutes. As pointed out above, the horizontal maps are chain homotopy equivalences over $\Z[\pi]$. Since $X$ is a finite CW complex we see that  each  $C_k^{\op{cell}}(X,Y; \Z[\pi])$ is a finitely generated  free $\Z[\pi]$-module.
Thus we obtain from the above that the left vertical map is an isomorphism.
Therefore the right vertical map is a chain homotopy equivalence.
In particular it induces an isomorphism of homology groups.
\end{proof}

The remainder of this appendix is dedicated to the proof of the Twisted Poincar\'e Duality Theorem~\ref{thm:poincareduality}.
Even though the theorem is well-known and often used, there are not many satisfactory proofs in the literature. The proof which is closest to ours in spirit is the proof of Sun \cite{Sun17}.
For closed manifolds Kwasik-Sun \cite{KwasikSun18} provide a proof by using the work of Kirby-Siebenmann to reduce the proof to the case of triangulated manifolds.

The proof of the  Twisted Poincar\'e Duality Theorem~\ref{thm:poincareduality}
is modelled on the proof of untwisted  Poincar\'e Duality that is given in  Bredon's book~\cite[Chapter VI.8]{Br93}. The logic of his proof is unchanged, but some arguments and definitions have to be adjusted for the twisted setting.

\section[Preparations for the proof]{Preparations for the proof of the Twisted Poincar\'e Duality Theorem}

We fix some notation that we will use for the remainder of the appendix. Let $M$ be a \emph{connected} manifold, let $x_0\in M$ and denote by $\pi:=\pi_1(M,x_0)$ the fundamental group.
Finally let $R$ be a ring and let $A$ be an  $(R,\Z[\pi])$-bimodule.

We write $\map{p}{\widetilde{M}}{M}$ for the universal cover of $M$.  For a subset $X\subseteq M$ (not necessarily connected) we consider the (co)-homology of $X$ with respect to the coefficient system coming from $M$ by setting
\begin{align*}
\Ch{*}{X}{A}&:= A\otimes_{\Z[\pi]} \Ch{*}{p^{-1}(X)}{\Z},\\
\coCh{*}{X}{A}&:= \Hom_{\Z[\pi]} \left(\overline{\Ch{*}{p^{-1}(X)}{\Z}}, A\right),
\end{align*}
with generalisation to pairs $Y \subseteq X \subseteq M$ by
\begin{align*}
\Ch{*}{X,Y}{A}&:= A\otimes_{\Z[\pi]} \Ch{*}{p^{-1}(X),p^{-1}(Y)}{\Z},\\
\coCh{*}{X,Y}{A}&:= \Hom_{\Z[\pi]} \left(\overline{\Ch{*}{p^{-1}(X),p^{-1}(Y)}{\Z}}, A\right).
\end{align*}

We summarise the basic properties of  (co-) homology with twisted coefficients in the following theorem, which should be compared to the untwisted case.

\begin{theorem}\label{thm:basicproperties}
Let $M$ be a connected manifold with fundamental group $\pi$, and let $A$ be an $(R,\Z[\pi])$-bimodule.
\begin{enumerate}[font=\normalfont,leftmargin=0.9cm]
	\item Given $Y\subseteq X\subseteq M$ there is a long exact sequence of pairs in homology
	\[\begin{tikzcd}[column sep=small] \cdots\arrow[r]& H_k(Y;A) \arrow[r] & H_k(X; A) \arrow[r]& H_k(X,Y;A)\arrow[r] & H_{k-1}(Y;A)\arrow[r] &\cdots
	\end{tikzcd}		
	\]
	and cohomology
	\[\begin{tikzcd}[column sep=small] \cdots\arrow[r]& H^k(X,Y;A) \arrow[r] & H^k(X;A) \arrow[r]& H^k(Y;A)\arrow[r] & H^{k+1}(X,Y;A)\arrow[r] &\cdots
	\end{tikzcd}.		
	\]
	\item Suppose we have a chain of subspaces $Z\subseteq Y\subseteq X\subseteq M$ such that  the closure of $Z$ is contained in the interior of $Y$. Then the inclusion $(X\setminus Z, Y\setminus Z)\to (X,Y)$ induces an isomorphism in homology and cohomology i.e.\ for all $k\in \N_0$ we have
	\[ H_k(X\setminus Z,Y\setminus Z;A)\xrightarrow{\cong} H_k(X,Y;A)\quad\mbox{and}\quad H^k(X\setminus Z,Y\setminus Z;A)\xleftarrow{\cong} H^k(X,Y;A) \]
	\item If $U_1\subseteq U_2\subseteq M$ and $ V_1\subseteq V_2\subseteq M$ are open subsets in $M$, then there are long exact sequences in homology
	\[\begin{tikzcd}[column sep=tiny, row sep = tiny] \ldots\arrow[r]& H_k(U_1\cap V_1, U_2\cap V_2;A) \arrow[r] &\begin{array}{c} H_k(U_1,U_2;A)\\ \oplus\\ H_k(V_1,V_2;A)\end{array} \arrow[r]& H_k(U_1\cup V_2, U_2\cup V_2;A) \\ \arrow[r]& H_{k-1}(U_1\cap V_1, U_2\cap V_2;A)\arrow[r] &\ldots
	\end{tikzcd}		
	\]
	and cohomology
		\[\begin{tikzcd}[column sep=tiny, row sep = tiny]
 & & \cdots \arrow[r] & H^{k-1}(U_1\cap V_1, U_2\cap V_2;A) & \\
  \arrow[r] &  H^k(U_1\cup V_1, U_2\cup V_2;A)  \arrow[r] &  \begin{array}{c} H^k(U_1,U_2;A)\\ \oplus\\ H^k(V_1,V_2;A)\end{array} \arrow[r] & H^k(U_1\cap V_2, U_2\cap V_2;A)  \arrow[r] & \cdots
	\end{tikzcd}		
	\]	
	\item Suppose the inclusion $Y\to X$ is a homotopy equivalence, then the inclusion induced maps
\[ H_k(Y; A)\xrightarrow{\cong} H_k(X;A)\quad\mbox{and}\quad H^k(Y;A)\xleftarrow{\cong} H^k(X;A) \]
are isomorphisms.
\item Let $U_1\subseteq U_2 \subseteq \ldots$ be a sequence of open  sets in $M$ and let $U=\bigcup_{i\in\N} U_i$, then for each $k\in \N_0$ the inclusions induce an isomorphism
\[ {\varinjlim}_{i\in\N} H_k(U_i;A) \,\,\xrightarrow{\cong}\,\, H_k(U;A). \]
\end{enumerate}
\end{theorem}
The proofs are essentially the same as in the classical case. Therefore we will only sketch the arguments and focus on what is different. We also warn the reader that we give the ``philosophically wrong proof'' of statement (4). This is due to the fact that we developed the theory of twisted coefficients only for inclusions and hence a homotopy inverse does not fit in our theory. Therefore statement (4) will be deduced in a slightly round-about way using the following elementary lemma~\cite[Theorem III.3.4 \& remark after proof]{Br93}\label{lem:homotopyliftingproperty}.

\begin{lemma}\textbf{\textup{(Covering Homotopy Theorem)}}
	Given a covering $\map{p}{\widetilde{X}}{X}$, a homotopy $\map{H}{Y\times I}{X}$, and a lift $\widetilde{h}\colon Y\to\widetilde{X}$ of $H(-,0)$,  there exists a unique lift $\map{\widetilde{H}}{Y\times I}{\widetilde{X}}$ of $H$ with $\widetilde{h}=\widetilde{H}(-,0)$.
\end{lemma}

\begin{proof}[Proof of Theorem~\ref{thm:basicproperties}]
Recall that $\map{p}{\widetilde{M}}{M}$ denotes the universal cover.
For statement (1) we consider the short exact sequence $0\to\Ch{*}{Y}{\Z[\pi]}\to\Ch{*}{X}{\Z[\pi]}\to\Ch{*}{X,Y}{\Z[\pi]}\to 0$ of \emph{free} $\Z[\pi]$-modules.
Since the modules are free the sequence stays exact after applying the functors $A\otimes_{\Z[\pi]} -$ and $\Hom_{\Z[\pi]}(-, A)$.

Recall the proof of statement (2) and (3) in the classical case as in~\cite[Chapter IV.17]{Br93}. The main ingredient is to show that the inclusion of chain complexes $C^{\mathcal{U}}_*(X;\Z[\pi])\to C^{*}(X;\Z[\pi])$ induces an isomorphism on homology \cite[Theorem IV.17.7]{Br93}. Here $\mathcal{U}$ is an open cover of $X$ and $C^{\mathcal{U}}_*(X;\Z[\pi])$ is the free abelian group generated by simplices $\sigma$ for which there is a $U\in \mathcal{U}$ such that $\sigma\colon \Delta^*\to p^{-1}(U)$.
This is done by defining the barycentric subdivision $\map{\Upsilon_*}{\Ch{*}{\wti X}{\Z}}{\Ch{*}{\wti X}{\Z}}$ and a chain homotopy $T$ between $\Upsilon_*$ and the identity \cite[Lemma IV.17.1]{Br93}. The important thing for us to observe is that both maps are natural \cite[Claim (1) in proof of Lemma IV.17.1]{Br93}. Hence for a twisted chain $\Upsilon(e\otimes_{\Z[\pi]} \sigma):= e\otimes_{\Z[\pi]} \Upsilon(\sigma)$ is well-defined, because
\begin{align*}
\Upsilon(e\otimes_{\Z[\pi]}\gamma\sigma)&= e\otimes_{\Z[\pi]} \Upsilon(\gamma\sigma) \\
&= e\otimes_{\Z[\pi]}\gamma \Upsilon(\sigma) \quad\mbox{(naturality of $\Upsilon$)} \\
&= e\gamma\otimes_{\Z[\pi]} \Upsilon(\sigma)= \Upsilon(e\gamma\otimes_{\Z[\pi]} \sigma).
\end{align*}
The same holds for $T$ and from now on one can follow the classical proofs. Alternatively, one could invoke Lemma ~\ref{lem:algebraiclem}.

Next we prove statement (4).
Let $\map{f}{X}{Y\subseteq X}$ be a homotopy inverse of the inclusion and $\map{H}{X\times I}{X}$ a homotopy between $\id_X$ and $f$. Since $\map{p}{\widetilde{X}}{X}$ is a covering and $\id_{{\widetilde{X}}}$ is a lift of $H(p(-),0)$, we get by Lemma~\ref{lem:homotopyliftingproperty} a lift $\map{\widetilde{H}}{\widetilde{X}\times I}{\widetilde{X}}$ of the homotopy $H$. One easily verifies  that the inclusion $\widetilde{Y}\to\widetilde{X}$ induces a homotopy equivalence where a homotopy inverse is given by $\widetilde{H}(-,1)$. Hence the inclusion induced map  $H_k(\Ch{*}{\widetilde{Y}}{\Z})\to H_k(\Ch{*}{\widetilde{X}}{\Z})$ is an isomorphism for every $k$. Thus the claim follows from Lemma~\ref{lem:algebraiclem}.

The proof of Statement (5) is almost verbatim  the same proof as in the classical case.
\end{proof}

\section{The main technical theorem}

Given a group $\pi$  we can view $\Z$ as a $\Z[\pi]$-module with trivial $\pi$-action. We denote this module by  $\Z^{\operatorname{triv}}$. Let $p \colon \wt M \to M$ be the covering projection. We have the following useful lemma, concerning the chain map $\Ch{*}{X}{\Ztriv} \to \Ch{*}{X}{\Z}$ defined by $k \otimes_{\Z[\pi]} \wt \sigma \mapsto k \cdot p(\sigma)$.

\begin{lemma}\label{lem:untwisted}
Given any subset $X\subseteq M$ the chain map above is an isomorphism between $\Ch{*}{X}{\Ztriv}$ and $\Ch{*}{X}{\Z}$, and induces one between $\coCh{*}{X}{\Z}$ and $\coCh{*}{X}{\Z^{\operatorname{triv}}}$, where $\Ch{*}{X}{\Z}$ and $\coCh{*}{X}{\Z}$ are the untwisted singular chain complexes.
\end{lemma}

\begin{proof}
	The isomorphism is given by lifting a simplex, which is always possible since a simplex is simply connected. If one has two different choices of lifts, then they differ by an element in $\pi$. But the action of $\Z[\pi]$ on $\Z$ is trivial and hence this indeterminacy vanishes.
\end{proof}

We will keep the notational difference between $\Ch{*}{X}{\Z}$ and $\Ch{*}{X}{\Z^{\operatorname{triv}}}$ to emphasise where our simplices live.

As above let $R$ be a ring and let $A$ be an $(R,\Z[\pi])$-bimodule.
Let $K\subseteq M$ be a compact subset of $M$. We define the (twisted) Cech cohomology groups
\begin{align*}
\Cech{p}{K}{A}:=\lim_{\underset{K\subseteq U \subseteq M}{\longrightarrow}} H^p(U;A),
\end{align*}
where the direct limit runs over all open sets in $M$ containing $K$. Since cohomology is contravariant, we define the order on open sets in the reversed way i.e.\ $U\leq V$ if $V\subseteq U$.

Now we assume that $M$ is oriented. Being oriented gives us, for any compact subset $K\subseteq M$, a preferred element $\theta_K\in H_n(M,M\setminus K;\Ztriv)\cong H_n(M,M\setminus K;\Z)$, which restricts for all $x\in K$ to the generator in $H_n(M,M\setminus\SET{x};\Ztriv)$.

For any open set $U\subseteq M$ containing $K$,
the inclusion induced map $H_n(U,U\setminus K;\Ztriv)\to H_n(M,M\setminus K;\Ztriv)$
is an isomorphism by Theorem~\ref{thm:basicproperties} (2).
Let $\ex_U\colon H_n(M,M\setminus K;\Ztriv)\to H_n(U,U\setminus K;\Ztriv)$ be the inverse of this inclusion induced  isomorphism i.e.\ if $\map{j}{U}{M}$ is the inclusion then $j_*\circ \ex_U =\id$.
We then obtain a map
\begin{align*}
	D_U\colon H^p(U;A)&\longrightarrow H_n(M,M\setminus K; A) \\
	\phi&\longmapsto j_*(\phi\acap \ex_U(\theta_K)).
\end{align*}
Given another open set $V\subseteq U$ denote by $\map{i}{V}{U}$ the inclusion. Then one easily calculates:
\begin{align*}
\PD_V(i^*\phi)=j_*i_*(i^*\phi \acap \ex_V(\theta_K)) = j_*(\phi\acap i_*\ex_V(\theta_K) ) = j_* (\phi\acap \ex_U(\theta_K))=\PD_U(\phi).
\end{align*}
Or,  in other words, the following diagram commutes:
\[\begin{tikzcd}[row sep=tiny]
H^p(U;A)\arrow[dd,"i^*"] \arrow[dr,"\PD_U"] \\
& H_{n-p}(M,M\setminus K; A). \\
H^p(V;A) \arrow[ur,"\PD_V"]
\end{tikzcd}
\]
By the universal property of the direct limit we obtain the \emph{dualising} map $\map{\PD_K}{\Cech{p}{K}{A}}{H_{n-p}(M,M\setminus K;A)}$.

In the remainder of this section we will prove the following theorem.

\begin{theorem}\textbf{\textup{(Poincar\'e Duality Theorem)}}\label{thm: pdclosed}
	The map $\map{\PD_K}{\Cech{p}{K}{A}}{H_{n-p}(M,M\setminus K;A)}$ is a left $R$-module isomorphism for all compact subsets $K\subseteq M$.
\end{theorem}

Here, as above, $A$ is an $(R,\Z[\pi])$-bimodule.
In the subsequent section we will see that
the Twisted Poincar\'e Duality Theorem~\ref{thm:poincareduality} is a reasonably straightforward consequence of Theorem~\ref{thm: pdclosed}.

The proof of Theorem~\ref{thm: pdclosed} will be an application of the following lemma.

\begin{lemma}\textbf{\textup{(Bootstrap lemma)}}
For each compact subspace $K\in M$ let  $P_M(K)$ be a statement. If $P_M(\cdot)$ satisfies the following three conditions:
	\begin{enumerate}[leftmargin=1cm,font=\normalfont]
		\item $P_M(K)$ holds true for all compact subsets $K\subseteq M$ with the property that for all $x\in K$ the inclusions $\SET{x}\to K$ and $M\setminus K\to M\setminus\SET{x}$ are deformation retracts,
		\item If $P_M(K_1), P_M(K_2)$ and $P_M(K_1\cap K_2)$ are true, then $P_M(K_1\cup K_2)$ is true,
		\item If $\cdots \subseteq K_2\subseteq K_1$ and $P_M(K_i)$ is true for all $i\in\N$, then $P_M(\bigcap_{i\in\N}  K_i )$ is true.
	\end{enumerate}	
Then $P_M(K)$ is true for all $K\subseteq M$.
\end{lemma}
\begin{proof}
	See \cite[Lemma VI.7.9]{Br93}.
\end{proof}

The idea is to apply the bootstrap lemma to the statement that
the conclusion of Theorem~\ref{thm: pdclosed} holds for a given compact set $K$.
It turns out that condition (3) is the easiest to verify. It follows from formal properties about direct limits. For the verification of condition (1) we have to do one explicit calculation. This is the content of the next lemma.

\begin{lemma}\label{lem:explicitcalc}
	Let $x\in M$ be a point. The map $\PD_{\SET{x}}\colon \Cech{0}{\SET{x}}{A}\to H_n(M,M\setminus\SET{x}; A)$ is an $R$-module isomorphism.	
\end{lemma}

\begin{proof}
Let $\map{p}{\widetilde{M}}{M}$ be the universal cover.
Since $x$ is a point in a manifold we can calculate the dualising map $\PD_{\SET{x}}$ by taking the limit over open neighbourhoods $U$ of $x$ with the following two properties:
\begin{enumerate}[leftmargin=1cm]
	\item $U$ is contractible,
	\item for any connected component $\overline{U}\subseteq p\inv(U)$ the map $p|_{\overline{U}}$ is a homeomorphism.
\end{enumerate}
This can be done, since any neighbourhood of $x$ contains a neighbourhood with these two properties. Let $U$ be such a neighbourhood of $x$ and $\overline{U}\subseteq p\inv(U)$ a fixed connected component. This choice of connected component gives us an isomorphism $H^0(U; A)\cong A$ as follows. Let $f\in H^0(U; A)$ be arbitrary and $\overline{x}\in\overline{U}$ be a point in our connected component. Then we get an element in $A$ by evaluating $f([\overline{x}])$. Conversely, given an element $e\in A$ we can construct a function in $H^0(U; A)$ by setting $f([\overline{x}])=e$ for all $\overline{x}\in\overline{U}$. Note that there is a unique way to extend $f$ equivariantly to $\Ch{0}{p\inv (U)}{\Z}$.

We are now going to construct a representative of the orientation class $\theta_K\in H_n(M,M\setminus\SET{x};\Ztriv)$ for which it is very simple to calculate the dualising map.
Let $\overline{x}$ be the preimage of $x$ in $\overline{U}$. Now take a cycle $\sum_{i=1}^{d}k_i \sigma_i$ which generates $H_n(\overline{U},\overline{U}\setminus\SET{\overline{x}};\Z)$. By Theorem~\ref{thm:basicproperties} (2) and Lemma~\ref{lem:untwisted} one easily sees that $1\otimes_{\Z[\pi]} \sum_{i=1}^{d}k_i \sigma_i$ is a generator of $H_n(M,M\setminus\SET{x}; \Ztriv)$.

Using the isomorphism $H^0(U; A)\cong A$ from above the dualising map becomes $\PD_{\SET{x}}\colon A\to H_n(M,M\setminus\SET{x}; A), e\mapsto e\otimes_{\Z[\pi]} \sum_{i=1}^{d}k_i \sigma_i$. This is clearly an isomorphism, since on the chain level we have:
\[ \Ch{*}{U,U\setminus\SET{x}}{A}=A
\otimes_{\Z[\pi]} \bigoplus_{\gamma\in\pi}\Ch{*}{\gamma\overline{U},\gamma\overline{U}\setminus\SET{\gamma\overline{x}}}{\Z}\cong A\otimes_\Z \Ch{*}{\overline{U},\overline{U}\setminus\SET{\overline{x}}}{\Z}. \qedhere \]
\end{proof}

In order to verify condition (2) of the bootstrap lemma we will need the following lemma (compare~\cite[Lemma VI.8.2]{Br93}).

\begin{lemma}\label{lem:mvscommutative}
If $K$ and $L$ are two compact subsets of $M$, thenfor all $p\in \N_0$ the diagram
\[ \begin{tikzcd}[column sep=tiny,font=\scriptsize]
\cdots \arrow[r]\hspace{-0.1cm}&\hspace{-0.1cm}\Cech{p}{K\shortcup L}{A}\arrow[r]\arrow[d,"\PD_{K\cup L}"]\hspace{-0.1cm}&\hspace{-0.1cm}\begin{array}{c}
\Cech{p}{K}{A}\\\oplus\\\Cech{p}{L}{A}
\end{array}\arrow[r]\arrow[d,"\PD_{K}\oplus\PD_{L}"]\hspace{-0.1cm}&\hspace{-0.1cm} \Cech{p}{K\cap L}{A}  \arrow[r]\arrow[d,"\PD_{K\cap L}"]\hspace{-0.1cm}&\hspace{-0.1cm}\Cech{p+1}{K\shortcup L}{A}\arrow[r]\arrow[d,"\PD_{K\cup L}"]\hspace{-0.1cm}&\hspace{-0.1cm}\cdots\\
\cdots\arrow[r]\hspace{-0.1cm}&\hspace{-0.1cm} H_{n-p}(M,M\shortsms (K\shortcup L);A)\arrow[r]\hspace{-0.1cm}&\hspace{-0.1cm}\hspace{-0.2cm}\begin{array}{c} H_{n-p}(M,M\shortsms K;A)\\\oplus\\ H_{n-p}(M,M\shortsms L;A)\end{array} \hspace{-0.2cm}\arrow[r]\hspace{-0.1cm}&\hspace{-0.1cm} H_{n-p}(M,M\shortsms (K\cap L);A)\arrow[r]\hspace{-0.1cm}&\hspace{-0.1cm}H_{n-p-1}(M,M\shortsms (K\shortcup L);A)\arrow[r]\hspace{-0.1cm}&\hspace{-0.1cm}\cdots
\end{tikzcd}
\]
has exact rows and it commutes up to a sign depending only on $p$.
\end{lemma}

\begin{proof}
	The rows are exact by Mayer-Vietoris (see
 Theorem~\ref{thm:basicproperties} (3))
 and the fact that direct limit is an exact functor. The commutativity of the squares is clear except for the last one involving the boundary map. This will be a painful diagram chase. Let $U\supseteq K$ and $V\supseteq L$ be open neighbourhoods containing $K$ resp. $L$. The sequence in the top row comes from the short exact sequence ($\mathcal{U}=\SET{U,V}$):
	\[
	\begin{tikzcd}[column sep=small]
	0\arrow[r]& C^*_{\mathcal{U}}(U\cup V;A)\arrow[r]& \coCh{*}{U}{A}\oplus \coCh{*}{V}{A}\arrow[r]& \coCh{*}{U\cap V}{A}\arrow[r]&0.
	\end{tikzcd}
	\]
	An element $\phi\in\Cech{p}{K\cap L}{A}$ will already be represented by some element $f\in\coCh{p}{U\cap V}{A}$ for some $U$ and $V$ as above.
	We can extend $f$ to an element $\overline{f}\in\coCh{p}{M}{A}$ by
	\begin{align*}
	\overline{f}(\sigma)=\begin{cases}
	f(\sigma)&\mbox{if }\operatorname{Im}\sigma \subseteq \widetilde{U}\cap\widetilde{V} \\
	0&\mbox{else.}
	\end{cases}
	\end{align*}
	Note that $\overline{f}\in\coCh{p}{M}{A}$ since $p\inv(U\cap V)$ is an equivariant subspace and hence $\overline{f}$ is equivariant.
	If we consider $\overline{f}$ as an element in $\coCh{p}{U}{A}$ then the cohomology class $\delta(\phi)$ is represented by the cochain $h\in\coCh{p+1}{U\cup V}{A}$ which is given by
	\begin{align*}
	h(\sigma)=\begin{cases}
	\delta(\overline{f})(\sigma)&\mbox{if }\operatorname{Im}\sigma \subseteq \widetilde{U}\\
	0&\mbox{else.}
	\end{cases}
	\end{align*}
	Since $\phi$ is a cocycle we have $\delta(f)(\sigma)=0$ for $\sigma\in\Ch{*}{U\cap V}{A}$. It follows in particular  that if $\sigma$ is a simplex whose image is completely contained in $\widetilde{V}$, then $h(\sigma)=0$.
We can represent our orientation class $\theta\in H_n(M,M\setminus (K\cup L))$ by a cycle
\begin{align*}
 a=b+c+d+e\quad\mbox{with}\quad& b\in\Ch{n}{U\cap V}{\Ztriv}\quad c\in\Ch{n}{U\setminus (U\cap L) }{\Ztriv}
 \\  &d\in\Ch{n}{V\setminus (V\cap K)}{\Ztriv},\quad e\in\Ch{n}{M\setminus(K\cup L)}{\Ztriv}.
\end{align*}
Obviously $e$ does not play a role since we kill it in the end. With these representatives one computes that
	$\delta(\phi)(\theta)$ is represented by \[h\acap (b+c+d)= \delta(\overline{f})\cap c + h\acap d + \delta(f) \cap b = \delta(\overline{f})\cap c.\] The pairing of $h$ with $b$ is zero since $f$ was a cocycle in $\coCh{*}{U\cap V}{A}$ and the pairing of $h$ with $d$ is zero since $d$ consists of simplices with image in $\widetilde{V}$.
	
	The lower sequence comes from the short exact sequence:
	\[\begin{tikzcd}[column sep=tiny]
	0\arrow[r]&C_*(M,M\setminus (K\cup L);A)\arrow[r]&\begin{array}{c}C_*(M,M\setminus K;A)\\ \oplus\\C_*(M,M\setminus L);A)\end{array}\arrow[r]&C_*(M,M\setminus (K\cap L);A)\arrow[r]& 0.
	\end{tikzcd}
	\]
	Before we compute the other side $\partial (\phi\acap \ex_{U\cap V} (\theta))$ we want to recall that the cap product is natural on the chain complex level i.e.\ the following diagram commutes:
	\[ \begin{tikzcd}
	\coCh{p}{U}{A}\times\Ch{n}{U,U\setminus K}{\Ztriv}\arrow[r]\arrow[d,xshift=6ex] & \Ch{*}{U,U\setminus K}{A}\arrow[d] \\
	\coCh{p}{M}{A}\times\Ch{n}{M,M\setminus K}{\Ztriv}\arrow[r]\arrow[u,xshift=-10ex] &\Ch{*}{M,M\setminus K}{A}.	
	\end{tikzcd}
	\]
	Therefore we use the representatives from above. To construct the boundary map $\partial$, we take as the preimage of $\overline{f}\acap a\in\Ch{*}{M,M\setminus (K\cap L)}{A}$ the element $(\overline{f}\acap a,0)\in\Ch{*}{M,M\setminus K}{A}\oplus\Ch{*}{M,M\setminus L}{A}$. Then one computes in $\Ch{*}{M,M\setminus K}{A}$
	\begin{align*}
	\partial(\overline{f} \acap a)&=(-1)^{p+1}\cdot \delta(\overline{f})\acap a \pm f\acap \partial a \quad (\mbox{by Lemma~\ref{lem:capboundary}}) \\
									&=(-1)^{p+1}\cdot\delta(\overline{f})\acap a \quad(\mbox{since $f\acap\partial a\in\Ch{n-p-1}{M\setminus (K\cup L)}{A}$}) \\
									&=(-1)^{p+1}\cdot\delta(\overline{f})\acap b+c+d+e \\
									&=(-1)^{p+1}\cdot\delta(\overline{f}) \acap (c+d) \quad\mbox{(same reason as above)} \\
								    &=(-1)^{p+1}\cdot\delta(\overline{f})\acap c \quad\mbox{(since $d\in \Ch{n-p}{V\setminus (K\cap V)}{A}$)} 	
	\end{align*}
Therefore the element $\partial (\phi\acap \ex_{U\cap V} (\theta))$ is also represented by $(-1)^{p+1}\cdot \delta(\overline{f})\acap c\in\Ch{n-p-1}{M,M\setminus (K\cup L)}{A}$.
\end{proof}

\begin{proof}[Proof of Theorem~\ref{thm: pdclosed}]  Let $P_M(K)$ be the statement that the map $\PD_K$ is an isomorphism.
	Then it is sufficient to verify condition (1), (2) and (3) of the bootstrap lemma.
	We start by verifying (1). In the case that $K=\SET{x}$ is just a point we have already seen in Lemma~\ref{lem:explicitcalc} that the statement holds true.
	For a general compact $K$ with the property of (1) the statement follows from the following commutative diagram:
	\[\begin{tikzcd}
	\Cech{p}{K}{A}\arrow[d,"\simeq"]\arrow[r] &H_{n-p}(M,M\setminus K; A)\arrow[d,"\simeq"]\\
	\Cech{p}{\SET{x}}{A} \arrow[r,"\simeq"] & H_{n-p}(M,M\setminus \SET{x}; A),
	\end{tikzcd}\]
	where the vertical maps are isomorphisms by the homotopy invariance and the bottom row by the observation above. Hence condition (1) is verified.
	
	Condition (2) follows immediately from the five lemma and Lemma~\ref{lem:mvscommutative}.

Let $K_i$ be a sequence of compact subsets such that $P_M(K_i)$ holds for all $i\in\N$. We set $K=\bigcap_{i\in\N} K_i$. It is an exercise in point set topology of manifolds that each $K_i$ has a fundamental system  $U_{i,j}$ of open neighbourhoods. Fundamental system means that $U_{i,j}\subseteq U_{i,k}$ if $j<k$ and that for each open set $U$ containing $K_i$ there is a $j$ such that $U_{i,j}\subseteq U$.
Another exercise in the point set topology of manifolds shows that one can construct these sets such that $U_{1,j}\supseteq U_{2,j} \supseteq U_{3,j}\supseteq \ldots$ for all $j\in \N$. Then $U_{i,j}$ is a fundamental system of open neighbourhoods of $K$ with the order $(i,j)\leq(k,l) \Leftrightarrow i\leq k\ \wedge j\leq l$.
One has the natural isomorphism \cite[Appendix D5]{Br93}:
\[ \varinjlim_{i\in\N}\Cech{p}{K_i}{A} =  \varinjlim_{i\in\N} \varinjlim_{j\in\N} H^{p}(U_{i,j};A) \xrightarrow{\ \simeq\ }  \varinjlim_{i,j\in\N}H^{p}(U_{i,j};A)\cong\Cech{p}{K}{A}.  \]
Hence the theorem follows from the commutativity of the diagram:
\[
\begin{tikzcd}
 {\varinjlim}_{i\in\N}\Cech{p}{K_i}{A} \arrow[r]\arrow[d] & {\varinjlim}_{i\in\N} H_{n-p}(M,M\setminus K_i; A) \arrow[d] \\
\Cech{p}{K}{A} \arrow[r]& H_{n-p} (M,M\setminus K; A).
\end{tikzcd}
\]
\end{proof}

\section{Proof of the Twisted Poincar\'e Duality Theorem}
For the reader's convenience we recall the main theorem from the last section.  Here, as above, $R$ is a ring and $A$ is an $(R,\Z[\pi])$-bimodule.\\

\noindent \textbf{Theorem~\ref{thm: pdclosed}.} \emph{Let $M$ be a compact, oriented, connected $n$-dimensional manifold.
The map $\map{\PD_K}{\Cech{p}{K}{A}}{H_{n-p}(M,M\setminus K;A)}$ is an isomorphism of left $R$-modules for all compact subsets $K\subseteq M$ and al $p\in \N_0$.}\\

Furthermore, we also recall that we need to prove the following theorem.\\

\noindent \textbf{Theorem~\ref{thm:poincareduality}.} \emph{Let $M$ a compact, oriented, connected $n$-dimensional manifold. Let $S$ and $T$ be codimension 0 compact submanifolds of $\partial M$ such that $\partial S=\partial T=S\cap T$ and $\partial M=S\cup T$.
Let $[M]\in H_n(M,\partial M;\Z)$ be the fundamental class of $M$. The map
\[-\acap [M]\colon H^k(M,S;{A})\,\,\rightarrow \,\,H_{n-k}(M,T;A)\]
 defined by Lemma~\ref{lem:boundary-cap} is an isomorphism of left $R$-modules.}\\

In the remainder of this appendix we will explain how to deduce
Theorem~\ref{thm:poincareduality} from Theorem~\ref{thm: pdclosed}.
First note that if $M$ is a closed manifold, then we can set
$K=M$ in Theorem~\ref{thm: pdclosed}.
Evidently we have $\Cech{p}{M}{A}=H(M;A)$.
Thus  we obtain precisely the statement of  Theorem~\ref{thm:poincareduality} in the closed case.

Next let $M$ be a compact oriented manifold with nonempty boundary.
First we consider the case $R=\emptyset$ and $S=\partial M$.
By the Collar Neighbourhood Theorem~\ref{thm:collar} there exists a collar $\partial M\times [0,2]\subseteq M$ of the boundary such that $\partial M=\partial M\times\SET{0}$. We obtain the following chain of isomorphisms:
\begin{align*}
H^p (M; A) &\cong H^p(M\setminus (\partial M\times[0,1)) ;A) \quad\mbox{(homotopy)} \\
&\cong \Cech{p}{M\setminus (\partial M\times[0,1))}{A}\quad\mbox{(follows from considering the open}\\
&\hspace{6.4cm}\mbox{neighbourhoods $M\sms (\partial M\times [0,1-\frac{1}{n}])$)} \\
& \cong H_{n-p}(M\setminus\partial M,\partial M\times (0,1); A) \quad\mbox{(duality $K=M\setminus (\partial M\times[0,1)))$}\\
&\cong H_{n-p}(M,\partial M\times [0,1);A)\quad\mbox{(excision $U=\partial M$)}\\
&\cong H_{n-p}(M,\partial M; A),
\end{align*}
It follows from the definition of the dualising map and naturality of cap product that these isomorphisms are given by capping with a generator $[M]\in H_n(M,\partial M;\Ztriv)\cong H_n(M,\partial M; \Z)$ as in the classical case.

The proof of the general case of Theorem~\ref{thm:poincareduality} relies on the following lemma.

\begin{lemma}\label{lem:pd-commutative-diagram}
Let $M$ be a compact, oriented, connected $n$-dimensional manifold. Let $R$ and $S$ be compact codimension 0 submanifolds of $\partial M$ such that $\partial R=\partial S=R\cap S$ and $\partial M=R\cup S$. For each $p\in \N_0$ the following diagram commutes up to a sign:	\[
\begin{tikzcd}[column sep=tiny,font=\footnotesize]
\ldots \arrow[r]&H^{p}(M,\partial M;A)\arrow[r]\arrow[dd,"\acap {[M]}"]&H^{p}(M,R;A)\arrow[r]\arrow[dd,"\acap {\left[M\right]}"]& H^{p}(\partial M,R;A)  \arrow[r,"\delta"]\arrow[d,"{\text{\emph{Theorem}}~\ref{thm:basicproperties} (2)}"]&H^{p+1}(M,\partial M;A)\arrow[r]\arrow[dd,"\acap {[M]}"]&\ldots\\ &&& H_{n-p-1}(S,\partial S;A) \arrow[d,"\acap {\left[R\right]}"]  \\
\ldots\arrow[r]& H_{n-p}(M;A)\arrow[r]& H_{n-p}(M,S;A) \arrow[r,"\partial"]& H_{n-p-1}( S;A)\arrow[r]&H_{n-p-1}(M;A)\arrow[r]&\ldots
\end{tikzcd}
\]
\end{lemma}

\begin{proof}
The commutativity is a more or less direct consequence of Lemma~\ref{lem:capboundary} and the observation that $\partial_*[M]=[\partial M]$. More precisely, the proof in the untwisted case is given in detail in  \cite[p.~2892]{Fr23}. The proof in the twisted case is basically the same.
\end{proof}

The proof of the general case of Theorem~\ref{thm:poincareduality} follows from
the previously discussed Poincar\'e Duality isomorphisms
$\acap [M] \colon H^p(M;A)\xrightarrow{\cong} H_{n-p}(M,\partial M;A)$,
$\acap [R] \colon H^{p+1}(R,A)\xrightarrow{\cong} H_{n-p-1}(R,\partial R;A)$
together with Lemma~\ref{lem:pd-commutative-diagram} and  the five lemma.

\end{appendix}


\bibliographystyle{alpha}
\renewcommand{\MR}[1]{}
\bibliography{BibTopMfd}

\end{document}

%% file: submanifold-subcharts.pstex_t
\begin{picture}(0,0)%
\includegraphics{submanifold-subcharts.eps}%
\end{picture}%
%
%
\setlength{\unitlength}{1184sp}%
\begingroup\makeatletter\ifx\SetFigFont\undefined%
\gdef\SetFigFont#1#2#3#4#5{%
  \reset@font\fontsize{#1}{#2pt}%
  \fontfamily{#3}\fontseries{#4}\fontshape{#5}%
  \selectfont}%
\fi\endgroup%
\begin{picture}(25270,8130)(1507,-10430)
\put(1651,-3976){\makebox(0,0)[lb]{\smash{{\SetFigFont{12}{14.4}{\rmdefault}{\mddefault}{\updefault}{\color[rgb]{1,0,0}topological manifold $X$}%
}}}}
\put(15115,-7459){\makebox(0,0)[lb]{\smash{{\SetFigFont{12}{14.4}{\rmdefault}{\mddefault}{\updefault}{\color[rgb]{.5,.17,0}\begin{tabular}{c}submanifold chart\\ of  type $(\beta)$\end{tabular}}%
}}}}
\put(4441,-2911){\makebox(0,0)[lb]{\smash{{\SetFigFont{12}{14.4}{\rmdefault}{\mddefault}{\updefault}{\color[rgb]{0,0,1}submanifold $N$}%
}}}}
\put(6467,-9545){\makebox(0,0)[lb]{\smash{{\SetFigFont{12}{14.4}{\rmdefault}{\mddefault}{\updefault}{\color[rgb]{0,0,0}$x_1$}%
}}}}
\put(13253,-8281){\makebox(0,0)[lb]{\smash{{\SetFigFont{12}{14.4}{\rmdefault}{\mddefault}{\updefault}{\color[rgb]{0,0,0}$x_2$}%
}}}}
\put(2996,-8149){\makebox(0,0)[lb]{\smash{{\SetFigFont{12}{14.4}{\rmdefault}{\mddefault}{\updefault}{\color[rgb]{0,0,0}$x_2$}%
}}}}
\put(16846,-10215){\makebox(0,0)[lb]{\smash{{\SetFigFont{12}{14.4}{\rmdefault}{\mddefault}{\updefault}{\color[rgb]{0,0,0}$x_1$}%
}}}}
\put(17341,-2851){\makebox(0,0)[lb]{\smash{{\SetFigFont{12}{14.4}{\rmdefault}{\mddefault}{\updefault}{\color[rgb]{0,.56,0}$\partial X$}%
}}}}
\put(22630,-8074){\makebox(0,0)[lb]{\smash{{\SetFigFont{12}{14.4}{\rmdefault}{\mddefault}{\updefault}{\color[rgb]{0,0,0}$x_2$}%
}}}}
\put(26180,-9775){\makebox(0,0)[lb]{\smash{{\SetFigFont{12}{14.4}{\rmdefault}{\mddefault}{\updefault}{\color[rgb]{0,0,0}$x_1$}%
}}}}
\put(19231,-3121){\makebox(0,0)[lb]{\smash{{\SetFigFont{12}{14.4}{\rmdefault}{\mddefault}{\updefault}{\color[rgb]{0,.69,.69}\begin{tabular}{c}submanifold chart\\of type $(\gamma)$\end{tabular}}%
}}}}
\put(20881,-5971){\makebox(0,0)[lb]{\smash{{\SetFigFont{12}{14.4}{\rmdefault}{\mddefault}{\updefault}{\color[rgb]{0,.69,.69}\begin{tabular}{c}manifold\\[-0.05cm] chart of\\[-0.05cm] type (i)\end{tabular}}%
}}}}
\put(1522,-6353){\makebox(0,0)[lb]{\smash{{\SetFigFont{12}{14.4}{\rmdefault}{\mddefault}{\updefault}{\color[rgb]{.56,0,.56}\begin{tabular}{c}manifold\\[-0.05cm] chart of\\[-0.05cm] type (i)\end{tabular}}%
}}}}
\put(5425,-7069){\makebox(0,0)[lb]{\smash{{\SetFigFont{12}{14.4}{\rmdefault}{\mddefault}{\updefault}{\color[rgb]{.56,0,.56}\begin{tabular}{c}submanifold\\[-0.05cm] chart of \\[-0.05cm] type $(\alpha)$\end{tabular}}%
}}}}
\put(9916,-7111){\makebox(0,0)[lb]{\smash{{\SetFigFont{12}{14.4}{\rmdefault}{\mddefault}{\updefault}{\color[rgb]{.5,.17,0}\begin{tabular}{c}manifold chart\\ of type (ii)\end{tabular}}%
}}}}
\end{picture}%

%% file: collar-boundary.pstex_t
\begin{picture}(0,0)%
\includegraphics{collar-boundary.eps}%
\end{picture}%
%
%
\setlength{\unitlength}{1184sp}%
\begingroup\makeatletter\ifx\SetFigFont\undefined%
\gdef\SetFigFont#1#2#3#4#5{%
  \reset@font\fontsize{#1}{#2pt}%
  \fontfamily{#3}\fontseries{#4}\fontshape{#5}%
  \selectfont}%
\fi\endgroup%
\begin{picture}(22043,5094)(1365,-5706)
\put(1743,-2245){\makebox(0,0)[lb]{\smash{{\SetFigFont{12}{14.4}{\rmdefault}{\mddefault}{\updefault}{\color[rgb]{0,.56,0}$\partial M\times [0,2]$}%
}}}}
\put(4706,-1248){\makebox(0,0)[lb]{\smash{{\SetFigFont{12}{14.4}{\rmdefault}{\mddefault}{\updefault}{\color[rgb]{0,0,1}$M$}%
}}}}
\put(1380,-4369){\makebox(0,0)[lb]{\smash{{\SetFigFont{12}{14.4}{\rmdefault}{\mddefault}{\updefault}{\color[rgb]{0,.69,.69}$\partial X\times [0,2]$}%
}}}}
\put(6710,-5475){\makebox(0,0)[lb]{\smash{{\SetFigFont{12}{14.4}{\rmdefault}{\mddefault}{\updefault}{\color[rgb]{1,0,0}$X\setminus (\partial X\times [0,2])$}%
}}}}
\put(15151,-5488){\makebox(0,0)[lb]{\smash{{\SetFigFont{12}{14.4}{\rmdefault}{\mddefault}{\updefault}{\color[rgb]{1,0,0}$h_1(X\setminus (\partial X\times [0,2]))$}%
}}}}
\end{picture}%

%% file: structures-on-manifolds.pstex_t
\begin{picture}(0,0)%
\includegraphics{structures-on-manifolds.eps}%
\end{picture}%
%
%
\setlength{\unitlength}{1066sp}%
\begingroup\makeatletter\ifx\SetFigFont\undefined%
\gdef\SetFigFont#1#2#3#4#5{%
  \reset@font\fontsize{#1}{#2pt}%
  \fontfamily{#3}\fontseries{#4}\fontshape{#5}%
  \selectfont}%
\fi\endgroup%
\begin{picture}(26033,10093)(181,-12950)
\put(17206,-12436){\makebox(0,0)[lb]{\smash{{\SetFigFont{11}{13.2}{\rmdefault}{\mddefault}{\updefault}{\color[rgb]{0,.56,0}ANR}%
}}}}
\put(196,-7876){\makebox(0,0)[lb]{\smash{{\SetFigFont{11}{13.2}{\rmdefault}{\mddefault}{\updefault}{\color[rgb]{0,0,0}\begin{tabular}{c}finite simplicial\\ structure\end{tabular}}%
}}}}
\put(196,-12361){\makebox(0,0)[lb]{\smash{{\SetFigFont{11}{13.2}{\rmdefault}{\mddefault}{\updefault}{\color[rgb]{0,0,0}\begin{tabular}{c}homotopy type of\\a finite CW-complex\end{tabular}}%
}}}}
\put(646,-10351){\makebox(0,0)[lb]{\smash{{\SetFigFont{11}{13.2}{\rmdefault}{\mddefault}{\updefault}{\color[rgb]{0,0,0}finite CW-structure}%
}}}}
\put(8341,-9511){\makebox(0,0)[lb]{\smash{{\SetFigFont{11}{13.2}{\rmdefault}{\mddefault}{\updefault}{\color[rgb]{0,0,1}if $X$ closed}%
}}}}
\put(9361,-7996){\makebox(0,0)[lb]{\smash{{\SetFigFont{11}{13.2}{\rmdefault}{\mddefault}{\updefault}{\color[rgb]{0,0,0}manifold}%
}}}}
\put(8477,-4149){\makebox(0,0)[lb]{\smash{{\SetFigFont{11}{13.2}{\rmdefault}{\mddefault}{\updefault}{\color[rgb]{0,0,0}smooth structure}%
}}}}
\put(9047,-6039){\makebox(0,0)[lb]{\smash{{\SetFigFont{11}{13.2}{\rmdefault}{\mddefault}{\updefault}{\color[rgb]{0,0,0}PL-structure}%
}}}}
\put(10665,-5229){\makebox(0,0)[lb]{\smash{{\SetFigFont{12}{14.4}{\rmdefault}{\mddefault}{\updefault}{\color[rgb]{0,0,1}$\operatorname{dim}\leq 7$}%
}}}}
\put(16856,-4125){\makebox(0,0)[lb]{\smash{{\SetFigFont{11}{13.2}{\rmdefault}{\mddefault}{\updefault}{\color[rgb]{0,0,0}\begin{tabular}{c}smooth handle\\ structure\end{tabular}}%
}}}}
\put(16899,-6119){\makebox(0,0)[lb]{\smash{{\SetFigFont{11}{13.2}{\rmdefault}{\mddefault}{\updefault}{\color[rgb]{0,0,0}PL-handle structure}%
}}}}
\put(16680,-8302){\makebox(0,0)[lb]{\smash{{\SetFigFont{11}{13.2}{\rmdefault}{\mddefault}{\updefault}{\color[rgb]{0,0,0}\begin{tabular}{c}topological handle\\ structure\end{tabular}}%
}}}}
\put(10487,-7016){\makebox(0,0)[lb]{\smash{{\SetFigFont{12}{14.4}{\rmdefault}{\mddefault}{\updefault}{\color[rgb]{0,0,1}$\operatorname{dim}\leq 3$}%
}}}}
\put(23919,-6281){\makebox(0,0)[lb]{\smash{{\SetFigFont{11}{13.2}{\rmdefault}{\mddefault}{\updefault}{\color[rgb]{0,0,1}if $\operatorname{dim}=4$}%
}}}}
\put(13596,-7629){\makebox(0,0)[lb]{\smash{{\SetFigFont{11}{13.2}{\rmdefault}{\mddefault}{\updefault}{\color[rgb]{0,0,1}if $\operatorname{dim}>4$}%
}}}}
\put(5686,-7456){\makebox(0,0)[lb]{\smash{{\SetFigFont{12}{14.4}{\rmdefault}{\mddefault}{\updefault}{\color[rgb]{0,0,1}if $\operatorname{dim}\leq 4$}%
}}}}
\end{picture}%

%% file: connected-sum-well-defined.pstex_t
\begin{picture}(0,0)%
\includegraphics{connected-sum-well-defined.eps}%
\end{picture}%
%
%
\setlength{\unitlength}{1184sp}%
\begingroup\makeatletter\ifx\SetFigFont\undefined%
\gdef\SetFigFont#1#2#3#4#5{%
  \reset@font\fontsize{#1}{#2pt}%
  \fontfamily{#3}\fontseries{#4}\fontshape{#5}%
  \selectfont}%
\fi\endgroup%
\begin{picture}(22906,5434)(403,-6675)
\put(4563,-6384){\makebox(0,0)[lb]{\smash{{\SetFigFont{12}{14.4}{\rmdefault}{\mddefault}{\updefault}{\color[rgb]{0,0,0}$D^n$}%
}}}}
\put(10947,-2607){\makebox(0,0)[lb]{\smash{{\SetFigFont{12}{14.4}{\rmdefault}{\mddefault}{\updefault}{\color[rgb]{0,0,0}$N$}%
}}}}
\put(22067,-4751){\makebox(0,0)[lb]{\smash{{\SetFigFont{12}{14.4}{\rmdefault}{\mddefault}{\updefault}{\color[rgb]{1,0,1}$\psi$}%
}}}}
\put(12851,-2100){\makebox(0,0)[lb]{\smash{{\SetFigFont{12}{14.4}{\rmdefault}{\mddefault}{\updefault}{\color[rgb]{0,.56,.56}$X_1$}%
}}}}
\put(9634,-4781){\makebox(0,0)[lb]{\smash{{\SetFigFont{12}{14.4}{\rmdefault}{\mddefault}{\updefault}{\color[rgb]{1,0,1}$\Psi$}%
}}}}
\put(418,-2130){\makebox(0,0)[lb]{\smash{{\SetFigFont{12}{14.4}{\rmdefault}{\mddefault}{\updefault}{\color[rgb]{0,0,0}$M$}%
}}}}
\put(9936,-6456){\makebox(0,0)[lb]{\smash{{\SetFigFont{12}{14.4}{\rmdefault}{\mddefault}{\updefault}{\color[rgb]{0,0,0}$D^n$}%
}}}}
\put(23250,-3916){\makebox(0,0)[lb]{\smash{{\SetFigFont{12}{14.4}{\rmdefault}{\mddefault}{\updefault}{\color[rgb]{1,0,1}$C$}%
}}}}
\put(23294,-2731){\makebox(0,0)[lb]{\smash{{\SetFigFont{12}{14.4}{\rmdefault}{\mddefault}{\updefault}{\color[rgb]{.5,.17,0}$Y$}%
}}}}
\put(5133,-5283){\makebox(0,0)[lb]{\smash{{\SetFigFont{12}{14.4}{\rmdefault}{\mddefault}{\updefault}{\color[rgb]{.56,0,.56}$\Phi_2$}%
}}}}
\put(1826,-1801){\makebox(0,0)[lb]{\smash{{\SetFigFont{12}{14.4}{\rmdefault}{\mddefault}{\updefault}{\color[rgb]{0,.56,0}$D_1$}%
}}}}
\put(916,-5251){\makebox(0,0)[lb]{\smash{{\SetFigFont{12}{14.4}{\rmdefault}{\mddefault}{\updefault}{\color[rgb]{0,.56,0}$\Phi_1$}%
}}}}
\put(13680,-5236){\makebox(0,0)[lb]{\smash{{\SetFigFont{12}{14.4}{\rmdefault}{\mddefault}{\updefault}{\color[rgb]{0,.56,0}$\varphi_1$}%
}}}}
\put(5216,-1914){\makebox(0,0)[lb]{\smash{{\SetFigFont{12}{14.4}{\rmdefault}{\mddefault}{\updefault}{\color[rgb]{.56,0,.56}$D_2$}%
}}}}
\end{picture}%

%% file: simple-homotopy-type-closed-mfd.pstex_t
\begin{picture}(0,0)%
\includegraphics{simple-homotopy-type-closed-mfd.eps}%
\end{picture}%
%
%
\setlength{\unitlength}{1381sp}%
\begingroup\makeatletter\ifx\SetFigFont\undefined%
\gdef\SetFigFont#1#2#3#4#5{%
  \reset@font\fontsize{#1}{#2pt}%
  \fontfamily{#3}\fontseries{#4}\fontshape{#5}%
  \selectfont}%
\fi\endgroup%
\begin{picture}(12323,3044)(1157,-4570)
\put(6419,-2041){\makebox(0,0)[lb]{\smash{{\SetFigFont{12}{14.4}{\rmdefault}{\mddefault}{\updefault}{\color[rgb]{0,0,1}$B(M)$}%
}}}}
\put(4330,-2167){\makebox(0,0)[lb]{\smash{{\SetFigFont{12}{14.4}{\rmdefault}{\mddefault}{\updefault}{\color[rgb]{0.565,0.565,0.565}$\R^n$}%
}}}}
\put(9153,-2022){\makebox(0,0)[lb]{\smash{{\SetFigFont{12}{14.4}{\rmdefault}{\mddefault}{\updefault}{\color[rgb]{0,.56,0}triangulation of $B(M)$}%
}}}}
\put(1172,-3256){\makebox(0,0)[lb]{\smash{{\SetFigFont{12}{14.4}{\rmdefault}{\mddefault}{\updefault}{\color[rgb]{1,0,0}$M$}%
}}}}
\end{picture}%